\documentclass[11pt]{article}
\usepackage{amsmath,amssymb,amsthm,mathtools,bm}
\usepackage{bbm}
\usepackage{graphicx}
\usepackage{tikz}
\usepackage{pgfplots}
\pgfplotsset{compat=1.18}
\usepackage{mathtools}
\usepackage{enumitem}
\usepackage{subcaption}
\usepackage[colorlinks=true, linkcolor=blue]{hyperref}

\usepackage[margin=1in]{geometry}
\usepackage{parskip}

\usepackage[T1]{fontenc}

\usepackage{setspace}

\usepackage{titling}

\usepackage{titlesec}

\titlespacing*{\section}{0pt}{8pt plus 2pt minus 2pt}{6pt}
\titlespacing*{\subsection}{0pt}{6pt plus 2pt minus 2pt}{4pt}
\titlespacing*{\subsubsection}{0pt}{4pt plus 1pt minus 1pt}{3pt}

\usepackage{longtable}
\usepackage{booktabs}
\usepackage{appendix}

\theoremstyle{definition}
\newtheorem{assumption}{Assumption}[section]
\newtheorem{definition}[assumption]{Definition}
\newtheorem{remark}[assumption]{Remark}

\theoremstyle{plain}
\newtheorem{proposition}[assumption]{Proposition}
\newtheorem{lemma}[assumption]{Lemma}
\newtheorem{theorem}[assumption]{Theorem}
\newtheorem{corollary}[assumption]{Corollary}

\newcommand{\R}{\mathbb{R}}
\newcommand{\pr}{\mathbb{P}}
\newcommand{\E}{\mathbb{E}}
\newcommand{\N}{\mathbb{N}}
\newcommand{\C}{\mathbb{C}}
\newcommand{\D}{\mathbb{D}}

\newcommand{\qc}{\textrm{QC}}

\newcommand{\cF}{{\cal F}}
\newcommand{\cB}{{\cal B}}
\newcommand{\cH}{{\cal H}}
\newcommand{\cV}{{\cal V}}
\newcommand{\cA}{{\cal A}}
\newcommand{\cP}{{\cal P}}
\newcommand{\cI}{{\cal I}}
\newcommand{\cL}{{\cal L}}

\newcommand{\tb}{\tilde{b}}

\newcommand{\be}{\bm{e}}

\newcommand{\bZero}{\bm{0}}

\newcommand{\bOne}{\bm{1}}
\newcommand{\mbOne}{\mathbbm{1}}

\newcommand{\Z}{\mathbb{Z}}

\newcommand{\veps}{\varepsilon}

\newcommand{\tr}{\text{Tr}}
\newcommand{\Hess}{\text{Hessian}}
\newcommand{\grad}{\nabla_x}

\newcommand{\vertiii}[1]{{\left\vert\kern-0.25ex\left\vert\kern-0.25ex\left\vert #1 
    \right\vert\kern-0.25ex\right\vert\kern-0.25ex\right\vert}}

\title{Multilevel Picard scheme for solving high-dimensional drift control problems with state constraints}
\author{Yuan Zhong}
\date{}

\begin{document}
\onehalfspacing
\maketitle

\begin{abstract}
Motivated by applications to the dynamic control of queueing networks, 
we develop a simulation-based scheme, the so-called multilevel Picard (MLP) approximation, 
for solving high-dimensional drift control problems whose states 
are constrained to stay within the nonnegative orthant, over a finite time horizon. 
We prove that under suitable conditions, the MLP approximation overcomes 
the curse of dimensionality in the following sense: To approximate 
the value function and its gradient evaluated at a given time and state to  
within a prescribed accuracy $\veps$, the computational complexity 
grows at most polynomially in the problem dimension $d$ and $1/\veps$. 
To illustrate the effectiveness of the scheme, we carry out numerical experiments for 
a class of test problems that are related to the dynamic scheduling problem of parallel server systems in heavy traffic, 
and demonstrate that the scheme is computationally feasible up to dimension at least $20$.
\end{abstract}

\section{Introduction}\label{sec:introduction}
Drift controls are stochastic control problems where the system state can be adjusted at finite rates. 
They arise in a wide range of application domains, including manufacturing \cite{RubinoAta2009, OrmeciMatogluVandeVate2011}, communications \cite{AtaHarrisonShepp2005}, 
service systems \cite{Atar2005, atar2004scheduling, kim2018dynamic}, portfolio management \cite{DaiZhong2010}, and economics \cite{BarIlanMarionPerry2007}. 
In many of these problems, state constraints are a natural and critical feature; 
for example, in service and manufacturing systems, job or customer counts are nonnegative, 
and in portfolio management, the investor's portfolio is commonly required to stay within a solvency region. 

In this paper, we study problems in which states are constrained to stay within the nonnegative orthant 
by a ``reflecting'' mechanism at the boundary 
(henceforth, we refer to these as {\em drift controls with reflections}); See Section~\ref{sec:problem_description} for details.
Our primary motivation comes from the problem of dynamically controlling queueing networks in heavy traffic; 
see e.g., \cite{AtaHarrisonSi2024,williams2016stochastic} and the references therein for recent overviews of 
the extensive history of heavy-traffic analysis of queueing networks, 
and its connections to and synergy with drift controls with reflections.

Except in a few special cases, drift control problems (with reflections) are difficult to solve analytically, 
necessitating numerical solutions. 
A standard approach is to consider the associated Hamilton-Jacobi-Bellman (HJB) equations, 
which are partial differential equations (PDEs) that characterize the value function of the control problem.
Classical grid-based solvers for these PDEs, such as finite-element and finite-difference methods \cite{thomas2013numerical,axelsson2001finite}, 
suffer from the curse of dimensionality and quickly become computationally intractable as the problem dimension grows; cf. \cite{Bellman1957,Judd1998}. 
However, high-dimensional instances of these control problems arise in important practical applications \cite{HanJentzenE2018PNAS,Atar2005,AtaHarrisonSi2024}, 
rendering the use of grid-based methods infeasible.
Another popular approach is the Markov chain approximation method (see, e.g., \cite{KushnerDupuis2001}), 
which solves carefully designed controlled Markov chains that approximate the original control problem,
but it also suffers from the curse of dimensionality, since the method requires fine discretization of the state space, similar to grid-based methods.

In this paper, we develop, analyze and implement a multilevel Picard (MLP) method for solving high-dimensional drift control problems with reflections. 
MLP is a non-parametric, simulation-based framework that has been developed recently for solving broad classes of high-dimensional semilinear PDEs 
\cite{HutzenthalerJentzenKruse2019,HutzenthalerETAL2021,NeufeldNguyenWu2025,NeufeldWu2025,giles2019generalised,beck2020overcoming,becker2020numerical,hutzenthaler2019multilevel,hutzenthaler2021multilevel,hutzenthaler2022multilevel,hutzenthaler2022multilevel1,HJKN2020,HJKNW2020,NW2022}
 -- PDEs with nonlinearities only in lower-order terms -- 
and has been shown to provably overcome the curse of dimensionality under suitable conditions, 
in the sense that the computational complexity scales polynomially in the problem dimension and $1/\veps$, 
where $\veps$ is the prescribed accuracy. 
In contrast to existing literature, the presence of state constraints introduces considerable challenges to the analysis and implementation of the MLP method.
In particular, our approach relies heavily on so-called derivative processes of reflected diffusions \cite{LipshutzRamanan2018,LipshutzRamanan2019b}, 
which capture the sensitivity of the state process to its defining parameters, such as changes in the initial state. 
These processes, unlike those of unconstrained diffusions prevalent in prior works, 
feature discontinuous sample paths and a discontinuous dependence on the initial state at the boundary. 
To address these issues, we introduce new analytical and simulation tools and make the following contributions:
\begin{itemize}
  \item[(i)] Establish a Bismut--Elworthy--Li (BEL) formula for reflected diffusions (Theorem \ref{thm:bel}); 
  \item[(ii)] Prove continuous dependence on the initial condition for the derivative processes, in the interior of the state space (Proposition \ref{thm:continuity-DZ});
  \item[(iii)] Characterize the value function and its gradient for drift control problems with reflections as the unique fixed point of a contractive functional operator (Theorem \ref{thm:fixed-point-tildeV});
  \item[(iv)] Develop an MLP estimator for the value function and gradient at a given time--state pair, with computational complexity growing polynomially in the problem dimension and in $1/\veps$, for prescribed error $\veps$ (Section \ref{sec:mlp}); and
  \item[(v)] Carry out numerical experiments that demonstrate computational feasibility and competitive empirical performance of the MLP scheme in dimension up to $20$ (Section \ref{sec:numerical}).
\end{itemize}

Let us also briefly remark that besides MLP, other competitive approaches, 
including several recently developed numerical methods that are deep-learning based, such as the deep backward stochastic differential equation (BSDE) \cite{HanJentzenE2018PNAS} and deep splitting methods \cite{BeckJentzen2019}, 
have shown strong empirical performances in solving high-dimensional PDEs \cite{HanJentzenE2018PNAS,BeckJentzen2019,HurePhamWarin2020,WeinanHanJentzen2021,AtaHarrisonSi2024,ata2025analysis,AtaKasikaralar2023}. 
Different from the MLP scheme developed in this paper, these methods are parametric and often require extensive engineering and hyperparameter tuning in order to work well. 

\subsection{Notation}\label{ssec:notation}
We summarize the notation used throughout the paper.
Let $\Z$ denote the set of integers, $\N$ the set of positive integers, 
$\R$ the set of real numbers, $\R_+$ the set of nonnegative real numbers, 
and $\R_{++}$ the set of positive real numbers. 
For two real numbers $\alpha_1$ and $\alpha_2$, we use $\alpha_1 \vee \alpha_2$ to denote the maximum of $\alpha_1$ and $\alpha_2$.
For $d \in \N$, $\R^d$ denotes the $d$-dimensional Euclidean space, 
$\R_+^d := \{x = (x_i)_{i=1}^d \in \R^d: x_i \geq 0 ~~\forall i\}$ denotes the nonnegative orthant, 
and $\R_{++}^d := \{x = (x_i)_{i=1}^d \in \R^d: x_i > 0 ~~\forall i\}$ denotes the positive orthant. 
For a vector $x \in \R^d$, we use $x^{\top}$ to denote its transpose, 
and $\|x\|_1$, $\|x\|_2$ and $\|x\|_\infty$ to denote its $\ell_1$ norm, Euclidean norm and $\ell_\infty$ norm, respectively. 
For vectors $x, y \in \R^d$, we use both $x^\top y$ and $\langle x, y \rangle$ to denote the inner product of $x$ and $y$. 
The $d$-dimensional vector of all ones is denoted by $\bOne_d$, and when the context is clear, we drop the subscript and use the notation $\bOne$. 
The $j$th standard unit vector in $\R^d$ is denoted by $e_j$, with all components being zero but the $j$th component being one, $j=1, \cdots, d$, 
For a matrix $A \in \R^{d\times K}$, $d, K \in \N$, we use $\tr(A)$ to denote the trace of $A$, 
$A^{\top}$ to denote its transpose, and, with a slight abuse of notation,
$\|A\|_2$ and $\|A\|_\infty$ to denote the operator norms induced by the Euclidean norm $\|\cdot\|_2$ 
and by the $\ell_\infty$-norm $\|\cdot\|_\infty$ on $\R^K$, respectively 
(see e.g., Definition 5.6.1 in Chapter 5 of \cite{HornJohnson2012}).
The indicator function is denoted by $\mbOne\{\cdot\}$. 
A function $f: \R_+ \to \R^d$ is called a $d$-dimensional path, with the interpretation that $f(t)$ is the state of path $f$ at time $t$. 
The space of c\`adl\`ag (or RCLL, which stands for right-continuous with left-limits) paths is denoted by $\D(\R^d)$, 
and the space of continuous paths is denoted by $\C(\R^d)$. 
We use $\C_+(\R_+^d)$ to denote the set of paths $f \in \C(\R^d)$ with $f(0) \in \R_+^d$. 
We also use $\D_{\text{lim}}(\R_+^d)$ to denote the set of paths $f$ that have finite left limits at all $t>0$ 
and finite right limits at all $t\geq 0$, and $\D_{l,r}(\R_+^d) \subset \D_{\text{lim}}(\R_+^d)$ 
to denote the set of paths $f \in \D_{\text{lim}}(\R_+^d)$ that are either left-continuous or right-continuous at all $t>0$. 
For $f \in \D_{\text{lim}}(\R_+^d)$ and $T > 0$, we define the sup norm of $f$ on $[0,T]$ by
\begin{equation}\label{eq:sup-norm}
  \|f\|_T := \sup_{t\in [0,T]} \|f(t)\|_\infty.
\end{equation}
For a domain $U \subset \R^d$, $C^k(U)$ denotes the space of continuous functions $u : U \to \R$ that are $\ell$-times continuously differentiable. 
We also use the notation $C^{1,2}([0,T]\times U)$ to denote the space of continuous functions $u : [0,T] \times U \to \R$ that are continuously differentiable in $t$ and twice continuously differentiable in $x\in U$ 
($T = \infty$ is allowed here), 
and $C^{1,1}(\R_+^d \times \cA)$ to denote the space of continuous functions $u : \R_+^d \times \cA \to \R$ that are continuously differentiable in $x\in \R_+^d$ and $a\in \cA$, 
for a domain $\cA \subset \R^K$ for some $K \in \N$. 
For a real-valued function $V: \R_+ \times \R_+^d \to \R$, if exist, the state gradient and Hessian of $V$ with respect to $x\in \R_+^d$ are $D_x V$ and $D_{xx} V$, respectively. 
For a $\R^d$-valued function $\tb: \R^{d_1} \to \R^{d_2}$, we use $D\tb(x)$ to denote the Jacobian matrix of $\tb$ at $x$, 
i.e., 
\[
D\tb(x) = \left( \frac{\partial \tb_i(x)}{\partial x_j} \right)_{i=1,\cdots,d_2, j=1,\cdots,d_1}.
\]

\subsection{Solution Approach Overview}\label{ssec:overview}
We present an overview of our solution approach, highlighting key challenges and technical contributions, 
without going into all technical details. 
We start with an informal description of the control problem. Let $T>0$ be fixed. 
The state $Z(\cdot)$ evolves according to the following $d$-dimensional diffusion process with reflections:
\begin{equation}\label{eq:state}
Z(t) = Z(0) + \int_0^t b\bigl(Z(s), a(s)\bigr)\,ds + \sigma B(t) + RY(t), \quad t \in [0,T],
\end{equation}
with $Z(0)$ being the initial state, $\sigma B(\cdot)$ a $d$-dimensional Brownian motion with zero drift and covariance matrix $\sigma \sigma^{\top}$, 
$R$ the reflection matrix, and $Y(\cdot)$ the Skorokhod regulator that keeps $Z(\cdot)$ in the nonnegative orthant; 
see Sections~\ref{sec:problem_description} and \ref{ssec:skorokhod} for details.
Here, $a(\cdot)$ is the control process taking values in a compact action set $\cA$, 
which can adjust the drift coefficient $b(\cdot, a(\cdot))$, and in turn, the state process $Z(\cdot)$.
For $(t,x) \in [0,T] \times \R_+^d$, we wish to compute the value function $V(t,x)$, defined as the infimum of 
$J(t,x,a)$ over all non-anticipating control processes $a(\cdot)$, 
where $J(t,x,a)$ is the total expected discounted cost over the time horizon $[t,T]$ starting from $Z(t)=x$, under control $a(\cdot)$:
\begin{equation}\label{eq:cost}
J(x,a) := \mathbb{E} \left[ \int_t^T e^{-\beta (s-t)}\, c\bigl(Z(s), a(s)\bigr)\,ds + \int_t^T e^{-\beta (s-t)}\, \kappa^{\top} dY(s)
+ e^{-\beta (T-t)}\, \xi(Z(T)) \Big| Z(t)=x \right].
\end{equation}
Here, $\beta$ is the discount rate, and the total cost consists of three components: the running cost $c$, 
which depends on both the state and control; the penalty costs $\kappa_j$, incurred whenever the $j$th component of the state 
hits zero; and the terminal cost $\xi$. 

{\em A priori}, it is unclear whether the value function $V(t,x)$ is differentiable in $x \in \mathbb{R}^d_+$. 
In the sequel we will show that under suitable conditions, this is the case, but assuming sufficient smoothness, 
central to the development of a MLP scheme is a fixed-point formulation of the (state) gradient of the value function $D_x V$ 
(see Eqs. \eqref{eq:fixed-point} and \eqref{eq:F}), which we now formally derive. 

First, the following formal parabolic HJB equation for the value function $V$ is standard:
\begin{equation}\label{eq:hjb1}
V_t(t,x) + \frac{1}{2}\mathrm{Tr}\left(\sigma\sigma^{\top} D_{xx}V(t,x)\right) + \cH\left(x, D_x V(t,x)\right) - \beta V(t,x) = 0, 
\quad (t,x) \in [0,T) \times \mathbb{R}^d_+,
\end{equation}
with terminal condition
\begin{equation}\label{eq:hjb2}
V(T,x) = \xi(x), \quad x \in \mathbb{R}^d_+, 
\end{equation}
and oblique derivative boundary conditions 
\begin{equation}\label{eq:hjb3}
\langle R_j, D_x V (t,x) \rangle = -\kappa_j, \quad \text{if } x_j = 0, t<T.
\end{equation}
The {\em Hamiltonian} $\cH$ is given by  
\begin{equation}\label{eq:hamiltonian}
\cH(x,p) := \inf_{a \in \mathcal{A}} \left\{ (b(x,a)^{\top} p + c(x,a) \right\}, \quad (x,p) \in \mathbb{R}^d_+ \times \mathbb{R}^d.
\end{equation}
It is useful to note that given a function $\tb : \R_+^d \to \R^d$, the following equation is formally equivalent to Eq. \eqref{eq:hjb1}:
\begin{equation}\label{eq:hjb1-tilde}
    V_t(t,x) + \frac{1}{2}\mathrm{Tr}\left(\sigma\sigma^{\top} D_{xx}V(t,x)\right) + \tb(x)^{\top} D_x V(t,x) + \tilde\cH\left(x, D_x V(t,x)\right) - \beta V(t,x) = 0, 
\end{equation}
where the Hamiltonian $\tilde\cH$ is given by
\begin{equation}\label{eq:hamiltonian-tilde}
    \tilde \cH(x,p) := \inf_{a \in \mathcal{A}} \left\{ (b(x,a) - \tb(x))^{\top} p + c(x,a) \right\}.
\end{equation}
Second, assuming that $V$ is sufficiently smooth, given $t \in [0,T)$ and $x \in \mathbb{R}^d_+$, we apply Itô's lemma for reflected diffusions (see, e.g., \cite{HarrisonReiman81}) 
to the function $(s,z) \mapsto e^{-\beta (s-t)} V(s, z)$ 
over the reference process $\tilde Z^{t,x}$ to obtain the following Feynman-Kac representation for $V(t,x)$, making use of 
the HJB equation \eqref{eq:hjb1-tilde}, \eqref{eq:hjb2}, and \eqref{eq:hjb3}:
\begin{align}
    V(t,x) = &~\mathbb{E} \left[ \int_t^T e^{-\beta(s-t)} \tilde\cH\left(\tilde Z^{t,x}(s), D_x V(s, \tilde Z^{t,x}(s))\right) ds \right. \nonumber \\
    &~~+ \left.\int_t^T e^{-\beta(s-t)} \kappa^{\top} d \tilde Y^{t,x}(s) + e^{-\beta(T-t)}\xi(\tilde{Z}^{t,x}(T)) \right],
    \label{eq:feynman-kac-tilde}
\end{align}
where the process $\tilde Z^{t,x}$ is a reflected diffusion with drift term $\tb$ and diffusion coefficient $\sigma$ starting with $\tilde Z^{t,x}(t)=x$, 
i.e., it satisfies the following stochastic differential equation (SDE) with reflections:
\begin{equation}\label{eq:Z-t-x}
    \tilde{Z}^{t,x}(s) = x + \int_t^s \tb(\tilde{Z}^{t,x}(r)) dr + \sigma (B(s) - B(t)) + R \tilde{Y}^{t,x}(s), \quad s \in [t,T],
\end{equation}
and $\tilde Y^{t,x}$ is the Skorokhod regulator of the process $\tilde Z^{t,x}$. 

Third, to derive a fixed-point characterization of the value function gradient $D_x V$ from the Feynman-Kac representation \eqref{eq:feynman-kac-tilde}, 
which relates $V$ and $D_x V$, we formally differentiate both sides of Eq. \eqref{eq:feynman-kac-tilde} with respect to $x$ to obtain $D_x V$ on the left-hand side. 
On the right-hand side, instead of applying the chain rule, we apply a Bismut-Elworthy-Li type formula for reflected diffusions (see Section \ref{sec:bel} for details), 
and introduce a Malliavin weight of the form 
\begin{equation}\label{eq:Malliavin-weight}
\frac{1}{s-t} \int_t^s \left[\sigma^{-1} D \tilde{Z}^{t,x}(s)\right]^{\top} dB(s)
\end{equation}
to the right-hand side of Eq. \eqref{eq:feynman-kac-tilde}. 
Here, $D \tilde{Z}^{t,x}$ is the derivative process of $\tilde{Z}^{t,x}$ with respect to $x$, 
which takes values in $\R^{d\times d}$ and captures sensitivity of $\tilde{Z}^{t,x}$ with respect to changes in $x$; 
see Section \ref{sec:tech} for precise definitions and detailed properties.

The end result is the following formal, functional fixed-point equation for $D_x V$:
\begin{equation}\label{eq:fixed-point}
    D_x V(t,x) = F(D_x V)(t,x), \quad (t,x) \in [0,T) \times \mathbb{R}^d_+,
\end{equation}
where the functional map $F$ is formally given by
\begin{align}
   & F(v)(t,x) :=~\mathbb{E} \Bigg[ \int_t^T e^{-\beta(s-t)} \tilde \cH \left( \tilde{Z}^{t,x}(s), v(s, \tilde{Z}^{t,x}(s))\right) \left(\frac{1}{s - t} \int_t^s \left[\sigma^{-1} D \tilde{Z}^{t,x}(r)\right]^{\top} dB(r)\right) ds \nonumber \\
    & ~~~~ + \int_t^T e^{-\beta(s-t)} \kappa^{\top} d\left[D\tilde Y^{t,x}(s)\right] + e^{-\beta(T-t)}\xi(\tilde{Z}^{t,x}(T)) \left( \frac{1}{T - t} \int_t^T \left[\sigma^{-1} D \tilde{Z}^{t,x}(r)\right]^{\top} dB(r) \right) \Bigg].
    \label{eq:F}
\end{align}
Here, $D\tilde Y^{t,x}$ is the derivative process of $\tilde Y^{t,x}$ with respect to $x$; also see Section \ref{sec:tech} for details. 
Note the insertions of Malliavin weights of the form \eqref{eq:Malliavin-weight} when comparing the definition of $F$ in Eq. \eqref{eq:F} 
with the Feynman-Kac representation in Eq. \eqref{eq:feynman-kac-tilde}. 

So far, we have argued that if exists, $D_x V$ should formally satisfy the functional fixed-point equation Eq. \eqref{eq:fixed-point}. 
Using properties of the derivative processes $D \tilde{Z}^{t,x}$, 
we are able to prove that $F$ is a well-defined contraction mapping from an appropriate Banach space of functions into itself; 
see Proposition \ref{prop:F-contraction}. It then follows from Banach's fixed-point theorem that the mapping $F$ has a unique fixed point, 
and that, assuming $D_x V$ exists and coincides with that fixed point, it can be approximated by 
Picard iterations of the form $v^{(n)} = F(v^{(n-1)})$, $n \in \mathbb{N}$. 
Because it is computationally challenging to directly implement the iterative scheme (see, e.g., discussion in Section 2.1 of \cite{hutzenthaler2019multilevel}), 
we rewrite the iteration $v^{(n)} = F(v^{(n-1)})$ as a telescopic sum, 
so that for a given $(t,x) \in [0,T) \times \mathbb{R}^d_+$, we have
\begin{equation}\label{eq:telescopic-sum}
    v^{(n)}(t,x) = F(v^{(0)})(t,x) + \sum_{\ell=1}^{n-1} \left[F(v^{(\ell)}) - F(v^{(\ell-1)})\right](t,x), 
\end{equation}
and approximate the differences $\left[F(v^{(\ell)}) - F(v^{(\ell-1)})\right](t,x)$ 
by a combination of Monte Carlo simulations and further Picard iterations; 
the reason for further iterations is that $v^{(\ell)}$ and $v^{(\ell-1)}$ are not {\em a priori} known, so they themselves need to be approximated 
at random time-state pairs, and so on and so forth. 
In this sense, MLP is a recursive scheme that requires full history of all earlier levels 
when estimating each of the differences $\left[F(v^{(\ell)}) - F(v^{(\ell-1)})\right](t,x)$. 
A key insight from the MLP literature is that as the level $\ell$ increases, because of the contractive property of $F$, 
the difference $\left[F(v^{(\ell)}) - F(v^{(\ell-1)})\right]$ becomes exponentially smaller under appropriate norms, 
so a smaller number of samples suffice to approximate the difference well. 
By carefully choosing geometrically decreasing numbers of samples to use at larger levels, it can be established that 
under appropriate conditions, the total number of samples required to approximate the fixed point of $F$ at a given time-state pair $(t,x)$
to within $\veps$ prescribed error is polynomial in $d$ and $1/\veps$; see Section \ref{sec:mlp} for details. 

To summarize, we have (a) outlined a formal argument deriving the functional fixed-point equation Eq. \eqref{eq:fixed-point} for $D_x V$, 
and (b) described the high-level intuitions behind the MLP scheme for approximating the unique fixed point at a given time-state pair $(t,x)$, 
but an important missing link is to rigorously establish that $D_x V$ is well-defined and is none other than the unique fixed point of $F$.
A standard approach is to use the viscosity solution paradigm \cite{FlemingSoner2006,CrandallIshiiLions1992}, but techniques in existing literature cannot be directly applied, 
because they rely on the fact that unconstrained diffusions have continuous sample paths, and that their derivative processes depend continuously on initial conditions, 
both properties of which do not hold for reflected diffusions. 
To tackle these challenges, we develop two technical tools that are, to the best of our knowledge, new: 
First, a Bismut-Elworthy-Li formula for reflected diffusions, 
and second, continuous dependence of the derivative process on the initial condition, 
in the interior of the state space. 
Using these tools, we are able to prove that the antiderivative of the unique fixed point of $F$ 
is a viscosity solution of the HJB equation Eq. \eqref{eq:hjb1} -- \eqref{eq:hjb3}. 
By standard arguments, it can also be established that the value function $V$ is a viscosity solution, 
and that viscosity solutions are unique, so $V$ must be differentiable and coincides with the antiderivative of the unique fixed point of $F$.
Through these results, we fully justify the use of the MLP scheme to approximate $D_x V$ and $V$.

\subsection{Organization}\label{ssec:organization} 
The rest of the paper is organized as follows. 
Section~\ref{sec:related} reviews related literature on drift controls, MLP methods for high-dimensional PDEs, and reflected diffusions and their derivative processes.
Section~\ref{sec:problem_description} provides a complete problem description, together with several standing assumptions used throughout the paper. 
Section~\ref{sec:tech} collects technical preliminaries on the Skorokhod problem, reflected diffusions and their derivative processes. 
Here we recall fundamental concepts and properties established in prior literature, and develop some new results that are used later for our analysis. 
We also briefly recall the definition of viscosity solutions for parabolic PDEs with oblique derivative boundary conditions at the end of the section. 
We state the first main result of the paper, which provides a fixed-point characterization of the value function gradient, in Section \ref{sec:fixed-point-equation}. 
Then, in Section \ref{sec:mlp}, we develop the multilevel Picard method, and state the sample complexity bounds, our second main result. 
We conduct a set of numerical experiments to demonstrate the effectiveness of the MLP method in Section \ref{sec:numerical}. 
The rest of the paper is devoted to proofs of various results: 
Section~\ref{sec:bel} establishes a Bismut--Elworthy--Li formula for reflected diffusions, and 
Section~\ref{sec:viscosity-solutions} establishes the fixed point of $F$ as a viscosity solution of the HJB equation. 
We conclude the paper in Section \ref{sec:conclusion} with a brief summary and some future directions. 
Further proofs, implementation details of the MLP method, and numerical results are contained in the Appendix. 

\section{Literature Review}\label{sec:related}
The following streams of literature are most relevant to our work: (i) drift control problems; (ii) multilevel Picard approximations for solving PDEs; 
and (iii) reflected diffusions and derivative processes. We also briefly review related literature on dynamic control of queueing networks in heavy traffic 
in our discussion of drift control problems. 

\paragraph{Drift control problems.} Several works consider one-dimensional drift control problems. 
Due to their analytical tractability, it is often possible to derive closed-form solutions that can provide useful insights. 
Motivated by applications to power control in wireless communication, \cite{AtaHarrisonShepp2005} derives optimal policies in closed form 
for a one-dimensional drift control problem on a bounded interval with general control costs but zero cost on states. 
Extensions to the model in \cite{AtaHarrisonShepp2005} are analyzed in \cite{GhoshWeerasinghe2007,GhoshWeerasinghe2010}. 
\cite{OrmeciMatogluVandeVate2011,VandeVate2021} consider models in which fixed costs are incurred to change the drift, and derive optimality of control-band policies, 
with applications to capacity management in a build-to-order environment. 
Drift control problems also see applications to make-to-order manufacturing with cancellations \cite{RubinoAta2009}, 
staffing of volunteer gleaners \cite{AtaLeeSonmez2019}, and foreign reserve holdings \cite{BarIlanMarionPerry2007}. 
Our work focuses on multidimensional settings where closed-form solutions are usually unavailable. 

There is also extensive literature on multi-dimensional drift control problems, with and without state constraints. Motivated by applications to service systems, 
\cite{atar2004scheduling, harrison2004dynamic} analyze a queueing system with multiple customer classes and a single pool of servers that are capable of serving all customer types. 
Both papers study a formal diffusion approximation to the dynamic control problem of the queueing system, in a so-called Halfin-Whitt regime, and derive structural properties 
of the associated HJB equation, such as solution regularity. The diffusion approximation is a multi-dimensional drift control problem with no state constraints, 
i.e., the state space is the entire Euclidean space. In addition, 
\cite{atar2004scheduling} proves asymptotic optimality of scheduling policies that are derived based on solutions to the HJB equation in the diffusion limit, 
and \cite{harrison2004dynamic} carried out numerical experiments in dimension $2$. 
\cite{Atar2005} considers a significant generalization of the drift control problem in \cite{atar2004scheduling, harrison2004dynamic}, 
which arises as the formal diffusion approximation of a queueing system 
with multiple customer classes and multiple server pools, and derive and study properties of the corresponding HJB equation. 
\cite{AtaKasikaralar2023} considers a model that is similar to that considered in \cite{atar2004scheduling, harrison2004dynamic}, and 
develops a simulation-based computational method to solve the associated HJB equation, for problems up to dimension $500$. 

Drift control problems with state constraints arise across domains, notably in queueing theory and mathematical finance. 
As noted in the introduction, portfolio management often imposes a solvency constraint, 
yielding control problems with absorbing boundaries in which the state process is stopped upon hitting the boundary. 
\cite{DaiZhong2010} develops a penalty method for portfolio problems with proportional transaction costs: 
The associated singular control problem, in which system states may be instantaneously displaced, is approximated
by drift control problems, the latter of which are then solved by finite differences; 
numerical results are reported for problems with dimension up to $2$. 
In queueing theory, a primary motivation for our work, job or customer counts are nonnegative, 
so servers must idle when the corresponding buffers are empty. 
In heavy-traffic limits, diffusion approximations of such systems under fixed scheduling policies often lead to reflected diffusions \cite{Reiman1984,ChenYao2001,Williams1998}, 
where a regulator process keeps the state within the feasible region; 
server idling corresponds, roughly speaking, to the boundary local time. 
When dynamic controls are allowed, the corresponding diffusion approximations are often drift controls with reflections (see, e.g., \cite{budhiraja2011ergodic, borkar2004ergodic}), 
or singular controls with reflections (see, e.g., \cite{atar2006singular,PesicWilliams2016,HarrisonVanMieghem1997,Harrison2000,Harrison2006Corr}). 
In the ergodic setting, \cite{borkar2004ergodic} analyzes the viscosity solution properties of the value function, 
while \cite{budhiraja2011ergodic} justifies the diffusion approximation by proving convergence of the value functions in the heavy-traffic limit. 
In another stream of literature (see, e.g., \cite{HarrisonVanMieghem1997,Harrison2000,Harrison2006Corr}), 
Brownian control problems (BCPs) are developed to approximate queueing network control problems in heavy traffic, 
which can often be reduced to a lower-dimensional equivalent workload formulation (EWF). 
These EWFs are typically singular control problems constrained to remain in a convex polyhedral domain with oblique reflection at the boundary. 
Similar to drift controls that arise as approximations to singular controls in portfolio management problems, 
drift controls with reflections can serve as tractable proxies for EWFs. 

Much of prior works on drift controls with reflections focuses on their theoretical properties. 
More recently, \cite{AtaHarrisonSi2024} develop a simulation-based computational method 
to solve high-dimensional drift control problems with reflections up to problem dimension at least $30$, building on the deep BSDE method 
pioneered in \cite{HanJentzenE2018PNAS}. Their problem setting is similar to ours, but differ substantially in the solution approaches. 
First, the method developed in \cite{AtaHarrisonSi2024} is deep-learning based, hence parametric, 
but ours is non-parametric. Second, we focus on sample complexity guarantees for value function and gradient estimates with prescribed accuracy, 
whereas \cite{AtaHarrisonSi2024} follows a ``benchmarking'' approach that focuses on matching optimal performances or beating existing benchmarks. 
Finally, a key methodological difference is that \cite{AtaHarrisonSi2024} relies on a BSDE representation involving both the value function and its gradient, 
while we rely on a functional fixed-point equation that builds on a Feynman-Kac representation and a Bismut-Elworthy-Li type formula. 

As mentioned, drift control problems are closely related to singular control problems. Indeed, many singular control problems can be approximated 
by drift controls \cite{DaiZhong2010,menaldi1989optimal,williams1994regularity}, opening the door for numerical solutions to singular control problems 
via those of the approximating drift controls; see \cite{ata2024singular} for a computational method that can successfully solve high-dimensional singular control problems 
with state constraints, which is based on solving the approximating drift controls. 
In this vein, our method may potentially be used to produce useful benchmarks and/or approximate solutions for related singular control problems with reflections; see, e.g., \cite{atar2006singular,PesicWilliams2016}.

\paragraph{Multilevel Picard approximation.} 
Our work builds on the MLP framework to solve drift control problems with reflections, 
so it also contributes to and is related to the rapidly expanding literature on using multilevel Picard approximations to 
solve high-dimensional PDEs; see \cite{giles2019generalised,beck2020overcoming,becker2020numerical,hutzenthaler2019multilevel,hutzenthaler2021multilevel,hutzenthaler2022multilevel,hutzenthaler2022multilevel1,HJKN2020,HJKNW2020,NW2022}. 
The MLP framework delicately combines several powerful ideas in and of themselves, 
including nonlinear Feynman-Kac representations \cite{PardouxPeng1990,Peng1991,PardouxPeng1992}, Bismut-Elworthy-Li type identities in the case of gradient-dependent nonlinearities \cite{Bismut1984,ElworthyLi1994}, 
Picard iterations for BSDEs (see, e.g., Chapter 7 of \cite{YongZhou1999}), and multilevel Monte Carlo simulations \cite{Giles2008,Giles2015}. 
MLP methods admit rigorous convergence and complexity analysis and, under standard Lipschitz-type assumptions, overcome the curse of dimensionality in that computational cost grows at most polynomially in the dimension and $1/\veps$, 
where $\veps$ is the prescribed accuracy. 
Earlier works focus on semilinear PDEs in which the nonlinearity does not depend on the gradient. 
Our work is more closely related to recent works that address gradient-dependent nonlinearities, 
including \cite{HJK2022,HK2020,NeufeldNguyenWu2025,NeufeldWu2025}. 
Compared to these latter works, a major challenge for us is addressing the discontinuities introduced by the reflection mechanism; 
see, e.g., discussion in Section \ref{sec:introduction}. This requires us to delve deeply into the subject of derivative processes 
for reflected diffusions on polyhedral domains \cite{LipshutzRamanan2018,LipshutzRamanan2019a,LipshutzRamanan2019b}, a topic of substantial interest in its own right, 
derive new results on these processes, and establish key building blocks in the analysis of the MLP scheme using new techniques, 
such as the Bismut-Elworthy-Li type formula for reflected diffusions and properties of a contractive functional operator whose fixed point 
characterizes the value function and its gradient. Through these innovations, we extend the MLP framework to 
include sample complexity guarantees for semilinear HJB equations with oblique derivative boundary conditions, 
and, under appropriate conditions, establishes that MLP overcomes the curse of dimensionality. 
Finally, one related work that addresses reflections on the boundary in the MLP literature is \cite{BoussangeBeckerJentzenKuckuckPellissier2023}. 
Different from our work, they focus on normal reflections on a smooth boundary, and they do not consider gradient-dependent nonlinearities.

\paragraph{Reflected diffusions and derivative processes.} 
The development and analysis of our MLP scheme rely crucially on reflected diffusions and on their sensitivity 
with respect to data that define them. 
Foundational results on reflected diffusions establish existence and/or uniqueness, where the boundary of the state space may be smooth or non-smooth, 
and the directions of reflections may be normal or oblique 
\cite{Skorokhod1961, Skorokhod1962, Costantini1992, DupuisIshii1993, DupuisIshii2008corr, HarrisonReiman1981, LionsSznitman1984, 
Saisho1987, StroockVaradhan1971, Tanaka1979, TaylorWilliams1993, VaradhanWilliams1985}. 
A central object underpinning this theory is the deterministic Skorokhod map, 
whose fundamental properties -- most notably Lipschitz continuity -- have been extensively developed in \cite{DupuisIshii1991, DupuisRamanan1999I, DupuisRamanan1999II, HarrisonReiman81,REF}. 
Sensitivity properties of reflected diffusions have been analyzed more recently in \cite{MandelbaumRamanan2010,LipshutzRamanan2018,LipshutzRamanan2019a,LipshutzRamanan2019b,Andres2009,DeuschelZambotti2005}. 
We focus on reflected diffusions in the nonnegative orthant, with oblique reflections on the boundary \cite{HarrisonReiman1981}, 
and, in particular, notions of derivatives with respect to infinitesimal perturbations to the driving data and the initial condition. 
\cite{MandelbaumRamanan2010} proves the existence of directional derivatives of Skorokhod maps in the nonnegative orthant with oblique reflections, 
and characterize the types of path discontinuities that arise. Because that approach is not the most convenienct for our setting, 
we adopt the axiomatic framework developed in \cite{LipshutzRamanan2018,LipshutzRamanan2019b}, 
which characterize derivative processes as right-continuous regularizations of directional derivatives of reflected diffusions. 
\cite{LipshutzRamanan2018,LipshutzRamanan2019b} further show that if the Skorokhod map satisfies a certain boundary jitter property -- 
which holds in our problem class -- then derivative processes exist and are pathwise unique. 
We use properties of derivative processes established in \cite{LipshutzRamanan2018,LipshutzRamanan2019b}, 
together with new properties proved in this paper, to analyze our MLP scheme and associated stochastic functional fixed-point equation. 
Related developments include \cite{DeuschelZambotti2005}, which studies derivative processes in the orthant with normal reflections, 
and \cite{Andres2009}, which extends \cite{DeuschelZambotti2005} to convex polyhedron with oblique reflections, up to the first hitting time to the non-smooth portion of the boundary.

\section{Problem Description}\label{sec:problem_description}
Fix a probability space $(\Omega, \mathcal{F}, \mathbb{P})$ with sample space $\Omega$, $\sigma$-algebra $\mathcal{F}$,  probability measure $\mathbb{P}$ 
and filtration $\{\mathcal{F}_t : t \in \mathbb{R}_+\}$.  
Suppose that the probability space $(\Omega, \mathcal{F}, \mathbb{P})$ is complete, and the filtration $\{\mathcal{F}_t\}$ satisfies the usual conditions 
in the sense of e.g., Definition~1.2.25 in Karatzas and Shreve (1998);  
that is, $\mathcal{F}_0$ contains all $\mathbb{P}$-null sets in $\mathcal{F}$, and $\{\mathcal{F}_t : t \in \mathbb{R}_+\}$ is right-continuous.
Let $B = \{B(t), t \in \mathbb{R}_+\}$ be a standard $d$-dimensional $\{\cF_t\}$-adapted Brownian motion defined on the probability space $(\Omega, \mathcal{F}, \mathbb{P})$.

Let $T \in (0, \infty)$. Given the problem dimension $d \in \N$, 
we consider a multi-dimensional drift rate stochastic control problem with state space $\mathbb{R}_+^d$ over a finite time horizon $[0, T]$. 
The state $Z(\cdot)$ evolves according to Eq. \eqref{eq:state}: 
\begin{equation*}
Z(t) = Z(0) + \int_0^t b\bigl(Z(s), a(s)\bigr)\,ds + \sigma B(t) + RY(t). 
\end{equation*}
Here, the initial state is $Z(0) \in \mathbb{R}_+^d$, the control process $a(\cdot)$ is a $K$-dimensional $\{\mathcal{F}_t\}$-adapted process 
with action space $\mathcal{A} \subset \mathbb{R}^d$, so that for all $t\geq 0$, $a \in \mathcal{A}$, 
the function $b : \mathbb{R}_+^d \times \mathbb{R}^K \to \mathbb{R}^d$ is the drift coefficient, 
the matrix $\sigma \in \mathbb{R}^{d\times d}$ is the constant diffusion coefficient, 
the matrix $R \in \mathbb{R}^{d\times d}$ is the so-called reflection matrix, 
and $Y = \{Y(t) : t\in \mathbb{R}_+\}$ is a $d$-dimensional $\{\mathcal{F}_t\}$-adapted regulator process that ensures a.s. $Z(t) \in \mathbb{R}_+^d$ for all $t\in \mathbb{R}_+$; 
we describe the process $Y$ in more detail in the sequel.
Here, though we do not explicit spell out the dependence of the problem data ($b$, $\sigma$, $R$, etc) on $d$, 
we are implicitly considering a sequence of drift control problems with reflections, indexed by the problem dimension $d$.
The following conditions are assumed on $b$, $\mathcal{A}$, $\sigma$, $R$ and $Y$, respectively, throughout the paper. 
\begin{assumption}\label{as:Theta}
The action space $\mathcal{A}$ is a compact subset of $\R^K$. Furthermore, 
there exists a constant $C_\mathcal{A} > 0$, independent of the dimension $K$, such that for all $a \in \mathcal{A}$, 
\begin{equation}
\|a\|_\infty \leq C_\mathcal{A}.
\end{equation}
\end{assumption}
\begin{assumption}\label{as:b}
There exists a constant $C_b > 0$, independent of $d$, such that
\begin{equation}
\| b(x, a) \|_\infty \leq C_b, \quad \text{ for all } (x, a) \in \mathbb{R}_+^d \times \mathcal{A}, 
\end{equation}
and for all $x, y$ and $a \in \mathcal{A}$, 
\begin{equation}
\|b(x,a) - b(y, a)\|_\infty \leq C_b \|x-y\|_\infty.
\end{equation}
\end{assumption}

\begin{assumption}\label{as:sigma} 
The covariance matrix $\sigma \sigma^{\top}$ is uniformly elliptic in the following sense.  
There exists a constant $C_\sigma > 0$, 
independent of $d$, such that for all vectors $y \in \mathbb{R}^d$,
\[
 C_\sigma^{-1} \|y\|_2^2 \le y^{\top} \sigma \sigma^{\top} y \le C_\sigma \|y\|_2^2.
\]
In other words, $\|\sigma\|_2 \le C_\sigma$ and $\|\sigma^{-1}\|_2 \le C_\sigma$.
\end{assumption}

\begin{assumption}\label{as:R}
The reflection matrix $R$ takes the following form. 
$R = I-Q^{\top}$, where $I$ is the $d\times d$ identity matrix, 
and $Q = (q_{ij})_{i,j=1}^{d}$ is a nonnegative $d \times d$ matrix with zero diagonal entries and all row sums being no larger than $1$. 
In other words, $q_{ii} = 0$ and $\sum_{j=1}^d q_{ij} \le 1$ for all $i = 1,\dots,d$. 
Furthermore, we assume the following uniform contraction condition on $Q$ (cf. \cite{BlanchetChenSiGlynn2021}): There exists $\delta \in (0, 1)$ 
and $\alpha_0 \in (0, \infty)$, both of which are independent of $d$, such that 
\begin{equation}\label{eq:Q-contraction}
\|\bOne^{\top} Q^n\|_\infty \leq \alpha_0 (1-\delta)^n, \quad n\geq 1.
\end{equation}
\end{assumption}
The state process $Z(\cdot)$ is constrained to stay within the nonnegative orthant $\mathbb{R}_+^d$ via the regulator process $Y$, 
where, for each $j = 1, \cdots, d$, $Y_j = \{Y_j(t) : t\in \R_+\}$ is a non-decreasing process with $Y_j(0) = 0$, and $Y_j$ only increases when $Z_j$ is zero, i.e.,
\( \int_0^T \mathbbm{1}_{\{Z_j(t) > 0\}}\,dY_j(t) = 0 \).
In other words, under a control process $a(\cdot)$, if $Z$ solves the reflected stochastic differential equation \eqref{eq:state}, 
and if we define the process $X$ by
\begin{equation}
X(t) = Z(0) + \int_0^t b\bigl(Z(s), a(s)\bigr)\,ds + \sigma B(t), 
\end{equation} 
then $(Z,Y)$ solves the Skorokhod problem for $X$ (cf. Section \ref{ssec:skorokhod}). 

We are interested in minimizing the total expected discounted cost incurred over the time horizon $[0, T]$. 
More specifically, let $\beta \geq 0$ be the discount rate. 
Let $c : \mathbb{R}_+^d \times \mathcal{A} \to \mathbb{R}_+$ 
be the running cost function, and let 
$\xi : \mathbb{R}_+^d \to \mathbb{R}_+$ be the terminal cost function. 
Furthermore, for each $j = 1, \cdots, d$, whenever the regulator process $Y_j$ needs to increase to keep $Z_j$ nonnegative, a ``pushing'' penalty is incurred at rate $\kappa_j \geq 0$.

Let $t \in [0, T]$.  
Under a control process $a(\cdot)$, the total expected discounted cost-to-go incurred from time $t$ to $T$, 
conditioned on $Z(t) = x$, is
\begin{equation}\label{eq:J(x,t,theta)}
J(t,x,a) := \mathbb{E} \left[ \int_t^T e^{-\beta (s-t)}\, c\bigl(Z(s), a(s)\bigr)\,ds + \int_t^T e^{-\beta (s - t)}\, \kappa^{\top} dY(s)
+ e^{-\beta (T - t)}\, \xi(Z(T)) \Big| Z(t)=x \right].
\end{equation}
Define the value function \( V : [0,T] \times \mathbb{R}_+^d \to \mathbb{R} \) by
\begin{equation}\label{eq:V}
V(t,x) := \inf_{a(\cdot)} J(t,x,a).
\end{equation}
Our goal is to efficiently compute the value function $V(t,x)$ for given $(t,x) \in [0, T]\times \R_+^d$.

To ensure that the cost-to-go and value functions are well-defined, we assume that the cost functions $c$ and $\xi$ satisfy certain growth conditions. 
\begin{assumption}\label{as:cost-poly}
Both the running cost function $c$ and the terminal cost function $\xi$ are continuously differentiable, i.e., $c \in C^{1,1}(\R_+^d \times \cA)$, and $\xi \in C^1(\R_+^d)$.
Furthermore, $c$ and $\xi$ satisfy the following polynomial growth condition: 
There exists positive constants $\alpha_w$ and $\alpha_0 \geq 1$,  
both of which are independent from the problem dimension $d$, 
such that for all $x \in \R_+^d$ and $a \in \cA$, 
\begin{equation}\label{eq:c-xi-poly}
|c(x, a)| \le \alpha_w \left(1+(\log d)^{\alpha_0/2} + \|x\|_\infty^{\alpha_0}\right), \text{ and }
|\xi(x)| \le \alpha_w \left(1+(\log d)^{\alpha_0/2} + \|x\|_\infty^{\alpha_0}\right).
\end{equation}
\end{assumption}
Let $\alpha_w$ and $\alpha_0$ be as in Assumption \ref{as:cost-poly}. 
Define the weight function $w : \R_+^d \to \R$ by
\begin{equation}\label{eq:w}
w(x) = 1+(\log d)^{\alpha_0/2} + \|x\|_\infty^{\alpha_0}.
\end{equation}
Then, under Assumption \ref{as:cost-poly}, for all $(x, a) \in \R_+^d \times \cA$, we have
\begin{equation}\label{eq:c-xi-w}
|c(x, a)|\le \alpha_w w(x), \text{ and } |\xi(x)| \le \alpha_w w(x).
\end{equation}
Let us briefly remark that if the cost functions $c$ and/or $\xi$ do not initially satisfy Assumption \ref{as:cost-poly}, they can often be scaled to satisfy the assumption. 
For example, if $c(x,a) = h^\top x$ is a linear function of $x$, with $\|h\|_\infty$ being uniformly bounded in $d$, we can scale $c(x,a)$ 
by a factor of $1/d$ so it satisfies the assumption. 

\section{Technical Preliminaries}\label{sec:tech}

In this section, we introduce reflected diffusions and their associated derivative processes, 
and recall some useful properties established in prior works \cite{LipshutzRamanan2018,LipshutzRamanan2019a,LipshutzRamanan2019b}. 
To do so, we need to introduce the deterministic Skorokhod problem (Section \ref{ssec:skorokhod}), 
the concept of directional derivatives (Section \ref{ssec:directional-derivatives}) and
the derivative problem (Section \ref{ssec:deriv-prob}). 
We also prove some new results, specifically, Proposition \ref{prop:moment-Z}, Corollary \ref{cor:moment-w}, Lemma \ref{lem:bound-DZ}, 
Proposition \ref{thm:continuity-DZ}, and Corollary \ref{cor:conv-bel}, 
which are useful in later sections. 
For readability, their proofs are deferred to Appendix~\ref{app:tech-proofs}.
Finally, our last subsection here recalls the definition of viscosity solution for the parabolic HJB equation \eqref{eq:hjb1} -- \eqref{eq:hjb3}.

\subsection{Skorokhod Problem}\label{ssec:skorokhod}

\begin{definition}[Skorokhod Problem]\label{def:skorokhod}
Given \( f \in \C_+(\mathbb{R}^d) \) and reflection matrix \( R \), the pair of functions \( (h,g) \in \C(\mathbb{R}_+^d)\times \C(\mathbb{R}_+^d) \) 
solves the {\em Skorokhod problem} (associated with \( R \)) for input \( f \), if the following properties hold:
\begin{itemize}
  \item[(i)] \( h(t) = f(t) + Rg(t) \in \mathbb{R}_+^d \) for all \( t \ge 0 \);
  \item[(ii)] For each \( j = 1, \cdots, d \), \( g_j \) is a non-decreasing function with \( g_j(0) = 0 \);
  \item[(iii)] \( g_j \) increases only when \( h_j \) is zero, i.e.
  \(
  \int_0^{\infty} \mathbbm{1}_{\{h_j(t) > 0\}}\, dg_j(t) = 0, j = 1, \cdots, d.
  \)
\end{itemize}
\end{definition}

The seminal paper \cite{HarrisonReiman1981} establishes that for any \( f \in \C_+(\mathbb{R}^d) \),  
there exists a unique solution \( (h, g) \) to the Skorokhod problem for \( f \). 
Thus, the Skorokhod map $\Gamma :  \C_+(\mathbb{R}^d) \to \C(\mathbb{R}_+^d) $ defined by 
\( \Gamma: f \mapsto h \) is well-defined.
It is also well known that the Skorokhod map \( \Gamma \) is Lipschitz continuous \cite{borkar2004ergodic,HarrisonReiman1981,REF,DupuisIshii1991,DupuisRamanan1999I,DupuisRamanan1999II}. 
Furthermore, by Lemma 3 of \cite{BlanchetChenSiGlynn2021}, under Assumption \ref{as:R}, the Lipschitz constant does not grow with $d$.

\begin{proposition}[Lemma 3 of \cite{BlanchetChenSiGlynn2021}]\label{prop:skorokhod-lipschitz}
There exists a constant \( C_\Gamma > 0 \), which is independent of the problem dimension $d$, 
such that if \( (h_i, g_i) \) is a solution to the Skorokhod problem for \( f_i \in \C_+(\mathbb{R}^d) \), \( i = 1,2 \), then for all \( t \in [0,\infty) \),
\(
\| h_1 - h_2 \|_t + \| g_1 - g_2 \|_t \le C_\Gamma \| f_1 - f_2 \|_t.
\)
\end{proposition}

\subsection{Directional Derivatives}\label{ssec:directional-derivatives}

For \( f \in \C_+(\R^d) \), \( \psi \in \C(\R^d) \), and \( \varepsilon > 0 \), define
\begin{equation}\label{eq:eps-dir-deriv}
\nabla_\psi^\varepsilon \Gamma(f) := \frac{\Gamma(f + \varepsilon \psi) - \Gamma(f)}{\varepsilon}.
\end{equation}
Here, it is assumed that $\nabla_\psi^\varepsilon \Gamma(f)$ is only defined for $f$, $\psi$ and \( \varepsilon > 0 \) such that  
\( f + \varepsilon \psi \in \C_+(\R^d) \), i.e., \( f(0) + \varepsilon \psi(0) \in \mathbb{R}_+^d \).

\begin{definition}
Given \( f \in \C_+(\R^d) \) and \( \psi \in \C(\R^d) \),  
the {\em directional derivative} of \( \Gamma \) evaluated at \( f \) in the direction \( \psi \) is a function
\[
\nabla_\psi \Gamma(f): \mathbb{R}_+ \to \mathbb{R}^d,
\]
defined as the pointwise limit
\[
\nabla_\psi \Gamma(f)(t) := \lim_{\varepsilon \to 0} \nabla_\psi^\varepsilon \Gamma(f)(t), \qquad t \geq 0.
\]
\end{definition}

\begin{proposition}[Proposition 2.17 in \cite{LipshutzRamanan2018}]\label{prop:nabla-veps}
Given \( f \in \C_+(\R^d) \) and \( \psi \in \C(\R^d) \) such that  
\( \nabla_\psi \Gamma(f) \) exists, suppose that \(\{\psi_\varepsilon\}_{\varepsilon > 0}\) is a family of functions in \( \C(\R^d) \) such that 
$\psi_\varepsilon \to \psi$ in \( \C(\R^d) \) under the uniform topology, as $\veps \downarrow 0$. Then, 
\begin{equation}
\lim_{\varepsilon \to 0} \nabla^\veps_{\psi_\varepsilon} \Gamma(f)(t) = \nabla_\psi \Gamma(f)(t), \quad t \geq 0. 
\end{equation}
\end{proposition}

\begin{proposition}[Proposition 6.3 in \cite{LipshutzRamanan2018}]\label{prop:nabla-lipschitz}
Given \( f \in \C_+(\R^d) \) and \( \psi_1, \psi_2 \in \C(\R^d) \),  
suppose that both \( \nabla_{\psi_1} \Gamma(f) \) and \( \nabla_{\psi_2} \Gamma(f) \) exist.  
Then for all \( t \in [0,\infty) \),
\begin{equation}
\| \nabla_{\psi_1} \Gamma(f) - \nabla_{\psi_2} \Gamma(f) \|_t \le C_\Gamma \| \psi_1 - \psi_2 \|_t,
\end{equation}
where we recall that $C_\Gamma$ is the Lipschitz constant of the Skorokhod map $\Gamma$ 
from Proposition \ref{prop:skorokhod-lipschitz}.
\end{proposition}

\subsection{The Derivative Problem}\label{ssec:deriv-prob}
Denote the $j$th column vector of the reflection matrix $R$ by $R_j$; this is also a direction of reflection on the face $\{x \in \R_+^d : x_j = 0\}$.
For $x \in \R_+^d$, define the set $\cI(x)$ of indices $j$ for which $x_j=0$ by
\begin{equation}
\cI(x) := \left\{j = 1, 2, \cdots, d : x_j = 0\right\}.
\end{equation}
In particular, if $x_j > 0$ for all $j$, in which case we simply write $x>0$, then $\cI(x) = \emptyset$.
For $x \in \R_+^d$, also define the set $R(x)$ of reflection directions at $x$ by
\begin{equation}
R(x) := \{R_j : j \in \cI(x)\}.
\end{equation}
For $x\in \R_+^d$, if $\cI(x) \neq \emptyset$, i.e., if $x \in \partial \R_+^d$ lies on the boundary of $\R_+^d$, 
define the subspace $H_x$ by
\begin{equation}
H_x := \left\{y \in \R^d : y_j = 0 \text{ for all } j \in \cI(x)\right\}, 
\end{equation}
and the set $G_x$ of allowed directions of perturbation by
\begin{equation}
G_x := \left\{y \in \R^d : y_j\geq 0 \text{ for all } j \in \cI(x) \right\}.
\end{equation}
If $\cI(x) = \emptyset$, i.e., if $x > 0$, set $H_x := \R^d$ and $G_x := \R^d$. As explained in, e.g., \cite{LipshutzRamanan2019b}, 
if $x \in \partial \R_+^d$, then for $y \in H_x$, $\cI(x + \veps y) = \cI(x)$ for sufficiently small $\veps>0$, 
and for $y \in G_x$, $x+\veps y \in \R_+^d$ for sufficiently small $\veps>0$.
\begin{definition}\label{df:deriv-prob}
Let \( f \in \C_+(\R^d) \), and suppose that \( (h,g) \) solves the Skorokhod problem associated with \( R \) for \( f \).  
Let \( \psi \in \D(\mathbb{R}^d) \). Then \( (\phi, \eta) \in \D(\mathbb{R}^d)\times \D(\mathbb{R}^d) \) is a solution to the {\em derivative problem} along \( h \) for \( \psi \) if
$\eta(0) \in \text{span} \left[ R(h(0)) \right]$, and if, for all \( t \ge 0 \), the following conditions hold:
\begin{itemize}
  \item[(i)] \( \phi(t) = \psi(t) + \eta(t) \) for all $t \geq 0$;
  \item[(ii)] \( \phi(t) \in H_{h(t)} \) for all $t \geq 0$; 
  \item[(iii)] for all \( 0 \le s < t \),
  \[
  \eta(t) - \eta(s) \in \text{span} \left[ \bigcup_{r \in (s,t]} R(h(r)) \right].
  \]
\end{itemize}
If there exists a unique solution \( (\phi, \eta) \) to the derivative problem along \( h \) for \( \psi \), we write $\phi = \Lambda_h[\psi]$, 
and refer to \( \Lambda_h \) as the derivative map along \( h \).
\end{definition}

\begin{lemma}[Lemma 5.1 in \cite{LipshutzRamanan2018}]\label{lem:linearity}
Suppose \( (\phi_i, \eta_i) \) solves the derivative problem along \( h \) for \( \psi_i \in \D(\mathbb{R}^d) \), $i = 1, 2$, respectively.  
Then, for all \( \alpha_1, \alpha_2 \in \mathbb{R} \), the pair
\(
(\alpha_1 \phi_1 + \alpha_2 \phi_2,\; \alpha_1 \eta_1 + \alpha_2 \eta_2)
\)
solves the derivative problem along \( h \) for \( \alpha_1 \psi_1 + \alpha_2 \psi_2 \).
\end{lemma}

\begin{lemma}[Lemma 5.2 in \cite{LipshutzRamanan2018}]\label{lem:shift}
Let \( f \in \C_+(\R^d) \), and suppose that \( (h,g) \) solves the Skorokhod problem for \( f \).  
Let \( \psi \in \D(\mathbb{R}^d) \), and suppose \( (\phi, \eta) \in \D(\mathbb{R}^d)\times \D(\mathbb{R}^d) \) solves the {\em derivative problem} along \( h \) for \( \psi \).
Define $h^s$, $\phi^s$, $\psi^s$, and $\eta^s$ by 
\(
h^s(\cdot) := h(s+ \cdot), \phi^s(\cdot) := \phi(s+\cdot), \psi^s(\cdot) := \phi(s) + \psi(s+\cdot) - \psi(s), \eta^s(\cdot) := \eta(s+\cdot) - \eta(s).
\)
Then, $(\phi^s, \eta^s)$ solves the derivative problem along $h^s$ for $\psi^s$.
\end{lemma}

\begin{proposition}[Theorem 5.4 of \cite{LipshutzRamanan2018}]
There exists a constant \( C_\Lambda > 0\), which is independent of the problem dimension $d$, 
such that if \( (h,g) \) solves the Skorokhod problem for \( f \),  
and \( (\phi_i, \eta_i) \) solves the derivative problem along \( h \) for \( \psi_i \in \D(\mathbb{R}^d) \), \( i = 1,2 \),  
then for all \( t \in (0,\infty) \),
\(
\| \phi_1 - \phi_2 \|_t \le C_\Lambda \| \psi_1 - \psi_2 \|_t.
\)
Furthermore, the constant \( C_\Lambda \) does not depend on \( h \in \C(\R_+^d) \).
\end{proposition}
\begin{remark}
Theorem 5.4 of \cite{LipshutzRamanan2018} does not directly establish growth rate of $C_\Lambda$ 
in the problem dimension $d$, but an inspection of its proof shows that under Assumption \ref{as:R}, $C_\Lambda$ does not grow with $d$. 
\end{remark}

The following proposition is one of the main results established in \cite{LipshutzRamanan2018},  
specialized to the more restrictive setting of this paper.  
It provides a sharp characterization of the relationship between the derivative problem and the notion of directional derivative.

\begin{proposition}[Theorem 3.11 of \cite{LipshutzRamanan2018}]\label{prop:LR2018-main}
Given \( f \in \C_+(\R^d) \), suppose that \( (h,g) \) solves the Skorokhod problem for \( f \).  
Then for all \( \psi \in \C(\R^d) \), \( \nabla_\psi \Gamma(f) \) exists, lies in \( \D_{\ell, r}(\mathbb{R}^d) \),  
and \( \Lambda_h[\psi] \) is equal to the right-continuous regularization of \( \nabla_\psi \Gamma(f) \); i.e.,
\(
\Lambda_h[\psi](t) = \nabla_\psi \Gamma(f)(t+) \text{ for all } t \ge 0.
\)
\end{proposition}

\subsection{Reflected Diffusions}\label{ssec:ref-dif}

Let $\tilde{b}: \mathbb{R}_+^d \to \mathbb{R}^d$ be a continuously differentiable function that satisfies the following assumption.

\begin{assumption}\label{as:tb}
There exists a constant $C_{\tilde{b}} \in (0, \infty)$ and $\alpha_{\tb} \in (0,1]$, both of which are independent of $d$, such that
$\| D\tb(x) \|_\infty \le C_{\tb}$ and 
$\| \tb(x) \|_\infty \le C_{\tb}$ for all $x \in \mathbb{R}_+^d$, 
and for all $x, y \in \R_+^d$, 
\(
\| D\tb(x) - D\tb(y) \|_\infty \le C_{\tilde{b}} \|x - y\|^{\alpha_{\tb}}_\infty.
\)

\end{assumption}

\begin{definition}[Reflected diffusion]
Given drift coefficient function $\tb(\cdot)$, initial condition $x \in \mathbb{R}_+^d$, reflection matrix $R$, and the Brownian motion $B(\cdot)$, 
the associated reflected diffusion is a $d$-dimensional continuous $\{\mathcal{F}_t\}$-adapted process $\tilde Z^x = \{\tilde Z^x(t) : t \ge 0\}$ such that a.s.,
\( \tilde Z^x(t) \in \mathbb{R}_+^d \) for all $t \ge 0$,
and
\begin{equation}\label{eq:reflected-diffusion}
\tilde Z^x(t) = x + \int_0^t \tb\bigl( \tilde Z^x(s) \bigr)\,ds + \sigma B(t) + R \tilde Y^x(t),
\end{equation}
where $\tilde Y^x = \{\tilde Y^x(t): t \ge 0\}$ is a $d$-dimensional continuous $\{\mathcal{F}_t\}$-adapted process such that a.s., for each $j = 1, \cdots, d$,
$\tilde Y_j^x(0) = 0$, the process $\tilde Y_j^x(\cdot)$ is non-decreasing, and 
\(
\int_0^\infty \mathbbm{1}_{\{\tilde Z_j^x(s) > 0\}}\,d\tilde Y_j^x(s) = 0.
\)
\end{definition}
Denote the ``free'' part of the process $\tilde Z^x$ by $\tilde X^x$. More precisely, 
\begin{equation}\label{eq:tx}
\tilde X^x(t) = x + \int_0^t \tb\bigl( \tilde Z^x(s) \bigr)\,ds + \sigma B(t), \quad t \geq 0, 
\end{equation}
and $\tilde Z^x = \Gamma \left(\tilde X^x\right)$.
The following result is well-known (see, e.g., Theorem 4.3 of \cite{REF}).

\begin{proposition}
Suppose $\tb(\cdot)$, $\sigma$ and $R$ satisfy Assumptions \ref{as:tb}, \ref{as:sigma} and \ref{as:R}. 
Then, for each $x \in \mathbb{R}_+^d$ and $\{\mathcal{F}_t\}$-Brownian motion $B(\cdot)$, there exists an associated reflected diffusion $\tilde Z^x$, and $\tilde Z^x$ is a strong Markov process. Furthermore, if $\bar Z^x$ is another such reflected diffusion, then a.s.\ $\tilde Z^x = \bar Z^x$. In other words, pathwise uniqueness holds.
\end{proposition}

\begin{lemma}[Lemma 4.13 of \cite{LipshutzRamanan2019b}]\label{lem:Z>0}
Suppose Assumption \ref{as:sigma} holds. Then, 
$\pr\left(\tilde Z^x(t) > 0\right) = 1$ for all $t > 0$.
\end{lemma}

We establish some moment bounds on $\tilde Z^x$ in the following proposition. 
\begin{proposition}\label{prop:moment-Z}
Let $x \in \mathbb{R}^d_+$. Then for all $t \ge 0$ and $\alpha \geq 1$, there exists a constant $C = C(t, \alpha, C_\sigma, C_{\tb}, C_{\Gamma})$ that only depends on 
$t$, $\alpha$, $C_\sigma$, $C_{\tb}$ and $C_{\Gamma}$, such that
\begin{equation}\label{eq:moment1}
\mathbb{E} \left[ \| \tilde{Z}^x \|_t^\alpha \right] \leq C \left(1 + (\log d)^{\alpha/2} + \|x\|^\alpha_\infty\right).
\end{equation}
Furthermore, for any $x, y\in \R_+^d$, a.s.
\begin{equation}\label{eq:moment2}
\| \tilde{Z}^y - \tilde{Z}^x \|_t \leq C \|y - x\|_\infty.
\end{equation}
\end{proposition}

The following corollary is a simple but useful consequence of the moment bounds on $\tilde Z^x$ in Proposition \ref{prop:moment-Z}. 
Recall the definition of the weight function $w$ in Eq. \eqref{eq:w}.
\begin{corollary}\label{cor:moment-w}
Let $\alpha \geq 1$, and let $T>0$. Then, there exists a positive constant $C_{\alpha,T}$ that depends only on $\alpha$ and $T$ 
(as well as $C_\sigma$, $C_{\tb}$ and $C_{\Gamma}$) but not on $d$, such that 
for all $t \in [0, T]$, for all $x \in \R_+^d$,
\begin{equation}
\E\left[w^\alpha \left(\tilde Z^x(t)\right)\right] \leq C_{\alpha, T} w^{\alpha}(x).
\end{equation}
\end{corollary}

\subsection{Derivative Processes}\label{ssec:deriv-proc}

\begin{definition}[Derivative process]
Let $x \in \mathbb{R}_+^d$, let $B(\cdot)$ be a $\{\mathcal{F}_t\}$-Brownian motion, and let $\tilde Z^x$ be an associated reflected diffusion. 
A \emph{derivative process along $\tilde Z^x$} is an RCLL $\{\mathcal{F}_t\}$-adapted process $D\tilde Z^x = \{ D\tilde Z^x(t): t \ge 0 \}$ taking values in $\mathbb{R}^{d\times d}$ such that a.s., for all $t \geq 0$,
\(
D_j\tilde Z^x(t) \in H_{\tilde Z^x(t)} \text{ for } j=1,\dots,d,
\)
where $D_j\tilde Z^x(t)$ denotes the $j$th column of $D\tilde Z^x(t)$, 
and $D\tilde Z^x$ satisfies
\[
D\tilde Z^x(t) = I + \int_0^t D\tb\bigl( \tilde Z^x(s) \bigr) D\tilde Z^x(s)\,ds + \tilde \eta(t),
\]
where $\tilde \eta = \{ \tilde \eta(t) : t \in \mathbb{R}_+ \}$ is an RCLL $\{\mathcal{F}_t\}$-adapted process taking values in $\mathbb{R}^{d\times d}$ such that a.s.\ $\tilde \eta_j(0) \in \text{span}[R(x)]$, $j=1, 2, \cdots, d$, and for all $0 \le s < t < \infty$,
\[
\tilde \eta_j(t) - \tilde \eta_j(s) \in \text{span} \left[ \bigcup_{r \in (s,t]} R\bigl( \tilde Z^x(r) \bigr) \right].
\]
In addition to $D\tilde Z^x$, we also define a {\em derivative process along $\tilde Y^x$} as the process $R^{-1} \tilde \eta$, which we denote by $D\tilde Y^x$.
\end{definition}
\begin{remark}
Besides the definition of the standard derivative process $D\tilde Z^x$ in \cite{LipshutzRamanan2019b}, 
we also define the derivative process $D\tilde Y^x$, which is a regularized version of the pathwise derivative of 
$\tilde Y^x$ with respect to $x$; see Proposition \ref{prop:LR2019b-main} in the sequel. The process $D\tilde Y^x$ is useful in several ways; 
for example, an important use of $D\tilde Y^x$ is in characterizing the derivative of 
\(
\E\left[\int_0^T e^{-\beta s} \kappa^{\top} d\tilde Y^{x}(s) \right],
\)
the total expected pushing penalty cost, with respect to the initial condition $x$.
\end{remark}

\begin{proposition}[Theorem 3.6 in \cite{LipshutzRamanan2019b}]\label{prop:deriv-unique}
Let $x \in \mathbb{R}_+^d$ and $B(\cdot)$ be as above. Then pathwise uniqueness holds for the derivative process $D\tilde Z^x$ along $\tilde Z^x$, 
and, as an easy consequence, also for the derivative process $D\tilde Y^x$ along $\tilde Y^x$.
\end{proposition}
The following proposition is the main result in \textsc{\cite{LipshutzRamanan2019b}}, specialized to the more restrictive setting of this paper.
\begin{proposition}[Theorem 3.13 in \cite{LipshutzRamanan2019b}]\label{prop:LR2019b-main}
Suppose  Assumptions \ref{as:tb}, \ref{as:sigma} and \ref{as:R} are in place.
Let $x \in \mathbb{R}_+^d$. Then, there exists a pathwise unique derivative process $D\tilde Z^x$ along $\tilde Z^x$, and for all $y \in G_x$, a.s.\ the following hold.
\begin{itemize}
\item[(i)] The pathwise derivative of $\tilde Z^x$ in the direction $y$, defined by
\(
\partial_y \tilde Z^x(t) := \lim_{\varepsilon \to 0} \frac{\tilde Z^{x + \varepsilon y}(t) - \tilde Z^x(t)}{\varepsilon}, t\ge 0,
\)
exists. The pathwise derivative of $\tilde Y^x$ in the direction $y$, defined for all $t \ge 0$ by
\(
\partial_y \tilde Y^x(t) := \lim_{\varepsilon \to 0} \frac{\tilde Y^{x + \varepsilon y}(t) - \tilde Y^x(t)}{\varepsilon},
\)
also exists.
\item[(ii)] The pathwise derivatives $\partial_y \tilde Z^x = \{\partial_y \tilde Z^x(t) : t \in \mathbb{R}_+\}$ and 
$\partial_y \tilde Y^x = \{\partial_y \tilde Y^x(t) : t \in \mathbb{R}_+\}$ take values in $\D_{\ell, r}(\mathbb{R}^d)$ and are continuous at times $t > 0$ when $\tilde Z^x(t) > 0$, or when $\tilde Z^x_j = 0$ for more than one $j \in \{1, 2, \cdots, d\}$.

\item[(iii)] The right-continuous regularization of the pathwise derivative $\partial_y \tilde Z^x$ ($\partial_y \tilde Y^x$, respectively) is equal to the derivative process $D\tilde Z^x$ 
($D\tilde Y^x$, respectively) evaluated in the direction $y$, i.e.
\[
\lim_{s \downarrow t} \partial_y \tilde Z^x(s) = D\tilde Z^x(t) y, \ \text{ and } \ \lim_{s \downarrow t} \partial_y \tilde Y^x(s) = D\tilde Y^x(t) y, \qquad t \geq 0. 
\]
\end{itemize}
\end{proposition}
\begin{remark}
(a) In part (iii) of Proposition \ref{prop:LR2019b-main}, both $D\tilde Z^x(t) y$ and $D\tilde Y^x(t) y$ are simply matrix-vector multiplications. 
We also denote these quantities as $D\tilde Z^x(t; y)$ and $D\tilde Y^x(t; y)$, respectively, to emphasize their dependences on $y$, 
the direction in which derivatives are taken. 

(b) Let us note the following simple but useful observation. For all $j = 1, 2, \cdots, d$, the $j$th standard unit vector $e_j \in G_x$ 
for any $x\in \R_+^d$, i.e., $e_j$ is always a valid direction in which pathwise derivatives can be taken. 
Then, 
\(
D\tilde Z^x(t; e_j) = D\tilde Z^x(t)e_j = D_j\tilde Z^x(t), \text{ and } D\tilde Y^x(t; e_j) = D_j \tilde Y^x(t),
\)
where $D_j\tilde Z^x(t)$ ($D_j \tilde Y^x(t)$, respectively) denotes the $j$th column of the matrix $D\tilde Z^x(t)$ ($D\tilde Y^x(t)$, respectively). 
In other words, the $j$th column of $D\tilde Z^x(t)$ ($D\tilde Y^x(t)$, respectively) is the regularized $j$th partial derivative of $\tilde Z^x(t)$ ($\tilde Y^x(t)$, respectively) in $x$.

\end{remark}

\begin{proposition}[Theorem 4 of \cite{LipshutzRamanan2019a}]\label{prop:Lambda-Dconv}
Let $\{f_n\}_{n \in \mathbb{N}}$ be a sequence in $\C_+(\mathbb{R}^d)$ that converges to $f \in \C_+(\mathbb{R}^d)$ as $n \to \infty$.
Let $(h,g)$ solve the Skorokhod problem for $f$, and for each $n \in \mathbb{N}$, let $(h_n,g_n)$ solve the Skorokhod problem for $f_n$.
Let $\psi \in \C(\mathbb{R}^d)$ with $\psi(0) \in H_{h(0)}$ and $\{\psi_n\}_{n \in \mathbb{N}}$ be a sequence in $\C(\mathbb{R}^d)$ such that
$\psi_n \to \psi$ in $\C(\mathbb{R}^d)$ as $n \to \infty$. Then,
\(
\Lambda_{h_n}(\psi_n) \to \Lambda_h(\psi) \ \ \text{in } \ \D(\mathbb{R}^d), \ \text{as } \ n \to \infty.
\)
\end{proposition}
We derive a simple bound for the derivative process $D \tilde{Z}^x$ in the following lemma. 
\begin{lemma}\label{lem:bound-DZ}
Let $x \in \mathbb{R}^d_+$, and let $j \in \{1, 2, \cdots, d\}$. There exists a positive constant $C$, independent from $x$ and $j$, such that a.s.
\(
\| D_j \tilde{Z}^x \|_t := \sup_{s \in [0,t]} \|D_j \tilde{Z}^x(s)\|_\infty \leq C.
\)
Furthermore, $C = C(t, C_{\Lambda}, C_{\tb})$ depends only on $t$, $C_{\Lambda}$ and $C_{\tb}$, but not on the problem dimension $d$. 
\end{lemma}

The next result establishes the continuity of derivative processes \( D\tilde{Z}^x \) in the initial state \( x \in \mathbb{R}_{++}^d \), 
in an appropriate sense. Its proof, deferred to Appendix \ref{app:tech-proofs}, is an application of Proposition \ref{prop:Lambda-Dconv}.
\begin{proposition}\label{thm:continuity-DZ}
Let \( x \in \mathbb{R}_{++}^d \) and \( (y^{(n)})_{n \in \mathbb{N}} \) be a fixed sequence in \( \mathbb{R}_{++}^d \) with \( y^{(n)} \to x \) as \( n \to \infty \).  
Let \( D\tilde{Z}^x \) be the derivative process along \( \tilde{Z}^x \), and for each \( n \in \mathbb{N} \),  
let \( D\tilde{Z}^{y^{(n)}} \) be the derivative process along \( \tilde{Z}^{y^{(n)}} \).  
Then a.s. for each $j$, 
\(
D_j\tilde{Z}^{y^{(n)}} \to D_j\tilde{Z}^x \ \text{in } \D(\mathbb{R}^d), \text{ as } n \to \infty.
\)
\end{proposition}

\begin{remark}\label{rm:continuity}
Proposition \ref{thm:continuity-DZ} shows the continuity of the derivative process with respect to the initial condition $x$, 
as long as $x$ is in the interior of the nonnegative orthant, i.e., $x \in \R_{++}^d$. 
This is, in some sense, best possible; see also Proposition \ref{prop:Lambda-Dconv} (Theorem 4 in \cite{LipshutzRamanan2019a} 
and the ensuing discussion). Indeed, for a one-dimensional reflected Brownian motion with initial condition $x$, 
let $\tau_x$ be its first hitting time to zero. 
Then, at $x=0$, the derivative process is everywhere zero, and for $x>0$, the derivative process is $1$ before $\tau_x$, 
and zero after. It can be verified that the derivative process at $x>0$ does not converge to that at $x=0$ in the Skorokhod topology. 
\end{remark}
\begin{corollary}\label{cor:conv-bel}
Let $j \in \{1, 2, \cdots, d\}$, and let \( (y^{(n)}) \), $x$, and \( D_j\tilde{Z}^{y^{(n)}} \) and \( D_j\tilde{Z}^x \) be the same as in Proposition \ref{thm:continuity-DZ}. 
In particular, $x \in \R_{++}^d$.
Define the continuous martingales
\( I(y^{(n)}), n \in \mathbb{N}, \) and \( I(x) \), by
\(
I(y^{(n)})(t) = \int_0^t \left[\sigma^{-1} D_j\tilde Z^{y^{(n)}}(s)\right]^{\top} \, dB(s),\ \text{ and } \ I(x)(t) = \int_0^t \left[\sigma^{-1} D_j\tilde Z^x(s)\right]^{\top} \, dB(s).
\)
Then, \( \mathbb{E} \left[ \|I(y^{(n)}) - I(x)\|_t^2 \right] \to 0 \text{ as } n \to \infty \)
for each \( t \in \mathbb{R}_+ \).
\end{corollary}

\subsection{Viscosity Solution}
Recall the parabolic HJB equation with oblique boundary conditions in Eqs. \eqref{eq:hjb1} -- \eqref{eq:hjb3}, reproduced here for convenience:
\begin{equation*}
  V_t(t,x) + \frac{1}{2}\mathrm{Tr}\left(\sigma\sigma^{\top} D_{xx}V(t,x)\right) + \tb(x)^{\top} D_x V(t,x) + \tilde\cH\left(x, D_x V(t,x)\right) - \beta V(t,x) = 0, 
\end{equation*}
where 
\(
\cH(x,p) := \inf_{a \in \mathcal{A}} \left\{ (b(x,a)^{\top} p + c(x,a) \right\} \) for \( (x,p) \in \mathbb{R}^d_+ \times \mathbb{R}^d.
\)
The terminal condition is given by 
\(
V(T,x) = \xi(x), \text{ for all} \ x \in \mathbb{R}^d_+, 
\)
and the oblique boundary conditions are
\(
\langle R_i, D_x V (t,x) \rangle = -\kappa_i, \text{ if } x_i = 0, t<T.
\)
Define $\Phi(t,x,r,q,p,X)$ by
\begin{align}
\Phi(t,x,r,p,q,X) = -q - \frac{1}{2} \mathrm{Tr}(\sigma \sigma^{\top} X) - \cH (x,p) + \beta r.
\end{align}
Define also $\Phi_*(t,x,r,q,p,X)$ by
\begin{align}
\Phi_*(t,x,r,q,p,X) =
\begin{cases}
\Phi(t,x,r,q,p,X), & \text{if } x > 0; \\
\Phi(t,x,r,q,p,X) \wedge \min_{i \in \mathcal{I}(x)} \left\{ - R_i^{\top} p - \kappa_i \right\}, & \text{if } x \in \partial \mathbb{R}_+^d,
\end{cases}
\end{align}
and $\Phi^*(t,x,r,q,p,X)$ by
\begin{align}
\Phi^*(t,x,r,q,p,X) =
\begin{cases}
\Phi(t,x,r,q,p,X), & \text{if } x > 0 \\
\Phi(t,x,r,q,p,X) \vee \max_{i \in \mathcal{I}(x)} \left\{ - R_i^{\top} p - \kappa_i \right\}, & \text{if } x \in \partial \mathbb{R}_+^d.
\end{cases}
\end{align}
\begin{definition}
$u \in C([0,T] \times \mathbb{R}_+^d)$ is a {\em viscosity solution} of the HJB equation \eqref{eq:hjb1} -- \eqref{eq:hjb3}, 
if $u$ satisfies the terminal boundary condition Eq. \eqref{eq:hjb3}, and 
if for all $\varphi \in C^{1,2}((0,T) \times \mathbb{R}_+^d)$ and $(t,x) \in (0,T) \times \mathbb{R}_+^d$ such that $(t,x)$ is a strict maximum of $u - \varphi$,
\begin{equation}\label{eq:supersolution}
\Phi_*(t,x,u(t,x), \varphi_t(t,x), \varphi_x(t,x), \varphi_{xx}(t,x)) \leq 0,
\end{equation}
and for all $\varphi \in C^{1,2}((0,T) \times \mathbb{R}_+^d)$ and
$(t,x) \in (0,T) \times \mathbb{R}_+^d$ such that $(t,x)$ is a strict minimum of $u - \varphi$, 
\begin{equation}\label{eq:subsolution}
\Phi^*(t,x,u(t,x), \varphi_t(t,x), \varphi_x(t,x), \varphi_{xx}(t,x)) \geq 0.
\end{equation}
\end{definition}

\section{Stochastic Fixed-Point Equation}\label{sec:fixed-point-equation}
In this section, we provide a fix-point characterization of the value function gradient.  
To do so, we first rigorously define the functional map $F$, formally introduced in Eq. \eqref{eq:F}, 
and establish several useful properties. 
We then proceed to prove the fixed-point characterization in Theorem \ref{thm:fixed-point-tildeV}, our main result of the section. 

\subsection{Properties of the Functional Map \texorpdfstring{$F$}{F}}\label{sec:fixed-point}
All proofs of the results in this subsection are deferred to Appendix \ref{sec:proofs-fixed-point}.
Let \( T > 0 \), and let \( \mathcal{B} \) be the space of measurable functions 
\( v : [0, T) \times \mathbb{R}^d_+ \to \mathbb{R}^d \) with
\(
\sup_{(t,x) \in [0,T) \times \mathbb{R}^d_+} \frac{(T - t)^{1/2} \|v(t,x)\|_\infty}{w(x)} < \infty,
\)
and \( w : \R^d \to \R_+ \) defined in Eq. \eqref{eq:w}.

\begin{lemma}\label{lem:norm-rho-banach}
Let \( \rho \in \mathbb{R} \). Denote
\begin{equation}\label{eq:norm-rho}
\|v\|_\rho := \sup_{(t,x) \in [0,T) \times \mathbb{R}^d_+} \frac{e^{\rho t}(T - t)^{1/2} \|v(t,x)\|_\infty}{w(x)}.
\end{equation}
Then, for all \( v \in \mathcal{B} \), \( \|v\|_\rho < \infty \). 
Furthermore, \( \|\cdot\|_\rho \) defines a norm on the space \( \mathcal{B}\), and \( (\mathcal{B}, \|\cdot\|_\rho) \) is a Banach space.
\end{lemma}

Recall the reference process \( \tilde{Z}^{t,x} \) defined in Eq. \eqref{eq:Z-t-x}, i.e., 
\begin{equation*}
  \tilde{Z}^{t,x}(s) = x + \int_t^s \tb(\tilde{Z}^{t,x}(r)) dr + \sigma (B(s) - B(t)) + R \tilde{Y}^{t,x}(s), \quad s \in [t,T],
  \end{equation*}
and the functional map $F$ formally defined in Eq. \eqref{eq:F}, where
for \( (t,x) \in [0,T) \times \mathbb{R}^d_+ \),
\begin{align*}
&~F(v)(t,x) :=~\mathbb{E} \Bigg[ \int_t^T e^{-\beta(s-t)} \tilde \cH \left( \tilde{Z}^{t,x}(s), v(s, \tilde{Z}^{t,x}(s))\right) \left(\frac{1}{s - t} \int_t^s \left[\sigma^{-1} D \tilde{Z}^{t,x}(r)\right]^{\top} dB(r)\right) ds \nonumber \\
& ~~~~ + \int_t^T e^{-\beta(s-t)} \kappa^{\top} d\left[D\tilde Y^{t,x}(s)\right] + e^{-\beta(T-t)}\xi(\tilde{Z}^{t,x}(T)) \left( \frac{1}{T - t} \int_t^T \left[\sigma^{-1} D \tilde{Z}^{t,x}(r)\right]^{\top} dB(r) \right) \Bigg].
\end{align*}
Here, the {\em Hamiltonian} \( \tilde \cH : \R_+^d \times \R^d \to \R \) is defined by Eq. \eqref{eq:hamiltonian-tilde}, i.e., 
for each \( (x,p) \in \mathbb{R}^d_+ \times \mathbb{R}^d \),
\(
\tilde \cH(x,p) := \inf_{a \in \mathcal{A}} \left\{ (b(x,a) - \tb(x))^{\top} p + c(x,a) \right\}.
\)


\begin{lemma}\label{lem:Fv-bounded}
For \( v \in \mathcal{B} \), \( \|F(v)\|_0 < \infty \).
\end{lemma}

Lemmas \ref{lem:Fv-bounded} 
establishes that \( F \) is a well-defined map from  
\( (\mathcal{B}, \|\cdot\|_0) \) into itself.  
By the equivalence of the norms \( \|\cdot\|_\rho \), \( \rho \in \mathbb{R} \),  
\( F \) is also a well-defined map from \( (\mathcal{B}, \|\cdot\|_\rho) \) into itself  
for any \( \rho \in \mathbb{R} \). The next proposition shows that  
\( F \) is contractive for sufficiently large \( \rho \).

\begin{proposition}\label{prop:F-contraction}
For sufficiently large \( \rho \in \mathbb{R} \), the map
\[
F: (\mathcal{B}, \|\cdot\|_\rho) \to (\mathcal{B}, \|\cdot\|_\rho)
\]
where $F$ is defined in Eq. \eqref{eq:F}, is contractive.
\end{proposition}

The following lemma establishes certain continuity-preserving property of the map $F$.
\begin{lemma}\label{lem:Fv-continuous}
For any \( v \in \mathcal{B} \) such that $v$ is continuous on \( [0,T) \times \mathbb{R}_{++}^d \), 
we have that \( F(v) \) is continuous on \( [0,T) \times \mathbb{R}_{++}^d \).
\end{lemma}

\subsection{Fixed-Point Characterization}
\begin{theorem}\label{thm:fixed-point-tildeV}
\begin{enumerate}
\item[(i)] There exists a unique \( \tilde{v} \in \mathcal{B} \) that is the fixed point of the map \( F \), 
i.e., 
\begin{equation}\label{eq:fixed-point-tildev}
  \tilde v = F(\tilde v).
\end{equation}
\end{enumerate}
For the rest of the paper, we use $\tilde v$ to denote the unique fixed point of $F$ in $\mathcal{B}$. Then: 
\begin{enumerate}
\item[(ii)] $\tilde v$ is continuous on $[0, T)\times \R_{++}^d$.
\item[(iii)] Define the function $\tilde V : [0, T)\times \R_+^d \to \R$ by
\begin{align}
  \tilde{V}(t,x) =&~\mathbb{E} \left[ \int_t^T e^{-\beta(s - t)} \tilde \cH(\tilde Z^{t,x}(s), \tilde{v}(s, \tilde Z^{t,x}(s))) ds \right. \nonumber\\
  &~+ \int_t^T e^{-\beta(s - t)} \kappa^{\top} d\tilde Y^{t,x}(s) + e^{-\beta(T - t)} \xi(\tilde Z^{t,x}(T)) \Bigg]. \label{eq:tildeV}
  \end{align}
Then $\tilde V$ is differentiable in $x\in \R_+^d$, with 
\[
\tilde{v}(t,x) = D_x \tilde{V}(t,x), \quad (t,x) \in [0,T) \times \mathbb{R}_+^d.
\]
\item[(iv)] $\tilde V$ is the unique viscosity solution of the HJB equation \eqref{eq:hjb1} -- \eqref{eq:hjb3} 
that satisfies the polynomial growth condition \eqref{eq:c-xi-poly} (with $\alpha_w$ appropriately defined if necessary), 
and $\tilde V = V$, 
where $V$ is the value function defined in Eq. \eqref{eq:V}.
\end{enumerate}
\end{theorem}
\begin{proof} Part (i) follows immediately from Proposition \ref{prop:F-contraction} and Banach's fixed point theorem. 
For part (ii), consider the sequence $\{v^{(n)}\}_{n \geq 0}$ of functions in $\mathcal{B}$ defined recursively as follows. 
$v^{(0)} \equiv 0$, and $v^{(n)} = F(v^{(n-1)})$, $n \in \N$. 
Then by Lemma \ref{lem:Fv-continuous}, for all $n \ge 0$, $v^{(n)}$ is continuous on $[0,T)\times \R_{++}^d$. 
Furthermore, $v^{(n)} \to \tilde v$ under norm $\|\cdot\|_\rho$, for some sufficiently large $\rho$, 
which in turn imply that on any compact subset of $[0,T)\times \R_{++}^d$, 
the convergence $v^{(n)} \to \tilde v$ is uniform. Therefore, the limit $\tilde v$ is continuous on $[0,T)\times \R_{++}^d$. 
Part (iii) follows by applying the Bismut-Elworthy-Li formula Eq. \eqref{eq:bel-main} to $\tilde V$, noting that 
$\tilde \cH(x, \tilde{v}(s, x))$ and $\xi(x)$ are measurable functions of polynomial growth. 
Finally, part (iv) follows from the following facts: (a) $\tilde V$ is a viscosity solution 
of the HJB equation \eqref{eq:hjb1} -- \eqref{eq:hjb3} (Proposition \ref{prop:viscosity-tildeV}); 
(b) the value function $V$ defined in Eq. \eqref{eq:V} is also a viscosity solution (Proposition \ref{prop:viscosity-V}); 
and (c) the viscosity solution is unique (Proposition \ref{prop:viscosity-unique}). 
\end{proof}

\section{Multi-Level Picard Scheme}\label{sec:mlp}
In this section, we develop the multilevel Picard (MLP) scheme for our drift control problem with reflections, 
and establish complexity bounds on the scheme under appropriate assumptions to be detailed. 
\subsection{The MLP Scheme}
Given $(t, x) \in [0, T)\times \R_+^d$, we wish to approximate $\tilde v(t,x)$, 
as well as $\tilde V(t,x)$ defined in Eq. \eqref{eq:tildeV}. 
We do so by making use of the fact that $\tilde v$ 
is the unique fixed point to the functional map $F$ introduced in Eq. \eqref{eq:F}, and 
that $\tilde V$ can be computed from $\tilde v$. 
For notational convenience, we introduce the related functional map $\bar F$ 
that relates $\tilde v$ to $(\tilde V, \tilde v)$: 
For $v \in \cB$, define $\bar F(\bar v)$ by
\begin{align}
\bar F(v)(t,x) & := \mathbb{E} \left[ \int_t^T e^{-\beta(s - t)} 
\tilde \cH(\tilde{Z}^{t,x}(s), \tilde{v}(s, \tilde{Z}^{t,x}(s))) 
\left( 1, \frac{1}{s - t} \int_t^s \left[\sigma^{-1} D \tilde{Z}^{t,x}(r)\right]^\top dB(r)  \right) ds \right. \nonumber \\
&\qquad \quad + \left(\int_t^T e^{-\beta(s - t)} \kappa^{\top} d\tilde Y^{t,x}(s), \int_t^T e^{-\beta(s-t)} \kappa^{\top} d\left[D\tilde Y^{t,x}(s)\right]\right) \nonumber \\
&\qquad \quad \left. + e^{-\beta(T - t)} \xi(\tilde{Z}^{t,x}(T)) 
\left( 1, \frac{1}{T - t} \int_t^T \left[\sigma^{-1} D \tilde{Z}^{t,x}(r)\right]^\top dB(r) \right) \right].
\label{eq:barF}
\end{align}
Here, for a vector $y \in \R^d$, $(1, y)$ is the $(d+1)$-dimensional vector that concatenates $1$ with $y$. 
The fact that $F$ is contractive under an appropriate norm
suggests the following natural iterative scheme to approximate $(\tilde V, \tilde v)$. 
Let
\begin{equation}\label{eq:initial}
v^{(0)}_j(t,x) \equiv 0, j=1, 2, \cdots, d, \quad \text{ and } \quad V^{(0)} \equiv 0.
\end{equation}
Recursively define $(V^{(n)}, v^{(n)})$ by
\begin{equation}\label{eq:iterations}
(V^{(n)}, v^{(n)}) := \bar F\left(v^{(n-1)}\right), n\in \N.
\end{equation}
We expect $(V^{(n)}, v^{(n)})$ to approximate $(\tilde V, \tilde v)$ well for large $n$. 
Thus, we now focus on the task of efficiently approximating $(V^{(n)}, v^{(n)})$. 
First, we rewrite the iteration in Eq. \eqref{eq:iterations} as a telescopic sum:
\begin{equation}
(V^{(n)}, v^{(n)}) = \bar F\left(v^{(0)}\right)+\sum_{\ell=1}^{n-1} \left[\bar F\left(v^{(\ell)}\right) -  \bar F\left(v^{(\ell-1)}\right)\right].
\end{equation}
In particular, for each $j = 1, 2, \cdots, d$, 
\begin{equation}
v^{(n)}_j(t,x) = \left[\bar F\left(v^{(0)}\right)\right]_{j+1} (t,x)+\sum_{\ell=1}^{n-1} \left[\bar F\left(v^{(\ell)}\right) -  \bar F\left(v^{(\ell-1)}\right)\right]_{j+1} (t,x),
\end{equation}
and 
\begin{equation}
V^{(n)}(t,x) = \left[\bar F\left(v^{(0)}\right)\right]_1 (t,x)+\sum_{\ell=1}^{n-1} \left[\bar F\left(v^{(\ell)}\right) -  \bar F\left(v^{(\ell-1)}\right)\right]_1 (t,x).
\end{equation}
Thus, to approximate $(V^{(n)}, v^{(n)})(t,x)$, we approximate $\left[\bar F\left(v^{(0)}\right)\right](t,x)$ and 
$\left[\bar F\left(v^{(\ell)}\right) -  \bar F\left(v^{(\ell-1)}\right)\right](t,x)$, for $\ell = 1, 2, \cdots, n-1$, respectively. 
Furthermore, because of the contractive property of $F$, we expect that we can use fewer samples 
to approximate $\left[\bar F\left(v^{(\ell)}\right) - \bar F\left(v^{(\ell-1)}\right)\right](t,x)$, for larger $\ell$. 
In preparation for the description of the multilevel Picard scheme, we provide a closer examination of the terms 
$\left[\bar F\left(v^{(0)}\right)\right]_j (t,x)$ and $\left[\bar F\left(v^{(\ell)}\right) -  \bar F\left(v^{(\ell-1)}\right)\right]_j (t,x)$, $\ell=1,2,\cdots,n-1$, 
$j = 1, 2, \cdots, d+1$.

For notational simplicity, we write 
\begin{equation}\label{eq:bel}
\mathcal{V}(s; t, x) := \frac{1}{\sqrt{s-t}} \int_t^s \left[\sigma^{-1} D\tilde{Z}^{t,x}(r)\right]^{\top} dB(r), 
\end{equation}
and 
\begin{equation}\label{eq:pushing-penalty}
\mathcal{P}(T; t,x) := \int_t^T e^{-\beta (s-t)} \kappa \cdot d\left[D\tilde Y^{t,x}(s)\right].
\end{equation}
Then, by the definition of $\bar F$, we have, for $j = 1, 2, \cdots, d$, 
\begin{eqnarray}
\left[\bar F\left(v^{(0)}\right)\right]_{j+1}(t,x) &=& \E\Bigg[\int_t^T \frac{e^{-\beta(s-t)}}{\sqrt{s-t}} \tilde \cH\left(\tilde{Z}^{t,x}(s), v^{(0)}\right) \mathcal{V}_j(s; t,x) ds \nonumber \\
& & + \mathcal{P}_j(T; t,x) + \frac{e^{-\beta (T-t)}}{\sqrt{T-t}}\xi\left(\tilde{Z}^{t,x}(T)\right) \mathcal{V}_j(T; t,x) \Bigg] \nonumber \\
&=& \E\Bigg[\int_t^T \frac{e^{-\beta(s-t)}}{\sqrt{s-t}} \underbar{\em c}\left(\tilde{Z}^{t,x}(s)\right) \mathcal{V}_j(s; t,x) ds \nonumber \\
& & + \mathcal{P}_j(T; t,x) + \frac{e^{-\beta (T-t)}}{\sqrt{T-t}}\xi\left(\tilde{Z}^{t,x}(T)\right) \mathcal{V}_j(T; t,x)\Bigg], 
\label{eq:phi-v0}
\end{eqnarray}
where 
\begin{equation}
\underbar{\em c}(x) = \inf_{a \in \mathcal{A}} \left\{c(x, a)\right\},
\end{equation}
and for $\ell \in \N$,
\begin{align}\label{eq:phi-vl-phi-vl-1}
& \left[\bar F\left(v^{(\ell)}\right)-\bar F\left(v^{(\ell-1)}\right)\right]_{j+1}(t,x) \nonumber \\
=& \ \E\left[\int_t^T \frac{e^{-\beta(s-t)}}{\sqrt{s-t}} \left[\tilde \cH\left(\tilde{Z}^{t,x}(s), v^{(\ell)}\left(s,\tilde{Z}^{t,x}(s)\right)\right) - \tilde \cH\left(\tilde{Z}^{t,x}(s), v^{(\ell-1)}\left(s,\tilde{Z}^{t,x}(s)\right)\right) \right] \mathcal{V}_j (s; t, x) ds \right].
\end{align}
We also have 
\begin{eqnarray}
\left[\bar F\left(v^{(0)}\right)\right]_1(t,x) &=& \E\Bigg[\int_t^T e^{-\beta(s-t)} \underbar{\em c}\left(\tilde{Z}^{t,x}(s)\right) ds \nonumber \\
& & \int_t^T e^{-\beta (s-t)}\kappa\cdot d\tilde Y^{t,x}(s) + e^{-\beta (T-t)}\xi\left(\tilde{Z}^{t,x}(T)\right) \Bigg], 
\label{eq:phi-v0-1}
\end{eqnarray}
and for $\ell \in \N$,
\begin{align}
& \left[\bar F\left(v^{(\ell)}\right)-\bar F\left(v^{(\ell-1)}\right)\right]_1(t,x) \nonumber \\
=& \ \E\left[\int_t^T e^{-\beta(s-t)} \left[\tilde \cH\left(\tilde{Z}^{t,x}(s), v^{(\ell)}\left(s,\tilde{Z}^{t,x}(s)\right)\right) - \tilde \cH\left(\tilde{Z}^{t,x}(s), v^{(\ell-1)}\left(s,\tilde{Z}^{t,x}(s)\right)\right) \right]  ds \right].
\label{eq:phi-vl-phi-vl-1-1}
\end{align}
We now make use of the identities in Eqs. \eqref{eq:phi-v0}, \eqref{eq:phi-vl-phi-vl-1}, \eqref{eq:phi-v0-1} and \eqref{eq:phi-vl-phi-vl-1-1}
to develop a multi-level Picard scheme 
for approximating $(V^{(n)}, v^{(n)})(t,x)$. 

First, we develop unbiased estimators 
for the expectations in \eqref{eq:phi-v0} and \eqref{eq:phi-vl-phi-vl-1}. To this end, we treat  
integration in time as expectation involving a randomly sampled time point. 
More specifically, consider the integral $\int_t^s e^{-\beta(r-t)} (r-t)^{-1/2} dr$. 
It can be computed that 
\begin{equation}
\int_t^s e^{-\beta(r-t)} (r-t)^{-1/2} dr = \left\{
\begin{array}{ll}
\sqrt{\frac{\pi}{\beta}}\left(1-2\Phi\left(-\sqrt{2\beta(s-t)}\right)\right), & \beta > 0; \\
2\sqrt{s-t}, & \beta = 0.
\end{array}\right.
\end{equation}
where $\Phi(\cdot)$ is the distribution function of the standard normal (with mean $0$ and standard deviation $1$). 
This motivates us to consider sampling a random time point $S \in (t, T)$ 
with the following distribution:
\begin{equation}\label{eq:random-time}
\pr(S < s) = \left\{
\begin{array}{ll}
\frac{1-2\Phi\left(-\sqrt{2\beta(s-t)}\right)}{1-2\Phi\left(-\sqrt{2\beta(T-t)}\right)}, & \beta > 0; \\
\frac{\sqrt{s-t}}{\sqrt{T-t}}, & \beta = 0.
\end{array}\right.
\end{equation}
Also define the normalizing constant $C(\beta, T-t)$ by
\begin{equation}
C(\beta, T-t) = \left\{
\begin{array}{ll}
\sqrt{\frac{\pi}{\beta}}\left(1-2\Phi\left(-\sqrt{2\beta(T-t)}\right)\right), & \beta > 0; \\
2\sqrt{T-t}, & \beta = 0.
\end{array}\right.
\end{equation}
When the context is clear, we often drop the $\beta$ dependence and simply write $C(T-t) = C(\beta, T-t)$.
Then, we can rewrite Eq. \eqref{eq:phi-v0} as 
\begin{align}
& \left[\bar F\left(v^{(0)}\right)\right]_{j+1}(t,x) \nonumber \\
=& \ \E\Bigg[ C(\beta, T-t)\cdot \underbar{\em c}\left(\tilde{Z}^{t,x}(S)\right) \mathcal{V}_j(S; t,x)
+ \mathcal{P}_j(T; t,x) + \frac{e^{-\beta (T-t)}}{\sqrt{T-t}}\xi \left(\tilde{Z}^{t,x}(T)\right)\mathcal{V}_j(T; t,x)\Bigg], 
\label{eq:v0-ex}
\end{align}
Eq. \eqref{eq:phi-vl-phi-vl-1} as 
\begin{align}
& \left[\bar F\left(v^{(\ell)}\right)-\bar F\left(v^{(\ell-1)}\right)\right]_{j+1}(t, x) \nonumber \\
=& \ \E\left[C(\beta, T-t)\left(\tilde \cH\left(\tilde{Z}^{t,x}(S), v^{(\ell)}\left(S,\tilde{Z}^{t,x}(S)\right)\right) - \tilde \cH\left(\tilde{Z}^{t,x}(S), v^{(\ell-1)}\left(S,\tilde{Z}^{t,x}(S)\right)\right) \right) \mathcal{V}_j (S; t, x)\right], \label{eq:vl-vl-1-ex}
\end{align}
Eq. \eqref{eq:phi-v0-1} as 
\begin{align}
& \left[\bar F\left(v^{(0)}\right)\right]_1(t,x) \nonumber \\
=& \ \E\Bigg[ C(\beta, T-t)\cdot \sqrt{S-t}\cdot \underbar{\em c}\left(\tilde{Z}^{t,x}(S)\right) 
+ \int_t^T e^{-\beta (s-t)}\kappa\cdot d\tilde Y^{t,x}(s) + e^{-\beta (T-t)}\xi \left(\tilde{Z}^{t,x}(T)\right)\Bigg], 
\label{eq:v0-ex-1}
\end{align}
and Eq. \eqref{eq:phi-vl-phi-vl-1-1} as
\begin{align}
& \left[\bar F\left(v^{(\ell)}\right)-\bar F\left(v^{(\ell-1)}\right)\right]_1(t, x) \nonumber \\
=& \ \E\left[C(\beta, T-t)\sqrt{S-t}\left(\tilde \cH\left(\tilde{Z}^{t,x}(S), v^{(\ell)}\left(S,\tilde{Z}^{t,x}(S)\right)\right) - \tilde \cH\left(\tilde{Z}^{t,x}(S), v^{(\ell-1)}\left(S,\tilde{Z}^{t,x}(S)\right)\right) \right)\right], \label{eq:vl-vl-1-ex-1}
\end{align}
where $S$ follows the distribution in Eq. \eqref{eq:random-time}, 
independently from the process $\tilde{Z}^{t,x}$. 

Second, in order to obtain estimates for $\bar F\left(v^{(0)}\right)(t,x)$ 
and $\bar F\left(v^{(\ell)}\right)-\bar F\left(v^{(\ell-1)}\right)(t, x)$, $\ell = 1, 2, \cdots, n$, 
we need to be able to efficiently and accurately simulate the following tuple: 
\begin{equation}\label{eq:tuple}
  \left(S, \tilde{Z}^{t,x}(S), \tilde{Z}^{t,x}(T), \int_t^T e^{-\beta (s-t)}\kappa\cdot d\tilde Y^{t,x}(s), \mathcal{P}(T; t,x), \mathcal{V}(S; t,x), \mathcal{V}(T; t,x)\right);
  \end{equation}
see expressions in Eqs. \eqref{eq:v0-ex} -- \eqref{eq:vl-vl-1-ex-1}. 
In addition, we also need to be able to efficiently compute the Hamiltonian $\tilde \cH$ (recall its definition in Eq. \eqref{eq:hamiltonian-tilde}), 
since computing $\tilde \cH(x,p)$ for any given input $(x,p) \in \R_+^d \times \R^d$ requires us to solve an optimization problem\footnote{Note 
that the computation of the gradient-dependent nonlinear term, in our case $\tilde \cH(x,p)$, is often implicitly assumed to only take constant time 
in prior literature. In stochastic control applications, this need not be the case.}. 
These considerations motivate the following definitions. 
\begin{definition}[Polynomial-complexity simulation]\label{df:poly-sim}
A simulation algorithm that generates one sample of the tuple \eqref{eq:tuple}, exactly or approximately, 
has \emph{polynomial computational complexity (in the dimension $d$)}, if there exist positive constants $\alpha_1$ and $\alpha_2$ 
that are independent of $d$, $t$ and $x$, 
such that its expected runtime is no more than $\alpha_1\left(1+d^{\alpha_2}\right)$. 
\end{definition}
\begin{definition}\label{df:poly-opt}
An optimization algorithm for solving $\tilde \cH (x,p)$ given any input $(x,p) \in \R_+^d \times \R^d$ 
has \emph{polynomial computational complexity (in the dimension $d$)}, if there exist positive constants $\alpha_1$ and $\alpha_2$ 
that are independent of $d$, $x$, and $p$, 
such that its expected runtime is no more than $\alpha_1\left(1+d^{\alpha_2}\right)$. 
\end{definition}
Note that Definition \ref{df:poly-sim} does not require an algorithm to be accurate; in other words,  
an algorithm for simulating the tuple \eqref{eq:tuple} may have polynomial complexity and incur large simulation error at the same time. 
In Section \ref{sec:exact-simulation}, for the important special case 
where components of $\tilde Z^{t,x}$ are independent and $\kappa = 0$, 
we describe an exact simulation algorithm with polynomial complexity; in fact, the complexity is linear in $d$. 
In more general settings where Euler–-Maruyama approximations are needed, there is an inherent tradeoff between efficiency, 
as measured by runtime, and accuracy, as measured by simulation error. 
If the number of discretization steps in an Euler–-Maruyama scheme is polynomial in $d$, then the scheme has polynomial complexity. 
We leave the question of whether a polynomial number of discretization steps suffices to obtain 
efficient MLP estimators with any prescribed accuracy as an important future research direction.

 

We are now ready to describe the MLP scheme.

Let $\Theta = \cup_{k=1}^{\infty} \Z^k$ be a parameter space, 
and $n, M \in \N$ be given hyperparameters.
Then define for all $t \in [0, T)$, $x \in \R_+^d$, and $\theta \in \Theta$, 
the following (random) functions recursively:
\begin{equation}\label{eq:mlp-v0}
\left(V^\theta_{0,M}, v^\theta_{0,M}\right)(t,x)\equiv 0;
\end{equation}
and 
\begin{align}
& \left(V^\theta_{n,M}, v^\theta_{n,M}\right)(t,x) \\
=~& \frac{1}{M^n}\sum_{m=1}^{M^n} \Bigg[C(\beta, T-t) \cdot 
\underbar{\em c}\left(\tilde{Z}^{t,x}_{(\theta, 0, m)}\left(S^{(\theta,0,m)}\right)\right) \left(\sqrt{S^{(\theta,0,m)}-t}, \mathcal{V}^{(\theta,0,m)}\left(S^{(\theta, 0, m)}; t,x\right)\right) \label{eq:Fv0-1} \\
& + \left(\int_t^T e^{-\beta (s-t)}\kappa\cdot d\tilde Y^{t,x}_{(\theta, 0, m)}(s), \mathcal{P}^{(\theta, 0, m)}(T; t,x)\right) \\
& + \frac{e^{-\beta (T-t)}}{\sqrt{T-t}}\xi \left(\tilde{Z}^{t,x}_{(\theta, 0, m)}(T)\right)\left(\sqrt{T-t}, \mathcal{V}^{(\theta,0,m)}\left(T; t,x\right)\right) \Bigg] \label{eq:Fv0-2}\\
& + \sum_{\ell=1}^{n-1} \sum_{m=1}^{M^{n-\ell}} 
\frac{C(\beta, T-t)}{M^{n-\ell}} \left(\sqrt{S^{(\theta, \ell, m)}-t}, \mathcal{V}{(\theta,\ell,m)}\left(S^{(\theta, \ell, m)}; t,x\right)\right) \times \label{eq:Fvl-1} \\
& \times \Bigg[\tilde \cH \left(\tilde{Z}^{t,x}_{(\theta, \ell, m)}\left(S^{(\theta,\ell,m)}\right), v^{(\theta, \ell, m)}_{\ell, M}\left(S^{(\theta,\ell,m)}, \tilde{Z}^{t,x}_{(\theta, \ell, m)}\left(S^{(\theta,\ell,m)}\right)\right) \right) \label{eq:Fvl-2}\\
& ~~~- \tilde \cH \left(\tilde{Z}^{t,x}_{(\theta, \ell, m)}\left(S^{(\theta,\ell,m)}\right), v^{(\theta, -\ell, m)}_{\ell-1, M}\left(S^{(\theta,\ell,m)}, \tilde{Z}^{t,x}_{(\theta, \ell, m)}\left(S^{(\theta,\ell,m)}\right)\right) \right)\Bigg]. \label{eq:mlp-rec}
\end{align}
Here, $\left(S^\theta, \tilde{Z}^{t,x}_\theta(S^\theta), \tilde{Z}^{t,x}_\theta(T), \int_t^T e^{-\beta (s-t)}\kappa\cdot d\tilde Y^{t,x}_\theta(s), \mathcal{V}^\theta(S^\theta; t,x), \mathcal{P}^\theta(T; t,x), \mathcal{V}^\theta(T; t,x) \right)$ 
for different $\theta \in \Theta$ 
are i.i.d. copies of the tuple \eqref{eq:tuple}. 
The summation involved in expressions \eqref{eq:Fv0-1} and \eqref{eq:Fv0-2} is a Monte Carlo estimation 
of $\left[F\left(v^{(0)}\right)\right]_j(t,x)$. Similarly, the summation involved in expressions \eqref{eq:Fvl-1} -- \eqref{eq:mlp-rec} 
is a Monte Carlo estimate of the telescopic sum $\sum_{\ell=1}^{n-1} \left[F\left(v^{(\ell)}\right)-F\left(v^{(\ell-1)}\right)\right]_j(t, x)$. 
Note that the number of simulations performed at level $\ell$ to estimate $\left[F\left(v^{(\ell)}\right)-F\left(v^{(\ell-1)}\right)\right]_j(t, x)$ is $M^{n-\ell}$, 
which becomes fewer as $\ell$ increases. 
Also note that the MLP scheme is full-history recursive: In order to simulate one instance of $\left[v^\theta_{n,M}(t,x)\right]_j$, 
samples $v^{\tilde{\theta}}_{\ell, M}\left(S^{\tilde{\theta}}, \tilde{Z}^{t,x}_{\tilde{\theta}}\left(S^{\tilde{\theta}}\right)\right)$ are required for all $\ell < n$ 
(see expressions \eqref{eq:Fvl-2}--\eqref{eq:mlp-rec}), 
each of which in turn requires simulations of $v^{\theta'}_{\ell', M}$ at random times and states for all $\ell' < \ell$, and so on and so forth. 

\subsection{Complexity Bounds}

Here we state the main result of this section, 
on the computational complexity of the MLP scheme. 
First, we state an assumption on the Lipschitz property of the Hamiltonian $\tilde \cH$. 

\begin{assumption}\label{asmp:lipschitz-H}
There exist positive constants $C_1, C_2, \cdots, C_d$ such that 
for any $p, \tilde{p} \in \R^d$, and for any $y \in \R_+^d$, 
\begin{equation}
\left|\tilde \cH(y,p) - \tilde \cH(y,\tilde{p})\right| \leq \sum_{j=1}^d C_j |p_j - \tilde{p}_j|, 
\end{equation}
where we recall the definition of the Hamiltonian $\tilde \cH$ in \eqref{eq:hamiltonian-tilde}. 
Furthermore, there exists a positive constant $C_{\tilde \cH}\geq 1$ independent of the problem dimension $d$, 
such that $C_1 +\cdots +C_d \leq C_{\tilde \cH}$.
\end{assumption}
\begin{remark}
Let us provide some remarks on Assumption \ref{asmp:lipschitz-H}. On the one hand, 
since the action set $\mathcal{A}$ is compact, and $b$ and $\tb$ are bounded by Assumptions \ref{as:b} and \ref{as:tb}, 
it is not difficult to see that there must exist $C'> 0$ that is independent of $d$ such that 
\[
\left|\tilde \cH (y, p)-\tilde \cH (y,\tilde{p})\right| \leq C' \sum_{j=1}^d |p_j-\tilde{p}_j|, 
\]
for any $p, \tilde{p} \in \R^d$, and for any $y \in \R_+^d$. 
If we take $C_j = C'$ for all $j$, then Assumption \ref{asmp:lipschitz-H} will not be satisfied. 
By upper bounding the sum $C_1 + \cdots + C_d$ by a constant $C_{\tilde \cH}$ that does not grow with $d$, 
we are, roughly speaking, constraining the degree of nonlinearity of the Hamiltonian $\tilde \cH$, 
which intuitively means that the reference processes $\tilde Z^{t,x}$ under drift $\tb$ is not far from 
the optimally controlled process. 
This degree of nonlinearity directly affects the required complexity to approximate 
$V(t,x)$ and $D_x V(t,x)$ to within desired accuracy. 
On the other hand, practically speaking, as we will see in the numerical experiments (Section \ref{sec:numerical}), 
even when $C_{\tilde \cH}$ is large ($C_{\tilde \cH}$ can go up to around $20$ in the experiments), 
the MLP scheme can still perform well. 
\end{remark}

Our next assumption is more technical, and concerns the sizes of $\cV_j(s; t,x)$, $j = 1, 2, \cdots, d$, 
where we recall the definition of $\cV(s; t,x)$ in Eq. \eqref{eq:bel}, a key term that appears in the Bismut-Elworthy-Li type 
characterization of $\tilde v$ as the fixed point of map $F$.
\begin{assumption}\label{as:sigma-D}
There exists a positive constant $C_{\sigma, D}$ that is independent of the problem dimension $d$, such that 
for all $j=1,2, \cdots, d$, $x\in \R_+^d$, $s, t \in [0, T)$ with $s > t$, 
\begin{equation}\label{eq:sigma-D}
\|\sigma^{-1} D_j \tilde Z^{t,x}(s)\|_2^2 \leq C_{\sigma, D}.
\end{equation}
\end{assumption}
\begin{remark} 
(a) An immediate consequence of Assumption \ref{as:sigma-D} is that the variance of $\cV_j(s; t,x)$ is uniformly upper bounded by $C_{\sigma, D}$, i.e., 
for all $j=1,2, \cdots, d$, $x\in \R_+^d$, $s, t \in [0, T)$ with $s > t$, 
\begin{equation}
E\left[\cV_j^2(s; t,x)\right] = \frac{1}{s-t}\E\left[\int_t^s \|\sigma^{-1} D_j \tilde Z^{t,x}(r)\|_2^2 dr\right] \leq C_{\sigma, D}.
\end{equation}
(b) Assumption \ref{as:sigma-D} shares resemblance to Lemma \ref{lem:bound-DZ}, but cannot be derived from the lemma, 
and is in some sense stronger than what Lemma \ref{lem:bound-DZ} states. 
An immediate consequence of Lemma \ref{lem:bound-DZ} is that $\|D_j \tilde Z^x(t)\|_\infty$ is uniformly upper bounded by 
a constant independent from the problem dimension $d$, which only implies that $\|D_j \tilde Z^x(t)\|_2^2 \sim O(d)$.

(c) Since Assumption \ref{as:sigma-D} cannot be derived from other more primitive assumptions stated so far in the paper, 
it is natural to wonder how restrictive it is. We argue that this is not a restrictive assumption, 
by showing that it is satisfied by a broad class of reference processes. Consider reference processes $\tilde Z^{t,x}$ 
that are reflected Brownian motions, whose drift coefficients are constant and not state dependent, i.e., 
there exists a constant vector $\tb$ such that 
\begin{equation}
\tilde{Z}^{t,x}(s) = x + \tb (s-t) + \sigma (B(s) - B(t)) + R \tilde{Y}^{t,x}(s), \quad s \in [t,T].
\end{equation}
Then, by Theorem 6 of \cite{KellaWhitt1996}, also Theorem 1.1 of \cite{KellaRamasubramanian2012},  
for each $j=1, 2, \cdots, d$, the derivative process $D_j \tilde Z^{t,x}$ is nonnegative, 
and satisfies $\|D_j \tilde Z^{t,x}\|_1 \leq 1$. By Assumption \ref{as:sigma}, 
it then easily follows that 
\[
\|\sigma^{-1} D_j \tilde Z^{t,x}(r)\|_2^2 \leq \left[\|\sigma^{-1}\|_2 \|D_j \tilde Z^{t,x}(r)\|_2\right]^2 
\leq \left[\|\sigma^{-1}\|_2 \|D_j \tilde Z^{t,x}(r)\|_1\right]^2
\leq \|\sigma^{-1}\|_2^2 \leq C_\sigma^2, 
\]
so that Assumption \ref{as:sigma-D} holds.
\end{remark}
The final assumption before we state the main result concerns the growth rate of the pushing penalties $\kappa$. 
\begin{assumption}\label{as:kappa}
Recall the constants $\alpha_w$ and $\alpha_0$ from Assumption \ref{as:cost-poly}. Then
\begin{equation}\label{eq:kappa}
\|\kappa\|_1 \leq \alpha_w \left(1+\left(\log d\right)^{\alpha_0/2}\right).
\end{equation}
\end{assumption}
Assumption \ref{as:kappa} keeps the growth rate of $\kappa$ no larger than that of the cost functions $c$ and $\xi$ 
to facilitate the analysis of the MLP scheme. The assumption can be made to hold by a scaling of the 
problem parameters, if needed; cf. comparable assumptions made in prior literature, e.g., Assumption A4 in \cite{BlanchetChenSiGlynn2021}.

We also introduce the following notation. For random functions $\bar V : [0,T]\times \R_+^d \to \R$ and $\bar v : [0,T]\times \R_+^d \to \R^d$, 
define, for $s \in [0,T]$: 
\begin{equation}\label{eq:vertiii}
  \vertiii{(\bar V,\bar v)}_{s} := \sup_{x \in \mathbb{R}_+^d} 
  \frac{\mathbb{E}[ \bar V(s,x)^2 ]^{1/2}}{w(x)} \vee \max_{j=1,\dots,d} \sup_{x \in \mathbb{R}_+^d} 
  \frac{(T - s)^{1/2} \cdot \mathbb{E}[ \bar v_j^2(s,x)^2 ]^{1/2}}{w(x)}, 
\end{equation}
where we recall the definition of $w(x)$ in Eq. \eqref{eq:w}. 
In other words, $\vertiii{(\bar V,\bar v)}_{s}^2$ is the maximum weighted squared mean of $\bar V$ and $\bar v_j$, $j=1,2,\cdots,d$, over all $x \in \R_+^d$, 
at a given time $s$.
We are now ready to state the main result of this section, whose proof is deferred to Appendix \ref{sec:proof-main1}.
\begin{theorem}\label{thm:main1}
Suppose Assumptions \ref{asmp:lipschitz-H}, \ref{as:sigma-D}, and \ref{as:kappa} hold. 
Suppose that there is an exact simulation algorithm for the tuple \eqref{eq:tuple} with polynomial complexity (cf. Definition \ref{df:poly-sim}), 
and an optimization algorithm for solving $\tilde \cH$ with polynomial complexity (cf. Definition \ref{df:poly-opt}).
Recall the norm $\|\cdot\|_0$ defined in Eq. \eqref{eq:norm-rho} with $\rho = 0$.
Fix $\veps>0$ and $t \in [0,T)$. Then there exist integers $n, M \in \N$ such that under the MLP scheme \eqref{eq:mlp-v0}--\eqref{eq:mlp-rec}, 
\begin{equation}\label{eq:error}
  \frac{\vertiii{\left(V_{n,M}^{0}, v_{n,M}^{0}\right) - \left(\tilde V, \tilde{v}\right)}_{t}}{\max\{1, \|\tilde v\|_0\}} \leq \veps,
\end{equation}
where we call the definitions of $\tilde v$ and $\tilde V$ in Eqs. \eqref{eq:fixed-point-tildev} and \eqref{eq:tildeV}, respectively.
Moreover, the expected runtime needed to generate one realization of $\big(V^{0}_{n,M}(t,x), v^{0}_{n,M}(t,x)\big)$ so that Ineq. \eqref{eq:error} 
holds is upper bounded by
\begin{equation}\label{eq:sample-comp}
C_{RT} (1+d^{\alpha_2}) \veps^{-4},
\end{equation}
where $\alpha_2$ is from Definitions \ref{df:poly-sim} and \ref{df:poly-opt}, 
and $C_{RT}$ is independent of $d$ and $\veps$. 
\end{theorem}
\begin{remark}
(i)  The constant $C_{RT}$ takes the form in expression \eqref{eq:C_RT}. A closer examination of the expression \eqref{eq:C_RT} 
and calculus show that $C_{RT}$ is super exponential in the time horizon $T$ and the constant $C_{\tilde \cH}$ from Assumption \ref{asmp:lipschitz-H}. 
On the one hand, this suggests that if $T$ or $C_{\tilde \cH}$ is large, the performance of the MLP scheme would suffer. 
On the other hand, in our numerical experiments, we observe that the MLP scheme can perform quite well even under a large $C_{\tilde \cH}$, for example,  
implying that the practical performance of the scheme can be much better that what the theory predicts. 

(ii) An inspection of the complexity analysis of the MLP scheme in Section \ref{sec:proof-main1} 
shows that to to get within $\veps$ prescribed error, i.e., for Ineq. \eqref{eq:error} to hold, 
the scheme makes $O\left(C_{RT} \veps^{-4}\right)$ calls to the subroutine that simulates the tuple \eqref{eq:tuple} 
and solves the optimization problem in \eqref{eq:hamiltonian-tilde} $O\left(C_{RT} \veps^{-4}\right)$ times. 
The extra $(1+d^{\alpha_2})$ factor in the sample complexity bound \eqref{eq:sample-comp} only comes from complexities 
of the simulation algorithm for the tuple \eqref{eq:tuple} and the optimization algorithm for $\tilde \cH$. 

(iii) A restriction of Theorem \ref{thm:main1} is that we require the simulation of \eqref{eq:tuple} to be exact. 
The case where exact simulation is not possible, but where Euler–-Maruyama scheme can be used to simulate the tuple \eqref{eq:tuple} with a (small) bias, 
is an important future research topic. Some works have been carried out in this direction for regular diffusions 
-- see, e.g., \cite{NeufeldNguyenWu2025,NeufeldWu2025} -- but theoretically, 
it is unclear what the impact of simulation errors due to Euler–-Maruyama discretizations is  
on the MLP scheme for drift controls with reflections. 
Our numerical experiments suggest that the impact of time discretization may be minimal; 
see results for a two-dimensional test problem at the end of Section \ref{ssec:computational-results}.

(iv) The flavor of our result, as reflected in Ineq. \eqref{eq:error}, is somewhat different from similar results in prior literature, 
which often establish {\em a priori} bounds on $\|v\|_0$ to remove the denominator $\max\{1, \|\tilde v\|_0\}$ in \eqref{eq:error}. 
We are not particularly concerned with obtaining such bounds; we prefer to treat $\|v\|_0$ as the baseline, 
and obtain error bounds as a fraction of the baseline. Adopting this latter approach (or not) does not alter  
the main qualitative insight, in that under appropriate conditions, the MLP scheme provably overcomes the curse of dimensionality. 
\end{remark}

\section{Numerical Experiments}\label{sec:numerical}

In this section, we study the numerical performance, stability, and scalability of the multilevel Picard (MLP) scheme 
for a class of drift control problems with reflections, 
under a class of reference processes that allow exact simulation. 
Detailed numerical results are presented after describing the benchmark problems and experimental setup, 
where we consider problems in dimensions two, five, and twenty. 
We conclude this section with additional experiments involving alternative reference processes.


\subsection{Benchmark Problem}\label{ssec:benchmark}
We evaluate the performances of the proposed multilevel Picard (MLP) scheme on the following drift control problem with reflections. 
Let $T > 0$ be fixed.
The state process $Z(\cdot)$ evolves according to a $d$-dimensional diffusion process with reflections as follows:
\begin{equation}\label{eq:pss-dcr1}
    Z(t) = x + \gamma t + \int_0^t Ga(s) ds + \sigma B(t) + Y(t), \quad t \in [0,T],
\end{equation}
where $x \in \R_+^d$ is the initial state, $\gamma \in \R^d$ is a ``base'' drift, $G$ is a given $d\times K$ matrix, 
$a(\cdot)$ is a $K$-dimensional control process taking values in the compact action set $\mathcal{A}:=\left[0,C_\mathcal{A}\right]^K$, 
$B(\cdot)$ is a standard $d$-dimensional Brownian motion, $\sigma$ is a diagonal diffusion matrix with diagonal entries $\sigma_j > 0$, $j=1,2,\cdots,d$, 
and $Y(\cdot)$ is the $d$-dimensional Skorokhod regulator with normal reflections 
that keeps $Z(\cdot) \in \R_+^d$; i.e., the reflection matrix $R$ is the identity matrix. 
Let $h \in \R_+^d$ be the vector of holding costs, and let $\beta \geq 0$ be the discount rate. 
Given a time-state pair $(t,x) \in [0,T] \times \R_+^d$, we wish to compute the value function $V(t,x)$, defined as 
\begin{equation}\label{eq:pss-dcr2}
V(t,x) = \inf_{a(\cdot)} J(t,x,a), 
\end{equation}
where the infimum is taken over all non-anticipating control processes $a(\cdot)$, and where $J(t,x,a)$ is given by
\begin{equation}\label{eq:pss-dcr3}
    J(t,x,a) := \mathbb{E} \left[ \int_t^T e^{-\beta (s-t)}\, h^{\top} Z(s) ds\Big| Z(t)=x \right],
\end{equation}

We are interested in the problem thus described for several reasons. 
On the technical side, a primary reason is that the problem \eqref{eq:pss-dcr1} -- \eqref{eq:pss-dcr3} appears to be 
a simplest non-trivial instance of drift controls with reflections; for example,  
under the zero control $a(\cdot) \equiv 0$, the state evolves as independent reflected Brownian motions in different dimensions. 
Indeed, recall that to compute $V(t,x)$ using MLP, we need to be able to efficiently simulate 
(quantities associated with) a reference process $\tilde Z^{t,x}$; cf. Definition \ref{df:poly-sim} and the ensuing discussion. 
Because the reflection matrix $R$ is simply the identity matrix, 
we do not need to solve general linear complementarity problems \cite{AtaHarrisonSi2024} to simulate the sample paths of $\tilde Z^{t,x}$,
allowing us to focus solely on the sampling complexity of the MLP scheme. 
There is also considerable flexibility in choosing what reference process to use, because of the freedom in choosing the drift $\tb$, 
but if we let 
\begin{equation}\label{eq:rbm-reference}
\tilde Z^{t,x}(s) = x + \gamma (s-t)  + \sigma \left(B(s)-B(t)\right) + \tilde Y^{t,x}(s), s \in [t, T], 
\end{equation}
i.e., set $a(\cdot) \equiv 0$, 
then $\tilde Z_j^{t,x}$ are simply independent one-dimensional reflected Brownian motions with constant drift and diffusion coefficient. 
As a consequence, the tuple in expression \eqref{eq:tuple}, the basic building block for the MLP scheme, can be simulated efficiently and exactly. 
To see this, first note that because 
our drift control problem does not penalize reflections on the boundary, and has zero terminal condition, 
we do not need to explicitly simulate $\mathcal{P}(T; t,x)$ and $\mathcal{V}(T; t,x))$
(recall their definitions in Eqs. \eqref{eq:bel} and \eqref{eq:pushing-penalty}). 
Second, given the randomly sampled time point $S$, the pairs $(Z^{t,x}_j(S), \mathcal{V}_j(S; t,x))$ are independent, $j = 1, 2, \cdots, d$. 
Third, under the reference process $\tilde Z^{t,x}$ in Eq. \eqref{eq:rbm-reference}, the weights $\mathcal{V}_j(s; t, x)$ have simple closed form. 
More specifically, we have
\begin{align}
\mathcal{V}_j(s; t, x) &~:= \frac{1}{\sqrt{s-t}} \int_t^s \left[\sigma^{-1} D_j\tilde{Z}^{t,x}(r)\right]^{\top} dB(r) \nonumber \\
&~= \frac{1}{\sqrt{s-t}} \int_t^s \sigma_j^{-1}\mbOne_{\{r < \tau_j\}} dB_j(r) = \frac{B_j(\tau_j\wedge s) - B_j(t)}{\sigma_j \sqrt{s-t}}, \label{eq:pss-mlp-V}
\end{align}
where $\tau_j := \inf\{r \ge t : \tilde Z_j^{t,x}(r) = 0\}$ is the first hitting time to zero of $\tilde Z_j^{t,x}$ since time $t$. 
The second equality above follows from the following two facts: (a) the diffusion coefficient matrix $\sigma$ is diagonal; and 
(b) because $\tilde Z_j^{t,x}$ are independent reflected Brownian motions with initial condition $x_j$ at time $t$, we have
\[
\frac{\partial \tilde Z_{j'}^{t,x}(r)}{\partial x_{j}} = 0 \quad \text{if } j' \neq j, \quad \text{and} \quad
\frac{\partial \tilde Z_j^{t,x}(r)}{\partial x_j} = \left\{ \begin{array}{ll}
1 & \text{if } r < \tau_j, \\
0 & \text{if } r > \tau_j.
\end{array}\right.
\]
Thus, the simulation of the tuple in expression \eqref{eq:tuple} reduces to the simulation of $S$, followed by 
independent simulations of 
\[
\left(\tilde Z^{t,x}_j(s), \frac{B_j(\tau_j\wedge s) - B_j(t)}{\sigma_j \sqrt{s-t}}\right)
\]
given $S=s$, $j = 1, 2, \cdots, d$. 
It follows that the key is to simultaneously simulate $\tau_j$, $Z^{t,x}_j(s)$, and $B_j(\tau_j\wedge s)$ 
efficiently and exactly, $j=1,2,\cdots,d$. We defer the detailed algorithm to Appendix \ref{sec:exact-simulation}, 
but let us point out here that at a high level, the simulation of the triple 
\begin{equation}\label{eq:triple}
\left(\tau_j, \tilde Z^{t,x}_j(s), B_j(\tau_j\wedge s)\right)
\end{equation} 
involves simulations of the marginals of one-dimensional 
reflected Brownian motions and so-called Brownian meanders (see, e.g., Chapter 1 in \cite{DevroyeKarasozenKohlerKorn2010} and references therein).  

Another advantage of the reference process $\tilde Z^{t,x}$ in Eq. \eqref{eq:rbm-reference} is that
the Hamiltonian $\tilde \cH$ defined in Eq. \eqref{eq:hamiltonian-tilde} can be efficiently computed 
(cf. Definition \ref{df:poly-opt}). 
Indeed, it is not difficult to check that 
\begin{align}
\tilde \cH (x, p) &~= \inf_{a \in \left[0, C_{\mathcal{A}}\right]^K} \left\{\left(Ga\right)^\top p + h^\top x\right\} \nonumber\\
&~= \inf_{a \in \left[0, C_{\mathcal{A}}\right]^K} \left\{\left(Ga\right)^\top p \right\} + h^\top x \nonumber\\
&~= \inf_{a \in \left[0, C_{\mathcal{A}}\right]^K} \left\{a^\top \left(G^\top p\right) \right\} + h^\top x \nonumber \\
&~= C_{\mathcal{A}} \sum_{k=1}^K \min\left\{\left(G^\top p\right)_k, 0\right\} + h^\top x. \label{eq:pss-mlp-cH}
\end{align}
As we will see in the sequel, the matrix $G$ is sparse in all our test problems, 
allowing fast computation of $G^\top p$, hence that of $\tilde \cH$ as well. 
Again, the fact that $\tilde \cH$ can be easily computed allows us to focus on the sampling complexity of the MLP scheme. 

On the application side, the problem \eqref{eq:pss-dcr1} -- \eqref{eq:pss-dcr3} 
is closely related to the dynamic scheduling of parallel server systems (PSSs), 
an important dynamic control problem in operations management (see, e.g., \cite{PesicWilliams2016,HarrisonLopez1999}). 
PSSs form an important subclass of stochastic processing networks \cite{Harrison2000,Harrison2006Corr}, 
a widely used framework for modeling dynamic resource allocation scenarios that arise in diverse sectors, 
including manufacturing, computer systems and service systems. 
Rather than providing a general and thorough description of parallel server systems and its connections with our problem, 
which would distract us from the main goal of the section, 
we instead point out the high-level connections as we introduce details of the test instances in Section \ref{ssec:test-instances}.

We end this subsection with a description of the MLP scheme for the problem \eqref{eq:pss-dcr1} -- \eqref{eq:pss-dcr3}, 
under the reference process $\tilde Z^{t,x}$ defined in Eq. \eqref{eq:rbm-reference}. 
In addition to the standard MLP scheme, we also implement a variance reduction technique 
for the problem \eqref{eq:pss-dcr1} -- \eqref{eq:pss-dcr3}, to be described shortly. 

Recall that $\Theta = \cup_{k=1}^{\infty} \Z^k$ is a parameter space, 
and $n, M \in \N$ are given hyperparameters.
Define for all $t \in [0, T)$, $x \in \R_+^d$, and $\theta \in \Theta$ 
the following (random) functions recursively:
\begin{equation}\label{eq:pss-mlp-v0}
\left(V^\theta_{0,M}, v^\theta_{0,M}\right)(t,x)\equiv 0;
\end{equation}
and, for $j = 1, 2, \cdots, d$, 
\begin{align}
&~\left(V^\theta_{n,M}, v^\theta_{n,M}\right)(t,x) \\
=&~\sum_{m=1}^{M^n} \frac{C(\beta, T-t)}{M^n} \cdot \Bigg(\sqrt{S^{(\theta,0,m)}-t} \cdot \left[h^\top \tilde Z^{t,x}_{(\theta, 0, m)} \left(S^{(\theta, 0, m)}\right)\right], \label{eq:pss-mlp-v0-1}\\
&~\qquad \qquad \qquad \qquad h \odot \tilde Z^{t,x}_{(\theta, 0, m)} \left(S^{(\theta, 0, m)}\right) \odot \mathcal{V}^{(\theta,0,m)}\left(S^{(\theta, 0, m)}; t,x\right)\Bigg) \label{eq:pss-mlp-v0-2}\\
&~+ \sum_{\ell=1}^{n-1} \sum_{m=1}^{M^{n-\ell}} 
\frac{C(\beta, T-t)}{M^{n-\ell}} \cdot \left(\sqrt{S^{(\theta, \ell, m)}-t}, \mathcal{V}{(\theta,\ell,m)}\left(S^{(\theta, \ell, m)}; t,x\right)\right)\times \label{eq:pss-mlp-v0-3}\\
&~\times \Bigg[\bar \cH \left(v^{(\theta, \ell, m)}_{\ell, M} \right)
 - \bar \cH \left(v^{(\theta, -\ell, m)}_{\ell-1, M} \right)\Bigg]\left(\tilde Z^{t,x}_{(\theta, \ell, m)} \left(S^{(\theta, \ell, m)}\right), S^{(\theta, \ell, m)}\right). \label{eq:pss-mlp-v0-4}
\end{align}
Here, for two $d$-dimensional vectors $x$ and $y$, 
$x \odot y := \left(x_jy_j\right)_{j=1}^d$ denotes their Hadamard product,  
and for different $\theta \in \Theta$, $\left(S^\theta, \tilde{Z}^{t,x}_\theta(S^\theta), \mathcal{V}^\theta(S^\theta; t,x)\right)$
are i.i.d. copies of the tuple $\left(S, \tilde{Z}^{t,x}(S), \mathcal{V}(S; t,x)\right)$, 
where we recall the definition of $S$ in Eq. \eqref{eq:random-time}, 
that of $\tilde{Z}^{t,x}$ in Eq. \eqref{eq:rbm-reference}, 
and the simplified form of $\mathcal{V}$ in Eq. \eqref{eq:pss-mlp-V}. 
In Eq. \eqref{eq:pss-mlp-v0-4}, $\bar \cH(p)$ is defined to be
\begin{equation}\label{eq:pss-mlp-cH-bar}
\bar \cH(p) := C_{\mathcal{A}} \sum_{k=1}^K \min\left\{\left(G^\top p\right)_k, 0\right\} = \tilde \cH (x, p) - h^\top x.
\end{equation}
By Eq. \eqref{eq:pss-mlp-cH}, 
$\tilde \cH (x, p) - \tilde \cH (x, \tilde{p}) = \bar \cH(p) - \bar \cH(\tilde{p})$, 
so Eq. \eqref{eq:pss-mlp-v0-4} is simply the reduction of Eqs. \eqref{eq:Fvl-2} -- \eqref{eq:mlp-rec} to the problem described here. 
Finally, the term \eqref{eq:pss-mlp-v0-1} -- \eqref{eq:pss-mlp-v0-2} is a variance-reduced version of the corresponding term \eqref{eq:Fv0-1} of Section \ref{sec:mlp}. 
To see this, not that a direct reduction of \eqref{eq:Fv0-1} to the problem \eqref{eq:pss-dcr1} -- \eqref{eq:pss-dcr3} results in the following expression: 
\[
  \sum_{m=1}^{M^n} \frac{C(\beta, T-t)}{M^n} \cdot \left[h^\top \tilde Z^{t,x}_{(\theta, 0, m)} \left(S^{(\theta, 0, m)}\right)\right] \cdot \left(\sqrt{S^{(\theta,0,m)}-t}, \mathcal{V}^{(\theta,0,m)}\left(S^{(\theta, 0, m)}; t,x\right)\right),
\]
which is different from \eqref{eq:pss-mlp-v0-1} -- \eqref{eq:pss-mlp-v0-2}. 
From the simplification of $\mathcal{V}$ in Eq. \eqref{eq:pss-mlp-V}, we see that $\mathcal{V}_j$ 
is independent of $\tilde Z_i^{t,x}$, whenever $i \neq j$. 
Furthermore, as an Itô integral, $\E[\mathcal{V}_j(s; t,x)] = 0$ for all $j = 1, 2, \cdots, d$, 
so we have 
$\E\left[h_i \tilde Z_i^{t,x}(s) \mathcal{V}_j(s; t,x)\right] = 0$ whenever $i \neq j$.
Thus, the term \eqref{eq:pss-mlp-v0-2} is an variance-reduced estimator of 
\(
\E\left[\left(h^\top \tilde Z^{t,x}(S) \right) \cdot \mathcal{V}(S; t,x)\right].
\)
Indeed, in the numerical experiments, we observe that the estimator \eqref{eq:pss-mlp-v0-2} is substantially more efficient 
than the vanilla estimator $\left(h^\top \tilde Z^{t,x}(S) \right) \cdot \mathcal{V}(S; t,x)$, allowing faster convergence of the MLP scheme.

\subsection{Test Instances}\label{ssec:test-instances}

We consider three test problems, in dimensions $2$, $5$ and $20$, respectively. 
For a $d$-dimensional test example, $d=2, 5, 20$, we fix the drift vector $\gamma = -\bOne_d$, 
where $\bOne_d$ is the $d$-dimensional vector of all ones, 
fix the diffusion matrix $\sigma = I_d$, where $I_d$ is the $d\times d$ identity matrix, 
and fix $G = (G_{ij})$ to be a $d\times (d-1)$ matrix defined by
\begin{equation}\label{eq:G}
G_{ij} = \left\{\begin{array}{ll}
    0 & \text{ if } i<j \text{ or } i>j+1; \\
    1 & \text{ if } i=j; \\
    -1 & \text{ if } i=j+1,
\end{array}
    \right.
\end{equation}
so that the dimension of the control space is $K = d-1$.
We also fix the time horizon $T = 0.2$, but allow the drift upper bound $C_{\mathcal{A}}$, and the holding cost vector $h$ to vary 
for each test problem. 

Under the matrix $G$ defined in \eqref{eq:G}, we derive a formal interpretation of the optimal policy structure. 
For notational convenience, we use $v(t,x)$ to denote the state gradient of the value function $V$. 
Recall the expression for the Hamiltonian $\tilde \cH$ 
derived in Eq. \eqref{eq:pss-mlp-cH}. We then have 
\begin{align}
  \tilde \cH (x, v) &~= \inf_{a \in \left[0, C_{\mathcal{A}}\right]^{d-1}} \left\{a^\top \left(G^\top v\right) \right\} + h^\top x \nonumber \\
  &~= C_{\mathcal{A}} \sum_{k=1}^{d-1} \min\left\{\left(G^\top v\right)_k, 0\right\} + h^\top x \nonumber \\
  &~= C_{\mathcal{A}} \sum_{k=1}^{d-1} \min\left\{v_k-v_{k+1}, 0\right\} + h^\top x. \label{eq:pss-mlp-v-v}
\end{align}
Eq. \eqref{eq:pss-mlp-v-v} suggests that the following bang-bang type optimal policy structure: 
When $v_k-v_{k+1} > 0$, set $a_k = 0$; otherwise, set $a_k = C_{\mathcal{A}}$, $k=1, 2, \cdots, d-1$. 
Though the optimal policy structure is simple and intuitive, the main difficulty in computing it is that $v$ is unknown and needs to be estimated. 

A $d$-dimensional test problem with $G$ defined in Eq. \eqref{eq:G} and the formal policy structure derived above 
also have natural connections to a PSS 
with $d$ servers and $d$ buffers, in a heavy-traffic regime to be described, where server $i$ can serve buffers $i$ and $i+1$, for $i = 1, 2, \cdots, d-1$, 
and server $d$ can only serve buffer $d$; see Figure \ref{fig:parallel-server-system} for a schematic illustration. 
In the process flexibility literature, the structure is known as an open chain (see, e.g., \cite{WangWangZhang2021}), and 
in the case $d=2$, the system is often called the $N$-network in the queueing network literature \cite{Harrison1998,BellWilliams2001,TezcanDai2010,GhamamiWard2013}. 
The heavy-traffic regime is such that all servers have unit service rate, irrespective of the buffer they are serving, and 
that all arrival rates to the buffers are equal to $1-1/r$, with $r \to \infty$. 
Intuitively speaking, because buffer-$i$ arrival rate is close to the server-$i$ service rate,  
server $i$ will need to serve buffer $i$ most of the time, 
but can occasionally serve buffer $i+1$ at a (diffusion-scale) rate up to $C_{\mathcal{A}}$
(in which case server $i$ should serve buffer $i+1$ at the maximum rate possible, by the optimal policy structure described above), 
at the expense of letting buffer $i$ content 
grow at the same rate, $i=1, 2, \cdots, d-1$. Thus, the key trade-off is to decide whether to let server $i$ serve buffer $i$ or $i+1$, $i=1, 2, \cdots, d-1$, 
in order to balance the contents in the respective buffers. 

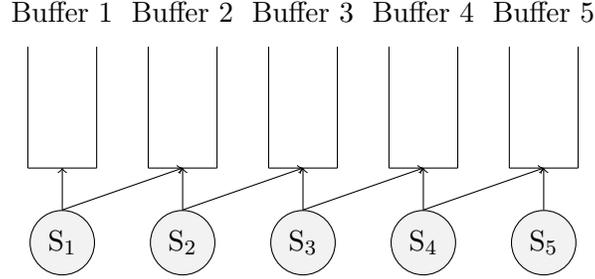
\begin{figure}[t]
  \centering
  \begin{tikzpicture}[x=1.6cm, y=1.0cm]
    \foreach \i in {1,...,5} {
      \node[minimum width=0.9cm, minimum height=1.6cm] (B\i) at (\i,1.8) {};
      \draw (B\i.south west) -- (B\i.south east);
      \draw (B\i.south west) -- (B\i.north west);
      \draw (B\i.south east) -- (B\i.north east);
      \node[anchor=south, yshift=2mm] at (B\i.north) {$\mathrm{Buffer}\ \i$};
    }

    \foreach \i in {1,...,5} {
      \node[draw, circle, minimum size=7mm, fill=gray!10] (S\i) at (\i,0) {$\mathrm{S}_{\i}$};
    }

    \foreach \i in {1,...,4} {
      \pgfmathtruncatemacro{\ipone}{\i+1}
      \draw[->] (S\i.north) -- (B\i.south);
      \draw[->] (S\i.north) -- (B\ipone.south);
    }
    \draw[->] (S5.north) -- (B5.south);

  \end{tikzpicture}
  \caption{Illustration of a $5$-dimensional open-chain system with $G$ defined in \eqref{eq:G}. Server $i$ connects to buffers $i$ and $i+1$ for $i=1,\cdots,4$, and server $5$ connects to buffer $5$ only.}
  \label{fig:parallel-server-system}
\end{figure}

\paragraph{Benchmark problems and control policies.} In general, none of our test problems can be solved analytically. Thus, 
we seek benchmark problems and control policies with competitive performances that are also relatively easy to simulate and estimate, 
against which we compare the performances of the MLP scheme. 
To this end, we consider a so-called equivalent workload formulation (EWF) 
of a formal diffusion approximation 
for the dynamic scheduling of the PSS under the heavy-traffic regime thus described, 
which can be derived using the approach developed in \cite{PesicWilliams2016}; we omit details here.
In the EWF, the state $Z(\cdot)$ evolves according to a $d$-dimensional process with RCLL sample paths as follows\footnote{
  Strictly speaking, in the EWF formulated in \cite{PesicWilliams2016}, the process $Z(\cdot)$ evolves according to 
  \[
    Z(t) = x -\bOne_d t + B(t) + \tilde G U(t) \in \R_+^d, \quad \text{for all } t\geq 0,
  \]
  where $U(\cdot)$ is non-decreasing with $U(0) = 0$, and 
  $\tilde G = [G \ e_d]$ is the $d\times d$ matrix formed by the concatenation of $G$ and the $d$th standard unit vector $e_d = (0,\cdots,0,1)^\top$. 
  However, it is not difficult to show that the representation in Eq. \eqref{eq:pss-singular-control} is equivalent.}: 
\begin{equation}\label{eq:pss-singular-control}
Z(t) = x -\bOne_d t + B(t) + G U(t) + Y(t) \in \R_+^d, \quad \text{for all } t\geq 0,
\end{equation}
where $U(\cdot)$ is a $(d-1)$-dimensional control process with non-decreasing sample paths in each dimension and $U(0) = 0$, 
$Y(\cdot)$ is the Skorokhod regulator that keeps $Z(\cdot)$ in the nonnegative orthant, 
and the objective is to minimize, 
for all $(t,x) \in [0,T] \times \R_+^d$, the expected cost-to-go under $U(\cdot)$ and starting from $Z(t) = x$, 
\begin{equation}\label{eq:pss-singular-control-cost}
  J(t,x,U) := \mathbb{E} \left[ \int_t^T e^{-\beta (s-t)}\, h^{\top} Z(s) ds\Big| Z(t)=x \right],
\end{equation}
over all non-anticipating control process $U(\cdot)$. 

The EWF is a singular control problem, in which system states can be instantaneously displaced; 
recall that the process $U(\cdot)$ is allowed to contain jumps. 
It is different from drift controls, in which system states can only be adjusted at finite rates, 
but it is a useful benchmark for our drift control problem of interest. 
Indeed, its value function lower bounds that of the drift control,  
because given any drift upper bound $C_{\mathcal{A}}$, and any control $a(\cdot)$ taking values in $[0, C_{\mathcal{A}}]^{d-1}$, 
the process $U(t):=\int_0^t a(s) ds$ is a feasible control for the EWF.
Furthermore, it is not difficult to show that as $C_{\mathcal{A}}$ increases to infinity, 
the drift-control value function monotonically decreases to that of the singular control for each time-state pair $(t,x)$; we omit details here.
Thus, if we can solve the EWF, we obtain a (tight) lower bound for the drift control value functions.

Generally speaking, similar to drift controls, singular controls are (arguably more) challenging to solve, 
but sufficient conditions are identified in \cite{PesicWilliams2016} on the system parameters 
under which a simple policy, known as least control, is optimal for the EWF. 
Under least control, the state process $Z(\cdot)$ evolves as follows:
\begin{equation}\label{eq:least-control-1}
Z_1(t) = x_1 - t + B_1(t) + \bar Y_1(t), \quad t\geq 0,
\end{equation}
where $\bar Y_1$ is the one-dimensional Skorokhod regulator of the process $\bar X_1$, with
\begin{equation}\label{eq:least-control-2}
  \bar X_1(t) := x_1 - t + B_1(t), \quad t\geq 0.
\end{equation}
For $i = 2, 3, \cdots, d$, 
recusively define $\bar X_i(\cdot)$, $\bar Y_i(\cdot)$, and $Z_i(\cdot)$ by 
\begin{equation}\label{eq:least-control-3}
  \bar X_i(t) := x_i - t + B_i(t) - \bar Y_{i-1}(t), \quad t\geq 0, 
\end{equation}
$\bar Y_i$ is the Skorokhod regulator of $\bar X_i$, and 
\begin{equation}\label{eq:least-control-4}
  Z_i := \bar X_i + \bar Y_i.
\end{equation}
Exploiting the closed form of the one-dimensional Skorokhod map, and noting that $\bar X_i$ is the difference of a Brownian motion and $\bar Y_{i-1}$, 
the recursive representation \eqref{eq:least-control-1} -- \eqref{eq:least-control-4} of $Z(\cdot)$ allows for easy and efficient simulation 
of its sample paths. In turn, this implies that the performance of the least control policy can be accurately estimated.

The following result is a straightforward implication of Theorem 7.3 in \cite{PesicWilliams2016}, 
and identifies a simple sufficient condition on the holding cost vector $h$ 
under which least control is optimal for the EWF considered here. 
\begin{corollary}\label{cor:pss-singular-control-least-control}
Least control is optimal for the EWF \eqref{eq:pss-singular-control} -- \eqref{eq:pss-singular-control-cost} if 
\begin{equation}\label{eq:least-control-sufficient-condition}
h_1 \geq h_2 \geq \cdots \geq h_d.
\end{equation} 
\end{corollary}
Corollary \ref{cor:pss-singular-control-least-control} allows us to obtain computable lower bounds -- the expected cost incurred under least control -- for the drift-control value functions, 
under the condition \eqref{eq:least-control-sufficient-condition}. Furthermore, 
as $C_{\mathcal{A}}$ increases, the lower bounds should become tighter.
Finally, even when the condition \eqref{eq:least-control-sufficient-condition} is not satisfied, 
the least control performances can still provide meaningful comparisons with the drift-control value functions. 

In summary, we will use the expected cost incurred under least control as a benchmark 
for the drift-control value function approximations obtained under the MLP scheme. 

\subsection{Computational Results}\label{ssec:computational-results}

All experiments were conducted on a MacBook Pro (Apple M4 Max, 64GB RAM) with 16 CPU cores (12 performance and 4 efficiency). 
Since the MLP scheme allows parallelized implementation by dividing the Monte Carlo simulations across different processors, 
we used multithreading to utilize all 12 performance cores. All our codes are written in the Julia programming language 
and can be found at \url{https://github.com/zhongyuanx/Multilevel_Picard}.

\paragraph{Setup and methodology.}
We fix the time horizon at $T=0.2$. For each configuration, we run the MLP estimator 
under the reference process $\tilde Z^{t,x}$ defined in Eq. \eqref{eq:rbm-reference}, 
and across levels $n=1,2,3,4,5$ with 
$(n,M) = (1,196608), (2,768), (3,192), (4,60), (5,48)$ to examine convergence behavior of the MLP scheme, unless stated otherwise. 
Note that the sample size hyperparameters $M$ are chosen to be larger than those typically used in prior literature. 
This has to do with the more computationally intensive nature of our test problem. 
We also use larger $M$ so the MLP scheme imitates closer the behavior of the Picard iterations \eqref{eq:iterations}. 
To quantify estimator variability, we compute each MLP estimator $5$ times 
and report averages with plus/minus the sample standard deviations (which roughly corresponds to a $97\%$ confidence interval) 
as a fraction of the average value in parentheses. 
We use $\tilde D_i$ to denote the averaged MLP estimator of the value function gradient with respect to the $i$-th dimension, $i=1,2, \cdots, d$. 
From now on, with a slight abuse of terminology, we use the term ``MLP estimator'' to refer to the averaged estimator across $5$ MLP runs. 
For benchmarking, we compute the expected cost of the least-control (LC) policy whenever applicable; see Corollary~\ref{cor:pss-singular-control-least-control}. 
When we do so, time is discretized in steps of size $10^{-6}$, and we approximate the expected cost by averaging over $10^4$ sample paths. 
We report the approximate $97\%$ confidence interval for consistency. 
CPU times are reported as wall-clock times with multithreading enabled.

\paragraph{Two-dimensional $N$-network.} 
We carry out more extensive experiments for this low-dimensional test example, 
because we can examine the behavior of the MLP scheme at a more comprehensive set of states, 
and this setting clearly illustrates a number of useful insights about the MLP scheme in general. 
We consider two cost configurations. In the first case, $h_1 = h_2 = 1.0$, 
and the least control policy is optimal by Corollary~\ref{cor:pss-singular-control-least-control}. 
We vary the drift upper bound $C_{\mathcal{A}}$ over the set $\{1.0, 2.0, 5.0\}$. 
Under each drift upper bound, we compute both value function and gradient estimates at time $t=0$ and 
states on the diagonal grid $\{(z,z): z=0.0, 0.1, \dots, 1.0\}$. 
The MLP estimators are computed across levels $n=1,2,3,4,5$. 
We also compute the expected costs under least control at the same time-state pairs for comparison. 
The total time for all simulations and computations thus far described is under $11$ hours, across $33$ combinations 
of state and $C_{\mathcal{A}}$ pairs, and $11$ performance simulations for least control, 
with the majority of the time being used for computing the MLP estimates at level $5$.
\begin{longtable}{@{}r r c c c@{}}
  \caption{Per-level MLP results at $x=(0.4,0.4)$: Entries are value function estimates with confidence intervals as percentages. 
  $h_1 = h_2 = 1.0$.}\label{tab:summary-04}\ \\
  \toprule
  \textbf{Level} & \textbf{$M$} & \textbf{$C_{\mathcal A}=1$} & \textbf{$C_{\mathcal A}=2$} & \textbf{$C_{\mathcal A}=5$} \\
  \midrule
  \endfirsthead
  \toprule
  \textbf{Level} & \textbf{$M$} & \textbf{$C_{\mathcal A}=1$} & \textbf{$C_{\mathcal A}=2$} & \textbf{$C_{\mathcal A}=5$} \\
  \midrule
  \endhead
  1 & 196608 & 0.144095 (\(\pm\) 0.16\%) & 0.144050 (\(\pm\) 0.20\%) & 0.143981 (\(\pm\) 0.17\%) \\
  2 & 768     & 0.141176 (\(\pm\) 0.15\%) & 0.138000 (\(\pm\) 0.12\%) & 0.128984 (\(\pm\) 0.83\%) \\
  3 & 192     & 0.142151 (\(\pm\) 0.22\%) & 0.141356 (\(\pm\) 0.28\%) & 0.140676 (\(\pm\) 1.33\%) \\
  4 & 60      & 0.141834 (\(\pm\) 0.19\%) & 0.139630 (\(\pm\) 0.56\%) & 0.137248 (\(\pm\) 2.43\%) \\
  5 & 48      & 0.141933 (\(\pm\) 0.13\%) & 0.140871 (\(\pm\) 0.22\%) & 0.137891 (\(\pm\) 0.98\%) \\
  \bottomrule
\end{longtable}

\begin{longtable}{@{}r r c c c c@{}}
  \caption{Per-level gradient estimates at $x=(0.4,0.4)$ by $C_{\mathcal A}$. 
  $\tilde D_i$ denotes the gradient estimate of $V$ with respect to the $i$-th dimension, $i=1,2$.
  Entries show $\tilde D_1$ and $\tilde D_2$ with confidence intervals as percentages.
  $h_1 = h_2 = 1.0$.}\label{tab:summary-04-dv}\ \\
  \toprule
  \textbf{Level} & \textbf{$M$} & \textbf{$\tilde D_i$} & \textbf{$C_{\mathcal A}=1$} & \textbf{$C_{\mathcal A}=2$} & \textbf{$C_{\mathcal A}=5$} \\
  \midrule
  \endfirsthead
  \toprule
  \textbf{Level} & \textbf{$M$} & \textbf{$\tilde D_i$} & \textbf{$C_{\mathcal A}=1$} & \textbf{$C_{\mathcal A}=2$} & \textbf{$C_{\mathcal A}=5$} \\
  \midrule
  \endhead
  1 & 196608 & $\tilde D_1$ & 0.146201 (\(\pm\) 0.48\%) & 0.146353 (\(\pm\) 0.88\%) & 0.145878 (\(\pm\) 0.82\%) \\
  &         & $\tilde D_2$ & 0.145658 (\(\pm\) 0.63\%) & 0.145510 (\(\pm\) 0.73\%) & 0.145056 (\(\pm\) 1.28\%) \\
  \midrule
  2 & 768     & $\tilde D_1$ & 0.154633 (\(\pm\) 1.01\%) & 0.165305 (\(\pm\) 0.37\%) & 0.196850 (\(\pm\) 1.96\%) \\
  &         & $\tilde D_2$ & 0.137847 (\(\pm\) 0.64\%) & 0.128804 (\(\pm\) 1.24\%) & 0.104050 (\(\pm\) 3.28\%) \\
  \midrule
  3 & 192     & $\tilde D_1$ & 0.153319 (\(\pm\) 0.59\%) & 0.156788 (\(\pm\) 3.20\%) & 0.166375 (\(\pm\) 8.39\%) \\
  &         & $\tilde D_2$ & 0.141971 (\(\pm\) 1.91\%) & 0.143662 (\(\pm\) 2.82\%) & 0.135054 (\(\pm\) 3.42\%) \\
  \midrule
  4 & 60      & $\tilde D_1$ & 0.154707 (\(\pm\) 0.69\%) & 0.160418 (\(\pm\) 2.44\%) & 0.166113 (\(\pm\) 6.44\%) \\
  &         & $\tilde D_2$ & 0.139907 (\(\pm\) 0.15\%) & 0.133035 (\(\pm\) 3.91\%) & 0.127715 (\(\pm\) 14.06\%) \\
  \midrule
  5 & 48      & $\tilde D_1$ & 0.153580 (\(\pm\) 0.56\%) & 0.158890 (\(\pm\) 3.04\%) & 0.168431 (\(\pm\) 3.55\%) \\
  &         & $\tilde D_2$ & 0.140074 (\(\pm\) 1.10\%) & 0.139301 (\(\pm\) 1.39\%) & 0.129148 (\(\pm\) 5.37\%) \\
  \bottomrule
\end{longtable}

\begin{longtable}{@{}rcccc@{}}
  \caption{Value function estimates at different initial states with $(n,M) = (5,48)$ 
  under different $C_{\mathcal{A}} \in \{1.0, 2.0, 5.0\}$. $h_1 = h_2 = 1.0$. 
  Baseline denotes the expected cost under least control.
  }\label{tab:summary-level5}\ \\
  \toprule
  \textbf{Initial state} & \textbf{$C_{\mathcal A}=1$} & \textbf{$C_{\mathcal A}=2$} & \textbf{$C_{\mathcal A}=5$} & \textbf{Baseline} \\
  \midrule
  \endfirsthead
  \toprule
  \textbf{$x=(u,u)$} & \textbf{$C_{\mathcal A}=1$} & \textbf{$C_{\mathcal A}=2$} & \textbf{$C_{\mathcal A}=5$} & \textbf{Baseline} \\
  \midrule
  \endhead
  $(0.0,0.0)$ & 0.074775 (\(\pm\) 0.11\%) & 0.073355 (\(\pm\) 0.58\%) & 0.070788 (\(\pm\) 6.07\%) & 0.067109 (\(\pm\) 0.47\%) \\
  $(0.1,0.1)$ & 0.079748 (\(\pm\) 0.18\%) & 0.078134 (\(\pm\) 0.55\%) & 0.074847 (\(\pm\) 3.39\%) & 0.072869 (\(\pm\) 0.48\%) \\
  $(0.2,0.2)$ & 0.093623 (\(\pm\) 0.11\%) & 0.091832 (\(\pm\) 0.30\%) & 0.088047 (\(\pm\) 2.25\%) & 0.087328 (\(\pm\) 0.47\%) \\
  $(0.3,0.3)$ & 0.114781 (\(\pm\) 0.07\%) & 0.112981 (\(\pm\) 0.51\%) & 0.111153 (\(\pm\) 3.09\%) & 0.109659 (\(\pm\) 0.44\%) \\
  $(0.4,0.4)$ & 0.141933 (\(\pm\) 0.13\%) & 0.140871 (\(\pm\) 0.22\%) & 0.137891 (\(\pm\) 0.98\%) & 0.137153 (\(\pm\) 0.41\%) \\
  $(0.5,0.5)$ & 0.173203 (\(\pm\) 0.14\%) & 0.171984 (\(\pm\) 0.20\%) & 0.170499 (\(\pm\) 1.23\%) & 0.170906 (\(\pm\) 0.36\%) \\
  $(0.6,0.6)$ & 0.207703 (\(\pm\) 0.06\%) & 0.206914 (\(\pm\) 0.28\%) & 0.204436 (\(\pm\) 1.44\%) & 0.206718 (\(\pm\) 0.32\%) \\
  $(0.7,0.7)$ & 0.244357 (\(\pm\) 0.09\%) & 0.243187 (\(\pm\) 0.29\%) & 0.240461 (\(\pm\) 0.89\%) & 0.244262 (\(\pm\) 0.28\%) \\
  $(0.8,0.8)$ & 0.282219 (\(\pm\) 0.04\%) & 0.281848 (\(\pm\) 0.32\%) & 0.277789 (\(\pm\) 1.65\%) & 0.282588 (\(\pm\) 0.25\%) \\
  $(0.9,0.9)$ & 0.321070 (\(\pm\) 0.03\%) & 0.320558 (\(\pm\) 0.28\%) & 0.316143 (\(\pm\) 0.59\%) & 0.320007 (\(\pm\) 0.22\%) \\
  $(1.0,1.0)$ & 0.360570 (\(\pm\) 0.08\%) & 0.359739 (\(\pm\) 0.17\%) & 0.353010 (\(\pm\) 0.99\%) & 0.360583 (\(\pm\) 0.20\%) \\
  \bottomrule
\end{longtable}

Tables \ref{tab:summary-04}, \ref{tab:summary-04-dv} and \ref{tab:summary-level5} summarize the MLP results for the two-dimensional $N$-network. 
Table \ref{tab:summary-04} provides value function estimates at a particular state $(0.4,0.4)$ 
under different levels and values of $C_{\mathcal{A}}$, but all qualitative observations 
hold for other states as well. 
From Table \ref{tab:summary-04}, we observe that the MLP estimates stabilize within four to five levels, 
across all values of $C_{\mathcal{A}}$, though larger oscillations are observed when $C_{\mathcal{A}} = 5.0$.
Similarly, estimator variability, as measured by standard errors, increases with $C_{\mathcal{A}}$, 
with the variability being notably higher when $C_{\mathcal{A}}=5.0$ compared to lower values of $C_{\mathcal{A}}$, 
especially at higher levels.  
The observation that variability increases with $C_{\mathcal{A}}$ is consistent with 
our theoretical findings; cf. Theorem \ref{thm:main1}. 
Table \ref{tab:summary-04-dv} contains gradient estimates at the same state $(0.4,0.4)$, 
and is also representative of other states. 
Similar to the value function estimates, gradient estimates stabilize within four to five levels 
and become more variable as $C_{\mathcal{A}}$ increases. 
It should be noted that gradient estimates tend to exhibit higher variability than value function estimates, 
which can be explained by the additional Malliavin weights in formulas for gradients. 
Finally, Table \ref{tab:summary-level5} summarizes the value function estimates at level $5$ 
across all diagonal states $(0.0, 0.0), (0.1, 0.1), \cdots$, and $(1.0, 1.0)$, 
under different values of $C_{\mathcal{A}}$. 
We observe that (i) value function estimates decrease and become more variable as $C_{\mathcal{A}}$ increases; 
(ii) value function estimates of drift controls can already be close to the optimal value function of the singular control, 
even when $C_{\mathcal{A}}$ is small ($1.0$ or $2.0$); 
and (iii) it is possible for value function estimates of the drift control problems to be lower than that under least control. 
The reason for observation (i) is intuitively clear: As $C_{\mathcal{A}}$ increases, theory predicts lower value function which approaches 
the baseline, and the MLP scheme becomes more variable with larger $C_{\mathcal{A}}$. 
Observation (ii) suggests that drift controls can effectively approximate singular controls even with low drift upper bounds; 
this phenomenon has also been observed in \cite{DaiZhong2010,ata2024singular}.
Observation (iii) indicates that being a simulation-based scheme, MLP can produce underestimates,
especially when the resulting estimates exhibit a nonnegligible degree of variability.

We now turn to the second cost configuration, $h_1 = 1.0$ and $h_2 = 2.0$. In this case, 
Corollary \ref{cor:pss-singular-control-least-control} no longer holds, 
leaving the question of an optimal policy open. 
Intuitively speaking, in the corresponding queueing setting, 
because buffer $2$ is more expensive, when buffer $2$ content is much larger than buffer $1$ content, 
server $1$ should turn to help serve buffer $2$. This suggests that a switching curve type policy may be (heavy-traffic) optimal. 
A switching curve policy works as follows: There is an increasing function $f: \R_+ \to \R_+$ 
such that at state $z = (z_1, z_2)$, if $f(z_1) < z_2$, then server $1$ should help serve buffer $2$ content, 
and if $f(z_1) > z_2$, server $1$ does not do so.  
To examine this hypothesis, we vary the drift upper bound $C_{\mathcal{A}}$ over the set $\{1.0, 2.0\}$. 
Under each drift upper bound, we compute both value function and gradient estimates at time $t=0$ and 
states in the square grid $\{(z_1,z_2): z_1,z_2 \in \{0.0, 0.1, \dots, 1.0\}\}$. 
The MLP estimators are computed at level $n=4$ for computational efficiency, 
with each MLP estimator taking less than $42$ seconds to compute, for each initial state. 
Similar to the first cost configuration, we also compute the expected costs under least control at the same time-state pairs for comparison. 

Figures \ref{fig:2d-policy} and \ref{fig:2d-V} summarize our findings. 
Figure \ref{fig:2d-policy-C1} plots the policy produced by the gradient estimates when $C_{\mathcal{A}} = 1.0$. 
When $\tilde D_1 - \tilde D_2 > 0$, we plot ``x'', indicating that server $1$ does not help serve buffer $2$ (i.e., drift control should be set to zero), 
and when $\tilde D_1 - \tilde D_2 \le 0$, we plot ``+'', indicating that server $1$ helps serve buffer $2$ (i.e., drift control should be set to $C_{\mathcal{A}}$). 
It is visible from the figure that the corresponding policy is a switching curve policy, which lets server $1$ serve buffer $2$ 
when its content is significantly larger than that in buffer $1$. 
Figure \ref{fig:2d-policy-C2} produces a similar plot when $C_{\mathcal{A}} = 2.0$; we see that the policies produced are similar. 
Finally, Figure \ref{fig:2d-V} plots a comparison between value function estimates across states when $C_{\mathcal{A}} = 2.0$ 
and those under least control. Even in this case with low $C_{\mathcal{A}}$, value function estimates are higher than 
those under least control for a significant fraction of states, with a cost reduction up to $7\%$. 
\begin{figure}[t]
  \centering
  \begin{subfigure}[t]{0.48\textwidth}
    \centering
    \includegraphics[width=0.9\linewidth]{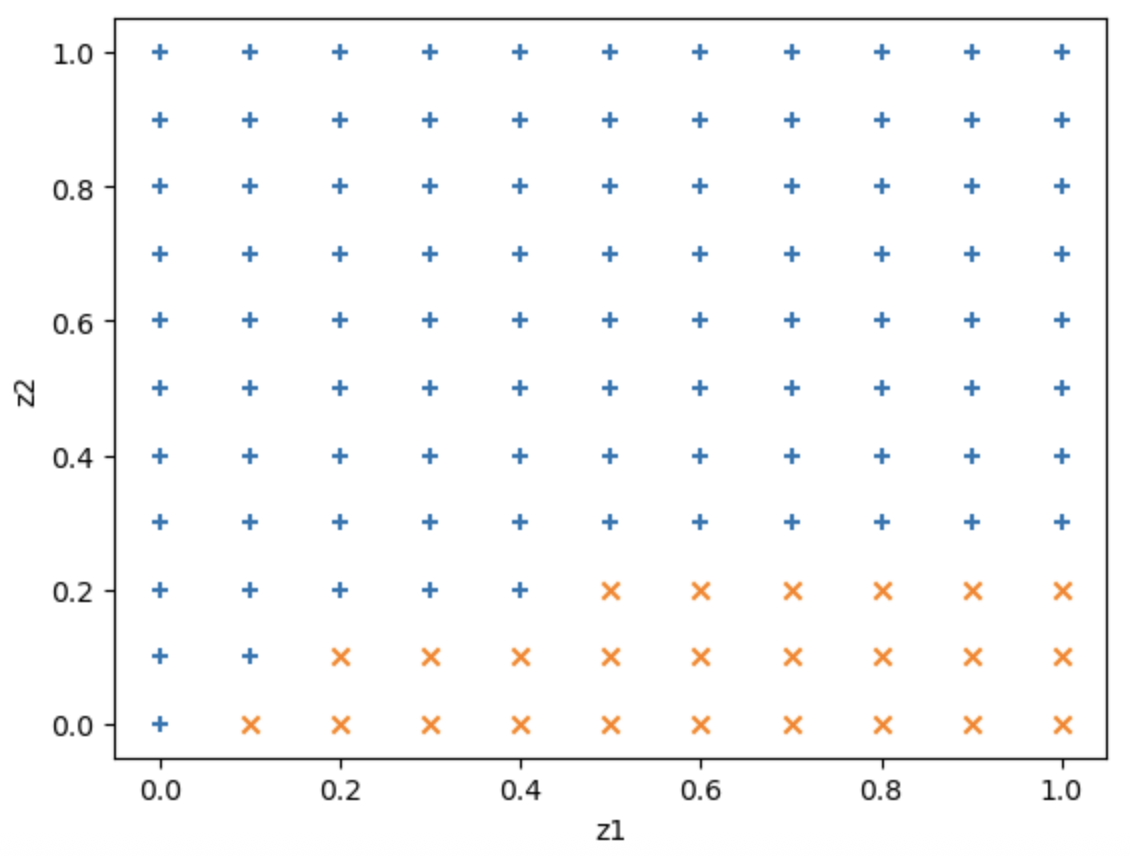}
    \caption{Policy plot under $C_{\mathcal{A}}=1.0$}
    \label{fig:2d-policy-C1}
  \end{subfigure}\hfill
  \begin{subfigure}[t]{0.48\textwidth}
    \centering
    \includegraphics[width=0.9\linewidth]{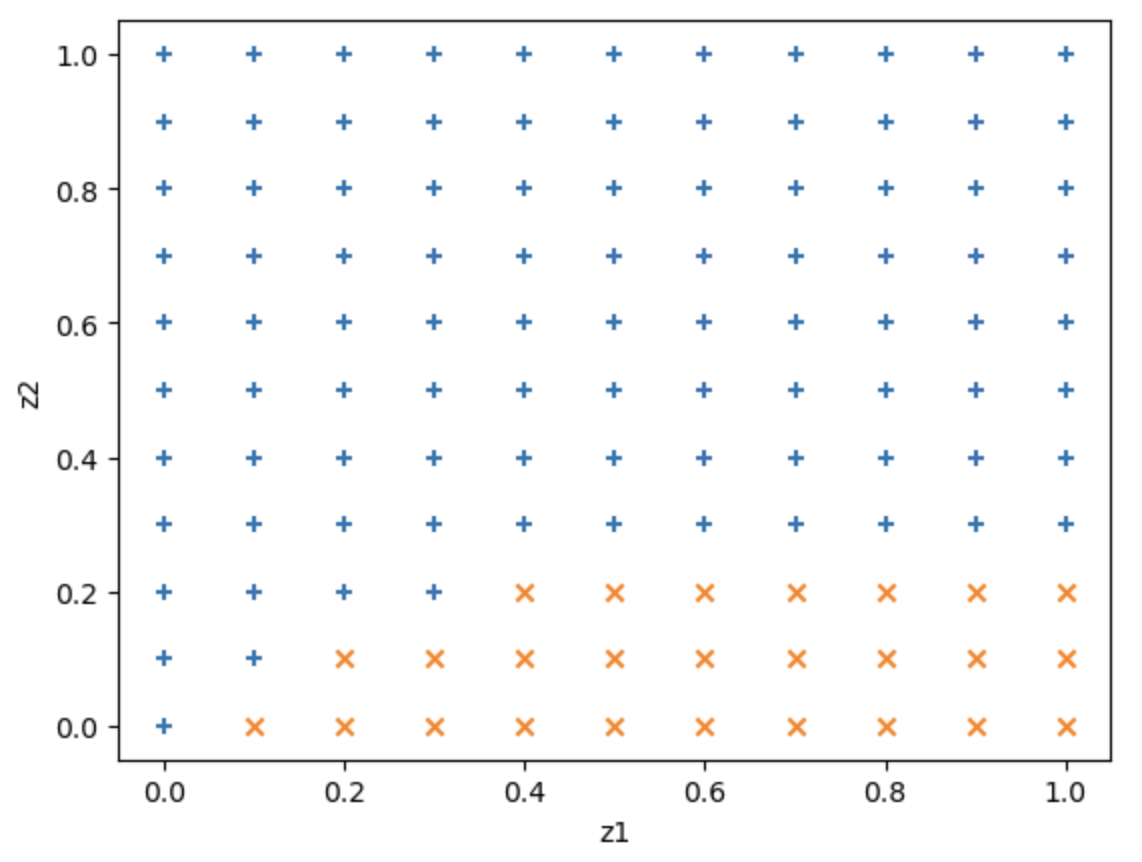}
    \caption{Policy plot under $C_{\mathcal{A}}=2.0$}
    \label{fig:2d-policy-C2}
  \end{subfigure}
  \caption{Policy plots under different $C_{\mathcal{A}}$ when $h_1 = 1.0$ and $h_2 = 2.0$. ``+'' indicates that for the estimated $\tilde D_1$ and $\tilde D_2$, 
  $\tilde D_1- \tilde D_2 \leq 0$, and ``x'' indicates that $\tilde D_1 - \tilde D_2 > 0$. The policy interpretation is that 
  at the ``+'' states, server $1$ helps to serve buffer $2$ at the maximum rate $C_{\mathcal{A}}$ possible, and does not do so at the ``x'' states.}
  \label{fig:2d-policy}
\end{figure}

\begin{figure}[t]
  \centering
  \includegraphics[width=.42\linewidth]{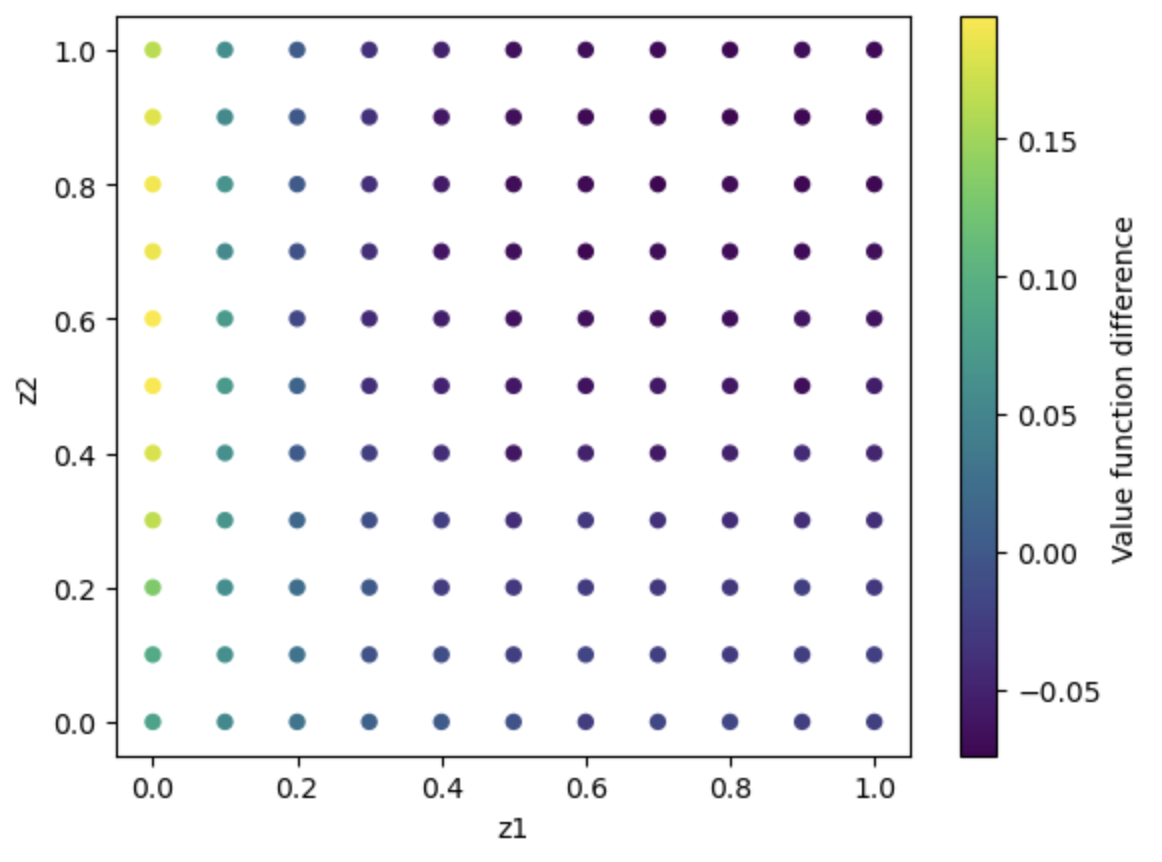}
  \caption{Plot of value function differences between that under $C_{\mathcal{A}}=2.0$ and that under least control, when $h_1 = 1.0$ and $h_2 = 2.0$. 
  Value function estimates for drift control with $C_{\mathcal{A}}=2.0$ are computed at level $n=4$. A negative value 
  indicates that the value function estimate under drift control is lower than that under least control, and vice versa.}
  \label{fig:2d-V}
\end{figure}


\paragraph{Higher-dimensional examples.}
For the $5$-dimensional test example, we take $h$ to be $h=(1.0, 2.0, 3.0, 4.0, 5.0)$. 
In this case, the least control policy is not optimal, so it is of interest to see how much performance gain we can obtain 
by solving the drift control problems optimally using the MLP scheme. 
We fix $C_{\mathcal{A}}=1.0$, and compute the value function and gradient estimates at time $t=0$ and the initial state $x=(1,1,1,1,1)$. 
Table \ref{tab:five-dim-averages} reports the value function and gradient estimates in the first two coordinates at different levels, together with the simulated performance estimate under least control. 
We report the approximate $97\%$ confidence intervals as percentages next to each estimate. 
For all value function estimates, the confidence intervals are within $0.3\%$ of the estimates, and 
but for gradient estimates, the confidence intervals can go up to $\pm 14\%$ of the estimate.
The discrepancy between the confidence intervals for value function and gradient estimates can be explained as follows. 
First, as explained before, because of the extra Malliavin weight used for gradient estimation, 
the gradient estimators tend to exhibit higher variance than the corresponding value function estimators. 
Second, the gradients are typically much smaller than the corresponding value function in absolute magnitude. 
In terms of convergence properties of the algorithm, the MLP estimators converge reliably within four to five levels, and the MLP value function estimate achieves a statistically significant $2.8\%$ improvement relative to the least-control baseline, 
even when the drift upper bound is relatively low. The total computation time across all levels is about $24$ minutes.

\begin{table}[t]
  \centering
  \caption{Estimates by level for the five-dimensional example. ``Est $V$'' is the value function estimate, 
  and $\tilde D_1$ and $\tilde D_2$ are the gradient estimates in the first and second coordinates, respectively.}
  \label{tab:five-dim-averages}
  \begin{tabular}{r r r r r}
    \toprule
    Level & $M$ & Est $V$ & $\tilde D_1$ & $\tilde D_2$ \\
    \midrule
    1 & 196608 & 2.707390 $\pm$ 0.11\% & 0.196734 $\pm$ 1.55\% & 0.398452 $\pm$ 1.76\% \\
    2 & 768 & 2.619617 $\pm$ 0.08\% & 0.194838 $\pm$ 8.40\% & 0.390922 $\pm$ 3.68\% \\
    3 & 192 & 2.628396 $\pm$ 0.03\% & 0.201183 $\pm$ 4.71\% & 0.390435 $\pm$ 3.45\% \\
    4 & 60 & 2.628077 $\pm$ 0.09\% & 0.185118 $\pm$ 14.15\% & 0.405743 $\pm$ 5.21\% \\
    5 & 48 & 2.627142 $\pm$ 0.07\% & 0.202458 $\pm$ 9.25\% & 0.393184 $\pm$ 2.66\% \\
    \midrule
    \multicolumn{2}{c}{Least control} & 2.703674 $\pm$ 0.30\% &  &  \\
    \bottomrule
  \end{tabular}
\end{table}

Finally, we turn to the $20$-dimensional test example. The cost vector is given by 
\[
h = (1.0, 2.0, 3.0, 4.0, 5.0, 1.0, 2.0, 3.0, 4.0, 5.0, 1.0, 2.0, 3.0, 4.0, 5.0, 1.0, 2.0, 3.0, 4.0, 5.0),
\]
and we set $C_{\mathcal{A}}=1.0$. We compute the MLP estimators across levels $1, \cdots, 5$ at time $t=0$ and state $x=\mathbf{1}_{20}$. 
To improve accuracy, the sample size hyperparameters $M$ are increased to $786432, 6144, 384, 240$ and $72$, 
for levels $1, 2, 3, 4$ and $5$, respectively. 
The total computation time for all $5$ levels is less than $9$ hours, 
with the level $4$ computation taking $4$ hours and the level $5$ computation taking less than $5$ hours. 
Table \ref{tab:twenty-dim-averages} summarizes the result.
\begin{table}[t]
  \centering
  \caption{Averages by level for the twenty-dimensional example. ``Est $V$'' is the the value function estimate, 
  and $\tilde D_1$ and $\tilde D_2$ are the gradient estimates in the first and second coordinates, respectively.}
  \label{tab:twenty-dim-averages}
  \begin{tabular}{r r r r r}
    \toprule
    Level & $M$ & Est $V$ & $\tilde D_1$ & $\tilde D_2$ \\
    \midrule
    1 & 786432 & 10.830523 $\pm$ 0.08\% & 0.197911 $\pm$ 0.47\% & 0.397100 $\pm$ 0.90\% \\
    2 & 6144 & 10.508915 $\pm$ 0.02\% & 0.200700 $\pm$ 9.73\% & 0.397782 $\pm$ 6.14\% \\
    3 & 384 & 10.503990 $\pm$ 0.02\% & 0.212257 $\pm$ 10.04\% & 0.408682 $\pm$ 3.49\% \\
    4 & 240 & 10.514296 $\pm$ 0.04\% & 0.198493 $\pm$ 11.83\% & 0.403514 $\pm$ 4.96\% \\
    5 & 72 & 10.492808 $\pm$ 0.05\% & 0.219981 $\pm$ 32.91\% & 0.411891 $\pm$ 14.47\% \\
    \midrule
    \multicolumn{2}{c}{Least control} & 10.787542 $\pm$ 0.33\% &  &  \\
    \bottomrule
  \end{tabular}
\end{table}
All qualitative observations made in the five-dimensional example hold here as well: 
Gradient estimators are typically more variable than the corresponding value function estimators, 
with the value function estimators converge quickly and reliably in four to five levels. 
The MLP value function estimate also achieves a statistically significant improvement 
that is more than $2.5\%$ relative to the least-control baseline. 

Let us now provide some remarks regarding the robustness of the value function estimators under the MLP scheme. 
On the one hand, if we use less sampling for each estimator, by setting the hyperparameter $M$ to be smaller, 
we obtain both value function and gradient estimators that are more variable and with potentially larger bias 
-- this phenomenon is both observed in experiments and predicted by the theory. 
On the other hand, in many of our experiments, we also observed that once the hyperparameter $M$ is sufficiently large, 
even when the gradient estimators are still variable, 
the corresponding value function estimators become relatively stable.
For example, in the twenty-dimensional example, at level $4$, under the smaller hyperparameter $M=120$, not the $M=240$ used in Table \ref{tab:twenty-dim-averages}, 
the value function estimate is $10.5199 \pm 0.08\%$ and statistically indistinguishable from the estimate at $M=240$, 
though the corresponding gradient estimates are much more variable at $\tilde V_1 = 0.111320 \pm 37.34\%$ and $\tilde V_2 = 0.307129 \pm 29.33\%$. 
The discrepancy between the accuracy of the value function estimates and that of the gradient estimates 
has also been observed in prior literature; see, e.g., the example with explicit solution in Section 3.5 of \cite{hutzenthaler2019multilevel_arxiv}, 
where the function $u$ itself can be approximated more accurately than its gradient\footnote{The paper \cite{hutzenthaler2019multilevel_arxiv} is the arXiv version of \cite{hutzenthaler2019multilevel}, 
the former of which contains this additional example.}. 


\paragraph{Alternative reference processes.} We now provide some brief remarks on the use of alternative reference processes.
So far, all our experiments use a reference process that has independent reflected Brownian motions in each dimension; cf. Eq. \eqref{eq:rbm-reference}. 
An attractive feature of this reference process is that exact simulation is available, allowing significant speedup in the implementation of the MLP scheme. 
However, in more general settings, exact simulation is not always possible, in which case we need to resort to Euler--Maruyama time discretization 
for simulating reference processes and their derivatives; see \cite{LipshutzRamanan2019a} for the development and analysis of such a scheme. 
Since the implementation of the MLP scheme under reference processes that require the use of Euler--Maruyama discretization is not the focus 
of this paper, we refer the readers to \cite{LipshutzRamanan2019a}, and omit a detailed description of the corresponding MLP scheme. 
Though computationally more intensive, time discretization is also more flexible in terms of the choice of reference processes; 
in particular, in applying the MLP scheme, if we can use a reference process whose drift is sufficiently close to the optimal one, 
for example, by making use of structural properties of the optimal policy, 
we may achieve faster convergence compared to the use of a more naïve reference process. 

To illustrate the convergence properties under a ``smarter'' reference process, 
we carry out numerical experiments for the two-dimensional example with cost configuration $h_1 = 1.0$ and $h_2 = 2.0$, 
where we conjectured the optimality of a switching-curve type policy. 
To this end, 
we use the following drift function for the reference process: 
\begin{equation}\label{eq:tb-switching-curve}
  \tb(x) = \frac{C_{\mathcal{A}}}{2}\left[\tanh(2x_2 - x_1) + 1\right] \times (1, -1)^{\top}. 
\end{equation}
The reason for using \eqref{eq:tb-switching-curve} is that it can be viewed as a smoothed implementation of the following switching curve policy: 
Set $a = C_{\mathcal{A}}$ if $2x_2 - x_1 > 0$, and $a = 0$ otherwise. The factor $2$ in front of $x_2$ is chosen to 
leverage some crude information about the policy structure we gained from Figure \ref{fig:2d-policy}.
The drift upper bound is set to $C_{\mathcal{A}}=5.0$. 
The number of time discretization steps is set to $50$ -- this means that if we need to simulate the reference process at a random time $S$, 
then the time interval $[t, S]$ is discretized into $50$ subintervals of equal length, with $t$ being the initial time. We compute the MLP estimators across levels $n=1,2,3$ with 
hyperparameters $(n,M) = (1,196608), (2,768), (3,192)$, at time-state pair $(t,x) = (0, (1.0, 1.0))$; the computation takes slight more than one hour to finish.
For comparison, we also compute MLP estimators across levels $n=1, \cdots, 5$ with 
hyperparameters $(n,M) = (1,196608), (2,768), (3,192), (4, 60), (5,48)$ using the reference process in Eq. \eqref{eq:rbm-reference}, at the same time-state pair, 
with the computation taking no more than half an hour in total. 

Tables \ref{tab:switching-curve-averages} and \ref{tab:switching-curve-averages-comparison} summarize the computational results. 
On the one hand, it is evident that the MLP scheme under reference drift \eqref{eq:tb-switching-curve} converges faster and more reliably (in about three iterations)
compared to that under the two-dimensional reflected Brownian motion reference process. 
In particular, variations around the gradient estimates are much smaller in Table \ref{tab:switching-curve-averages} 
compared to those in Table \ref{tab:switching-curve-averages-comparison}. On the other hand, 
the MLP scheme under the alternative switching curve reference process 
takes much longer to compute primarily because of time discretization. 
Another difficulty in implementing a successful MLP scheme using time-discretized reference processes in practice 
is the choice of the reference process; we leave the topic of finding a good reference process as an important future direction. 
\begin{table}[t]
  \centering
  \caption{Estimates by level using the reference drift in \eqref{eq:tb-switching-curve}. 
  ``Est $V$'' is the value function estimate, and $\tilde D_1$ and $\tilde D_2$ 
  are the gradient estimates in the first and second coordinates, respectively. 
  $h_1 = 1.0$, $h_2 = 2.0$, and $C_{\mathcal{A}} = 5.0$. The initial state is $(1.0, 1.0)$.} 
  \label{tab:switching-curve-averages}
  \begin{tabular}{r r r r r}
    \toprule
    Level & $M$ & Est $V$ & $\tilde D_1$ & $\tilde D_2$ \\
    \midrule
    1 & 196608 & 0.483742 $\pm$ 0.13\% & 0.255845 $\pm$ 1.63\% & 0.366238 $\pm$ 0.88\% \\
    2 & 768 & 0.458417 $\pm$ 0.15\% & 0.237146 $\pm$ 4.06\% & 0.371009 $\pm$ 1.45\% \\
    3 & 192 & 0.453238 $\pm$ 0.53\% & 0.222646 $\pm$ 3.95\% & 0.366354 $\pm$ 2.47\% \\
    \bottomrule
  \end{tabular}
\end{table}

\begin{table}[t]
  \centering
  \caption{Estimates by level using the reference process \eqref{eq:rbm-reference}. 
  ``Est $V$'' is the value function estimate, and $\tilde D_1$ and $\tilde D_2$ 
  are the gradient estimates in the first and second coordinates, respectively. 
  $h_1 = 1.0$, $h_2 = 2.0$, and $C_{\mathcal{A}} = 5.0$. The initial state is $(1.0, 1.0)$.} 
  \label{tab:switching-curve-averages-comparison}
  \begin{tabular}{r r r r r}
    \toprule
    Level & $M$ & Est $V$ & $\tilde D_1$ & $\tilde D_2$ \\
    \midrule
    1 & 196608 & 0.541150 $\pm$ 0.16\% & 0.196083 $\pm$ 0.90\% & 0.391997 $\pm$ 0.56\% \\
    2 & 768 & 0.444181 $\pm$ 0.53\% & 0.202141 $\pm$ 25.54\% & 0.371300 $\pm$ 6.81\% \\
    3 & 192 & 0.462128 $\pm$ 1.27\% & 0.180898 $\pm$ 35.63\% & 0.325975 $\pm$ 13.03\% \\
    4 & 60 & 0.450551 $\pm$ 1.10\% & 0.204257 $\pm$ 11.62\% & 0.341118 $\pm$ 24.76\% \\
    5 & 48 & 0.453102 $\pm$ 3.47\% & 0.079234 $\pm$ 209.12\% & 0.313554 $\pm$ 59.63\% \\
    \bottomrule
  \end{tabular}
\end{table}

\section{A Bismut--Elworthy--Li Formula}\label{sec:bel}

\begin{theorem}[Bismut--Elworthy--Li formula]\label{thm:bel}
Recall the definition of the reflected diffusion $\tilde Z^x$ in Eq. \eqref{eq:reflected-diffusion}, 
and let $\zeta : \mathbb{R}^d \to \mathbb{R}$ be a measurable function such that $\E[\zeta^2(\tilde Z^x(t))] < \infty$ for all $t > 0$. Let $x \in \mathbb{R}_+^d$. Then, for any $t > 0$,
\begin{equation}\label{eq:bel-main}
D_x \mathbb{E} \bigl[ \zeta(\tilde Z^x(t)) \bigr]
= \mathbb{E} \left[ \zeta(\tilde Z^x(t)) \cdot \frac{1}{t} \int_0^t \left[\sigma^{-1} D\tilde Z^x(s)\right]^{\top} dB(s) \right].
\end{equation}
\end{theorem}

The proof of Theorem \ref{thm:bel} relies crucially on the following proposition. 
\begin{proposition}\label{prop:bel}
Suppose Assumption \ref{as:tb} holds for $\tb$. Let $x \in \mathbb{R}_+^d$, and suppose that a.s.\ $\nabla_\psi \Gamma\left(\tilde X^x\right)$ exists for all $\psi \in \C(\mathbb{R}^d)$ 
(recall the definition of $\tilde X^x$ in Eq. \eqref{eq:tx}), 
and lies in $\D_{\text{lim}}(\mathbb{R}^d)$. Let $\upsilon(\cdot)$ be a $d$-dimensional progressively measurable process such that a.s.\ there exists a constant $C_\upsilon > 0$ with $\|\upsilon(t)\| \le C_\upsilon$ for all $t \geq 0$.

Let $\Upsilon$ be the $d$-dimensional $\{\cF_t\}$-adapted process defined to be the time integral of $\upsilon(\cdot)$, i.e., 
\[
\Upsilon(t) := \int_0^t \upsilon(s)\,ds, \quad t\geq 0.
\]
For each $\veps \geq 0$, let $\tilde Z^x(\cdot; \veps \Upsilon)$ be the unique strong solution of the reflected stochastic differential equation:
\[
\tilde Z^x(t; \veps \Upsilon) = x + \int_0^t \tb\bigl( \tilde Z^x(s; \veps \Upsilon) \bigr)\,ds + \sigma B(t) + R\,\tilde Y^x(t; \veps \Upsilon).
\]
(Note that $\tilde Z^x(\cdot; 0 \cdot \Upsilon) = \tilde Z^x$.)
Then, a.s., the process $\partial_\Upsilon \tilde Z^x$ exists, where
\[
\partial_\Upsilon \tilde Z^x(t) := \lim_{\varepsilon \downarrow 0} \frac{\tilde Z^x(t; \varepsilon \Upsilon) - \tilde Z^x(t)}{\varepsilon}, \qquad t \ge 0,
\]
and is characterized as the unique $\{\mathcal{F}_t\}$-adapted process that satisfies $\partial_\Upsilon \tilde Z^x = \nabla_\Psi \Gamma \bigl(\tilde X^x\bigr)$, 
where the process $\Psi$ satisfies for all $t \ge 0$,
\[
\Psi(t) = \int_0^t D\tb\bigl( \tilde Z^x(s) \bigr) \partial_\Upsilon \tilde Z^x(s)\,ds + \Upsilon(t).
\]
\end{proposition}

\begin{definition}\label{def:deriv-M}
Let \( x \in \mathbb{R}_+^d \). Let \( B(\cdot), \tilde{Z}^x, \Upsilon(\cdot), \upsilon(\cdot) \) be as in Proposition \ref{prop:bel}.  
A derivative process along \( \tilde{Z}^x \) in the direction \( \Upsilon(\cdot) \) is a \( J \)-dimensional \( \{\mathcal{F}_t\} \)-adapted RCLL process
$D_\Upsilon \tilde{Z}^x = \{ D_\Upsilon \tilde{Z}^x(t) : t \in \mathbb{R}_+ \}$ 
that a.s.\ satisfies, for all \( t \ge 0 \), \( D_\Upsilon \tilde{Z}^x(t) \in H_{\tilde{Z}^x(t)} \), and
\[
D_\Upsilon \tilde{Z}^x(t) = \int_0^t D\tb(\tilde{Z}^x(s)) D_\Upsilon \tilde{Z}^x(s)\,ds + \Upsilon(t) + \eta_\Upsilon(t),
\]
where \( \eta_\Upsilon = \{ \eta_\Upsilon(t) : t \in \mathbb{R}_+ \} \) is a \( J \)-dimensional \( \{\mathcal{F}_t\} \)-adapted RCLL process such that a.s. 
$\eta_\Upsilon(0) = 0$, and for all $0 \le s < t < \infty$,
\[
\eta_\Upsilon(t) - \eta_\Upsilon(s) \in \text{span} \left[ \bigcup_{r \in (s,t]} R(\tilde{Z}^x(r)) \right].
\]
\end{definition}

Similar to Proposition \ref{prop:deriv-unique}, pathwise uniqueness holds for a derivative process along \( \tilde{Z}^x \) in the direction \( \Upsilon \).

\begin{proposition}\label{prop:deriv-M-unique}
Let \( x \in \mathbb{R}_+^d \). Then pathwise uniqueness holds for the derivative process \( D_\Upsilon \tilde{Z}^x \) along \( \tilde{Z}^x \) in the direction \( \Upsilon \).
\end{proposition}
We omit the proof of Proposition \ref{prop:deriv-M-unique}, because it is standard 
and follows the same line of reasoning as in the proof of Theorem 3.6 in \cite{LipshutzRamanan2019b}.

We establish the following proposition, stated in a similar fashion to Theorem 3.13 of \cite{LipshutzRamanan2019b}. 
\begin{proposition}\label{prop:bel2}
Let $x \in \R_+^d$, and let \( B(\cdot), \tilde{Z}^x, \Upsilon(\cdot), \upsilon(\cdot) \) be as in Proposition \ref{prop:bel}.
There exists a pathwise unique derivative process \( D_\Upsilon \tilde{Z}^x \) along \( \tilde{Z}^x \) in the direction \( \Upsilon \), and a.s.\ the following hold:
\begin{itemize}
  \item[(i)] The pathwise derivative of \( \tilde{Z}^x \) in the direction \( \Upsilon \), defined for all \( t \ge 0 \) by
  \[
  \partial_\Upsilon \tilde{Z}^x(t) := \lim_{\varepsilon \to 0} \frac{\tilde{Z}^x(t; \varepsilon \Upsilon) - \tilde{Z}^x(t)}{\varepsilon},
  \]
  exists.
  \item[(ii)] The pathwise derivative \( \partial_\Upsilon \tilde{Z}^x \) takes values in \( \D_{\ell, r}(\mathbb{R}^d) \), and is continuous at times \( t > 0 \) 
  when $\tilde Z^x(t) > 0$, or when $\tilde Z^x_j = 0$ for more than one $j \in \{1, 2, \cdots, d\}$.

  \item[(iii)] The right-continuous regularization of the pathwise derivative \( \partial_\Upsilon \tilde{Z}^x \) is equal to the derivative process \( D_\Upsilon \tilde{Z}^x \), i.e.,
  \[
  \lim_{s \downarrow t} \partial_\Upsilon \tilde{Z}^x(s) = D_\Upsilon \tilde{Z}^x(t), \qquad t \ge 0.
  \]
\end{itemize}
\end{proposition}

In the following result, we provide an alternative characterization of the derivative process \( D_\Upsilon \tilde{Z}^x \). 
To do so, we first introduce the following processes:
For all \( s \in [0,\infty) \), \( D \tilde Z^x (s,t) \) is the derivative process along \( \tilde{Z}^x \) starting at state \( \tilde{Z}^x(s) \) at time \( s \).  
Note that a.s.\ for all \( s \in \mathbb{R}_+ \), \( D {\tilde{Z}^x}(s, \cdot) \) exists and is pathwise unique.

\begin{proposition}\label{prop:bel3}
Almost surely, \( D_\Upsilon \tilde{Z}^x(t) = \int_0^t D{\tilde{Z}^x}(s,t) \upsilon(s)\,ds \), for all \( t \ge 0 \).
\end{proposition}

\begin{proof}[Proof of Theorem \ref{thm:bel}]
We first establish the Bismut–Elworthy–Li formula for bounded continuously differentiable functions with bounded gradients.

Let \( \zeta \) be such a function. Let \( x \in \mathbb{R}_+^d \) and \( t \in (0,\infty) \).  
For \( s \in [0,t] \), let \( \upsilon(s) := \frac{1}{t} \partial_j\tilde{Z}^x(s) \), 
where $\partial_j\tilde{Z}^x$ is the pathwise derivative of $\tilde Z^x$ in the direction of $e_j$, the $j$th standard unit vector, 
and let \( \Upsilon(s) := \int_0^s \upsilon(r)\,dr \), \( s \in [0,t] \).

Then, by Proposition \ref{prop:bel3}, a.s.
\[
D_\Upsilon \tilde{Z}^x(t) = \int_0^t D{\tilde{Z}^x}(s,t) \upsilon(s)\,ds = \frac{1}{t} \int_0^t D{\tilde{Z}^x}(s,t) \partial_j\tilde{Z}^x(s)\,ds
\]

For each \( s \in [0,t] \), \( D{\tilde{Z}^x}(s,t) \partial_j\tilde{Z}^x(s) = D_j \tilde{Z}^x(t) \), where $D_j \tilde{Z}^x(t)$ is the $j$th column 
of $D{\tilde{Z}^x}(t)$.
This is because a.s.\ for all \( s \in [0,t] \),
\[
D{\tilde{Z}^x}(s,t) \partial_j \tilde{Z}^x(s) = \partial \tilde{Z}^x(s,t+) \partial_j \tilde{Z}^x(s) = \partial_j \tilde{Z}^x(t+) = D_j \tilde{Z}^x(t).
\]

Thus, $D_\Upsilon \tilde{Z}^x(t) = D_j \tilde{Z}^x(t)$ a.s. As a result, 
\begin{align*}
\left. \frac{d}{d\varepsilon} \mathbb{E} \left[ \zeta\left( \tilde{Z}^x(t; \varepsilon \Upsilon) \right) \right] \right|_{\varepsilon = 0}
& = \mathbb{E} \left[ D\zeta\left( \tilde{Z}^x(t) \right) \partial_\Upsilon \tilde{Z}^x(t) \right] \\
& = \mathbb{E} \left[ D\zeta\left( \tilde{Z}^x(t) \right) D_\Upsilon \tilde{Z}^x(t) \right] \\
& = \mathbb{E} \left[ D\zeta\left( \tilde{Z}^x(t) \right) D_j\tilde{Z}^x(t) \right] \\
& = \mathbb{E} \left[ D\zeta\left( \tilde{Z}^x(t) \right) \partial_j\tilde{Z}^x(t) \right] \\
& = D_j \mathbb{E} \left[ \zeta( \tilde{Z}^x(t) ) \right].
\end{align*}
Here, the first equality follows by the chain rule and the assumption that $D\zeta$ is uniformly bounded, 
the second equality follows from Lemma \ref{lem:Z>0} and Proposition \ref{prop:bel2}, 
the third equality follows from the fact that a.s. $D_\Upsilon \tilde{Z}^x(t) = D_j \tilde{Z}^x(t)$, the fourth equality follows from 
Lemma \ref{lem:Z>0} and Proposition \ref{prop:LR2019b-main}, 
and the fifth equality follows by the chain rule and uniform boundedness of $D\zeta$ again. 

By Girsanov’s theorem:
\begin{align*}
&\mathbb{E} \left[
\zeta(\tilde{Z}^x(t; \varepsilon \Upsilon)) \exp\left\{
- \frac{\veps}{t} \int_0^t \langle \sigma^{-1} \partial_j \tilde{Z}^x(s), dB(s) \rangle
- \frac{\varepsilon^2}{2t^2} \int_0^t \| \sigma^{-1} \partial_j \tilde{Z}^x(s) \|^2 ds
\right\} \right] \\
= & \mathbb{E} \left[
\zeta(\tilde{Z}^x(t; \varepsilon \Upsilon)) \exp\left\{
- \frac{\veps}{t} \int_0^t \langle \sigma^{-1} D_j \tilde{Z}^x(s), dB(s) \rangle
- \frac{\varepsilon^2}{2t^2} \int_0^t \| \sigma^{-1} D_j \tilde{Z}^x(s) \|^2 ds
\right\} \right]\\
= & \mathbb{E} \left[ \zeta( \tilde{Z}^x(t) ) \right].
\end{align*}
Applying \( \left. \frac{d}{d\varepsilon} \right|_{\varepsilon=0} \) to both sides, we get:
\[
\left. \frac{d}{d\varepsilon} \mathbb{E} \left[ \zeta( \tilde{Z}^x(t; \varepsilon \Upsilon) ) \right] \right|_{\varepsilon=0} 
- \mathbb{E} \left[ \zeta( \tilde{Z}^x(t) ) \left( \frac{1}{t} \int_0^t \langle \sigma^{-1} D_j \tilde{Z}^x(s), dB(s) \rangle \right) \right] = 0.
\]
Hence,
\[
D_j \mathbb{E} \left[ \zeta( \tilde{Z}^x(t) ) \right]
= \mathbb{E} \left[ \zeta( \tilde{Z}^x(t) )\cdot \frac{1}{t}\int_0^t \langle \sigma^{-1} D_j \tilde{Z}^x(s), dB(s) \rangle \right].
\]
Because $j$ is arbitrary, we have established the Bismut-Elworthy-Li formula for any bounded continuously differentiable function $\zeta$ with bounded gradients. 
By a standard density argument, the Bismut-Elworthy-Li formula \eqref{eq:bel-main} is true for any bounded continuous function $\zeta$. 
To show that the Bismut-Elworthy-Li formula holds for any measurable function $\zeta$ with $\E[\zeta^2(\tilde Z^x(t))] < \infty$, 
the proof follows exactly the same reasoning as steps 3 and 4 in the proof of Theorem 3.8 in \cite{Banos2018}, for which we omit the details. 
This concludes the proof of the theorem.
\end{proof}

\section{Viscosity Solutions}\label{sec:viscosity-solutions}
Recall that $\tilde{v} \in \mathcal{B}$ is the unique solution of the fixed-point equation $v = F(v)$, 
where $F$ is the functional map defined in Eq. \eqref{eq:F}, 
and $\mathcal{B}$ is the space of measurable functions defined in Section \ref{sec:fixed-point-equation}.
Recall also the definition of the function $\tilde{V}$ defined in Eq. \eqref{eq:tildeV}, i.e.,
\begin{align*}
\tilde{V}(t,x) =&~\mathbb{E} \Bigg[ \int_t^{\tau} e^{-\beta(s - t)} \tilde \cH(\tilde Z^{t,x}(s), \tilde{v}(s, \tilde Z^{t,x}(s))) ds
+ \int_t^{\tau} e^{-\beta(s - t)} \kappa^{\top} d\tilde Y^{t,x}(s) + e^{-\beta(T - t)} \xi(\tilde Z^{t,x}(T)) \Bigg].
\end{align*}
The following results can be established. 
\begin{proposition}\label{prop:viscosity-tildeV}
$\tilde{V}$ is a viscosity solution of the HJB equation \eqref{eq:hjb1} -- \eqref{eq:hjb3} 
and it satisfies the polynomial growth condition \eqref{eq:c-xi-poly}. 
\end{proposition}
\begin{proposition}\label{prop:viscosity-V}
Recall the value function $V$ defined in Eq. \eqref{eq:V}. 
$V$ is a viscosity solution of the HJB equation \eqref{eq:hjb1} -- \eqref{eq:hjb3}, 
and it satisfies the polynomial growth condition \eqref{eq:c-xi-poly}. 
\end{proposition}
\begin{proposition}\label{prop:viscosity-unique}
There exists at most one viscosity solution of the HJB equation \eqref{eq:hjb1} -- \eqref{eq:hjb3} 
that satisfies the polynomial growth condition \eqref{eq:c-xi-poly}.
\end{proposition}
We skip the proofs of Propositions \ref{prop:viscosity-V} and \ref{prop:viscosity-unique}, 
because they follow standard machinery, similar to the proofs of Theorems 3.4 and 3.7 of \cite{borkar2004ergodic}, respectively, 
and do not involve the derivative process $D\tilde Z^x$. 
We only provide a proof sketch of Proposition \ref{prop:viscosity-unique} in Appendix \ref{app:viscosity-unique}.

\begin{proof}[Proof of Proposition \ref{prop:viscosity-tildeV}]
Let $\varphi \in C^{1,2}\left((0,T)\times \R_+^d\right)$ and $(t,x) \in (0,T) \times \mathbb{R}_+^d$ be such that $(t,x)$ is a strict maximum of $\tilde{V} - \varphi$, 
so that for all $(s,y)$ in a neighborhood of $(t,x)$ with $(s,y) \neq (t,x)$,
\begin{align}
\varphi(s,y) - \tilde{V}(s,y) > \varphi(t,x) - \tilde{V}(t,x).
\end{align}
By shifting $\varphi$ by a constant if necessary, we can assume that $\varphi(t,x) - \tilde{V}(t,x) = 0$, without loss of generality. 
We consider two cases: (i) $x > 0$ componentwise, and (ii) $x \in \partial \R_+^d$, i.e., there exists $i =1,2, \cdots, J$ such that $x_i = 0$. 

Case (i): $x \in \R_{++}^d$. In this case, 
\begin{align*}
\Phi_*(t,x,u(t,x), \varphi_t(t,x), \varphi_x(t,x), \varphi_{xx}(t,x)) & = \Phi(t,x,u(t,x), \varphi_t(t,x), \varphi_x(t,x), \varphi_{xx}(t,x)) \\
& = \Phi(t,x,\varphi(t,x), \varphi_t(t,x), \varphi_x(t,x), \varphi_{xx}(t,x)).
\end{align*}
Furthermore, because $t \in (0,T)$, $x\in\R_{++}^d$, $(t,x)$ is a strict maximum of $\tilde V-\varphi$, and both $\tilde V$ and $\varphi$ 
are differentiable in $x$ at $(t,x)$, we must have
\(
\tilde v (t,x) = \tilde V_x (t,x) = \varphi_x(t,x).
\)
Therefore, 
\begin{align}
& \quad \Phi_*(t,x,u(t,x), \varphi_t(t,x), \varphi_x(t,x), \varphi_{xx}(t,x)) \\
& = \Phi(t,x,\varphi(t,x), \varphi_t(t,x), \varphi_x(t,x), \varphi_{xx}(t,x)) \\
& = \left(-\varphi_t - \frac{1}{2} \mathrm{Tr}(\sigma \sigma^{\top} \varphi_{xx}) + \beta \varphi \right)(t,x) - \cH (x, \varphi_x (t,x)) \\
& = \left(-\varphi_t - \frac{1}{2} \mathrm{Tr}(\sigma \sigma^{\top} \varphi_{xx}) + \beta \varphi \right)(t,x) - \tilde \cH (x, \varphi_x (t,x)) - \tb^{\top}\varphi_x(t,x) \\
& = \left(-\varphi_t - \frac{1}{2} \mathrm{Tr}(\sigma \sigma^{\top} \varphi_{xx}) -\tb^{\top}\varphi_x + \beta \varphi \right)(t,x) - \tilde \cH (x, \varphi_x (t,x)) \\
& = \left(-\varphi_t - \frac{1}{2} \mathrm{Tr}(\sigma \sigma^{\top} \varphi_{xx}) -\tb^{\top}\varphi_x + \beta \varphi \right)(t,x) - \tilde \cH (x, \tilde v (t,x)).
\end{align}
Here, in the third equality, we used the fact that $\cH (x, p) = \tilde \cH (x, p) - \tb(x)^{\top} p$ for all $x$ and $p$, 
and in the last equality, we used the fact that $\tilde v (t,x) = \varphi_x(t,x)$.

Suppose, by way of contradiction, that 
\begin{align}
&~\Phi(t,x,\varphi(t,x), \varphi_t(t,x), \varphi_x(t,x), \varphi_{xx}(t,x)) \nonumber \\
= &~\left(-\varphi_t - \frac{1}{2} \mathrm{Tr}(\sigma \sigma^{\top} \varphi_{xx}) -\tb^{\top}\varphi_x + \beta \varphi \right)(t,x) - \tilde \cH (x, \tilde v (t,x)) > 0.
\label{eq:viscosity2}
\end{align}
By part (ii) of Theorem \ref{thm:fixed-point-tildeV}, 
the function $(s,y) \mapsto \left(-\varphi_t - \frac{1}{2} \mathrm{Tr}(\sigma \sigma^{\top} \varphi_{xx}) -\tb^{\top}\varphi_x + \beta \varphi \right)(s,y) - \tilde \cH (x, \tilde v (s,y))$ 
is continuous at $(t,x)$. Thus, we can find $\delta, \delta'>0$ such that the following properties hold:
\begin{itemize}
\item[(a)] If $\|(s,y) - (t,x)\|_2 \leq \delta$, $(s,y) \in (0,T)\times \R_{++}^d$, and $(s,y) \neq (t,x)$, then
$\varphi(s,y) - \tilde{V}(s,y) > 0$.
\item[(b)] If $\|(s,y) - (t,x)\|_2 \leq \delta$ and $(s,y) \in (0,T)\times \R_{++}^d$, then
\begin{equation}\label{eq:viscosity3}
\left(-\varphi_t - \frac{1}{2} \mathrm{Tr}(\sigma \sigma^{\top} \varphi_{xx}) -\tb^{\top}\varphi_x + \beta \varphi \right)(s,y) - \tilde \cH (x, \tilde v (s,y))  > \delta'.
\end{equation}
\item[(c)] If $\|(s,y) - (t,x)\|_2 \leq \delta$, then $y_i > 0$ for $i = 1, 2, \cdots, d$ (since $x_i > 0$ for all $i$ and $y$ is close to $x$).  
\end{itemize}
Define $\tau$ to be the stopping time that either $\delta/2$ amount of time has passed since time $t$, 
or $\tilde Z^{t,x}$ has deviated from $x$ by more than $\delta/2$, i.e.,
\begin{equation}
\tau = \inf \left\{s\ge t: \|\tilde Z^{t,x}(s) - x\|_2\ge \frac{\delta}{2} \right\} \wedge \left(t+\frac{\delta}{2}\right). 
\end{equation}
Then, whenever $s \in [t, \tau)$, we have $|s-t| \leq \delta/2$ 
and $\|\tilde Z^{t,x}(s) - x\|_2 \le \delta/2$, in which case 
\[
\left\|(s, \tilde Z^{t,x}(s)) - (t,x)\right\|_2 \leq \sqrt{ (\delta/2)^2 + (\delta/2)^2} < \delta,
\]
so that properties (a), (b) and (c) above are satisfied for all $(s, \tilde Z^{t,x}(s))$.
Furthermore, it is not difficult to see that $\E[\tau] > t$. 

Next, by applying Itô's lemma to $(s,x) \mapsto e^{-\beta (s-t)} \varphi(s,x)$, we have
\begin{eqnarray*}
\varphi(t,x) &=& \mathbb{E} \Bigg[ e^{-\beta(\tau - t)} \varphi(\tau, \tilde{Z}^{t,x}(\tau)) \\
&& \quad - \int_t^{\tau} e^{-\beta(s - t)} \left( \varphi_t + \frac{1}{2} \mathrm{Tr}(\sigma \sigma^{\top} \varphi_{xx}) + \tb^{\top} \varphi_x - \beta \varphi \right)(s, \tilde{Z}^{t,x}(s)) ds \\
&& \quad - \sum_{j=1}^d \int_t^{\tau} e^{-\beta(s - t)} R_j^{\top}\varphi_x(s, \tilde{Z}^{t,x}(s)) d\tilde{Y}^{t,x}_j(s) \Bigg].
\end{eqnarray*}
By the definition of $\tilde V$, we have
\begin{eqnarray*}
\tilde{V}(t,x) &=& \mathbb{E} \Bigg[ \int_t^{\tau} e^{-\beta(s - t)} \tilde \cH(\tilde{Z}^{t,x}(s), \tilde{v}(s, \tilde{Z}^{t,x}(s))) ds \\
& & \quad + \int_t^{\tau} e^{-\beta(s - t)} \kappa^{\top}d\tilde{Y}^{t,x}(s) + e^{-\beta(\tau - t)} \tilde{V}(\tau, \tilde{Z}^{t,x}(\tau)) \Bigg].
\end{eqnarray*}
Thus, we have
\begin{align}
0 &= \varphi(t,x) - \tilde{V}(t,x) \nonumber \\
&= \mathbb{E} \Bigg[ e^{-\beta(\tau - t)} \left( \varphi(\tau, \tilde{Z}^{t,x}(\tau)) - \tilde{V}(\tau, \tilde{Z}^{t,x}(\tau)) \right) \nonumber \\
&\quad - \int_t^{\tau} e^{-\beta(s - t)} \left[\left( \varphi_t + \frac{1}{2} \mathrm{Tr}(\sigma \sigma^{\top} \varphi_{xx}) + \tb^{\top} \varphi_x - \beta \varphi \right)(s, \tilde{Z}^{t,x}(s)) 
+ \tilde \cH (\tilde{Z}^{t,x}(s), \tilde{v}(s, \tilde Z^{t,x}(s))) \right] ds \nonumber \\
&\quad - \sum_{j=1}^d \int_t^{\tau} e^{-\beta(s - t)} \left[R_j^{\top} \varphi_x(s, \tilde{Z}^{t,x}(s)) + \kappa_j\right] d\tilde{Y}^{t,x}_j(s) \Bigg].
\label{eq:viscosity1}
\end{align}
By property (c) above, for $s \in [t, \tau)$, $\tilde Z_i^{t,x}(s) > 0$ for all $i$, so by definition, $d\tilde Y_i^{t,x}(s) = 0$ for all $i$ .
Thus, continuing from Ineq. \eqref{eq:viscosity1}, we have
\begin{align*}
0 &\ge -\mathbb{E} \left[ \int_t^{\tau} e^{-\beta(s - t)} \left[\left( \varphi_t + \frac{1}{2} \mathrm{Tr}(\sigma \sigma^{\top} \varphi_{xx}) + \tb^{\top} \varphi_x - \beta \varphi \right)(s, \tilde{Z}^{t,x}(s))\right.\right.  \\
& \qquad \quad \left.\left. + \tilde \cH (\tilde{Z}^{t,x}(s), \tilde{v}(s, \tilde Z^{t,x}(s))) \right] ds\right]  \\
& \ge \E \left[\int_t^\tau e^{-\beta(s-t)} \delta' ds\right] \ge \delta' e^{-\beta \delta/2} \E[\tau-t] > 0, 
\end{align*}
reaching a contradiction. Here, in the second inequality, we used Ineq. \eqref{eq:viscosity3}, 
and in the last inequality, we used the fact that $\E[\tau] > t$. 
Therefore, we have established that when $x > 0$, for $\varphi \in C^{1,2}$ and $(t,x) \in (0,T) \times \mathbb{R}_+^d$ a strict maximum of $\tilde V - \varphi$,
\begin{equation}
\Phi_*(t,x,\tilde V(t,x), \varphi_t(t,x), \varphi_x(t,x), \varphi_{xx}(t,x)) \leq 0.
\end{equation}
This completes the proof for case (i). 

Case (ii): $x \in \partial \R_+^d$. In this case, we show directly that there exists some $j \in \cI(x)$ such that 
\begin{equation}\label{eq:oblique-boundary1}
-R_j^{\top} \varphi_x(t,x) - \kappa_j \leq 0,
\end{equation}
which immediately implies Ineq. \eqref{eq:supersolution}. The proof of Eq. \eqref{eq:oblique-boundary1} relies on the following lemma, 
whose proof is relegated to Appendix \ref{sec:appendix-oblique-boundary}.
\begin{lemma}\label{lem:oblique-boundary}
For $x\in \partial \R_+^d$, let $j \in \cI(x)$. If $t \in (0,T)$, then
\begin{equation}\label{eq:oblique-boundary2}
R_j^{\top} \tilde{v}(t,x) + \kappa_j = 0.
\end{equation}
\end{lemma}
For $j \notin \cI(x)$, $x_j > 0$. In this case, because $(t,x)$ is a strict maximum of $\tilde V - \varphi$, 
we must have
\begin{equation}
\tilde v_j (t,x) = \left[\tilde V_x (t,x)\right]_j = \left[\varphi_x (t,x)\right]_j.
\end{equation}
Define the following $d\times d$ matrix \( R(x) \), which depends on \( x \in \partial \mathbb{R}_+^d \): 
\begin{itemize}
\item If \( j \in \mathcal{I}(x) \), then the \( j \)th column $R_j(x)$ of \( R(x) \) is \( R_j \);
\item If \( j \notin \mathcal{I}(x) \), then the \( j \)th column $R_j(x)$ of \( R(x) \) is \( e_j \), the $j$th standard unit vector.
\end{itemize}
Then, \( R(x) \) is invertible, and \( R(x)^{-1} \) has nonnegative entries, with positive diagonal entries.
Let \( \bOne := (1, \dots, 1)^{\top} \). Then,  
\(
\bOne^{\top}\varphi_x(t,x) \ge \bOne^{\top} \tilde{v}(t,x),
\)
because, by definition, \( \tilde{V} - \varphi \) has a strict maximum at \( (t,x) \), and $\bOne$ is a feasible direction 
to take directional derivative for both $\tilde V$ and $\varphi$.
Since $R(x)$ is invertible, we have
\(
\bOne = R(x) R(x)^{-1} \bOne = R(x) \alpha, 
\)
where $\alpha = R(x)^{-1} \bOne$ has all its components being positive. 
Then,
\[
\langle \bOne, \varphi_x(t,x) \rangle = \langle R(x) \alpha, \varphi_x(t,x) \rangle 
= \sum_j \alpha_j \langle R_j(x), \varphi_x(t,x) \rangle.
\]
Similarly, 
\[
\langle \bOne, \tilde v(t,x) \rangle 
= \sum_j \alpha_j \langle R_j(x), \tilde v(t,x) \rangle.
\]
If \( j \notin \mathcal{I}(x) \), then
\(
\langle R_j(x), \varphi_x(t,x) \rangle = \langle e_j, \varphi_x(t,x) \rangle = \left[\varphi_x(t,x)\right]_j= \tilde{v}_j (t,x) = \langle R_j(x), \tilde v(t,x) \rangle.
\)
If \( j \in \mathcal{I}(x) \), then
\(
\langle R_j(x), \tilde{v}(t,x) \rangle = \langle R_j, \tilde{v}(t,x) \rangle = -\kappa_j.
\)
Thus,
\begin{align*}
0 & \geq \langle \bOne, \tilde v(t,x) \rangle - \langle \bOne, \varphi_x(t,x) \rangle \\
& = \sum_j \alpha_j \langle R_j(x), \tilde v(t,x) \rangle - \sum_j \alpha_j \langle R_j(x), \varphi_x(t,x) \rangle \\
& = \sum_{j \notin \cI(x)}\alpha_j \langle R_j(x), \tilde v(t,x)  - \varphi_x(t,x) \rangle + \sum_{j \in \cI(x)} \alpha_j \langle R_j(x), \tilde v(t,x)  - \varphi_x(t,x) \rangle \\
& = \sum_{j \in \cI(x)} \alpha_j \langle R_j(x), \tilde v(t,x)  - \varphi_x(t,x) \rangle \\
& = \sum_{j \in \cI(x)} \alpha_j \left(-\kappa_j - \langle R_j, \varphi_x(t,x) \rangle\right),
\end{align*}
so there exists \( j \in \cI(x) \) such that \( - \langle R_j, \varphi_x(t,x) \rangle  -\kappa_j \leq 0 \), which is Ineq. \eqref{eq:oblique-boundary1}.

Thus, we have established Ineq. \eqref{eq:supersolution} for \( \tilde{V} \). Ineq. \eqref{eq:subsolution} 
can be established similarly for $\tilde V$. Therefore, we have proved that $\tilde V$ is a viscosity solution of the HJB equation \eqref{eq:hjb1} -- \eqref{eq:hjb3}.
\end{proof}

\section{Conclusion}\label{sec:conclusion}
In this paper, we developed, analyzed, and implemented a multilevel Picard method for solving high-dimensional drift control problems with reflections. 
We established complexity bounds for value function and gradient estimations, 
showing their polynomial dependence on problem dimension and inverse error tolerance under standard Lipschitz assumptions. 
Computationally, our numerical experiments indicate that the method achieves accurate value and gradient estimates 
and scales favorably with the problem dimension on instances motivated by dynamic control of queueing networks in heavy traffic. 
Moreover, the algorithm requires minimal hyperparameter tuning and admits straightforward parallelization, 
which further improves practicality.

There are many directions for future work, of which we highlight a few. On the theory side, 
one immediate direction is to extend the analysis of the MLP method to problems in which 
the reference processes and derivative processes must be simulated using time discretization; 
see e.g., \cite{NeufeldNguyenWu2025,NeufeldWu2025} for analyses of this type. 
A second potential direction concerns the use of the MLP method and its complexity analysis to 
provide complexity guarantees for neural network approximations for these problems; 
see \cite{NeufeldNguyen2024Rectified} and references therein for some exciting recent developments in this direction.
A third interesting direction concerns continuity properties of the gradient function. 
Even though the derivative processes have discontinuous dependence on the initial condition  
up to the boundary of the state space, in the Skorokhod topology, we only need the Malliavin weights (cf. the form of the map $F$ in Eq. \eqref{eq:F}), 
which may have continuous dependence on the initial condition, even up to the boundary. 
The latter property would imply the continuity of the gradient function in the entire space, not just in the interior, 
thereby establishing new regularity results for the value function of drift controls with reflections.  
On the computational side, it is fruitful to explore synergies between the MLP method 
and the deep BSDE method and their applications in solving control problems; for example, the MLP method 
may be used to gauge optimality gap of results produced by the deep BSDE method. 
Finally, while we carry out a non-parametric approach to solve the Picard iterations, 
it is also possible to solve the Picard iterations using parametric neural networks; 
see \cite{HanHuLongZhao2024} for a recent promising development in this direction. 

\bigskip
\noindent\textbf{Acknowledgements.} The author gratefully acknowledges the financial support from University of Chicago, Booth School of Business. 
The author would also like to thank Barış Ata, Amarjit Budhiraja, René Caldentey, Xinyun Chen, Arnulf Jentzen, Kavita Ramanan, Marty Reiman, Nian Si, Wouter van Eekelen 
and Amy Ward for insightful conversations and help with various aspects of the paper. 

\bibliographystyle{apalike}

\begin{thebibliography}{99}
  \bibitem{Andres2009}
S.~Andres (2009).
Pathwise differentiability for SDEs in a convex polyhedron with oblique reflection.
\emph{Annales de l'Institut Henri Poincar\'e, Probabilit\'es et Statistiques}, 45(1), 215--224.

  \bibitem{AsmussenGlynn2007}
S.~Asmussen and P.~W. Glynn (2007).
\newblock \emph{Stochastic Simulation: Algorithms and Analysis}.
\newblock Springer, Stochastic Modelling and Applied Probability, Vol.~57.
  \bibitem{AsmussenGlynnPitman1995}
S.~Asmussen, P.~W. Glynn, and J.~Pitman (1995).
\newblock Discretization error in simulation of one-dimensional reflecting Brownian motion.
\newblock \emph{Annals of Applied Probability}, 5(4):875--896.
  \bibitem{ata2025analysis} B. Ata and Y. Zhou (2025). \newblock Analysis and improvement of eviction enforcement. \newblock \emph{arXiv}. 
  \bibitem{ata2024singular} B. Ata, J.~M. Harrison and N. Si (2024b). \newblock Singular control of (reflected) Brownian motion: A computational method suitable for queueing applications. 
  \newblock \emph{Queueing Systems}, 108(3):215--251.
  \bibitem{AtaKasikaralar2023} B.~Ata and E.~Kasikaralar (2025). \newblock Dynamic Scheduling of a Multiclass Queue in the Halfin-Whitt Regime: A Computational Approach for High-Dimensional Problems. \newblock Forthcoming in \emph{Management Science}.
  \bibitem{AtaLeeSonmez2019} B.~Ata, D.~Lee, and E.~Sönmez (2019). \newblock Dynamic Volunteer Staffing in Multicrop Gleaning Operations. \newblock \emph{Operations Research}, 67(2):295--314.
  \bibitem{AtaHarrisonShepp2005} B.~Ata, J.~M. Harrison, and L.~A. Shepp (2005). \newblock Drift rate control of a Brownian processing system. \newblock \emph{Annals of Applied Probability}, 15(3):1764--1801.
  \bibitem{AtaHarrisonSi2024} B.~Ata, J.~M. Harrison, and N.~Si (2024a). \newblock Drift control of high-dimensional reflected Brownian motion: A computational method based on neural networks. \newblock \emph{Stochastic Systems}, 15(2):111--146. 
  \bibitem{atar2006singular} R. Atar and A. Budhiraja (2006). \newblock Singular control with state constraints on unbounded domain. 
  \newblock \emph{The Annals of Probability}, 34(5):1864--1909.
  \bibitem{atar2004scheduling} R. Atar, A. Mandelbaum and M. Reiman (2004). \newblock Scheduling a multi-class queue with many exponential servers: Asymptotic optimality in heavy-traffic. 
  \newblock \emph{Annals of Applied Probability}, 14(3):1084--1134.
  \bibitem{Atar2005} R.~Atar (2005). \newblock A diffusion model of scheduling control in queueing systems with many servers. \newblock \emph{Annals of Applied Probability}, 15(1B):820--852.
  \bibitem{axelsson2001finite} O. Axelsson and V.~A. Barker (2001). \newblock \emph{Finite element solution of boundary value problems: theory and computation}. \newblock SIAM.
  \bibitem{BarIlanMarionPerry2007} A.~Bar-Ilan, N.~Marion, and M.~Perry (2007). \newblock Drift control of international reserves. \newblock \emph{Journal of Economic Dynamics and Control}, 31(9):3110--3137.
  \bibitem{beck2020overcoming} C. Beck, F. Hornung, M. Hutzenthaler, A. Jentzen, and T. Kruse (2020). \newblock Overcoming the curse of dimensionality in the numerical approximation of Allen--Cahn partial differential equations via truncated full-history recursive multilevel Picard approximations. \newblock \emph{Journal of Numerical Mathematics}, 28(4):197--222.
  \bibitem{BeckJentzen2019} C.~Beck, S.~Becker, P.~Cheridito, A.~Jentzen, and A.~Neufeld (2021). \newblock Deep splitting method for parabolic PDEs. \newblock \emph{SIAM Journal on Scientific Computing}, 43(5):A3135--A3154. 
  \bibitem{becker2020numerical} S. Becker, R. Braunwarth, M. Hutzenthaler, A. Jentzen, and P. von Wurstemberger (2020). \newblock Numerical simulations for full history recursive multilevel Picard approximations for systems of high-dimensional partial differential equations. \newblock \emph{Communications in Computational Physics}, 28(5):2109--2138. doi:10.4208/cicp.OA-2020-0130.
  \bibitem{BellWilliams2001}
S.~L.~Bell and R.~J.~Williams (2001).
Dynamic scheduling of a system with two parallel servers in heavy traffic with resource pooling: Asymptotic optimality of a threshold policy.
\emph{The Annals of Applied Probability}, \textbf{11}(3), 608--649.
  \bibitem{Bellman1957} R.~Bellman (1957). \newblock \emph{Dynamic Programming}. \newblock Princeton University Press.
  \bibitem{Bismut1984} J.-M. Bismut (1984). \newblock Large Deviations and the Malliavin Calculus. \newblock \emph{Progress in Mathematics}, Vol. 45, Birkh\"auser.
  \bibitem{BlanchetChenSiGlynn2021}
J.~Blanchet, X.~Chen, N.~Si, and P.~W.~Glynn (2021).
Efficient steady-state simulation of high-dimensional stochastic networks.
\emph{Stochastic Systems}, \textbf{11}(2), 174--192.
  \bibitem{borkar2004ergodic} V. Borkar and A. Budhiraja (2004). \newblock Ergodic control for constrained diffusions: Characterization using HJB equations. 
  \newblock \emph{SIAM journal on control and optimization}, 43(4):1467--1492.
  \bibitem{BoussangeBeckerJentzenKuckuckPellissier2023}
V.~Boussange, S.~Becker, A.~Jentzen, B.~Kuckuck, and L.~Pellissier (2023).
\newblock Deep learning approximations for non-local nonlinear PDEs with Neumann boundary conditions.
\newblock \emph{Partial Differential Equations and Applications}, 4(51).
  \bibitem{budhiraja2011ergodic} A. Budhiraja, A.~P. Ghosh and C. Lee (2011). \newblock Ergodic rate control problem for single class queueing networks. 
  \newblock \emph{SIAM journal on control and optimization}, 49(4):1570--1606.
  \bibitem{ChenYao2001} H.~Chen and D.~D. Yao (2001). \newblock \emph{Fundamentals of Queueing Networks: Performance, Asymptotics, and Optimization}. \newblock Springer.
\bibitem{Costantini1992}
C.~Costantini (1992).
The Skorohod oblique reflection problem in domains with corners and application to stochastic differential equations.
\emph{Probability Theory and Related Fields}, \textbf{91}(1), 43--70.

  \bibitem{CrandallIshiiLions1992} M.~G. Crandall, H.~Ishii, and P.-L. Lions (1992). \newblock User's guide to viscosity solutions of second order partial differential equations. \newblock \emph{Bulletin of the American Mathematical Society}, 27(1):1--67.
  \bibitem{DaiZhong2010} M. Dai and Y. Zhong (2010). \newblock Penalty Methods for Continuous-Time Portfolio Selection with Proportional Transaction Costs. \newblock \emph{Journal of Computational Finance}, 13(3):1--31.
  \bibitem{VandeVate2021} J.~H. van de Vate (2021). \newblock Average cost Brownian drift control with proportional changeover costs. \newblock \emph{Stochastic Systems}, 11(13):218--263.
  \bibitem{DeuschelZambotti2005}
J.-D.~Deuschel and L.~Zambotti (2005).
Bismut--Elworthy's formula and random walk representation for SDEs with reflection.
\emph{Stochastic Processes and their Applications}, 115(6), 907--925.
  \bibitem{devroye} L.~Devroye (1986). \newblock \emph{Non-Uniform Random Variate Generation}. \newblock Springer-Verlag, New York.
  \bibitem{DevroyeKarasozenKohlerKorn2010}
L.~Devroye, B.~Karas\"ozen, M.~Kohler, and R.~Korn, eds. (2010).
\newblock \emph{Recent Developments in Applied Probability and Statistics: Dedicated to the Memory of J\"urgen Lehn}.
\newblock Physica–Springer, Heidelberg. \newblock doi:\,10.1007/978-3-7908-2598-5.

\bibitem{DupuisIshii1990}
P.~Dupuis and H.~Ishii (1990).
\newblock On oblique derivative problems for fully nonlinear second-order elliptic partial differential equations on nonsmooth domains.
\newblock \emph{Nonlinear Analysis}, 15(12):1123--1138. \newblock doi:\,10.1016/0362-546X(90)90048-L.

\bibitem{DupuisIshii1991}
P.~Dupuis and H.~Ishii (1991).
\newblock On oblique derivative problems for fully nonlinear second-order elliptic PDEs on domains with corners.
\newblock \emph{Hokkaido Mathematical Journal}, 20(1):135--164. \newblock doi:\,10.14492/hokmj/1381413798.

\bibitem{DupuisIshii1993}
P.~Dupuis and H.~Ishii (1993).
SDEs with oblique reflection on nonsmooth domains.
\emph{The Annals of Probability}, \textbf{21}(1), 554--580.

\bibitem{DupuisIshii2008corr}
P.~Dupuis and H.~Ishii (2008).
Correction: SDEs with oblique reflections on nonsmooth domains.
\emph{The Annals of Probability}, \textbf{36}(5), 1992--1997.

\bibitem{DupuisRamanan1999I}
P.~Dupuis and K.~Ramanan (1999).
Convex duality and the Skorokhod problem. I.
\emph{Probability Theory and Related Fields}, \textbf{115}(2), 153--195.

\bibitem{DupuisRamanan1999II}
P.~Dupuis and K.~Ramanan (1999).
Convex duality and the Skorokhod problem. II.
\emph{Probability Theory and Related Fields}, \textbf{115}(2), 197--236.

  \bibitem{hutzenthaler2019multilevel} W. E, M. Hutzenthaler, A. Jentzen, and T. Kruse (2019). \newblock On multilevel Picard numerical approximations for high-dimensional nonlinear parabolic partial differential equations and high-dimensional nonlinear backward stochastic differential equations. \newblock \emph{Journal of Scientific Computing}, 79(3):1534--1571.
  \bibitem{hutzenthaler2019multilevel_arxiv} W. E, M. Hutzenthaler, A. Jentzen, and T. Kruse (2019b). \newblock On multilevel Picard numerical approximations for high-dimensional nonlinear parabolic partial differential equations and high-dimensional nonlinear backward stochastic differential equations. \newblock \emph{arXiv:1708.03223}.
  \bibitem{WeinanHanJentzen2021} W.~E, J.~Han, and A.~Jentzen (2021). \newblock Algorithms for solving high dimensional PDEs: From nonlinear Monte Carlo to machine learning. \newblock \emph{Nonlinearity}, 35(1):278--310.
  \bibitem{ElworthyLi1994} K.~D. Elworthy and X.-M. Li (1994). \newblock Formulae for the derivatives of heat semigroups. \newblock \emph{Journal of Functional Analysis}, 125(1):252--286.
  \bibitem{FlemingSoner2006} W.~H. Fleming and H.~M. Soner (2006). \newblock \emph{Controlled Markov Processes and Viscosity Solutions}, 2nd ed. \newblock Springer.
  \bibitem{GhamamiWard2013}
S.~Ghamami and A.~R.~Ward (2013).
Dynamic scheduling of a two-server parallel server system with complete resource pooling and reneging in heavy traffic: Asymptotic optimality of a two-threshold policy.
\emph{Mathematics of Operations Research}, \textbf{38}(4), 761--824.
  \bibitem{GhoshWeerasinghe2007} M.~K. Ghosh and A.~P. Weerasinghe (2007). \newblock Optimal buffer size and admission control for a fluid model with Brownian input. \newblock \emph{Queueing Systems}, 56:133--146.
  \bibitem{GhoshWeerasinghe2010} M.~K. Ghosh and A.~P. Weerasinghe (2010). \newblock Optimal drift control with state costs and abandonments. \newblock \emph{Stochastics}, 82(1):43--62.
  \bibitem{giles2019generalised} M. B. Giles, A. Jentzen, and T. Welti (2019). \newblock Generalised multilevel Picard approximations. \newblock arXiv:1911.03188.
  \bibitem{Giles2008} M.~B. Giles (2008). \newblock Multilevel Monte Carlo path simulation. \newblock \emph{Operations Research}, 56(3):607--617.
  \bibitem{Giles2015} M.~B. Giles (2015). \newblock Multilevel Monte Carlo methods. \newblock \emph{Acta Numerica}, 24:259--328.
  \bibitem{HanHuLongZhao2024}
J.~Han, W.~Hu, J.~Long, and Y.~Zhao (2024).
\newblock Deep Picard iteration for high-dimensional nonlinear PDEs.
\newblock \emph{arXiv:2409.08526}.
  \bibitem{HanJentzenE2018PNAS} J.~Han, A.~Jentzen, and W.~E (2018). \newblock Solving high-dimensional partial differential equations using deep learning. \newblock \emph{Proceedings of the National Academy of Sciences}, 115(34):8505--8510.
  \bibitem{harrison2004dynamic} J.~M. Harrison and A. Zeevi (2004). \newblock Dynamic scheduling of a multiclass queue in the Halfin-Whitt heavy traffic regime. 
  \newblock \emph{Operations Research}, 52(2):243--257.
  \bibitem{HarrisonReiman81} J.~M. Harrison and M.~I. Reiman (1981). \newblock Reflected Brownian motion on an orthant. \newblock \emph{Annals of Probability}, 9(2):302--308.
  \bibitem{Harrison1998}
J.~M.~Harrison (1998).
Heavy traffic analysis of a system with parallel servers: Asymptotic optimality of discrete-review policies.
\emph{The Annals of Applied Probability}, \textbf{8}(3), 822--848.
  \bibitem{Harrison2000}
J.~M.~Harrison (2000).
Brownian models of open processing networks: Canonical representation of workload.
\emph{The Annals of Applied Probability}, 10(1), 75--103.

\bibitem{Harrison2006Corr}
J.~M.~Harrison (2006).
Correction: Brownian models of open processing networks: Canonical representation of workload.
\emph{The Annals of Applied Probability}, 16(3), 1703--1732.

  \bibitem{HarrisonVanMieghem1997}
  J.~M.~Harrison and J.~A.~Van Mieghem (1997).
  Dynamic control of Brownian networks: State space collapse and equivalent workload formulations.
  \emph{The Annals of Applied Probability}, 7(3), 747--771.
\bibitem{HarrisonReiman1981}
J.~M.~Harrison and M.~I.~Reiman (1981).
Reflected Brownian motion on an orthant.
\emph{The Annals of Probability}, \textbf{9}(2), 302--308.


\bibitem{HarrisonLopez1999}
J.~M.~Harrison and M.~J.~L\'opez (1999).
Heavy traffic resource pooling in parallel server systems.
\emph{Queueing Systems}, 33(4), 339--368.

  \bibitem{HornJohnson2012}
R.~A.~Horn and C.~R.~Johnson (2012).
\emph{Matrix Analysis} (2nd ed.).
Cambridge University Press, Cambridge.
  \bibitem{HurePhamWarin2020} C.~Hur\'e, H.~Pham, and X.~Warin (2020). \newblock Deep backward schemes for high-dimensional nonlinear PDEs. \newblock \emph{Mathematics of Computation}, 89(324):1547--1579.
  \bibitem{hutzenthaler2022multilevel1} M. Hutzenthaler and T. A. Nguyen (2022). \newblock Multilevel Picard approximations of high-dimensional semilinear partial differential equations with locally monotone coefficient functions. \newblock \emph{Applied Numerical Mathematics}, 181:151--175.
  \bibitem{HK2020} M. Hutzenthaler and T. Kruse (2020). \newblock Multilevel Picard approximations of high-dimensional semilinear parabolic differential equations with gradient-dependent nonlinearities. \newblock \emph{SIAM Journal on Numerical Analysis}, 58(2):929--961.
  \bibitem{hutzenthaler2021multilevel} M. Hutzenthaler, A. Jentzen, and T. Kruse (2021). \newblock Multilevel Picard iterations for solving smooth semilinear parabolic heat equations. \newblock \emph{Partial Differential Equations and Applications}, 2(6):1--31.
  \bibitem{HJK2022} M. Hutzenthaler, A. Jentzen, and T. Kruse (2022). \newblock Overcoming the curse of dimensionality in the numerical approximation of parabolic partial differential equations with gradient-dependent nonlinearities. \newblock \emph{Foundations of Computational Mathematics}, 22(4):905--966.
  \bibitem{HJKN2020} M. Hutzenthaler, A. Jentzen, T. Kruse, and T. A. Nguyen (2020). \newblock Multilevel Picard approximations for high-dimensional semilinear second-order PDEs with Lipschitz nonlinearities. \newblock arXiv:2009.02484.
  \bibitem{HJKNW2020} M. Hutzenthaler, A. Jentzen, T. Kruse, T. A. Nguyen, and P. von Wurstemberger (2020). \newblock Overcoming the curse of dimensionality in the numerical approximation of semilinear parabolic partial differential equations. \newblock \emph{Proceedings of the Royal Society A}, 476(2244):20190630.
  \bibitem{hutzenthaler2022multilevel} M. Hutzenthaler, T. Kruse, and T. A. Nguyen (2022). \newblock Multilevel Picard approximations for McKean--Vlasov stochastic differential equations. \newblock \emph{Journal of Mathematical Analysis and Applications}, 507(1):125761.
  \bibitem{HutzenthalerJentzenKruse2019} M.~Hutzenthaler, A.~Jentzen, and T.~Kruse (2019). \newblock Overcoming the curse of dimensionality in the numerical approximation of semilinear parabolic partial differential equations with gradient-dependent nonlinearities. \newblock \emph{Foundations of Computational Mathematics}, 19:1--69. 
  \bibitem{HutzenthalerETAL2021} M.~Hutzenthaler, A.~Jentzen, T. Kruse, T.~A. Nguyen and P.~von Wurstemberger (2020). \newblock Overcoming the curse of dimensionality in the numerical approximation of semilinear parabolic partial differential equations. \newblock \emph{Proceedings of the Royal Society A}, 476(2244):20190630. 
  \bibitem{Judd1998} K.~L. Judd (1998). \newblock \emph{Numerical Methods in Economics}. \newblock MIT Press.
  \bibitem{KellaRamasubramanian2012}
O.~Kella and S.~Ramasubramanian (2012).
\newblock Asymptotic irrelevance of initial conditions for Skorokhod reflection mapping on the nonnegative orthant.
\newblock \emph{Mathematics of Operations Research}, 37(2):301--312.
\bibitem{KellaWhitt1996}
O.~Kella and W.~Whitt (1996).
\newblock Stability and structural properties of stochastic storage networks.
\newblock \emph{Journal of Applied Probability}, 33(4):1169--1180.
  \bibitem{kim2018dynamic} J. Kim, R. Randhawa and A. Ward (2018). \newblock Dynamic scheduling in a many-server, multiclass system: The role of customer impatience in large systems.  
  \newblock \emph{Manufacturing $\&$ Service Operations Management}, 20(2):285--301.
  \bibitem{KushnerDupuis2001} H.~J. Kushner and P.~G. Dupuis (2001). \newblock \emph{Numerical Methods for Stochastic Control Problems in Continuous Time}. \newblock Springer.
  \bibitem{Lepingle1984}
D.~L\'epingle (1984).
\newblock Euler scheme for reflected stochastic differential equations.
\newblock \emph{Mathematics and Computers in Simulation}, 26(2):215--227.
  \bibitem{LionsSznitman1984} P.-L. Lions and A.-S. Sznitman (1984). \newblock Stochastic differential equations with reflecting boundary conditions. \newblock \emph{Communications on Pure and Applied Mathematics}, 37(4):511--537.
  \bibitem{LipshutzRamanan2018} D.~Lipshutz and K.~Ramanan (2018). \newblock On Directional Derivatives of Skorokhod Maps in Convex Polyhedral Domains. \newblock \emph{Annals of Applied Probability}, 28(2):688--750.
  \bibitem{LipshutzRamanan2019a} D.~Lipshutz and K.~Ramanan (2019a). \newblock A Monte Carlo method for estimating sensitivities of reflected diffusions in convex polyhedral domains. \newblock \emph{Stochastic Systems}, 9(1):1--44.
  \bibitem{LipshutzRamanan2019b} D.~Lipshutz and K.~Ramanan (2019b). \newblock Pathwise differentiability of reflected diffusions in convex polyhedral domains. \newblock \emph{Annales de l'Institut Henri Poincar\'e, Probabilit\'es et Statistiques}, 55(3):1442--1464.
  \bibitem{MandelbaumRamanan2010} A.~Mandelbaum and K.~Ramanan (2010). \newblock Directional derivatives of oblique reflection maps. \newblock \emph{Mathematics of Operations Research}, 35(3):527--558.
  \bibitem{OrmeciMatogluVandeVate2011} M.~Ormeci Matoglu, and J. Vande Vate (2011). \newblock Drift Control with Changeover Costs. \newblock \emph{Operations Research}, 59(2):427--439.
  \bibitem{menaldi1989optimal} J.~L. Menaldi and M.~I. Taksar (1989). \newblock Optimal correction problem of a multidimensional stochastic system. 
  \newblock \emph{Automatica}, 25(2):223--232.
  \bibitem{NeufeldNguyen2024Rectified}
A.~Neufeld and T.~A. Nguyen (2024).
\newblock Rectified deep neural networks overcome the curse of dimensionality in the numerical approximation of gradient-dependent semilinear heat equations.
\newblock \emph{arXiv:2403.09200}.
  \bibitem{NeufeldNguyenWu2025} A.~Neufeld, T.~A. Nguyen, and S.~Wu (2025). \newblock Multilevel Picard approximations overcome the curse of dimensionality for semilinear PDEs with gradient-dependent nonlinearities. \newblock \emph{Journal of Complexity}, 90:101946.

  \bibitem{NW2022} A. Neufeld and S. Wu (2025). \newblock Multilevel Picard approximation algorithm for semilinear partial integro-differential equations and its complexity analysis. \newblock \emph{Stochastics and Partial Differential Equations: Analysis and Computations}, 1--59.
  \bibitem{NeufeldWu2025} A.~Neufeld and S.~Wu (2025). \newblock Multilevel Picard algorithm for general semilinear parabolic PDEs with gradient-dependent nonlinearities. \newblock \emph{Journal of Numerical Mathematics}. 
    \bibitem{Banos2018} D.~Ba\~nos (2018). \newblock The Bismut--Elworthy--Li formula for mean-field stochastic differential equations. \newblock \emph{Annales de l'Institut Henri Poincar\'e, Probabilit\'es et Statistiques}, 54(1):220--233.
  \bibitem{PardouxPeng1990} E.~Pardoux and S.~Peng (1990). \newblock Adapted solution of a backward stochastic differential equation. \newblock \emph{Systems \& Control Letters}, 14(1):55--61.
  \bibitem{PardouxPeng1992} E.~Pardoux and S.~Peng (1992). \newblock Backward stochastic differential equations and quasilinear parabolic partial differential equations. \newblock In: \emph{Stochastic Partial Differential Equations and Their Applications} (Charlotte, NC, 1991), Lecture Notes in Control and Information Sciences, Vol.~176, Springer, Berlin, pp.~200--217.
  \bibitem{Peng1991} S.~Peng (1991). \newblock Probabilistic interpretation for systems of quasilinear parabolic partial differential equations. \newblock \emph{Stochastics and Stochastics Reports}, 37(1--2):61--74.
  \bibitem{PesicWilliams2016} M.~Pe\v{s}i\'c and R.~J. Williams (2016). \newblock Dynamic Scheduling for Parallel Server Systems in Heavy Traffic: Graphical Structure, Decoupled Workload Matrix and some Sufficient Conditions for Solvability of the Brownian Control Problem. \newblock \emph{Stochastic Systems}, 6(1):26--89. 
  \bibitem{REF} K.~Ramanan (2006). \newblock Reflected diffusions defined via the extended Skorokhod map. \newblock \emph{Electronic Journal of Probability}, 11:934--992.
  \bibitem{Reiman1984} M.~I. Reiman (1984). \newblock Open queueing networks in heavy traffic. \newblock \emph{Mathematics of Operations Research}, 9(3):441--458.
  \bibitem{RubinoAta2009} M.~Rubino and B.~Ata (2009). \newblock Dynamic scheduling in a make-to-order system with cancellations. \newblock \emph{Operations Research}, 57(1):215--231.
\bibitem{Saisho1987}
Y.~Saisho (1987).
Stochastic differential equations for multi-dimensional domain with reflecting boundary.
\emph{Probability Theory and Related Fields}, \textbf{74}(3), 455--477.

\bibitem{Skorokhod1961}
A.~V.~Skorokhod (1961).
Stochastic equations for diffusion processes in a bounded region. I.
\emph{Theory of Probability and Its Applications}, \textbf{6}(3), 264--274.

\bibitem{Skorokhod1962}
A.~V.~Skorokhod (1962).
Stochastic equations for diffusion processes in a bounded region. II.
\emph{Theory of Probability and Its Applications}, \textbf{7}(1), 3--23.

\bibitem{Steele2001}
J.~M. Steele (2001).
\newblock \emph{Stochastic Calculus and Financial Applications}.
\newblock Springer, New York. \newblock (Applications of Mathematics, Vol.~45). \newblock doi:\,10.1007/978-1-4684-9305-4.

\bibitem{StroockVaradhan1971}
D.~W.~Stroock and S.~R.~S.~Varadhan (1971).
Diffusion processes with boundary conditions.
\emph{Communications on Pure and Applied Mathematics}, \textbf{24}(2), 147--225.

\bibitem{Tanaka1979}
H.~Tanaka (1979).
Stochastic differential equations with reflecting boundary condition in convex regions.
\emph{Hiroshima Mathematical Journal}, \textbf{9}(1), 163--177.

\bibitem{TaylorWilliams1993}
L.~M.~Taylor and R.~J.~Williams (1993).
Existence and uniqueness of semimartingale reflecting Brownian motions in an orthant.
\emph{Probability Theory and Related Fields}, \textbf{96}(3), 283--317.

  \bibitem{TezcanDai2010}
T.~Tezcan and J.~G.~Dai (2010).
Dynamic control of $N$-systems with many servers: Asymptotic optimality of a static priority policy in heavy traffic.
\emph{Operations Research}, \textbf{58}(1), 94--110.
  \bibitem{thomas2013numerical} J.~W. Thomas (2013). \newblock \emph{Numerical partial differential equations: finite difference methods}. \newblock Springer. 
\bibitem{VaradhanWilliams1985}
S.~R.~S.~Varadhan and R.~J.~Williams (1985).
Brownian motion in a wedge with oblique reflection.
\emph{Communications on Pure and Applied Mathematics}, \textbf{38}(4), 405--443.

  \bibitem{WangWangZhang2021}
S.~Wang, X.~Wang, and J.~Zhang (2021).
A Review of Flexible Processes and Operations.
\emph{Production and Operations Management}, 30(6), 1804--1824.
  \bibitem{Williams1998} R.~J. Williams (1998). \newblock Reflecting diffusions and queueing networks. \newblock In \emph{Handbook of Stochastic Networks}, eds. S.~R. S. Varadhan and S.~R.~S. Varadhan. 
  \bibitem{williams2016stochastic} R. Williams (2016). \newblock Stochastic processing networks. \newblock \emph{Annual Review of Statistics and Its Application}, 3(1):323--345.
  \bibitem{williams1994regularity} S.~A. Williams, P.~L. Chow and J.~L. Menaldi (1994). \newblock Regularity of the Free Boundary in Singular Stochastic Control.
  \newblock \emph{Journal of Differential Equations}, 111(1):175--201.
  \bibitem{YongZhou1999} J.~Yong and X.~Y. Zhou (1999). \newblock \emph{Stochastic Controls: Hamiltonian Systems and HJB Equations}. \newblock Applications of Mathematics, Vol.~43. Springer Science \& Business Media.
\end{thebibliography}

\begin{thebibliography}{99}
\bibitem{beck2020overview} C.~Beck, S.~Hutzenthaler, A.~Jentzen, and T.~Kuckuck (2020). \newblock An overview on deep learning-based approximation methods for PDEs. \newblock In: \emph{Numerical Methods for High-Dimensional PDEs}. Springer.
\bibitem{BeckHutzenthalerJentzenKuckuck2023} C.~Beck, S.~Hutzenthaler, A.~Jentzen, and T.~Kuckuck (2023). \newblock A survey on deep learning approaches for high-dimensional PDEs. \newblock \emph{SeMA Journal}, 80: 1--57.
\bibitem{EHJSurvey2022} W.~E, J.~Han, and A.~Jentzen (2022). \newblock Deep learning-based numerical methods for high-dimensional PDEs and BSDEs. \newblock \emph{Acta Numerica}, 31: 1--87.
\bibitem{AtaHarrisonShepp2005} B.~Ata, J.~M. Harrison, and L.~A. Shepp (2005). \newblock Drift rate control of a Brownian processing system. \newblock \emph{Annals of Applied Probability}, 15(3):1764--1801.
\bibitem{OrmeciMatogluVandeVate2011} M.~Ormeci Matoglu, and J. Vande Vate (2011). \newblock Drift Control with Changeover Costs. \newblock \emph{Operations Research}, 59(2):427--439.
\bibitem{VandeVate2021} J.~H. van de Vate (2021). \newblock Average cost Brownian drift control with proportional changeover costs. \newblock \emph{Stochastic Systems}, 11(13):218--263.
\bibitem{GhoshWeerasinghe2007} M.~K. Ghosh and A.~P. Weerasinghe (2007). \newblock Optimal buffer size and admission control for a fluid model with Brownian input. \newblock \emph{Queueing Systems}, 56:133--146.
\bibitem{GhoshWeerasinghe2010} M.~K. Ghosh and A.~P. Weerasinghe (2010). \newblock Optimal drift control with state costs and abandonments. \newblock \emph{Stochastics}, 82(1):43--62.
\bibitem{RubinoAta2009} M.~Rubino and B.~Ata (2009). \newblock Dynamic scheduling in a make-to-order system with cancellations. \newblock \emph{Operations Research}, 57(1):215--231.
\bibitem{AtaLeeSonmez2019} B.~Ata, D.~Lee, and E.~Sönmez (2019). \newblock Dynamic Volunteer Staffing in Multicrop Gleaning Operations. \newblock \emph{Operations Research}, 67(2):295--314.
\bibitem{BarIlanMarionPerry2007} A.~Bar-Ilan, N.~Marion, and M.~Perry (2007). \newblock Drift control of international reserves. \newblock \emph{Journal of Economic Dynamics and Control}, 31(9):3110--3137.
\bibitem{AtaKasikaralar2023} B.~Ata and E.~Kasikaralar (2025). \newblock Dynamic Scheduling of a Multiclass Queue in the Halfin-Whitt Regime: A Computational Approach for High-Dimensional Problems. \newblock Forthcoming in \emph{Management Science}.
\bibitem{BeckHutzenthalerJentzenKuckuck2023b} C.~Beck, M.~Hutzenthaler, A.~Jentzen, and B.~Kuckuck (2023). \newblock An overview on deep learning methods for PDEs. \newblock \emph{GAMM-Mitteilungen}, 46(2):e202300008.
\bibitem{KushnerDupuis2001} H.~J. Kushner and P.~G. Dupuis (2001). \newblock \emph{Numerical Methods for Stochastic Control Problems in Continuous Time}. \newblock Springer.
\bibitem{AtaHarrisonSi2024} B.~Ata, J.~M. Harrison, and N.~Si (2024). \newblock Drift control of high-dimensional reflected Brownian motion: A computational method based on neural networks. \newblock \emph{Stochastic Systems}, 15(2):111--146. 
\bibitem{Atar2005} R.~Atar (2005). \newblock A diffusion model of scheduling control in queueing systems with many servers. \newblock \emph{Annals of Applied Probability}, 15(1B):820--852.

\bibitem{BudhirajaDupuis2000} A.~Budhiraja and P.~Dupuis (2000). \newblock Simple necessary and sufficient conditions for the stability of constrained processes. \newblock \emph{SIAM Journal on Applied Mathematics}, 59(5):1686--1700.

\bibitem{Budhiraja2000} A.~Budhiraja (2000). \newblock Variational representations for positive functionals of infinite dimensional Brownian motion. \newblock \emph{Probability and Mathematical Statistics}, 20:55--82. 
\bibitem{Banos2018} D.~Ba\~nos (2018). \newblock The Bismut--Elworthy--Li formula for mean-field stochastic differential equations. \newblock \emph{Annales de l'Institut Henri Poincar\'e, Probabilit\'es et Statistiques}, 54(1):220--233.

\bibitem{BeckJentzen2019} C.~Beck, S.~Becker, P.~Cheridito, A.~Jentzen, and A.~Neufeld (2021). \newblock Deep splitting method for parabolic PDEs. \newblock \emph{SIAM Journal on Scientific Computing}, 43(5):A3135--A3154. 

\bibitem{Bellman1957} R.~Bellman (1957). \newblock \emph{Dynamic Programming}. \newblock Princeton University Press.

\bibitem{Bismut1984} J.-M. Bismut (1984). \newblock Large Deviations and the Malliavin Calculus. \newblock \emph{Progress in Mathematics}, Vol. 45, Birkh\"auser.

\bibitem{ChenYao2001} H.~Chen and D.~D. Yao (2001). \newblock \emph{Fundamentals of Queueing Networks: Performance, Asymptotics, and Optimization}. \newblock Springer.

\bibitem{Dai1995} J.~G. Dai (1995). \newblock On positive Harris recurrence of multiclass queueing networks: A unified approach via fluid limit models. \newblock \emph{Annals of Applied Probability}, 5(1):49--77.

\bibitem{DaiWilliams1995} J.~G. Dai and R.~J. Williams (1995). \newblock Existence and uniqueness of semimartingale reflecting Brownian motions in convex polyhedra. \newblock \emph{Theory of Probability and Its Applications}, 40(1):1--40.

\bibitem{devroye} L.~Devroye (1986). \newblock \emph{Non-Uniform Random Variate Generation}. \newblock Springer-Verlag, New York.

\bibitem{WeinanHanJentzen2021} W.~E, J.~Han, and A.~Jentzen (2021). \newblock Algorithms for solving high dimensional PDEs: From nonlinear Monte Carlo to machine learning. \newblock \emph{Nonlinearity}, 35(1):278--310.

\bibitem{CrandallIshiiLions1992} M.~G. Crandall, H.~Ishii, and P.-L. Lions (1992). \newblock User's guide to viscosity solutions of second order partial differential equations. \newblock \emph{Bulletin of the American Mathematical Society}, 27(1):1--67.

\bibitem{ElworthyLi1994} K.~D. Elworthy and X.-M. Li (1994). \newblock Formulae for the derivatives of heat semigroups. \newblock \emph{Journal of Functional Analysis}, 125(1):252--286.

\bibitem{FahimTouziWarin2011} A.~Fahim, N.~Touzi, and X.~Warin (2011). \newblock A probabilistic numerical method for fully nonlinear parabolic PDEs. \newblock \emph{Annals of Applied Probability}, 21(4):1322--1364.

\bibitem{FlemingSoner2006} W.~H. Fleming and H.~M. Soner (2006). \newblock \emph{Controlled Markov Processes and Viscosity Solutions}, 2nd ed. \newblock Springer.

\bibitem{Giles2008} M.~B. Giles (2008). \newblock Multilevel Monte Carlo path simulation. \newblock \emph{Operations Research}, 56(3):607--617.

\bibitem{HanJentzenE2018PNAS} J.~Han, A.~Jentzen, and W.~E (2018). \newblock Solving high-dimensional partial differential equations using deep learning. \newblock \emph{Proceedings of the National Academy of Sciences}, 115(34):8505--8510.

\bibitem{Harrison1985} J.~M. Harrison (1985). \newblock \emph{Brownian Motion and Stochastic Flow Systems}. \newblock Wiley.

\bibitem{HarrisonReiman81} J.~M. Harrison and M.~I. Reiman (1981). \newblock Reflected Brownian motion on an orthant. \newblock \emph{Annals of Probability}, 9(2):302--308.

\bibitem{TassiulasEphremides1992} L.~Tassiulas and A.~Ephremides (1992). \newblock Stability properties of constrained queueing systems and scheduling policies for maximum throughput in multihop radio networks. \newblock \emph{IEEE Transactions on Automatic Control}, 37(12):1936--1948.

\bibitem{HenryLabordereTanTouzi2016} P.~Henry-Labord\`ere, X.~Tan, and N.~Touzi (2016). \newblock Branching diffusion representation of semilinear PDEs and Monte Carlo approximation. \newblock \emph{Annals of Applied Probability}, 26(6):3383--3413.

\bibitem{HurePhamWarin2020} C.~Hur\'e, H.~Pham, and X.~Warin (2020). \newblock Deep backward schemes for high-dimensional nonlinear PDEs. \newblock \emph{Mathematics of Computation}, 89(324):1547--1579.

\bibitem{McKeown1999} N.~McKeown (1999). \newblock The iSLIP scheduling algorithm for input-queued switches. \newblock \emph{IEEE/ACM Transactions on Networking}, 7(2):188--201.

\bibitem{HutzenthalerETAL2021} M.~Hutzenthaler, A.~Jentzen, T. Kruse, T.~A. Nguyen and P.~von Wurstemberger (2020). \newblock Overcoming the curse of dimensionality in the numerical approximation of semilinear parabolic partial differential equations. \newblock \emph{Proceedings of the Royal Society A}, 476(2244):20190630. 

\bibitem{HutzenthalerJentzenKruse2019} M.~Hutzenthaler, A.~Jentzen, and T.~Kruse (2019). \newblock Overcoming the curse of dimensionality in the numerical approximation of semilinear parabolic partial differential equations with gradient-dependent nonlinearities. \newblock \emph{Foundations of Computational Mathematics}, 19:1--69. 

\bibitem{NeufeldNguyenWu2025} A.~Neufeld, T.~A. Nguyen, and S.~Wu (2025). \newblock Multilevel Picard approximations overcome the curse of dimensionality for semilinear PDEs with gradient-dependent nonlinearities. \newblock Preprint. Available at: \url{https://personal.ntu.edu.sg/ariel.neufeld/MLP_gradient-Overcome.pdf}

\bibitem{NeufeldWu2025} A.~Neufeld and S.~Wu (2025). \newblock Multilevel Picard algorithm for general semilinear parabolic PDEs with gradient-dependent nonlinearities. \newblock Preprint. Available at: \url{https://personal.ntu.edu.sg/ariel.neufeld/MLP_gradient.pdf}

\bibitem{Judd1998} K.~L. Judd (1998). \newblock \emph{Numerical Methods in Economics}. \newblock MIT Press.

\bibitem{KaratzasShreve98} I.~Karatzas and S.~~E. Shreve (1998). \newblock \emph{Brownian Motion and Stochastic Calculus}, 2nd~~ed. \newblock Springer.

\bibitem{RobertsMassoulie2000} J.~W. Roberts and L.~Massouli\'e (2000). \newblock Bandwidth sharing and admission control for elastic traffic. \newblock \emph{Telecommunication Systems}, 15:185--201.

\bibitem{LionsSznitman1984} P.-L. Lions and A.-S. Sznitman (1984). \newblock Stochastic differential equations with reflecting boundary conditions. \newblock \emph{Communications on Pure and Applied Mathematics}, 37(4):511--537.

\bibitem{LipshutzRamanan2018} D.~Lipshutz and K.~Ramanan (2018). \newblock On Directional Derivatives of Skorokhod Maps in Convex Polyhedral Domains. \newblock \emph{Annals of Applied Probability}, 28(2):688--750.

\bibitem{LipshutzRamanan2019a} D.~Lipshutz and K.~Ramanan (2019a). \newblock A Monte Carlo method for estimating sensitivities of reflected diffusions in convex polyhedral domains. \newblock \emph{Stochastic Systems}, 9(1):1--44.

\bibitem{LipshutzRamanan2019b} D.~Lipshutz and K.~Ramanan (2019b). \newblock Pathwise differentiability of reflected diffusions in convex polyhedral domains. \newblock \emph{Annales de l'Institut Henri Poincar\'e, Probabilit\'es et Statistiques}, 55(3):1442--1464.

\bibitem{MandelbaumRamanan2010} A.~Mandelbaum and K.~Ramanan (2010). \newblock Directional derivatives of oblique reflection maps. \newblock \emph{Mathematics of Operations Research}, 35(3):527--558.

\bibitem{XuZhong2019SPII} K.~Xu and Y.~Zhong (2019). \newblock Information and Memory in Dynamic Resource Allocation. \newblock \emph{arXiv:1904.08365}. Available at: https://arxiv.org/abs/1904.08365

\bibitem{PesicWilliams2016} M.~Pe\v{s}i\'c and R.~J. Williams (2016). \newblock Dynamic Scheduling for Parallel Server Systems in Heavy Traffic: Graphical Structure, Decoupled Workload Matrix and some Sufficient Conditions for Solvability of the Brownian Control Problem. \newblock \emph{Stochastic Systems}, 6(1):26--89. 

\bibitem{Reiman1984} M.~I. Reiman (1984). \newblock Open queueing networks in heavy traffic. \newblock \emph{Mathematics of Operations Research}, 9(3):441--458.

\bibitem{Williams1998} R.~J. Williams (1998). \newblock Reflecting diffusions and queueing networks. \newblock In \emph{Handbook of Stochastic Networks}, eds. S.~R. S. Varadhan and S.~R.~S. Varadhan. 

\bibitem{REF} K.~Ramanan (2006). \newblock Reflected diffusions defined via the extended Skorokhod map. \newblock \emph{Electronic Journal of Probability}, 11:934--992.

\bibitem{DaiZhong2010} M. Dai and Y. Zhong (2010). \newblock Penalty Methods for Continuous-Time Portfolio Selection with Proportional Transaction Costs. \newblock \emph{Journal of Computational Finance}, 13(3):1--31.

\bibitem{atar2004scheduling} R. Atar, A. Mandelbaum and M. Reiman (2004). \newblock Scheduling a multi-class queue with many exponential servers: Asymptotic optimality in heavy-traffic. 
\newblock \emph{Annals of Applied Probability}, 14(3):1084--1134.

\bibitem{kim2018dynamic} J. Kim, R. Ramandeep and A. Ward (2018). \newblock Dynamic scheduling in a many-server, multiclass system: The role of customer impatience in large systems.  
\newblock \emph{Manufacturing $\&$ Service Operations Management}, 20(2):285--301.

\bibitem{williams2016stochastic} R. Williams (2016). \newblock Stochastic processing networks. \newblock \emph{Annual Review of Statistics and Its Application}, 3(1):323--345.

\bibitem{thomas2013numerical} J.~W. Thomas (2013). \newblock \emph{Numerical partial differential equations: finite difference methods}. \newblock Springer. 

\bibitem{axelsson2001finite} O. Axelsson and V.~A. Barker (2001). \newblock \emph{Finite element solution of boundary value problems: theory and computation}. \newblock SIAM.

\bibitem{ata2025analysis} B. Ata and Y. Zhou (2025). \newblock Analysis and improvement of eviction enforcement. \newblock \emph{arXiv}. 

\bibitem{harrison2004dynamic} J.~M. Harrison and A. Zeevi (2004). \newblock Dynamic scheduling of a multiclass queue in the Halfin-Whitt heavy traffic regime. 
\newblock \emph{Operations Research}, 52(2):243--257.

\bibitem{budhiraja2011ergodic} A. Budhiraja, A.~P. Ghosh and C. Lee (2011). \newblock Ergodic rate control problem for single class queueing networks. 
\newblock \emph{SIAM journal on control and optimization}, 49(4):1570--1606.

\bibitem{borkar2004ergodic} V. Borkar and A. Budhiraja (2004). \newblock Ergodic control for constrained diffusions: Characterization using HJB equations. 
\newblock \emph{SIAM journal on control and optimization}, 43(4):1467--1492.

\bibitem{williams1994regularity} S.~A. Williams, P.~L. Chow and J.~L. Menaldi (1994). \newblock Regularity of the Free Boundary in Singular Stochastic Control.
\newblock \emph{Journal of Differential Equations}, 111(1):175--201.

\bibitem{menaldi1989optimal} J.~L. Menaldi and M.~I. Taksar (1989). \newblock Optimal correction problem of a multidimensional stochastic system. 
\newblock \emph{Automatica}, 25(2):223--232.

\bibitem{ata2024singular} B. Ata, J.~M. Harrison and N. Si (2024). \newblock Singular control of (reflected) Brownian motion: A computational method suitable for queueing applications. 
\newblock \emph{Queueing Systems}, 108(3):215--251.

\bibitem{atar2006singular} R. Atar and A. Budhiraja (2006). \newblock Singular control with state constraints on unbounded domain. 
\newblock \emph{The Annals of Probability}, 34(5):1864--1909.

\bibitem{al2022extensions} A. Al-Aradi, A. Correia, D. de Freitas Naiff, G. Jardim, and Y. Saporito (2022). \newblock Extensions of the deep Galerkin method. \newblock \emph{Applied Mathematics and Computation}, 430:127287.

\bibitem{beck2020deep} C. Beck, S. Becker, P. Cheridito, A. Jentzen, and A. Neufeld (2020). \newblock Deep learning based numerical approximation algorithms for stochastic partial differential equations and high-dimensional nonlinear filtering problems. \newblock arXiv:2012.01194.

\bibitem{beck2021deep} C. Beck, S. Becker, P. Cheridito, A. Jentzen, and A. Neufeld (2021). \newblock Deep splitting method for parabolic PDEs. \newblock \emph{SIAM Journal on Scientific Computing}, 43(5):A3135--A3154.

\bibitem{beck2019machine} C. Beck, W. E, and A. Jentzen (2019). \newblock Machine learning approximation algorithms for high-dimensional fully nonlinear partial differential equations and second-order backward stochastic differential equations. \newblock \emph{Journal of Nonlinear Science}, 29:1563--1619.

\bibitem{berner2020numerically} J. Berner, M. Dablander, and P. Grohs (2020). \newblock Numerically solving parametric families of high-dimensional Kolmogorov partial differential equations via deep learning. \newblock \emph{Advances in Neural Information Processing Systems}, 33:16615--16627.

\bibitem{castro2022deep} J. Castro (2022). \newblock Deep learning schemes for parabolic nonlocal integro-differential equations. \newblock \emph{Partial Differential Equations and Applications}, 3(6):77.

\bibitem{cioica2022deep} P. A. Cioica-Licht, M. Hutzenthaler, and P. T. Werner (2022). \newblock Deep neural networks overcome the curse of dimensionality in the numerical approximation of semilinear partial differential equations. \newblock arXiv:2205.14398.

\bibitem{ew2017deep} W. E, J. Han, and A. Jentzen (2017). \newblock Deep learning-based numerical methods for high-dimensional parabolic partial differential equations and backward stochastic differential equations. \newblock \emph{Communications in Mathematics and Statistics}, 5(4):349--380.

\bibitem{ew2018deep} W. E and B. Yu (2018). \newblock The deep Ritz method: A deep learning-based numerical algorithm for solving variational problems. \newblock \emph{Communications in Mathematics and Statistics}, 6(1):1--12.

\bibitem{weinan2021algorithms} W. E, J. Han, and A. Jentzen (2021). \newblock Algorithms for solving high dimensional PDEs: from nonlinear Monte Carlo to machine learning. \newblock \emph{Nonlinearity}, 35(1):278.

\bibitem{frey2022convergence} R. Frey and V. Köck (2022). \newblock Convergence Analysis of the Deep Splitting Scheme: the Case of Partial Integro-Differential Equations and the associated FBSDEs with Jumps. \newblock arXiv:2206.01597.

\bibitem{frey2022deep} R. Frey and V. Köck (2022). \newblock Deep neural network algorithms for parabolic PIDEs and applications in insurance mathematics. \newblock In \emph{Methods and Applications in Fluorescence}, 272--277. Springer.

\bibitem{GPW2022} M. Germain, H. Pham, and X. Warin (2022). \newblock Approximation error analysis of some deep backward schemes for nonlinear PDEs. \newblock \emph{SIAM Journal on Scientific Computing}, 44(1):A28--A56.

\bibitem{gnoatto2022deep} A. Gnoatto, M. Patacca, and A. Picarelli (2022). \newblock A deep solver for BSDEs with jumps. \newblock arXiv:2211.04349.

\bibitem{gonon2023random} L. Gonon (2023). \newblock Random feature neural networks learn Black-Scholes type PDEs without curse of dimensionality. \newblock \emph{Journal of Machine Learning Research}, 24(189):1--51.

\bibitem{gonon2021deep} L. Gonon and C. Schwab (2021). \newblock Deep ReLU network expression rates for option prices in high-dimensional, exponential Lévy models. \newblock \emph{Finance and Stochastics}, 25(4):615--657.

\bibitem{gonon2023deep} L. Gonon and C. Schwab (2023). \newblock Deep ReLU neural networks overcome the curse of dimensionality for partial integrodifferential equations. \newblock \emph{Analysis and Applications}, 21(1):1--47.

\bibitem{grohs2023proof} P. Grohs, F. Hornung, A. Jentzen, and P. von Wurstemberger (2023). \newblock A proof that artificial neural networks overcome the curse of dimensionality in the numerical approximation of Black--Scholes partial differential equations. \newblock \emph{Memoirs of the American Mathematical Society}, 284:1--106.

\bibitem{han2018solving} J. Han, A. Jentzen, and W. E (2018). \newblock Solving high-dimensional partial differential equations using deep learning. \newblock \emph{Proceedings of the National Academy of Sciences}, 115(34):8505--8510.

\bibitem{han2020convergence} J. Han and J. Long (2020). \newblock Convergence of the deep BSDE method for coupled FBSDEs. \newblock \emph{Probability, Uncertainty and Quantitative Risk}, 5:1--33.

\bibitem{han2019solving} J. Han, L. Zhang, and W. E (2019). \newblock Solving many-electron Schrödinger equation using deep neural networks. \newblock \emph{Journal of Computational Physics}, 399:108929.

\bibitem{hure2020deep} C. Huré, H. Pham, and X. Warin (2020). \newblock Deep backward schemes for high-dimensional nonlinear PDEs. \newblock \emph{Mathematics of Computation}, 89(324):1547--1579.

\bibitem{hutzenthaler2020proof} M. Hutzenthaler, A. Jentzen, T. Kruse, and T. A. Nguyen (2020). \newblock A proof that rectified deep neural networks overcome the curse of dimensionality in the numerical approximation of semilinear heat equations. \newblock \emph{SN Partial Differential Equations and Applications}, 1:1--34.

\bibitem{ito2021neural} K. Ito, C. Reisinger, and Y. Zhang (2021). \newblock A neural network-based policy iteration algorithm with global $H^2$-superlinear convergence for stochastic games on domains. \newblock \emph{Foundations of Computational Mathematics}, 21(2):331--374.

\bibitem{jacquier2023deep} A. Jacquier and M. Oumgari (2023). \newblock Deep curve-dependent PDEs for affine rough volatility. \newblock \emph{SIAM Journal on Financial Mathematics}, 14(2):353--382.

\bibitem{jacquier2023random} A. Jacquier and Z. Zuric (2023). \newblock Random neural networks for rough volatility. \newblock arXiv:2305.01035.

\bibitem{jentzen2018proof} A. Jentzen, D. Salimova, and T. Welti (2021). \newblock A proof that deep artificial neural networks overcome the curse of dimensionality in the numerical approximation of Kolmogorov partial differential equations with constant diffusion and nonlinear drift coefficients. \newblock \emph{Communications in Mathematical Sciences}, 19(5):1167--1205. doi:10.4310/CMS.2021.v19.n5.a1.

\bibitem{lu2021deepxde} L. Lu, X. Meng, Z. Mao, and G. E. Karniadakis (2021). \newblock DeepXDE: A deep learning library for solving differential equations. \newblock \emph{SIAM Review}, 63(1):208--228.

\bibitem{nguwi2022deep} J. Y. Nguwi, G. Penent, and N. Privault (2022). \newblock A deep branching solver for fully nonlinear partial differential equations. \newblock arXiv:2203.03234.

\bibitem{nguwi2022numerical} J. Y. Nguwi, G. Penent, and N. Privault (2022). \newblock Numerical solution of the incompressible Navier--Stokes equation by a deep branching algorithm. \newblock arXiv:2212.13010.

\bibitem{nguwi2023deep} J. Y. Nguwi and N. Privault (2023). \newblock A deep learning approach to the probabilistic numerical solution of path-dependent partial differential equations. \newblock \emph{Partial Differential Equations and Applications}, 4(4):37.

\bibitem{raissi2019physics} M. Raissi, P. Perdikaris, and G. E. Karniadakis (2019). \newblock Physics-informed neural networks: A deep learning framework for solving forward and inverse problems involving nonlinear partial differential equations. \newblock \emph{Journal of Computational Physics}, 378:686--707.

\bibitem{reisinger2020rectified} C. Reisinger and Y. Zhang (2020). \newblock Rectified deep neural networks overcome the curse of dimensionality for nonsmooth value functions in zero-sum games of nonlinear stiff systems. \newblock \emph{Analysis and Applications}, 18(06):951--999.

\bibitem{sirignano2018dgm} J. Sirignano and K. Spiliopoulos (2018). \newblock DGM: a deep learning algorithm for solving partial differential equations. \newblock \emph{Journal of Computational Physics}, 375:1339--1364.

\bibitem{zhang2020learning} D. Zhang, L. Guo, and G. E. Karniadakis (2020). \newblock Learning in modal space: Solving time-dependent stochastic PDEs using physics-informed neural networks. \newblock \emph{SIAM Journal on Scientific Computing}, 42(2):A639--A665.

\bibitem{beck2020overcoming} C. Beck, F. Hornung, M. Hutzenthaler, A. Jentzen, and T. Kruse (2020). \newblock Overcoming the curse of dimensionality in the numerical approximation of Allen--Cahn partial differential equations via truncated full-history recursive multilevel Picard approximations. \newblock \emph{Journal of Numerical Mathematics}, 28(4):197--222.

\bibitem{beck2020overcomingElliptic} C. Beck, L. Gonon, and A. Jentzen (2020). \newblock Overcoming the curse of dimensionality in the numerical approximation of high-dimensional semilinear elliptic partial differential equations. \newblock arXiv:2003.00596.

\bibitem{becker2020numerical} S. Becker, R. Braunwarth, M. Hutzenthaler, A. Jentzen, and P. von Wurstemberger (2020). \newblock Numerical simulations for full history recursive multilevel Picard approximations for systems of high-dimensional partial differential equations. \newblock \emph{Communications in Computational Physics}, 28(5):2109--2138. doi:10.4208/cicp.OA-2020-0130.

\bibitem{hutzenthaler2019multilevel} W. E, M. Hutzenthaler, A. Jentzen, and T. Kruse (2019). \newblock On multilevel Picard numerical approximations for high-dimensional nonlinear parabolic partial differential equations and high-dimensional nonlinear backward stochastic differential equations. \newblock \emph{Journal of Scientific Computing}, 79(3):1534--1571.

\bibitem{hutzenthaler2021multilevel} M. Hutzenthaler, A. Jentzen, and T. Kruse (2021). \newblock Multilevel Picard iterations for solving smooth semilinear parabolic heat equations. \newblock \emph{Partial Differential Equations and Applications}, 2(6):1--31.

\bibitem{giles2019generalised} M. B. Giles, A. Jentzen, and T. Welti (2019). \newblock Generalised multilevel Picard approximations. \newblock arXiv:1911.03188.

\bibitem{HJK2022} M. Hutzenthaler, A. Jentzen, and T. Kruse (2022). \newblock Overcoming the curse of dimensionality in the numerical approximation of parabolic partial differential equations with gradient-dependent nonlinearities. \newblock \emph{Foundations of Computational Mathematics}, 22(4):905--966.

\bibitem{HJKNW2020} M. Hutzenthaler, A. Jentzen, T. Kruse, T. A. Nguyen, and P. von Wurstemberger (2020). \newblock Overcoming the curse of dimensionality in the numerical approximation of semilinear parabolic partial differential equations. \newblock \emph{Proceedings of the Royal Society A}, 476(2244):20190630.

\bibitem{HJKN2020} M. Hutzenthaler, A. Jentzen, T. Kruse, and T. A. Nguyen (2020). \newblock Multilevel Picard approximations for high-dimensional semilinear second-order PDEs with Lipschitz nonlinearities. \newblock arXiv:2009.02484.

\bibitem{hutzenthaler2020overcoming} M. Hutzenthaler, A. Jentzen, and P. von Wurstemberger (2020). \newblock Overcoming the curse of dimensionality in the approximative pricing of financial derivatives with default risks. \newblock \emph{Electronic Journal of Probability}, 25(101):1--73.

\bibitem{HK2020} M. Hutzenthaler and T. Kruse (2020). \newblock Multilevel Picard approximations of high-dimensional semilinear parabolic differential equations with gradient-dependent nonlinearities. \newblock \emph{SIAM Journal on Numerical Analysis}, 58(2):929--961.

\bibitem{hutzenthaler2022multilevel} M. Hutzenthaler, T. Kruse, and T. A. Nguyen (2022). \newblock Multilevel Picard approximations for McKean--Vlasov stochastic differential equations. \newblock \emph{Journal of Mathematical Analysis and Applications}, 507(1):125761.

\bibitem{hutzenthaler2022multilevel1} M. Hutzenthaler and T. A. Nguyen (2022). \newblock Multilevel Picard approximations of high-dimensional semilinear partial differential equations with locally monotone coefficient functions. \newblock \emph{Applied Numerical Mathematics}, 181:151--175.

\bibitem{NW2022} A. Neufeld and S. Wu (2025). \newblock Multilevel Picard approximation algorithm for semilinear partial integro-differential equations and its complexity analysis. \newblock \emph{Stochastics and Partial Differential Equations: Analysis and Computations}, 1--59.

\bibitem{NNW2023} A. Neufeld, T. A. Nguyen, and S. Wu (2023). \newblock Deep ReLU neural networks overcome the curse of dimensionality in the numerical approximation of semilinear partial integro-differential equations. \newblock \emph{Analysis and Applications}, accepted; arXiv:2310.15581.

\bibitem{PardouxPeng1990} E.~Pardoux and S.~Peng (1990). \newblock Adapted solution of a backward stochastic differential equation. \newblock \emph{Systems \& Control Letters}, 14(1):55--61.

\bibitem{Peng1991} S.~Peng (1991). \newblock Probabilistic interpretation for systems of quasilinear parabolic partial differential equations. \newblock \emph{Stochastics and Stochastics Reports}, 37(1--2):61--74.

\bibitem{PardouxPeng1992} E.~Pardoux and S.~Peng (1992). \newblock Backward stochastic differential equations and quasilinear parabolic partial differential equations. \newblock In: \emph{Stochastic Partial Differential Equations and Their Applications} (Charlotte, NC, 1991), Lecture Notes in Control and Information Sciences, Vol.~176, Springer, Berlin, pp.~200--217.

\bibitem{DaPratoZabczyk1997} G.~Da Prato and J.~Zabczyk (1997). \newblock Differentiability of the Feynman--Kac semigroup and a control application. \newblock \emph{Rendiconti Lincei -- Matematica e Applicazioni}, 8(3):183--188.

\bibitem{DaPrato2014} G.~Da Prato (2014). \newblock \emph{Introduction to Stochastic Analysis and Malliavin Calculus}. \newblock Vol.~13. Springer.

\bibitem{Giles2015} M.~B. Giles (2015). \newblock Multilevel Monte Carlo methods. \newblock \emph{Acta Numerica}, 24:259--328.

\bibitem{YongZhou1999} J.~Yong and X.~Y. Zhou (1999). \newblock \emph{Stochastic Controls: Hamiltonian Systems and HJB Equations}. \newblock Applications of Mathematics, Vol.~43. Springer Science \& Business Media.

\end{thebibliography}
\begingroup\small
\begin{singlespace}

\end{singlespace}
\endgroup

\newpage
\begin{appendices}
\section*{Appendix}
%

\section{Proofs for Section \ref{sec:tech}}\label{app:tech-proofs}

\subsection{Proof of Proposition \ref{prop:moment-Z}}
The following lemma is used in the proof of Proposition \ref{prop:moment-Z}.
\begin{lemma}\label{lem:sigmaB}
Let $B(\cdot)$ be a standard $d$-dimensional Brownian motion, and let $\sigma$ be a diffusion coefficient matrix 
that satisfies Assumption \ref{as:sigma}. Then, for any $c > 0$, 
\begin{equation}\label{eq:sigmaB-reflection}
\pr\left(\sup_{s\in [0,t]} \|\sigma B(s)\|_\infty \geq c \right) \leq 2\cdot \pr\left(\|\sigma B(t)\|_\infty \geq c\right).
\end{equation}
\end{lemma}
\begin{proof}
The proof of the lemma makes crucial use of the reflection principle. We have 
\begin{align}
&~\pr\left(\sup_{s\in [0,t]} \|\sigma B(s)\|_\infty \geq c \right) \nonumber\\
=&~\pr\left(\sup_{s\in [0,t]} \|\sigma B(s)\|_\infty \geq c, \|\sigma B(t)\|_\infty \geq c \right) 
+ \pr\left(\sup_{s\in [0,t]} \|\sigma B(s)\|_\infty \geq c, \|\sigma B(t)\|_\infty < c \right) \nonumber \\
\leq &~\pr\left(\|\sigma B(t)\|_\infty \geq c \right) + \pr\left(\sup_{s\in [0,t]} \|\sigma B(s)\|_\infty \geq c, \|\sigma B(t)\|_\infty < c \right). 
\label{eq:reflection1}
\end{align}
Thus, it suffices to show that 
\[
\pr\left(\sup_{s\in [0,t]} \|\sigma B(s)\|_\infty \geq c, \|\sigma B(t)\|_\infty < c \right) \leq \pr\left(\|\sigma B(t)\|_\infty \geq c \right), 
\]
for which we will use the reflection principle. Define the hitting time $\tau$ by 
\[
\tau := \inf \left\{s \in [0, t] : \sup_{s\in [0,t]} \|\sigma B(s)\|_\infty \geq c\right\}, 
\]
with the convention that if the process never hits $c$ during the time interval $[0, t]$, $\tau = \infty$. 
Define the process $X$ by $X(t) = \sigma B(t)$, $t \geq 0$, and define the process $X^*$ ``reflected'' at time $\tau$ by
\[
X^*(s) = \left\{\begin{array}{ll}
X(s) & \text{ if } s\leq \tau; \\
X(\tau) - (X(s)-X(\tau)) & \text{ if } s > \tau.
\end{array}\right.
\]
Then, $X$ and $X^*$ have the same law, and under the event that $\sup_{s\in [0,t]} \|X(s)\|_\infty \geq c$ and $\|X(t)\|_\infty < c$, 
we have $\tau \leq t$, $\|X(\tau)\|_\infty \geq c$, and 
\[
\|X^*(t)\|_\infty = \|X(\tau) - (X(t)-X(\tau))\|_\infty \geq 2\|X(\tau)\|_\infty - \|X(t)\|_\infty > 2c-c = c.
\]
Therefore, 
\begin{equation}\label{eq:reflection2}
\pr\left(\sup_{s\in [0,t]} \|\sigma B(s)\|_\infty \geq c, \|\sigma B(t)\|_\infty < c \right) 
\leq \pr\left(\|X^*(t)\|_\infty > c \right) = \pr\left(\|X(t)\|_\infty > c \right).
\end{equation}
By Ineqs. \eqref{eq:reflection1} and \eqref{eq:reflection2}, the lemma is established.
\end{proof}

\begin{proof}[Proof of Proposition \ref{prop:moment-Z}] 
Let $t \ge 0$ and $\alpha \ge 1$. For $s \in [0, t]$,
\begin{align*}
\tilde{Z}^x(s) &= x + \int_0^s \tb(\tilde{Z}^x(r))dr + \sigma B(s) + R \tilde{Y}^x(s) \\
&= \tilde{X}^x(s) + R \tilde{Y}^x(s),
\end{align*}
where we recall the definition of $\tilde X^x$ in Eq. \eqref{eq:tx}.  
By Lipschitz continuity of the Skorokhod map $\Gamma$, 
and the fact that the zero process solves the Skorokhod problem for the zero process itself, 
\begin{align}
\| \tilde{Z}^x \|_t^\alpha &\leq C_{\Gamma}^\alpha \| \bar{X}^x \|_t^\alpha \\
&=  C_{\Gamma}^\alpha \sup_{s \in [0, t]} \left\| x + \int_0^s \tb(\tilde{Z}^x(r))dr + \sigma B(s) \right\|^\alpha_\infty \nonumber \\
&\leq (3 C_{\Gamma})^\alpha \left( \|x\|^\alpha_\infty + \sup_{s \in [0, t]} \left\| \int_0^s \tb(\tilde{Z}^x(r))dr \right\|^\alpha_\infty + \sup_{s \in [0, t]} \|\sigma B(s)\|^\alpha_\infty \right). 
\label{eq:moment-Z1-app}
\end{align}
By Assumption \ref{as:tb}, we have that 
\begin{equation}\label{eq:moment-Z-tb-app}
\sup_{s \in [0, t]} \left\| \int_0^s \tb(\tilde{Z}^x(r))dr \right\|^\alpha_\infty 
\le \left(C_{\tb} t\right)^\alpha.
\end{equation}
Next, consider $\E\left[\sup_{s \in [0, t]} \|\sigma B(s)\|^\alpha_\infty\right]$. 
By Ineq. \eqref{eq:sigmaB-reflection} in Lemma \ref{lem:sigmaB}, we have
\begin{eqnarray}
\E\left[\sup_{s \in [0, t]} \|\sigma B(s)\|^\alpha_\infty\right] &=& \int_0^\infty \alpha c^{\alpha-1} \pr\left(\sup_{s \in [0, t]} \|\sigma B(s)\|_\infty \geq c \right) dc \nonumber \\
&\leq & 2 \int_0^\infty \alpha c^{\alpha-1} \pr\left(\|\sigma B(t)\|_\infty \geq c \right) dc \nonumber \\
&=& 2\E\left[\|\sigma B(t)\|^\alpha_\infty\right].
\label{eq:sigmaB-moment-app}
\end{eqnarray}
For each $i$, $\sum_{j=1}^d \sigma_{ij} B_j(t)$ is a zero-mean normal random variable 
with variance $t \cdot \sum_{j=1}^d \sigma_{ij}^2 \leq C_\sigma^2 \cdot t$, by Assumption \ref{as:sigma}.
Thus, 
\[
\E\left[\|\sigma B(t)\|^\alpha_\infty\right] \leq \left(C_\sigma^2 t\right)^{\alpha/2} \cdot \E\left[\max_{i=1, 2, \cdots, d} |N_i|^{\alpha}\right] 
:= C_\sigma^\alpha t^{\alpha/2} \cdot \E\left[\max_{i=1, 2, \cdots, d} |N_i|^{\alpha}\right], 
\]
for some standard normal random variables $N_1, \cdots, N_J$ that are possibly correlated. 
Standard concentration inequalities imply that there exists some positive constant $C'_\alpha$, which depends on $\alpha$ but not on $d$, 
such that 
\[
\E\left[\max_{i=1, 2, \cdots, d} |N_i|^{\alpha}\right] \leq C'_\alpha \left(\log d\right)^{\alpha/2}.
\]
Therefore, 
\begin{equation}\label{eq:sigmaB2-app}
\E\left[\|\sigma B(t)\|^\alpha_\infty\right] \leq C_\sigma^\alpha t^{\alpha/2} \cdot \E\left[\max_{i=1, 2, \cdots, d} |N_i|^{\alpha}\right] 
\leq C_\sigma^\alpha t^{\alpha/2} C'_\alpha \left(\log d\right)^{\alpha/2}.
\end{equation}
Taking expectation on both sides of Ineq. \eqref{eq:moment-Z1-app}, and using Ineqs. \eqref{eq:moment-Z-tb-app}, \eqref{eq:sigmaB-moment-app} and \eqref{eq:sigmaB2-app}, 
we have that 
\begin{equation}
\E\left[\| \tilde{Z}^x \|_t^\alpha\right] \leq (3 C_{\Gamma})^\alpha \left( \|x\|^\alpha_\infty + \left(C_{\tb} t\right)^\alpha + 2 C_\sigma^\alpha t^{\alpha/2} C'_\alpha \left(\log d\right)^{\alpha/2}\right).
\end{equation}
By taking 
\begin{align}
C(t, \alpha, \sigma, C_{\tb}, C_{\Gamma}) 
:= (3 C_{\Gamma})^\alpha \cdot \max \left\{ 1, (C_{\tb} t)^\alpha, 2 C_\sigma^\alpha t^{\alpha/2} C'_\alpha \right\},
\end{align}
we have established Ineq. \eqref{eq:moment1}. 

To establish Ineq. \eqref{eq:moment2}, let $x, y \in \R_+^d$. We have
\begin{align*}
\|\tilde{X}^y(t) - \tilde{X}^x(t)\|_\infty &= \left\| (y - x) + \int_0^t \left[ \tb(\tilde{Z}^y(s)) - \tb(\tilde{Z}^x(s)) \right] ds\right\|_\infty \\
&\leq \|y - x\|_\infty + C_{\tb} \int_0^t \| \tilde{Z}^y - \tilde{Z}^x \|_s ds, 
\end{align*}
so by Lipschitz property of the map $\Gamma$, 
\begin{align*}
\| \tilde{Z}^y - \tilde{Z}^x \|_t &\leq C_{\Gamma} \| \tilde{X}^y - \tilde{X}^x \|_t \leq C_\Gamma \|y - x\|_\infty + C_\Gamma C_{\tb} \int_0^t \| \tilde{Z}^y - \tilde{Z}^x \|_s ds
\end{align*}
By Grönwall's inequality, 
\begin{equation*}
\| \tilde{Z}^y - \tilde{Z}^x \|_t \leq C_\Gamma \|y - x\|_\infty \exp(C_\Gamma C_{\tb} t). 
\end{equation*}
By re-defining
\begin{align}
C(t, \alpha, \sigma, C_{\tb}, C_{\Gamma}) 
:= \max\left\{(3 C_{\Gamma})^\alpha \cdot \max \left\{ 1, (C_{\tb} t)^\alpha, 2 C_\sigma^\alpha t^{\alpha/2} C'_\alpha \right\}, 
C_\Gamma  \exp(C_\Gamma K_{\tb} t)\right\},
\end{align}
we have established both Ineqs. \eqref{eq:moment1} and \eqref{eq:moment2}. 
This concludes the proof of the proposition. 
\end{proof}

\subsection{Proof of Corollary \ref{cor:moment-w}}
\begin{proof}
By Proposition \ref{prop:moment-Z}, 
\begin{eqnarray*}
\E\left[w^\alpha \left(\tilde Z^x(t)\right)\right] &=& \E\left[\left(1+\left(\log d\right)^{\alpha_0/2} + \left\|\tilde Z^x(t)\right\|_\infty^{\alpha_0}\right)^\alpha \right] \\
&\leq & 3^{\alpha} \left(1 + \left(\log d\right)^{\alpha \alpha_0/2} + \E\left[\left\|\tilde Z^x\right\|_T^{\alpha \alpha_0}\right]\right) \\
&\leq & 3^{\alpha} \left(1 + \left(\log d\right)^{\alpha \alpha_0/2} + C \left(1 + (\log d)^{\alpha\alpha_0/2} + \|x\|^{\alpha\alpha_0}_\infty\right)\right) \\
&\leq & C_{\alpha, T} \left(1 + \left(\log d\right)^{\alpha \alpha_0/2} + \|x\|^{\alpha\alpha_0}_\infty\right) \\
&\leq & C_{\alpha, T} w^{\alpha}(x), 
\end{eqnarray*}
for some positive constant $C_{\alpha, T}$ that only depends on $\alpha$, $T$, and $C$, where $C$ is the positive constant from Proposition \ref{prop:moment-Z}. 
This completes the proof of the corollary.
\end{proof}

\subsection{Proof of Lemma \ref{lem:bound-DZ}}
\begin{proof}
Let 
\[
{\Psi}(t) = e_j + \int_0^t D \tb(\tilde{Z}^x(s)) D_j \tilde{Z}^x(s) ds.
\]
Then a.s.\ $D_j \tilde{Z}^x = \Lambda_{\tilde Z^x}[{\Psi}]$.
By the facts that $0 = \Lambda_{\tilde Z^x}[0]$ and the Lipschitz property of the map $\Lambda$, we have that a.s.
\begin{eqnarray*}
\| D_j \tilde{Z}^x \|_t 
&\leq & C_{\Lambda} \| {\Psi} \|_t \leq C_{\Lambda} \sup_{s\in [0,t]} \| {\Psi}(s) \|_\infty \\
&\leq & C_{\Lambda} \left( 1 + \int_0^t \| D\tb(\tilde Z^x(s)) D_j \tilde{Z}^x(s) \|_\infty ds \right)\\
&\leq & C_{\Lambda} \left( 1 + \int_0^t \| D\tb(\tilde Z^x(s)) \|_\infty \|D_j \tilde{Z}^x(s) \|_\infty ds \right)\\
&\leq & C_{\Lambda} \left( 1 + C_{\tb} \int_0^t \| D_j \tilde{Z}^x \|_s ds \right).
\end{eqnarray*}
By Grönwall’s inequality, we have that a.s.
\begin{align}
\| D_j \tilde{Z}^x \|_t \leq C_{\Lambda} \exp(C_{\Lambda} C_{\tb} t).
\end{align}
This concludes the proof of the lemma. 
\end{proof}

\subsection{Proof of Proposition \ref{thm:continuity-DZ}}

\begin{lemma}\label{lem:Z-conv}
  a.s.\ \( \tilde{Z}^{y^{(n)}} \to \tilde{Z}^x \) in \( \C(\mathbb{R}_+^d) \) as \( n \to \infty \).
  \end{lemma}
  
  \begin{proof}
  Let
  \[
  \tilde{X}^{y^{(n)}}(t) = y^{(n)} + \int_0^t \tb(\tilde{Z}^{y^{(n)}}(s))\,ds + \sigma B(t), \qquad n \in \mathbb{N},
  \]
  and let
  \[
  \tilde{X}^x(t) = x + \int_0^t \tb(\tilde{Z}^x(s))\,ds + \sigma B(t).
  \]
  Then,
  \[
  \tilde{Z}^{y^{(n)}} = \Gamma(\tilde{X}^{y^{(n)}}), \quad \text{ and } \quad \tilde{Z}^x = \Gamma(\tilde{X}^x).
  \]
  By Lipschitz continuity of the Skorokhod map \( \Gamma \) and Assumption \ref{as:tb}, for any \( t \ge 0 \),
  \begin{align*}
  \| \tilde{Z}^{y^{(n)}} - \tilde{Z}^x \|_t
  &\le C_\Gamma \| \tilde{X}^{y^{(n)}} - \tilde{X}^x \|_t \\
  &\le C_\Gamma \left( \| y^{(n)} - x \|_\infty + C_{\tb} \int_0^t \| \tilde{Z}^{y^{(n)}} - \tilde{Z}^x \|_s\,ds \right).
  \end{align*}
  By Grönwall’s inequality, \[
  \| \tilde{Z}^{y^{(n)}} - \tilde{Z}^x \|_t \le C_\Gamma \| y^{(n)} - x \|_\infty \cdot \exp(C_\Gamma C_{\tb} t).
  \]
  Thus, as \( n \to \infty \), \( \| y^{(n)} - x \|_\infty \to 0 \), and
  \[
  \| \tilde{Z}^{y^{(n)}} - \tilde{Z}^x \|_t \to 0 \quad \text{for every } t \in \mathbb{R}_+.
  \]
  This completes the proof of the lemma. \end{proof}
  
  \begin{proof}[Proof of Proposition \ref{thm:continuity-DZ}] 
  Let $j \in \{1, 2, \cdots, d\}$. 
  For each \( n \in \mathbb{N} \), define
  \begin{equation}\label{eq:Psi-yn}
  \Psi^{y^{(n)}}(t) := e_j + \int_0^t D\tb(\tilde{Z}^{y^{(n)}}(s)) D_j\tilde{Z}^{y^{(n)}}(s)\,ds,
  \end{equation}
  and define
  \begin{equation}\label{eq:Psi-x}
  \Psi^x(t) := e_j + \int_0^t D\tb(\tilde{Z}^x(s)) D_j\tilde{Z}^x(s)\,ds, 
  \end{equation}
  where $e_j$ is the $j$th standard unit vector in $\R^d$.
  Then a.s. \( D_j\tilde{Z}^x = \Lambda_{\tilde{Z}^x}[\Psi^x] \), and for \( n \in \mathbb{N} \),
  $D_j\tilde{Z}^{y^{(n)}} = \Lambda_{\tilde{Z}^{y^{(n)}}}[\Psi^{y^{(n)}}]$.
  
  Let \( t \in \mathbb{R}_+ \). Then
  \begin{align*}
  \| \Psi^{y^{(n)}}(t) - \Psi^x(t) \|_\infty
  &= \left\| \int_0^t D\tb(\tilde{Z}^{y^{(n)}}(s)) D_j\tilde{Z}^{y^{(n)}}(s)\,ds - \int_0^t D\tb(\tilde{Z}^x(s)) D_j\tilde{Z}^x(s)\,ds \right\|_\infty \\
  &\le \int_0^t \left\| D\tb(\tilde{Z}^{y^{(n)}}(s)) \left(D_j\tilde{Z}^{y^{(n)}}(s) - D_j\tilde{Z}^x(s)\right) \right\|_\infty \,ds \\
  &\quad + \int_0^t \left\| \left(D\tb(\tilde{Z}^{y^{(n)}}(s)) - D\tb(\tilde{Z}^x(s))\right)D_j\tilde{Z}^x(s) \right\|_\infty \,ds \\
  &\le C_{\tb} \int_0^t \| D_j\tilde{Z}^{y^{(n)}}(s) - D_j\tilde{Z}^x(s) \|_\infty\,ds \\
  &\quad + C_{\tb} \int_0^t \|D_j \tilde{Z}^x(s)\|_\infty \, \|\tilde{Z}^{y^{(n)}}(s) - \tilde{Z}^x(s)\|_\infty^{\alpha_{\tb}} ds.
  \end{align*}
  Here, the preceding inequality follows from Assumption \ref{as:tb}.
  
  We now examine the terms
  \[
  C_{\tb} \int_0^t \| D_j\tilde{Z}^{y^{(n)}}(s) - D_j\tilde{Z}^x(s) \|_\infty \,ds \quad \text{and} \quad
  C_{\tb} \int_0^t \|D_j \tilde{Z}^x(s)\|_\infty \, \|\tilde{Z}^{y^{(n)}}(s) - \tilde{Z}^x(s)\|_\infty^{\alpha_{\tb}} ds, 
  \]
  respectively. First consider the latter term. Because by Lemma \ref{lem:bound-DZ}, \( \|D_j \tilde{Z}^x\|_t < \infty \), 
  and by Proposition \ref{prop:moment-Z}, 
  \( \|\tilde{Z}^{y^{(n)}} - \tilde{Z}^x\|_t < C \|y^{(n)}-x\|_\infty \) for some $C$ that does not depend on $n$ or $x$, 
  \[
  C_{\tb} \int_0^t \|D_j \tilde{Z}^x(s)\|_\infty \, \|\tilde{Z}^{y^{(n)}}(s) - \tilde{Z}^x(s)\|_\infty^{\alpha_{\tb}} ds \to 0 
  \quad \text{as } n \to \infty.
  \]
  
  Now consider
  \[
  C_{\tb}\int_0^t \| D_j\tilde{Z}^{y^{(n)}}(s) - D_j\tilde{Z}^x(s) \|_\infty\,ds.
  \]
  
  Recall that a.s. \( D_j \tilde{Z}^{y^{(n)}} = \Lambda_{\tilde{Z}^{y^{(n)}}}(\Psi^{y^{(n)}}) \), \( n \in \mathbb{N} \)  
  and \( D_j \tilde{Z}^x = \Lambda_{\tilde{Z}^x}(\Psi^x) \). Thus, a.s.,
  \begin{align*}
  \|D_j \tilde{Z}^{y^{(n)}}(s) - D_j \tilde{Z}^x(s)\|_\infty &= \|\Lambda_{\tilde{Z}^{y^{(n)}}}(\Psi^{y^{(n)}})(s) - \Lambda_{\tilde{Z}^x}(\Psi^x)(s)\|_\infty \\
  & \leq \|\Lambda_{\tilde{Z}^{y^{(n)}}}(\Psi^{y^{(n)}})(s) - \Lambda_{\tilde{Z}^{y^{(n)}}}(\Psi^{x})(s)\|_\infty + \|\Lambda_{\tilde{Z}^{y^{(n)}}}(\Psi^{x})(s) - \Lambda_{\tilde{Z}^{x}}(\Psi^{x})(s)\|_\infty \\
  & \leq \|\Lambda_{\tilde{Z}^{y^{(n)}}}(\Psi^{y^{(n)}}) - \Lambda_{\tilde{Z}^{y^{(n)}}}(\Psi^{x})\|_s + \|\Lambda_{\tilde{Z}^{y^{(n)}}}(\Psi^{x})(s) - \Lambda_{\tilde{Z}^{x}}(\Psi^{x})(s)\|_\infty \\
  & \leq C_\Lambda \|\Psi^{y^{(n)}} - \Psi^{x}\|_s + \|\Lambda_{\tilde{Z}^{y^{(n)}}}(\Psi^{x})(s) - \Lambda_{\tilde{Z}^{x}}(\Psi^{x})(s)\|_\infty.
  \end{align*}
  Thus,
  \[
  \int_0^t\|D_j \tilde{Z}^{y^{(n)}}(s) - D_j \tilde{Z}^x(s)\|_\infty ds 
  \leq C_\Lambda \int_0^t  \|\Psi^{y^{(n)}} - \Psi^{x}\|_s ds 
  + \int_0^t \|\Lambda_{\tilde{Z}^{y^{(n)}}}(\Psi^{x})(s) - \Lambda_{\tilde{Z}^{x}}(\Psi^{x})(s)\|_\infty ds.
  \]
  By Proposition \ref{prop:Lambda-Dconv}, \( \Lambda_{\tilde{Z}^{y^{(n)}}}(\Psi^{x}) \to \Lambda_{\tilde{Z}^{x}}(\Psi^{x}) \) in \( \mathbb{D}(\mathbb{R}^d) \). 
  Furthermore, both \( \|\Lambda_{\tilde{Z}^{y^{(n)}}}(\Psi^{x})\|_t \) and \( \|\Lambda_{\tilde{Z}^{x}}(\Psi^{x})\|_t \) are uniformly bounded, so by the fact that 
  \[
  \Lambda_{\tilde{Z}^{y^{(n)}}}(\Psi^{x})(s) \to \Lambda_{\tilde{Z}^{x}}(\Psi^{x})(s) \text{ for } s \in [0,t] \text{ a.e.,}
  \]
  and dominated convergence theorem,
  \[
  \int_0^t \|\Lambda_{\tilde{Z}^{y^{(n)}}}(\Psi^{x})(s) - \Lambda_{\tilde{Z}^{x}}(\Psi^{x})(s)\|_\infty ds \to 0 \quad \text{as } n \to \infty.
  \]
  To summarize, we have established that
  \[
  \| \Psi^{y^{(n)}}(t) - \Psi^x(t) \|_\infty \leq C_n(t) + C_{\tb} C_\Lambda \int_0^t  \|\Psi^{y^{(n)}} - \Psi^{x}\|_s ds,
  \]
  where
  \[
  C_n(t) = C_{\tb} \int_0^t \|\Lambda_{\tilde{Z}^{y^{(n)}}}(\Psi^{x})(s) - \Lambda_{\tilde{Z}^{x}}(\Psi^{x})(s)\|_\infty ds
  + C_{\tb} \int_0^t \|D_j \tilde{Z}^x(s)\|_\infty \, \|\tilde{Z}^{y^{(n)}}(s) - \tilde{Z}^x(s)\|_\infty^{\alpha_{\tb}} ds.
  \]
  Since \( C_n(\cdot) \) is non-decreasing for each \( n \in \mathbb{N} \), 
  we also have
  \[
  \| \Psi^{y^{(n)}} - \Psi^x \|_t \leq C_n(t) + C_{\tb} C_\Lambda \int_0^t \|\Psi^{y^{(n)}} - \Psi^{x}\|_s ds.
  \]
  By Grönwall's inequality, and the fact that \( C_n(t) \to 0 \) as \( n \to \infty \) for each $t \in \mathbb{R}_+$, we have
  \[
  \| \Psi^{y^{(n)}} - \Psi^x \|_t  \to 0 \quad \text{as } n \to \infty.
  \]
  Because $x \in \R_{++}^d$, $\Psi^x(0) = e_j \in H_x$. Thus, by Proposition \ref{prop:Lambda-Dconv}, a.s.
  \[
  D_j\tilde{Z}^{y^{(n)}} \to D_j\tilde{Z}^x \quad \text{in } \mathbb{D}(\mathbb{R}^d) \text{ as } n \to \infty.
  \]
  This completes the proof of the proposition. 
  \end{proof}

\subsection{Proof of Corollary \ref{cor:conv-bel}} 
  By the Burkholder–Davis–Gundy inequality,
  \[
  \mathbb{E} \left[ \sup_{s\in [0, t]} \|I(y^{(n)})(s) - I(x)(s)\|_2^2 \right] \leq 4 \, \mathbb{E} \left[ \int_0^t 
  \|\sigma^{-1}(D_j\tilde{Z}^{y^{(n)}}(s) - D_j\tilde{Z}^x(s))\|_2^2 ds \right] \to 0.
  \]
  as \( n \to \infty \). The convergence follows from dominated convergence theorem, and the fact that \( D_j\tilde{Z}^{y^{(n)}} \to D_j\tilde{Z}^x \) in \( \mathbb{D}(\mathbb{R}^d) \), 
  and both $\|D_j\tilde{Z}^{y^{(n)}}\|_t$ and $\|D_j\tilde{Z}^{x}\|_t$ are uniformly bounded. 
  This completes the proof of the corollary.

  \section{Proof of Theorem \ref{thm:main1}}\label{sec:proof-main1}
  
  We prove Theorem \ref{thm:main1} in this section. 
  For a random variable $X$, write $\|X\| := (\E[X^2])^{1/2}$. For $i\in\{1,\dots,d\}$ and any $\theta\in\Theta$, denote $v^{\theta,i}_{n,M}(t,x):= [v^\theta_{n,M}(t,x)]_i$. 
  We show $v^0_{n,M}(t,x)$ is close to $\tilde v(t,x)$; the bound for $V^0_{n,M}(t,x)$ will follow similarly.
  By the triangle inequality, for each $i$,
  \begin{equation}
  \left\|v^{0,i}_{n,M}(t,x) - \tilde{v}_i(t,x)\right\| \leq 
  \left\|v^{0,i}_{n,M}(t,x) - \E\left[v^{0,i}_{n,M}(t,x)\right]\right\| + \left|\E\left[v^{0,i}_{n,M}(t,x)\right] - \tilde{v}_i(t,x)\right|.
  \end{equation}
  Because for different $\theta\in\Theta$, $\left(S^\theta, \tilde{Z}^{t,x}_\theta(S^\theta), \tilde{Z}^{t,x}_\theta(T), \mathcal{V}^\theta(S^\theta; t,x), \mathcal{P}^\theta(T; t,x), \mathcal{V}^\theta(T; t,x) \right)$
  are i.i.d. copies of the tuple
  \[  
  \left(S, \tilde{Z}^{t,x}(S), \tilde{Z}^{t,x}(T), \mathcal{P}(T; t,x), \mathcal{V}(S; t,x), \mathcal{V}(T; t,x)\right),
  \] we have
  \begin{align*}
  &~ Var\left(v^{0,i}_{n,M}(t,x)\right) \\
  =&~\frac{1}{M^n} \cdot Var\Bigg(C(T-t) \cdot \underbar{\em c}\left(\tilde{Z}^{t,x}\left(S\right)\right) \mathcal{V}_i\left(S; t,x\right) + \mathcal{P}_i(T; t,x) + \frac{e^{-\beta (T-t)}}{\sqrt{T-t}}\xi\left(\tilde{Z}^{t,x}(T)\right) \mathcal{V}_i\left(T; t,x\right)\Bigg) \\
  &~+ \sum_{\ell=1}^{n-1} \frac{C^2(T-t)}{M^{n-\ell}} \cdot 
  Var \Bigg( \mathcal{V}_i\left(S; t,x\right)\cdot 
  \Bigg[\tilde \cH\left(\tilde{Z}^{t,x}\left(S\right), v^{(0, \ell, 1)}_{\ell, M}\left(S, \tilde{Z}^{t,x}\left(S\right)\right) \right) \\
  & \qquad \qquad \qquad \qquad \qquad \qquad \qquad \qquad 
  - \tilde \cH\left(\tilde{Z}^{t,x}\left(S\right), v^{(0, -\ell, 1)}_{\ell-1, M}\left(S, \tilde{Z}^{t,x}\left(S\right)\right) \right)\Bigg]\Bigg) \\
  \leq &~ \frac{1}{M^n} \cdot \left\|C(T-t) \cdot \underbar{\em c}\left(\tilde{Z}^{t,x}\left(S\right)\right) \mathcal{V}_i\left(S; t,x\right)
  + \mathcal{P}_i(T; t,x) + \frac{e^{-\beta (T-t)}}{\sqrt{T-t}}\xi\left(\tilde{Z}^{t,x}(T)\right) \mathcal{V}_i\left(T; t,x\right)\right\|^2 \\
  &~ + \sum_{\ell=1}^{n-1} \frac{C^2(T-t)}{M^{n-\ell}} \cdot 
  \Bigg\|\mathcal{V}_i\left(S; t,x\right)\cdot 
  \Bigg[\tilde \cH\left(\tilde{Z}^{t,x}\left(S\right), v^{(0, \ell, 1)}_{\ell, M}\left(S, \tilde{Z}^{t,x}\left(S\right)\right) \right) \\
  & \qquad \qquad \qquad \qquad - \tilde \cH\left(\tilde{Z}^{t,x}\left(S\right), v^{(0, -\ell, 1)}_{\ell-1, M}\left(S, \tilde{Z}^{t,x}\left(S\right)\right) \right)\Bigg]\Bigg\|^2.
  \end{align*}
  In the preceding inequality, we used the simple fact that $\mathrm{Var}(X) \le \E[X^2]=\|X\|^2$. Applying the triangle inequality again and noting that
  \begin{equation}
  Var\left(v^{0,i}_{n,M}(t,x)\right) = \left\|v^{0,i}_{n,M}(t,x) - \E\left[v^{0,i}_{n,M}(t,x)\right]\right\|^2, 
  \end{equation}
  we have
  \begin{align}
  & ~ \left\|v^{0,i}_{n,M}(t,x) - \E\left[v^{0,i}_{n,M}(t,x)\right]\right\| \\
  \leq & ~ \frac{1}{M^{n/2}} \cdot\left\|C(T-t) \cdot \underbar{\em c}\left(\tilde{Z}^{t,x}\left(S\right)\right) \mathcal{V}_i\left(S; t,x\right)
  + \mathcal{P}_i(T; t,x) + \frac{e^{-\beta (T-t)}}{\sqrt{T-t}}\xi\left(\tilde{Z}^{t,x}(T)\right) \mathcal{V}_i\left(T; t,x\right)\right\|\\
  &~ + \sum_{\ell=1}^{n-1} \frac{C(T-t)}{M^{(n-\ell)/2}} 
  \Bigg\|\mathcal{V}_i\left(S; t,x\right)\cdot 
  \Bigg[\tilde \cH\left(\tilde{Z}^{t,x}\left(S\right), v^{(0, \ell, 1)}_{\ell, M}\left(S, \tilde{Z}^{t,x}\left(S\right)\right) \right) \\
  & \qquad \qquad \qquad \qquad \qquad \qquad \qquad \qquad 
  - \tilde \cH\left(\tilde{Z}^{t,x}\left(S\right), v^{(0, -\ell, 1)}_{\ell-1, M}\left(S, \tilde{Z}^{t,x}\left(S\right)\right) \right)\Bigg]\Bigg\| \\
  \leq & ~ \frac{1}{M^{n/2}} \cdot \left\|C(T-t) \cdot \underbar{\em c}\left(\tilde{Z}^{t,x}\left(S\right)\right) \mathcal{V}_i\left(S; t,x\right)
  + \mathcal{P}_i(T; t,x) + \frac{e^{-\beta (T-t)}}{\sqrt{T-t}}\xi\left(\tilde{Z}^{t,x}(T)\right) \mathcal{V}_i\left(T; t,x\right)\right\| \\
  &~ + \sum_{\ell=1}^{n-1} \frac{C(T-t)}{M^{(n-\ell)/2}} 
  \Bigg\| \left|\mathcal{V}_i\left(S; t,x\right) \right| \cdot 
  \sum_{i_1=1}^d C_{i_1}\Bigg| \left(v^{(0, \ell, 1), i_1}_{\ell, M} - v^{(0, -\ell, 1), i_1}_{\ell-1, M}\right)\left(S, \tilde{Z}^{t,x}\left(S\right)\right)\Bigg|\Bigg\| \\
  \leq & ~ \frac{1}{M^{n/2}} \cdot \left\|C(T-t) \cdot \underbar{\em c}\left(\tilde{Z}^{t,x}\left(S\right)\right) \mathcal{V}_i\left(S; t,x\right)
  + \mathcal{P}_i(T; t,x) + \frac{e^{-\beta (T-t)}}{\sqrt{T-t}}\xi\left(\tilde{Z}^{t,x}(T)\right) \mathcal{V}_i\left(T; t,x\right)\right\| \label{eq:variance-1}\\
  &~ + \sum_{\ell=1}^{n-1} \frac{C(T-t)}{M^{(n-\ell)/2}} \cdot 
  \sum_{i_1=1}^d C_{i_1} \Bigg\| \mathcal{V}_i\left(S; t,x\right) \cdot \left(v^{0, i_1}_{\ell, M} - \tilde{v}_{i_1}\right)\left(S, \tilde{Z}^{t,x}\left(S\right)\right) \Bigg\| \label{eq:variance-2}\\
  &~ + \sum_{\ell=1}^{n-1} \frac{C(T-t)}{M^{(n-\ell)/2}} \cdot 
  \sum_{i_1=1}^d C_{i_1} \Bigg\| \mathcal{V}_i\left(S; t,x\right) \cdot \left(v^{0, i_1}_{\ell-1, M} - \tilde{v}_{i_1}\right)\left(S, \tilde{Z}^{t,x}\left(S\right)\right) \Bigg\|. \label{eq:variance-3}
  \end{align}
  We now wish to bound $\left|\E\left[v^{0,i}_{n,M}(t,x)\right] - \tilde{v}_i(t,x)\right|$. 
  To do so, we establish the following lemma.
  \begin{lemma}
  We have, for all $n \in \N$, for all $i = 1, 2, \cdots, d$, 
  \begin{align}
  & \quad \E\left[v^{0,i}_{n,M}(t,x)\right] \\
  & = \E\left[\mathcal{P}_i(T; t,x) + \frac{e^{-\beta (T-t)}}{\sqrt{T-t}}\xi\left(\tilde{Z}^{t,x}(T)\right) \mathcal{V}_i\left(T; t,x\right)\right] \\
  & + \E\left[C(T-t)\mathcal{V}_i\left(S; t,x\right)\cdot 
  \tilde \cH\left(\tilde{Z}^{t,x}\left(S\right), v^{0}_{n-1, M}\left(S, \tilde{Z}^{t,x}\left(S\right)\right) \right)\right].
  \end{align}
  \end{lemma}
  We skip the proof of the lemma, which is standard and similar to, e.g., that of Lemma 3.3 (iii) in \cite{HJK2022}. 
  The main fact used is that $v^\theta_{\ell, M}$ are i.i.d for different $\theta \in \Theta$.
  
  Using the fact the $\tilde v$ is a fixed point of $F$, we also have
  \begin{align}
  &~ \left|\E\left[v^{0,i}_{n,M}(t,x)\right] - \tilde{v}_i(t,x)\right| \\
  = &~ C(T-t) \left| \E \left[\mathcal{V}_i\left(S; t,x\right)\cdot 
  \left(\tilde \cH\left(\tilde{Z}^{t,x}\left(S\right), v^{0}_{n-1, M}\left(S, \tilde{Z}^{t,x}\left(S\right)\right) \right) - \tilde \cH\left(\tilde{Z}^{t,x}\left(S\right), v\left(S, \tilde{Z}^{t,x}\left(S\right)\right) \right)\right)\right]\right| \\
  \leq &~ C(T-t) \sum_{i_1=1}^d C_{i_1} 
  \E\left[\left|\mathcal{V}_i\left(S; t,x\right) \cdot \left(v^{0, i_1}_{n-1, M} - \tilde{v}_{i_1}\right)\left(S, \tilde{Z}^{t,x}\left(S\right)\right)\right|\right] \\
  \leq &~ C(T-t) \sum_{i_1=1}^d C_{i_1} 
  \left\|\mathcal{V}_i\left(S; t,x\right) \cdot \left(v^{0, i_1}_{n-1, M} - \tilde{v}_{i_1}\right)\left(S, \tilde{Z}^{t,x}\left(S\right)\right)\right\|. 
  \label{eq:bias}
  \end{align}
  From expressions \eqref{eq:variance-1} -- \eqref{eq:variance-3} and \eqref{eq:bias}, it follows that
  \begin{align}
  &~ \left\|v^{0,i}_{n,M}(t,x) - \tilde{v}_i(t,x)\right\| \\
  \leq &~ \frac{1}{M^{n/2}} \cdot \left\|C(T-t) \cdot \underbar{\em c}\left(\tilde{Z}^{t,x}\left(S\right)\right) \mathcal{V}_i\left(S; t,x\right)
  + \mathcal{P}_i(T; t,x) + \frac{e^{-\beta (T-t)}}{\sqrt{T-t}}\xi\left(\tilde{Z}^{t,x}(T)\right) \mathcal{V}_i\left(T; t,x\right)\right\|\\
  &~ + \sum_{\ell=1}^{n-1} \frac{C(T-t)}{M^{(n-\ell)/2}} \cdot 
  \sum_{i_1=1}^d C_{i_1} \Bigg\| \mathcal{V}_i\left(S; t,x\right) \cdot \left(v^{0, i_1}_{\ell, M} - \tilde{v}_{i_1}\right)\left(S, \tilde{Z}^{t,x}\left(S\right)\right) \Bigg\| \\
  &~ + \sum_{\ell=1}^{n-1} \frac{C(T-t)}{M^{(n-\ell)/2}} \cdot 
  \sum_{i_1=1}^d C_{i_1} \Bigg\| \mathcal{V}_i\left(S; t,x\right) \cdot \left(v^{0, i_1}_{\ell-1, M} - \tilde{v}_{i_1}\right)\left(S, \tilde{Z}^{t,x}\left(S\right)\right) \Bigg\| \\
  &~ + C(T-t) \sum_{i_1=1}^d C_{i_1} 
  \left\|\mathcal{V}_i\left(S; t,x\right) \cdot \left(v^{0, i_1}_{n-1, M} - \tilde{v}_i\right)\left(S, \tilde{Z}^{t,x}\left(S\right)\right)\right\|\\
  \leq & ~ \frac{1}{M^{n/2}} \cdot \left\|C(T-t) \cdot \underbar{\em c}\left(\tilde{Z}^{t,x}\left(S\right)\right) \mathcal{V}_i\left(S; t,x\right)
  + \mathcal{P}_i(T; t,x) + \frac{e^{-\beta (T-t)}}{\sqrt{T-t}}\xi\left(\tilde{Z}^{t,x}(T)\right) \mathcal{V}_i\left(T; t,x\right)\right\|\\
  &~ + \sum_{\ell=0}^{n-1} \sum_{i_1=1}^d \frac{2 C(T-t) C_{i_1}}{M^{(n-\ell-1)/2}} \cdot 
  \Bigg\| \mathcal{V}_i\left(S; t,x\right) \cdot \left(v^{0, i_1}_{\ell, M} - \tilde{v}_i\right)\left(S, \tilde{Z}^{t,x}\left(S\right)\right) \Bigg\|\\
  \leq & ~ \frac{1}{M^{n/2}} \left[ \left\|C(T-t) \cdot \underbar{\em c}\left(\tilde{Z}^{t,x}\left(S\right)\right) \mathcal{V}_i\left(S; t,x\right)\right\|
  + \left\|\mathcal{P}_i(T; t,x)\right\| + \left\|\frac{e^{-\beta (T-t)}}{\sqrt{T-t}}\xi\left(\tilde{Z}^{t,x}(T)\right) \mathcal{V}_i\left(T; t,x\right) \right\| \right]\\
  &~ + \sum_{\ell=0}^{n-1} \sum_{i_1=1}^d \frac{2 C(T-t) C_{i_1}}{M^{(n-\ell-1)/2}} \cdot 
  \Bigg\| \mathcal{V}_i\left(S; t,x\right) \cdot \left(v^{0, i_1}_{\ell, M} - \tilde{v}_i\right)\left(S, \tilde{Z}^{t,x}\left(S\right)\right) \Bigg\|.
  \end{align}
  We now upper bound terms in the preceding expression respectively. 
  First, we have
  \begin{align*}
  & \quad \| C(T - t) \cdot\underbar{\em c}(\tilde{Z}^{t,x}(S)) \cdot \mathcal{V}_i(S; t,x) \| \\
  & = C(T - t) \E\left[\underbar{\em c}^2(\tilde{Z}^{t,x}(S)) \cdot \mathcal{V}_i^2(S; t,x)\right]^{1/2} \\
  & = C(T - t) \E\left[ \E\left[\underbar{\em c}^2(\tilde{Z}^{t,x}(s)) \cdot \mathcal{V}_i^2(S; t,x)\right] \Big| S=s\right]^{1/2} \\
  & = C(T - t) \left(\int_t^T \frac{e^{-\beta(s-t)}}{\sqrt{s-t}}\cdot \frac{1}{C(T-t)}\E\left[\underbar{\em c}^2(\tilde{Z}^{t,x}(s)) \cdot \mathcal{V}_i^2(s; t,x)\right] ds\right)^{1/2} \\
  &\leq \sqrt{C(T-t)}\left( \int_t^T (s - t)^{-1/2} \mathbb{E} \left[ \underbar{\em c}^2(\tilde{Z}^{t,x}(s)) \cdot \mathcal{V}_i^2(s; t,x) \right] ds \right)^{1/2} \\
  &\leq \sqrt{C(T-t)} \left( \int_t^T (s - t)^{-1/2} \mathbb{E}[\underbar{\em c}^4(\tilde{Z}^{t,x}(s))]^{1/2} 
  \mathbb{E}[\mathcal{V}_i^4(s; t,x)]^{1/2} ds \right)^{1/2}.
  \end{align*}
  By Corollary \ref{cor:moment-w}, 
  \begin{equation}\label{eq:underbar-c}
  \mathbb{E}[\underbar{\em c}^4(\tilde{Z}^{t,x}(s))]^{1/2} \leq \alpha_w^2 \E\left[w^4\left(\tilde{Z}^{t,x}(s)\right)\right]^{1/2} \leq \alpha_w^2 C_{4,T}^{1/2} w^2(x).
  \end{equation}
  By Assumption \ref{as:sigma-D}, and the Burkholder-Davis-Gundy inequality, we have
  \begin{align*}
  \mathbb{E}[\mathcal{V}_i^4(s; t,x)]^{1/2} 
  &= \mathbb{E} \left[ \left( \frac{1}{\sqrt{s - t}} \int_t^s \left[\sigma^{-1} D \tilde{Z}^{t,x}(r)\right]^{\top}_i dB(r) \right)^4 \right]^{1/2} \\
  &\leq \frac{C_4}{s - t} \E\left[\left(\int_t^s \|\sigma^{-1} D_i \tilde{Z}^{t,x}(r)\|^2 dr\right)^2\right]^{1/2} \\
  & \leq C_4 C_{\sigma, D}, 
  \end{align*}
  for some universal constant $C_4$.
  Therefore, 
  \begin{align*}
  & \quad \| C(T - t) \cdot\underbar{\em c}(\tilde{Z}^{t,x}(S)) \cdot \mathcal{V}_i(S; t,x) \| \\
  & \leq \sqrt{C(T-t)} \left( \int_t^T (s - t)^{-1/2} \mathbb{E}[\underbar{\em c}^4(\tilde{Z}^{t,x}(s))]^{1/2} 
  \mathbb{E}[\mathcal{V}_i^4(s; t,x)]^{1/2} ds \right)^{1/2} \\
  & \leq \sqrt{C(T-t)} \left(\int_t^T (s - t)^{-1/2} \alpha_w^2 C_{4,T}^{1/2} C_4 C_{\sigma, D} w^2(x) \right)^{1/2} 
  \leq \underline{C} w(x),
  \end{align*}
  for some positive constant $\underline{C}$ that is independent of $d$. 
  
  For $\| \mathcal{P}_i(T; t,x) \|$, we claim that \( \| \mathcal{P}_i(T; t,x) \| \leq C_\cP w(x) \) for some  
  constant \( C_\cP\) that is independent of $d$ and $i$. 
  By product rule, we have 
  \[
  \E\left[\int_t^T e^{-\beta(s-t)} \kappa^{\top} d\tilde Y^x (s)\right] = \E\left[e^{-\beta(T-t)} \kappa^{\top} \tilde Y^x (T) - \int_t^T \kappa^{\top} \tilde Y^x(s) de^{-\beta(s-t)}\right].
  \]
  Thus, 
  \begin{align*}
  D_x \E\left[\int_t^T e^{-\beta(s-t)} \kappa^{\top} d\tilde Y^x (s)\right] 
  =&~D_x \E\left[e^{-\beta(T-t)} \kappa^{\top} \tilde Y^x (T) - \int_t^T \kappa^{\top} \tilde Y^x(s) de^{-\beta(s-t)}\right] \\
  =&~\E\left[e^{-\beta(T-t)} \kappa^{\top} D\tilde Y^x (T)\right] - \E\left[\int_t^T \kappa^{\top} D\tilde Y^x(s) de^{-\beta(s-t)}\right],
  \end{align*}
  so
  \[
  \mathcal{P}_i(T; t,x) = \E\left[e^{-\beta(T-t)} \kappa^{\top} D_i\tilde Y^x (T) - \int_t^T \kappa^{\top} D_i\tilde Y^x(s) de^{-\beta(s-t)}\right].
  \]
  By Lemma \ref{lem:bound-DZ}, it follows that $\sup_{s\in[t,T]} \left\|D_i\tilde Y^x(s)\right\|_\infty$ is upper bounded by a constant that is independent of $d$ 
  and $i$. By Assumption \ref{as:kappa}, it follows that $\left|\kappa^\top D_i\tilde Y^x (s)\right|/\left(1+(\log d)^{\alpha_0/2}\right)$ is upper bounded by a constant that is independent of $d$, $i$, 
  and $s$. Therefore, 
  \[
  \left|\mathcal{P}_i(T; t,x)\right| \leq C_\cP w(x),
  \]
  for some constant $C_\cP$ that is independent of $d$ and $i$.
  
  For the term $\left\|\frac{e^{-\beta (T-t)}}{\sqrt{T-t}}\xi\left(\tilde{Z}^{t,x}(T)\right) \mathcal{V}_i\left(T; t,x\right) \right\|$, we have
  \begin{eqnarray*}
  \left\|\frac{e^{-\beta (T-t)}}{\sqrt{T-t}}\xi\left(\tilde{Z}^{t,x}(T)\right) \mathcal{V}_i\left(T; t,x\right) \right\| 
  &\leq& \left\| (T-t)^{-1/2} \xi(\tilde{Z}^{t,x}(T)) \mathcal{V}_i(T; t,x) \right\| \\
  &=& (T-t)^{-1/2}\E\left[\xi^2(\tilde{Z}^{t,x}(T)) \mathcal{V}_i^2(T; t,x)\right]^{1/2} \\
  &\leq &(T-t)^{-1/2} \E\left[\xi^4(\tilde{Z}^{t,x}(T))\right]^{1/4} \E\left[\mathcal{V}_i^4(T; t,x)\right]^{1/4} \\
  &\leq &(T-t)^{-1/2} \alpha_w C_{4,T}^{1/4} w(x) \cdot \sqrt{C_4 C_{\sigma, D}}.
  \end{eqnarray*}
  Next,
  \begin{align*}
  & \quad \| C(T - t)(v_{\ell,M}^{0,i_1} - \tilde{v}_{i_1})(S, \tilde{Z}^{t,x}(S)) \mathcal{V}_i(S; t,x) \| \\
  &= C(T - t) \left( \int_t^T \frac{e^{-\beta(s - t)}}{\sqrt{s-t}} \cdot \frac{1}{C(T-t)}\mathbb{E} \left[ (v_{\ell,M}^{0,i_1} - \tilde{v}_i)^2(s, \tilde{Z}^{t,x}(s))
  \cdot \mathcal{V}_i^2(s; t,x) \right] ds \right)^{1/2} \\
  &\leq \sqrt{C(T - t)}\left( \int_t^T (s - t)^{-1/2} (T - s)^{-1/2} \cdot 
  \mathbb{E}[(T-s)\cdot (v_{\ell,M}^{0,i_1} - \tilde{v}_i)^2(s, \tilde{Z}^{t,x}(s)) \cdot \mathcal{V}_i^2(s; t,x)] ds \right)^{1/2}
  \end{align*}
  Define, for a random function $v: [0,T) \times \mathbb{R}_+^d \to \mathbb{R}$, 
  \[
  \vertiii{v}_{s} := \sup_{x \in \mathbb{R}_+^d} \frac{(T - s)^{1/2} \cdot \mathbb{E}[v^2(s,x)]^{1/2}}{w(x)}, 
  \quad s \in [0,T).
  \]
  Then, we continue with the bounds above to get
  \begin{align*}
  & \quad \| C(T - t) (v_{\ell,M}^{0,i_1} - \tilde{v}_{i_1})(S, \tilde{Z}^{t,x}(S)) \cdot \mathcal{V}_i(S; t,x) \| \\
  &\leq \sqrt{C(T - t)}\left( \int_t^T (s - t)^{-1/2} (T - s)^{-1/2} \cdot 
  \mathbb{E}[(T-s)\cdot (v_{\ell,M}^{0,i_1} - \tilde{v}_i)^2(s, \tilde{Z}^{t,x}(s)) \cdot \mathcal{V}_i^2(s; t,x)] ds \right)^{1/2} \\
  &\leq \sqrt{C(T - t)}\Bigg( \int_t^T (s - t)^{-1/2} (T - s)^{-1/2} \cdot \\
  & \qquad \E\left[\mathbb{E} \left[ w^2(z) \cdot \frac{(T-s) (v_{\ell,M}^{0, i_1} - \tilde{v}_{i_1})^2(s, z) y_i^2}{w^2(z)} \Bigg|\tilde{Z}^{t,x}(s) = z, \mathcal{V}(s; t,x) = y\right]\right] ds \Bigg)^{1/2} \\
  &= \sqrt{C(T - t)}\Bigg( \int_t^T (s - t)^{-1/2} (T - s)^{-1/2} \cdot \\
  & \qquad \E\left[w^2(z) y_i^2\mathbb{E} \left[  \frac{(T-s) (v_{\ell,M}^{0, i_1} - \tilde{v}_{i_1})^2(s, z) }{w^2(z)} \right] \Bigg|\tilde{Z}^{t,x}(s) = z, \mathcal{V}(s; t,x) = y\right] ds \Bigg)^{1/2} \\
  &\leq \sqrt{C(T - t)}\left( \int_t^T (s - t)^{-1/2} (T - s)^{-1/2} \mathbb{E}[w^2(\tilde{Z}^{t,x}(s))\cdot \mathcal{V}_i^2(s; t,x)] 
  \cdot \vertiii{v_{\ell,M}^{0,i_1} - \tilde{v}_{i_1}}_{s}^2 ds \right)^{1/2} \\
  &\leq \sqrt{C(T - t)}\left( \int_t^T (s - t)^{-1/2} (T - s)^{-1/2} \mathbb{E}[w^4(\tilde{Z}^{t,x}(s))]^{1/2} \mathbb{E}[\mathcal{V}_i^4(s; t,x)]^{1/2}
   \vertiii{v_{\ell,M}^{0,i_1} - \tilde{v}_{i_1}}_s^2 ds \right)^{1/2} \\
   &\leq \sqrt{C(T - t)} \left( \int_t^T (s - t)^{-1/2} (T - s)^{-1/2} C_{4,T}^{1/2} w^2(x) C_4 C_{\sigma,D}
   \vertiii{v_{\ell,M}^{0,i_1} - \tilde{v}_{i_1}}_s^2 ds \right)^{1/2} \\
   &\leq \tilde C w(x) \left( \int_t^T (s - t)^{-1/2} (T - s)^{-1/2} 
   \vertiii{v_{\ell,M}^{0,i_1} - \tilde{v}_{i_1}}_s^2 ds \right)^{1/2}
  \end{align*}
  for some positive constant $\tilde C$ that is independent of $d$.
  
  In summary, we have established that
  \begin{align*}
  &\quad \| v_{n,M}^{0,i}(t,x) - \tilde{v}_{i}(t,x) \| \\
  &\leq \frac{1}{M^{n/2}} \left(\frac{\tilde C_0}{\sqrt{T-t}} w(x) + C_\cP w(x)\right) + \\
  &\qquad \sum_{\ell = 0}^{n-1} \sum_{i_1=1}^d \frac{2C_{i_1}\tilde C w(x) }{M^{(n - \ell - 1)/2}} 
  \left( \int_t^T (s - t)^{-1/2} (T - s)^{-1/2}  \vertiii{v_{\ell,M}^{0,i_1} - \tilde{v}_{i_1}}_s^2 ds \right)^{1/2} \\
  &\leq \frac{\tilde C_1w(x) }{M^{n/2} \sqrt{T-t}} 
  + \sum_{\ell=0}^{n-1} \sum_{i_1=1}^d \frac{2 \tilde C C_{i_1}w(x) }{M^{(n - \ell - 1)/2}} 
  \left( \int_t^T \frac{1}{(s - t)^{3/4}} \frac{1}{(T - s)^{3/4}} ds \right)^{1/3} 
  \left( \int_t^T  \vertiii{v_{\ell,M}^{0,i_1} - \tilde{v}_{i_1}}_s^6 ds \right)^{1/6},
  \end{align*}
  where the last inequality follows from Hölder's inequality. Thus, 
  \begin{align}
  & \quad \frac{(T - t)^{1/2}}{w(x)} 
  \mathbb{E} \left[ \left(v_{n,M}^{0,i}(t,x) - \tilde{v}_{i}(t,x)\right)^2 \right]^{1/2} \nonumber \\
  &\leq \frac{\bar C}{M^{n/2}} + \sum_{\ell=0}^{n-1} \sum_{i_1=1}^d \frac{2 \bar C C_{i_1}}{M^{(n - \ell - 1)/2}} 
  \left( \int_t^T \vertiii{v_{\ell,M}^{0} - \tilde{v}}_s^6 ds \right)^{1/6},
  \label{eq:picard-bound}
  \end{align}
  where, in the last inequality, with a slight abuse of notation, we defined
  \[
  \vertiii{v}_{s} := \max_{j=1,\dots,J} \sup_{x \in \mathbb{R}_+^d} 
  \frac{(T - s)^{1/2} \cdot \mathbb{E}[ v_j^2(s,x) ]^{1/2}}{w(x)}
  = \max_{j=1,\dots,J} \vertiii{v_j}_s,
  \]
  for random vector-valued function \( v : [0,T) \times \mathbb{R}_+^d \to \mathbb{R}^d \), used the fact that
  \[
  \int_t^T (s - t)^{-3/4} (T - s)^{-3/4} ds < \infty, 
  \]
  and defined the constant $\bar C$ by
  \[
  \bar C := \max\left\{\tilde C_1, \tilde C \sqrt{T-t} \left(\int_t^T (s - t)^{-3/4} (T - s)^{-3/4} ds\right)^{1/3} \right\}.
  \]
  Because the right-hand side of Ineq. \eqref{eq:picard-bound} is independent of $x \in \mathbb{R}_+^d$ and $i = 1, \cdots, d$,
  we can conclude that
  \begin{align*}
  \vertiii{v_{n,M}^{0} - \tilde{v}}_{t} 
  &\leq \frac{\bar C}{M^{n/2}} + \sum_{\ell=0}^{n-1} \sum_{i_1=1}^d \frac{2 \bar C C_{i_1}}{M^{(n - \ell - 1)/2}} 
  \left( \int_t^T \vertiii{v_{\ell,M}^{0} - \tilde{v}}_s^6 ds \right)^{1/6} \\
  & \leq \frac{\bar C}{M^{n/2}} + \sum_{\ell=0}^{n-1}  \frac{2 \bar C C_{\tilde \cH}}{M^{(n - \ell - 1)/2}} 
  \left( \int_t^T \vertiii{v_{\ell,M}^{0} - \tilde{v}}_s^6 ds \right)^{1/6}.
  \end{align*}
  By Lemma 3.11 in \cite{HJKN2020}, we have 
  \begin{eqnarray*}
  \vertiii{v_{n,M}^{0} - \tilde{v}}_{t} &\leq & \left(\bar C + 2 \bar C C_{\tilde \cH} (T-t)^{1/6} \sup_{s\in [t, T]} \vertiii{\tilde v}_s\right) \times \\
  & & \times \exp\left(\frac{M^3}{6}\right) M^{-n/2} \left(1+2 \bar C C_{\tilde \cH} (T-t)^{1/6}\right)^{n-1}. \\
  &= & \left(\bar C + 2 \bar C C_{\tilde \cH} (T-t)^{1/6} \|\tilde v\|_0\right) \times \\
  & & \times \exp\left(\frac{M^3}{6}\right) M^{-n/2} \left(1+2 \bar C C_{\tilde \cH} (T-t)^{1/6}\right)^{n-1},
  \end{eqnarray*}
  where we recall that the norm $\|\cdot\|_0$ was defined in Eq. \eqref{eq:norm-rho} with $\rho = 0$.
  
  By similar reasoning, we can also show that 
  \begin{equation*}
  \frac{\| V_{n,M}(t,x) - \tilde V(t,x) \|}{w(x)} \leq \frac{\bar C_V}{M^{n/2}} + 
  \sum_{\ell=0}^{n-1}  \frac{2 \bar C_V C_{\tilde \cH}}{M^{(n - \ell - 1)/2}} 
  \left( \int_t^T \vertiii{v_{\ell,M}^{0} - \tilde{v}}_s^6 ds \right)^{1/6},
  \end{equation*}
  for some positive constant $\bar C_V$ that is independent of $d$. This implies that 
  \begin{eqnarray*}
  \sup_{x\in \R_+^d} \frac{\| V_{n,M}(t,x) - \tilde V(t,x) \|}{w(x)} & \leq & \left(\bar C_V + 2 \bar C_V C_{\tilde \cH} (T-t)^{1/6} \|\tilde v\|_0\right) \times \\
  & & \times \exp\left(\frac{M^3}{6}\right) M^{-n/2} \left(1+2 \bar C C_{\tilde \cH} (T-t)^{1/6}\right)^{n-1}.
  \end{eqnarray*}
  Finally, with a slight abuse of notation, define 
  \begin{equation}
  \vertiii{(V,v)}_{s} := \sup_{x \in \mathbb{R}_+^d} 
  \frac{\mathbb{E}[ V(s,x)^2 ]^{1/2}}{w(x)} \vee \max_{j=1,\dots,J} \sup_{x \in \mathbb{R}_+^d} 
  \frac{(T - s)^{1/2} \cdot \mathbb{E}[ v_j^2(s,x)^2 ]^{1/2}}{w(x)}. 
  \end{equation}
  Then, by appropriately redefining $\bar C$, we have
  \begin{eqnarray*}
  \vertiii{\left(V_{n,M}^{0}, v_{n,M}^{0}\right) - \left(\tilde V, \tilde{v}\right)}_{t} 
  &\leq & \left(\bar C + 2 \bar C C_{\tilde \cH} (T-t)^{1/6} \|\tilde v\|_0\right) \times \\
  & & \times \exp\left(\frac{M^3}{6}\right) M^{-n/2} \left(1+2 \bar C C_{\tilde \cH} (T-t)^{1/6}\right)^{n-1}.
  \end{eqnarray*}
  
  We now follow a similar reasoning as the proof of Corollary 5.1 in \cite{HJK2022} to derive computational complexity bounds.
  To simplify notation, define a new constant $\bar C_T$ by
  \[
  \bar C_T := 1 + \bar C + 2\bar C (T-t)^{1/6}. 
  \]
  Then, 
  \begin{eqnarray*}
  \vertiii{\left(V_{n,M}^{0}, v_{n,M}^{0}\right) - \left(\tilde V, \tilde{v}\right)}_{t} 
  &\leq & \bar C_T C_{\tilde \cH} \max\{1, \|\tilde v\|_0\} \exp\left(\frac{M^3}{6}\right) M^{-n/2} \left(\bar C_T C_{\tilde \cH} \right)^{n-1} \\
  &=& \exp\left(\frac{M^3}{6}\right) M^{-n/2} \left(\bar C_T C_{\tilde \cH} \right)^n \max\{1, \|\tilde v\|_0\}.
  \end{eqnarray*}
  Therefore\footnote{For the remainder of the proof, we will treat parameters such as $n^{1/3}$ and $N^{1/3}$ as if they were guaranteed to be integers. 
  Rounding them up or down to a nearest integer would overburden the notation, but would have no effect on our order-of-magnitude estimates.}, 
  \begin{eqnarray*}
  \vertiii{\left(V_{n,n^{1/3}}^{0}, v_{n,n^{1/3}}^{0}\right) - \left(\tilde V, \tilde{v}\right)}_{t} 
  &\leq & e^{n/6} n^{-n/6} \left(\bar C_T C_{\tilde \cH} \right)^n \max\{1, \|\tilde v\|_0\}.
  \end{eqnarray*}
  
  \paragraph{Complexity analysis.} Let $RS$ denote the expected runtime of the exact simulation algorithm with polynomial complexity for the tuple \eqref{eq:tuple}, 
  and let $RO$ denote the runtime of the polynomial-complexity optimization algorithm for solving $\tilde \cH$.
  Let $\underline{RT} = RS + 2\cdot RO$ be a ``base'' runtime. 
  Define $RT(n,M)$ as the expected runtime needed to produce one realization of $\big(V^0_{n,M}(t,x), v^0_{n,M}(t,x)\big)$. The multilevel construction \eqref{eq:mlp-v0}--\eqref{eq:mlp-rec} yields the recurrence
  \begin{equation}
  RT(n,M) \;\le\; \underline{RT} \cdot M^n \, + \, \sum_{\ell=1}^{n-1} M^{\,n-\ell}\,\Big(\underline{RT} \, + \, RT(\ell,M) \, + \, RT(\ell{-}1,M)\Big), \quad n\ge 1,
  \end{equation}
  with $RT(0,M)=0$. Using the same reasoning as the proof of Lemma 3.6 in \cite{HJKNW2020}, we can establish that  
  \begin{equation}\label{eq:work-upper}
  RT(n,M) \;\le\; 5\cdot\underline{RT}\cdot (2M{+}1)^n \;\le\; 5\cdot\underline{RT}\cdot (3M)^n, \qquad n\in \N.
  \end{equation}
  For notational convenience, let
  \begin{equation}
  \veps(n,M) := \frac{\vertiii{\left(V_{n,M}^{0}, v_{n,M}^{0}\right) - \left(\tilde V, \tilde{v}\right)}_{t}}{\max\{1, \|\tilde v\|_0\}},
  \end{equation}
  where we recall the definition of $\vertiii{\cdot}_s$ in Eq. \eqref{eq:vertiii}.
  Since 
  \[
  \veps\left(n, n^{1/3} \right) \leq e^{n/6} n^{-n/6} \left(\bar C_T C_{\tilde \cH} \right)^n, 
  \]
  and $e^{n/6} n^{-n/6}\left(\bar C_T C_{\tilde \cH} \right)^n \to 0$ as $n \to \infty$, we have $\veps\left(n, n^{1/3}\right) \to 0$ as $n \to \infty$. 
  Let $\veps\in (0, 1/2)$, and define
  \begin{equation}
  N_\veps := \min \left\{n_0 \in \N : \veps\left(n, n^{1/3}\right) \leq \veps  \text{ for all } n > n_0\right\}.
  \end{equation}
  Write $N = N_\veps$ for notational simplicity. Then we claim that 
  \begin{equation}
  \veps \leq e^{N/6} N^{-N/6} \left(\bar C_T C_{\tilde \cH} \right)^N.
  \end{equation}
  Indeed, if $N=1$, then $e^{N/6} N^{-N/6} \left(\bar C_T C_{\tilde \cH} \right)^N = 1 > 1/2 > \veps$; 
  if $N \geq 2$, then by minimality of $N$, we must have 
  \[
  \veps < \veps\left(N, N^{1/3}\right) \leq e^{N/6} N^{-N/6} \left(\bar C_T C_{\tilde \cH} \right)^N.
  \]
  Thus, we have established the claim. Therefore,
  \begin{align*}
  RT\left((N+1), (N+1)^{1/3} \right) \le &~5\cdot\underline{RT}\cdot \left(3(N+1)\right)^{(N+1)/3} \\
  \le &~ 5\cdot\underline{RT} \cdot \left(3(N+1)\right)^{(N+1)/3} \cdot \veps^{-4} \left(e^{N/6} N^{-N/6} \left(\bar C_T C_{\tilde \cH} \right)^N\right)^4 \\
  \le &~ 5\cdot\underline{RT}\cdot \veps^{-4} \left[\left(3(N+1)\right)^{(N+1)/3} e^{2N/3} N^{-2N/3} \left(\bar C_T C_{\tilde \cH} \right)^{4N}\right] \\
  \le &~ 5\cdot\underline{RT} \cdot \veps^{-4} \sup_{N\in \N} \left[\left(3(N+1)\right)^{(N+1)/3} e^{2N/3} N^{-2N/3} \left(\bar C_T C_{\tilde \cH} \right)^{4N}\right] \\
  \le &~ 15\alpha_1 (1+d^{\alpha_2}) \cdot \veps^{-4} \sup_{N\in \N} \left[\left(3(N+1)\right)^{(N+1)/3} e^{2N/3} N^{-2N/3} \left(\bar C_T C_{\tilde \cH} \right)^{4N}\right],
  \end{align*}
  where $\alpha_1$ and $\alpha_2$ are from Definitions \ref{df:poly-sim} and \ref{df:poly-opt}.
  Under Assumption \ref{asmp:lipschitz-H}, the quantity 
  \begin{equation}\label{eq:C_RT}
  \sup_{N\in \N} \left[\left(3(N+1)\right)^{(N+1)/3} e^{2N/3} N^{-2N/3} \left(\bar C_T C_{\tilde \cH} \right)^{4N}\right]
  \end{equation}
  is a constant independent of $d$. This concludes the proof of Theorem \ref{thm:main1}.

\section{Proofs for Section \ref{sec:bel}}
\subsection{Proof of Proposition \ref{prop:bel}}
To prove Proposition \ref{prop:bel}, we introduce some notation and establish a useful lemma. 
For $\veps > 0$, define the process \( \partial_\Upsilon^\varepsilon \tilde{Z}^x(t) \) by
\begin{equation}\label{eq:deriv-eps-M}
\partial_\Upsilon^\varepsilon \tilde{Z}^x(t) := \frac{\tilde{Z}^x(t; \varepsilon \Upsilon) - \tilde{Z}^x(t)}{\varepsilon}, \quad t \ge 0,
\end{equation}
and define the process \( {\Psi}^\varepsilon \) by
\begin{equation}\label{eq:Psi-eps-M}
{\Psi}^\varepsilon(t) := \int_0^t \frac{\tb(\tilde{Z}^x(s; \varepsilon \Upsilon)) - \tb(\tilde{Z}^x(s))}{\varepsilon}\,ds + \Upsilon(t), \quad t \ge 0.
\end{equation}
Let
\begin{equation}\label{eq:X-eps-M}
\tilde{X}^x(t; \varepsilon \Upsilon) := x + \int_0^t \tb(\tilde{Z}^x(s; \varepsilon \Upsilon))\,ds + \varepsilon \Upsilon(t) + \sigma B(t), \qquad t \ge 0.
\end{equation}

Then, for $t \ge 0$, 
\begin{align*}
\tilde{X}^x(t; \varepsilon \Upsilon)
&= x + \int_0^t \tb(\tilde{Z}^x(s))\,ds + \sigma B(t) + \int_0^t \left[ \tb(\tilde{Z}^x(s; \varepsilon \Upsilon)) - \tb(\tilde{Z}^x(s)) \right]\,ds + \varepsilon \Upsilon(t) \\
&= \tilde{X}^x(t) + \varepsilon {\Psi}^\varepsilon(t).
\end{align*}

Thus, by definition,
\[
\partial_\Upsilon^\varepsilon \tilde{Z}^x = \frac{\Gamma(\tilde{X}^x + \varepsilon {\Psi}^\varepsilon) - \Gamma(\tilde{X}^x)}{\varepsilon}
= \nabla_{{\Psi}^\varepsilon} \Gamma(\tilde{X}^x).
\]

\begin{lemma}\label{lem:bel}
Suppose Assumptions \ref{as:tb} and \ref{as:sigma} hold.  
Let \( x \in \mathbb{R}_+^d \), and suppose a.s.\ \( \nabla_\psi \Gamma(\tilde{X}^x) \) exists for all \( \psi \in \C(\mathbb{R}^d) \), 
and lies in \( \D_{\ell, r}(\mathbb{R}^d) \).  
Let \( \Upsilon(\cdot) \) and \( \upsilon(\cdot) \) be the same as in Proposition \ref{prop:bel}.
Then, there exists a unique \( \{\mathcal{F}_t\} \)-adapted process \( \Phi \) such that a.s.\ \( \Phi \in \D(\mathbb{R}^d) \) and satisfies
\[
\Phi = \nabla_{{\Psi}} \Gamma(\tilde{X}^x),
\]
where \( {\Psi} \) is a \( J \)-dimensional continuous \( \{\mathcal{F}_t\} \)-adapted process with
\[
{\Psi}(t) = \int_0^t D\tb(\tilde{Z}^x(s)) \Phi(s)\,ds + \Upsilon(t), \quad t\geq 0.
\]

Furthermore, for any \( t \ge 0 \),
\[
\| {\Psi}\|_t < C_\upsilon t \cdot \exp \left(C_{\tb} C_\Gamma t\right), \ \ \text{and} \quad \lim_{\varepsilon \to 0} \| \Phi^\veps - \Phi \|_t  = 0, 
\]
where we recall the definition of $\Phi^\veps$ in Eq. \eqref{eq:Psi-eps-M}.
\end{lemma}

\begin{proof}
We first show the uniqueness of the process \( \Phi \).  
Suppose both \( \Phi \) and \( \tilde{\Phi} \) satisfy the properties stated in the lemma. Let
\[
\tilde{\Psi}(t) := \int_0^t D\tb(\tilde{Z}^x(s)) \tilde{\Phi}(s)\,ds + \Upsilon(t), \quad t\geq 0.
\]
Then \( \Phi = \nabla_{\hat{\Psi}} \Gamma(\tilde{X}^x) \) and \( \tilde{\Phi} = \nabla_{\tilde{\Psi}} \Gamma(\tilde{X}^x) \).

We have
\begin{align*}
\| \Phi - \tilde{\Phi} \|_t & = \| \nabla_{{\Psi}} \Gamma(\tilde{X}^x) - \nabla_{\tilde{\Psi}} \Gamma(\tilde{X}^x) \|_t \\
& \le C_\Gamma \| {\Psi} - \tilde{\Psi} \|_t \\
& = C_\Gamma\left\{\sup_{s\in [0, t]} \Biggl|\int_0^s D\tb (\tilde{Z}^x(s)) \left[\Phi(s) - \tilde{\Phi}(s)\right] ds\Biggr|\right\} \\
& \le C_\Gamma C_{\tb} \int_0^t \| {\Phi} - \tilde{\Phi} \|_s ds.
\end{align*}
Here, the first inequality follows from the Lipschitz property of the derivative map $\nabla \Gamma$, stated in Proposition \ref{prop:nabla-lipschitz}, 
and the third inequality follows from Assumption \ref{as:tb}. By Grönwall's inequality, we must have
$\|\Phi - \tilde{\Phi}\|_t = 0$ for all $t \geq 0$, implying that $\Phi = \tilde \Phi$.

Next, we establish the existence of $\Phi$ by a standard Picard iteration argument.

Let \( \Phi^0 = 0 \), and recursively define, for \( t \ge 0 \),
\[
{\Psi}^k(t) := \int_0^t D\tb(\tilde{Z}^x(s))\, {\Phi}^{k-1}(s)\,ds + \Upsilon(t),
\]
and let
\[
\Phi^k := \nabla_{{\Psi}^k} \Gamma(\tilde{X}^x).
\]

Then
\[
\|{\Psi}^1\|_t = \|\Upsilon\|_t \leq \int_0^t \|\upsilon\|_s ds \leq C_\upsilon t, \quad t \geq 0.
\]
For $k \geq 2$, we have
\begin{align*}
\| {\Psi}^k(t) - \Psi^{k-1}(t) \| 
&= \left\| \int_0^t D\tb(\tilde{Z}^x(s)) \left( \Phi^{k-1}(s) - \Phi^{k-2}(s) \right) ds \right\| \\
&\le C_{\tb} \int_0^t \| \Phi^{k-1}(s) - \Phi^{k-2}(s) \|\,ds \\
&\le C_{\tb} \int_0^t \| \Phi^{k-1} - \Phi^{k-2}\|_s\,ds.
\end{align*}
Thus, by Proposition \ref{prop:nabla-lipschitz} and Assumption \ref{as:tb}, for $k \geq 2$, 
\[
\| \Phi^k - \Phi^{k-1} \|_t \le C_\Gamma \| \Psi^{k} - \Psi^{k-1} \|_t
\le C_\Gamma C_{\tb} \int_0^t \| \Phi^{k-1} - \Phi^{k-2} \|_s\, ds.
\]

Noting that $\Phi^0 \equiv 0$, a standard Picard iteration argument gives
\[
\| \Phi^{k+1} - \Phi^k \|_t \le \frac{(C_\Gamma C_{\tb})^k}{k!} \int_0^t \| \Phi^1 \|_s\,ds \le \frac{(C_\Gamma C_{\tb})^k C_\upsilon t^2}{2\cdot k!}.
\]

This implies that for any $t \ge 0$,  \( \{ \| \Psi^{k+1} - \Psi^k \|_t \}_k \) is a Cauchy sequence.  
Thus, there exists \( \Psi \in \C(\mathbb{R}^d) \) such that \( \Psi^k \to \Psi \) in \( \C(\mathbb{R}^d) \).

Let
\[
\Phi := \nabla_{\Psi} \Gamma(\tilde{X}^x).
\]

Then, by Lipschitz continuity of the map \( \nabla \Gamma \),
\[
\Phi^k \to \Phi \quad \text{uniformly on compact intervals, as } k \to \infty.
\]

Since
\[
\Psi^k(t) = \int_0^t D\tb(\tilde{Z}^x(s)) \Phi^{k-1}(s)\,ds + \Upsilon(t),
\]
\( \Phi^k \to \Phi \) u.o.c.\ as \( k \to \infty \), and $D\tb$ and $\upsilon$ are both uniformly bounded, 
\( \Psi \) must satisfy
\[
\Psi(t) = \int_0^t D\tb(\tilde{Z}^x(s)) \Phi(s)\,ds + \Upsilon(t), \qquad t \ge 0.
\]
In summary, we have established the uniqueness and existence of the process $\Phi$.

Next, we establish bounds on $\|\Psi\|_t$. We have
\begin{align*}
\| \Psi \|_t &= \sup_{s \in [0,t]} \left| \int_0^s D\tb(\tilde{Z}^x(r)) \Phi(r)\,dr + \Upsilon(s) \right| \\
& \le C_\upsilon t + C_{\tb} \int_0^t \| \Phi \|_s\,ds
\le C_\upsilon t + C_{\tb} C_\Gamma \int_0^t \| \Psi \|_s\,ds.
\end{align*}
In the last inequality, we used the uniform boundedness of $D\tb$, the Lipschitz continuity of the map \( \nabla \Gamma \), and the fact that
$0 = \nabla_0 \Gamma(\tilde{X}^x)$, where, with a slight abuse of notation, $0$ denotes the $d$-dimensional process that takes value $0$ in all coordinates, for all times.

By Grönwall’s inequality,
\[
\| {\Psi} \|_t < C_\upsilon t \cdot \exp \left(C_{\tb} C_\Gamma t\right) \quad \text{for all } t \ge 0. 
\]

Finally, we consider $\| \Psi^\varepsilon - \Psi \|_t$. For $t \geq 0$, we have 
\begin{align*}
\| \Psi^\varepsilon (t) - \Psi (t) \| &= \left\| \int_0^t \frac{\tb(\tilde{Z}^x(s; \varepsilon \Upsilon)) - \tb(\tilde{Z}^x(s))}{\varepsilon} ds - \int_0^t D\tb(\tilde{Z}^x(s)) \Phi(s)\,ds \right\| \\
&= \left\| \int_0^t \left[ \frac{\tb(\tilde{Z}^x(s) + \varepsilon \partial_\Upsilon^\varepsilon \tilde{Z}^x(s)) - \tb(\tilde{Z}^x(s))}{\varepsilon} - D\tb(\tilde{Z}^x(s)) \Phi(s) \right] ds \right\|.
\end{align*}

For any given \( s \in [0,t] \), we have
\[
\frac{\tb(\tilde{Z}^x(s) + \varepsilon \Phi(s)) - \tb(\tilde{Z}^x(s))}{\varepsilon}
= \int_0^1 D\tb(\tilde{Z}^x(s) + \delta \varepsilon \Phi(s)) \Phi(s) d\delta,
\]
so we have
\begin{align*}
\| \Psi^\varepsilon(t) - \Psi(t) \| 
& = \Bigg|\int_0^t \Bigg[ \frac{\tb(\tilde{Z}^x(s) + \varepsilon \partial_\Upsilon^\varepsilon \tilde{Z}^x(s)) - \tb(\tilde{Z}^x(s)+ \varepsilon \Phi(s))}{\varepsilon} +  \\
& \quad + \int_0^1 \left( D\tb(\tilde{Z}^x(s) + \delta \varepsilon \Phi(s)) - D\tb(\tilde{Z}^x(s)) \right) \Phi(s) d\delta \Bigg] ds \Bigg| \\
& \le \int_0^t C_{\tb} \| \partial_\Upsilon^\varepsilon \tilde{Z}^x(s) - \Phi(s) \| ds \\
& \quad + \int_0^t \left(\int_0^1 C_{\tb} \|\delta \veps \Phi(s)\|^{\alpha_{\tb}} \|\Phi(s)\|\right) ds \\
& \le \int_0^t C_{\tb} \| \partial_\Upsilon^\varepsilon \tilde{Z}^x(s) - \Phi(s) \| ds
+ \int_0^t C_{\tb} \veps^{\alpha_{\tb}} \| \Phi(s) \|^{1 + {\alpha_{\tb}}} ds.
\end{align*}
We first derive a bound on the term $\int_0^t C_{\tb} \| \partial_\Upsilon^\varepsilon \tilde{Z}^x(s) - \Phi(s) \| ds$. 
To bound \( \| \partial_\Upsilon^\varepsilon \tilde{Z}^x(s) - \Phi(s) \| \), recall that
\[
\partial_\Upsilon^\varepsilon \tilde{Z}^x = \nabla^\veps_{\Psi^\varepsilon} \Gamma(\tilde{X}^x),
\]
So, we have
\begin{align*}
\| \partial_\Upsilon^\varepsilon \tilde{Z}^x(s) - \Phi(s) \|
&= \| \nabla^\veps_{\Psi^\varepsilon} \Gamma(\tilde{X}^x)(s) - \nabla_{\Psi} \Gamma(\tilde{X}^x)(s) \| \\
&\le \| \nabla^\veps_{\Psi^\varepsilon} \Gamma(\tilde{X}^x)(s) - \nabla^\veps_{\Psi} \Gamma(\tilde{X}^x)(s) \| 
+ \| \nabla^\veps_{\Psi} \Gamma(\tilde{X}^x)(s) - \nabla_{\Psi} \Gamma(\tilde{X}^x)(s) \| \\
&\le C_\Gamma \| \Psi^\varepsilon - \Psi \|_s + \| \nabla^\veps_{\Psi} \Gamma(\tilde{X}^x)(s) - \nabla_{\Psi} \Gamma(\tilde{X}^x)(s) \|, 
\end{align*}
where, in the last inequality, we used the Lipschitz property of the map $\Gamma$, 
which directly implies the Lipschitz property of the map $\nabla^\veps \Gamma$. Thus, 
\begin{align*}
\| \Psi^\varepsilon(t) - \Psi(t) \| &\le C_{\tb}\veps^{\alpha_{\tb}} \int_0^t \| \Phi(s) \|^{{\alpha_{\tb}}+1}\,ds
+ C_\Gamma C_{\tb} \int_0^t \| \Psi^\varepsilon - \Psi \|_s\,ds \\
&\quad + C_\Gamma \int_0^t \| \nabla_{\Psi^\varepsilon} \Gamma(\tilde{X}^x)(s) - \nabla_{\Psi} \Gamma(\tilde{X}^x)(s) \|\,ds \\
&\le C_{\tb}\veps^{\alpha_{\tb}} \int_0^t \| \Phi \|_s^{{\alpha_{\tb}}+1}\,ds
+ C_\Gamma C_{\tb} \int_0^t \| \Psi^\varepsilon - \Psi \|_s\,ds \\
&\quad + C_\Gamma \int_0^t \| \nabla_{\Psi^\varepsilon} \Gamma(\tilde{X}^x)(s) - \nabla_{\Psi} \Gamma(\tilde{X}^x)(s) \|\,ds.
\end{align*}
Therefore, we also have
\begin{align*}
\| \Psi^\varepsilon - \Psi \|_t 
&\le C_{\tb}\veps^{\alpha_{\tb}} \int_0^t \| \Phi \|_s^{{\alpha_{\tb}}+1}\,ds + C_\Gamma \int_0^t \| \nabla_{\Psi^\varepsilon} \Gamma(\tilde{X}^x)(s) - \nabla_{\Psi} \Gamma(\tilde{X}^x)(s) \|\,ds \\
&\quad + C_\Gamma C_{\tb} \int_0^t \| \Psi^\varepsilon - \Psi \|_s\,ds. 
\end{align*}

By Grönwall’s inequality, if we can show that for any \( t \geq 0 \), the term 
\begin{equation}\label{eq:lem:bel-t1}
C_{\tb}\veps^{\alpha_{\tb}} \int_0^t \| \Phi \|_s^{{\alpha_{\tb}}+1}\,ds + C_\Gamma \int_0^t \| \nabla_{\Psi^\varepsilon} \Gamma(\tilde{X}^x)(s) - \nabla_{\Psi} \Gamma(\tilde{X}^x)(s) \|\,ds
\end{equation}
goes to zero as \( \varepsilon \to 0 \), then
\[
\| \Psi^\varepsilon - \Psi \|_t \to 0.
\]
For the first term in expression \eqref{eq:lem:bel-t1},
\[
C_{\tb}\veps^{\alpha_{\tb}} \int_0^t \| \Phi \|_s^{{\alpha_{\tb}}+1}\,ds,
\]
since we have established that \( \| \Psi \|_t < C_\upsilon t \cdot \exp \left(C_{\tb} C_\Gamma t\right) \) for all \( t \geq 0 \),  
by Lipschitz continuity of the map \( \nabla \Gamma \) and the fact that \( 0 = \nabla_0 \Gamma \),  
\[
\| \Phi \|_t < C_\Gamma C_\upsilon t \cdot \exp \left(C_{\tb} C_\Gamma t\right) \quad \text{for all } t > 0, \quad \text{and} \quad \int_0^t \| \Phi \|_s^{{\alpha_{\tb}}+1} ds < \infty.
\]
Thus, as \( \varepsilon \to 0 \),
\[
C_{\tb}\veps^{\alpha_{\tb}} \int_0^t \| \Phi \|_s^{{\alpha_{\tb}}+1}\,ds \to 0.
\]
For the second term,
\[
C_\Gamma \int_0^t \| \nabla_{\Psi^\varepsilon} \Gamma(\tilde{X}^x)(s) - \nabla_{\Psi} \Gamma(\tilde{X}^x)(s) \|\,ds,
\]
by the Lipschitz continuity of both the maps \( \nabla \Gamma \) and \( \Gamma \), and the fact that  
\[
0 = \nabla_0 \Gamma = \nabla_0 \Gamma,
\]
we have
\[
\| \nabla^\veps_{\Psi} \Gamma(\tilde{X}^x)(s) - \nabla_{\Psi} \Gamma(\tilde{X}^x)(s) \| \le 2 C_\Gamma \| \Psi \|_t.
\]
Since \( \int_0^t \| \Psi \|_s ds < \infty \), and by definition,
\[
\nabla^\veps_{\Psi} \Gamma(\tilde{X}^x)(s) \to \nabla_{\Psi} \Gamma(\tilde{X}^x)(s) \ \text{ as } \veps \to 0, \ \text{ for all } s, 
\]
by dominated convergence theorem we have
\[
C_\Gamma \int_0^t \| \nabla_{\Psi^\varepsilon} \Gamma(\tilde{X}^x)(s) - \nabla_{\Psi} \Gamma(\tilde{X}^x)(s) \|\,ds \to 0
\quad \text{as } \varepsilon \to 0.
\]
In summary, we have established that as $\veps \to 0$, the term in expression \eqref{eq:lem:bel-t1} goes to zero, implying that 
$\| \Psi^\varepsilon - \Psi \|_t \to 0$ by Grönwall's inequality. This concludes the proof of the lemma.
\end{proof}

\begin{proof}[Proof of Proposition \ref{prop:bel}]
Let \( \Psi \) and \( \Phi \) be the same as in Lemma \ref{lem:bel}, so that
\[
\Psi(t) = \int_0^t D\tb(\tilde{Z}^x(s)) \Phi(s)\,ds + \Upsilon(t),
\]
and \( \Phi = \nabla_{\Psi} \Gamma(\tilde{X}^x) \). Also recall the definition of \( \Psi^\varepsilon \) in Eq. \eqref{eq:Psi-eps-M}.

Because \( \Psi^\varepsilon \to \Psi \) in \( \C(\mathbb{R}^d) \), by Proposition \ref{prop:nabla-veps}, the limit of
\[
\partial_\Upsilon^\varepsilon \tilde{Z}^x = \nabla_{\Psi^\varepsilon} \Gamma(\tilde{X}^x)
\]
is \( \nabla_{\Psi} \Gamma(\tilde{X}^x) \), so we must have that $\partial_\Upsilon \tilde{Z}^x$, the limit of $\partial_\Upsilon^\varepsilon \tilde{Z}^x$ 
as $\veps \to 0$, coincides with the process \( \nabla_{\Psi} \Gamma(\tilde{X}^x) \).
This completes the proof of the proposition. 
\end{proof}

\subsection{Proofs of Proposition \ref{prop:bel2}}
\begin{proof}
By Proposition \ref{prop:LR2018-main}, for all \( \psi \in \C(\R^d) \),  
\( \nabla_\psi \Gamma(\tilde{X}^x) \) exists, lies in \( \D_{\ell,r}(\mathbb{R}^d) \), and
\[
\nabla_\psi \Gamma(\tilde{X}^x)(t+) = \Lambda_{\tilde{Z}^x}[\Psi](t) \quad \text{for all } t > 0,
\]
and \( \nabla_\psi \Gamma(\tilde{X}^x) \) is continuous at times \( t > 0 \) when $\tilde Z^x(t) > 0$, or when $\tilde Z^x_j = 0$ for more than one $j \in \{1, 2, \cdots, d\}$.

Thus, by Proposition \ref{prop:bel}, \( \partial_\Upsilon \tilde{Z}^x \) exists, and
$\partial_\Upsilon \tilde{Z}^x = \nabla_{\Psi} \Gamma(\tilde{X}^x)$, 
where
\[
\Psi(t) = \int_0^t D\tb(\tilde{Z}^x(s)) \partial_\Upsilon \tilde{Z}^x(s)\,ds + \Upsilon(t).
\]

Consequently, \( \partial_\Upsilon \tilde{Z}^x(t+) \in D_{\ell, r}(\mathbb{R}^d) \), and
\[
\partial_\Upsilon \tilde{Z}^x(t+) = \Lambda_{\tilde{Z}^x}[\Psi](t) \ \text{ for all } t \ge 0.
\]

It remains to prove that there exists a unique derivative process \( D_\Upsilon \tilde{Z}^x \) along \( \tilde{Z}^x \) in the direction \( \Upsilon \), and part (iii) holds.
We claim that a.s.\ $D_\Upsilon \tilde{Z}^x = \Lambda_{\tilde{Z}^x}[\Psi]$.

First, since \( \Lambda_{\tilde{Z}^x}[\Psi] \) is the right-continuous regularization of \( \partial_\Upsilon \tilde{Z}^x \), we have
\[
\Psi(t) = \int_0^t D\tb(\tilde{Z}^x(s)) \Lambda_{\tilde{Z}^x}[\Psi](s)\,ds + \Upsilon(t).
\]
Next, because by definition \( \Lambda_{\tilde{Z}^x}[\Psi] \) solves the derivative problem along \( \tilde{Z}^x \) for \( \Psi \), we have
\[
\Lambda_{\tilde{Z}^x}[\Psi](t) = \Psi(t) + \eta_{\Psi}(t),
\]
with $\eta_{\Psi}(0) \in \text{span}[R(x)]$, $\Lambda_{\tilde{Z}^x}[\Psi](t) \in H_{\tilde{Z}^x(t)}$, and 
\[
\eta_{\Psi}(t) - \eta_{\Psi}(s) \in \text{span} \left[ \bigcup_{r \in (s,t]} R(\tilde{Z}^x(r)) \right]
\quad \text{for all } 0 \le s < t < \infty.
\]

Because \( \Psi(0) = 0 \), \( \eta_{\Psi}(0) \in \text{span}[R(x)] \), 
 \( \Lambda_{\tilde{Z}^x}[\Psi](0) \in H_x \), and $H_x \cap \text{span}[R(x)] = \{0\}$ by Lemma 8.1 of \cite{LipshutzRamanan2018},  
we must have 
\[
\eta_{\Psi}(0) = \Lambda_{\tilde{Z}^x}(0) = 0.
\]
Thus, \( \Lambda_{\tilde{Z}^x}[\Psi] \) is a derivative process along \( \tilde{Z}^x \) in the direction \( \Upsilon \).  
By pathwise uniqueness, a.s.\ \( D_M \tilde{Z}^x = \Lambda_{\tilde{Z}^x}[\Psi] \).
This concludes the proof of the proposition. 
\end{proof}

\subsection{Proof of Proposition \ref{prop:bel3}}
\begin{proof}
Let
\[
\tilde{\Phi}(t) := \int_0^t D{\tilde{Z}^x}(s,t) \upsilon(s)\,ds.
\]
We wish to show that \( \tilde{\Phi} \) is a derivative process along \( \tilde{Z}^x \) in the direction \( \Upsilon \),  
so that by pathwise uniqueness we can conclude that a.s.\ \( D_\Upsilon \tilde{Z}^x = \tilde{\Phi} \).

Let $\eta^s(t)$ be such that 
\[
D{\tilde{Z}^x}(s,t) = I + \int_s^t D\tb(\tilde{Z}^x(s)) D{\tilde{Z}^x}(s,r)\,dr + \eta^s(t), \qquad 0 \le s < t < \infty.
\]
Then,
\begin{align*}
\int_0^t D{\tilde{Z}^x}(s,t) \upsilon(s)\,ds
&= \int_0^t \left[ I + \int_s^t D\tb(\tilde{Z}^x(r)) D{\tilde{Z}^x}(s,r)\,dr + \eta^s(t) \right] \upsilon(s)\,ds \\
& = \Upsilon(t) + \int_0^t D\tb(\tilde{Z}^x(r)) \left[ \int_0^r D{\tilde{Z}^x}(s,r) \upsilon(s)\,ds \right] dr + \int_0^t \eta^s(t) \upsilon(s)\,ds.
\end{align*}
Let
\[
\eta_{\tilde{\Phi}}(t) := \int_0^t \eta^s(t) \upsilon(s)\,ds.
\]
Then,
\[
\tilde{\Phi}(t) = M(t) + \int_0^t D\tb(\tilde{Z}^x(r)) \tilde{\Phi}(r)\,dr + \eta_{\tilde{\Phi}}(t).
\]

By the definition of \( D{\tilde{Z}^x}(s,t) \), for each \( 0 \le s < t < \infty \),  
\[
D{\tilde{Z}^x}(s,t) \in H_{\tilde{Z}^x(t)}.
\]
Thus,
\[
\tilde{\Phi}(t) = \int_0^t D{\tilde{Z}^x}(s,t) \upsilon(s)\,ds \in H_{\tilde{Z}^x(t)}.
\]
We also have $\eta_{\tilde{\Phi}}(0) = 0 \in \text{span}[R(x)]$ and $\tilde{\Phi}(0) = 0 \in H_x$.

For all \( 0 \le s < t < \infty \),
\begin{align*}
\eta_{\tilde{\Phi}}(t) - \eta_{\tilde{\Phi}}(s)
&= \int_0^t \eta^r(t) \upsilon(r)\,dr - \int_0^s \eta^r(s) \upsilon(r)\,dr \\
&= \int_0^s (\eta^r(t) - \eta^r(s)) \upsilon(r)\,dr + \int_s^t \eta^r(t) \upsilon(r)\,dr.
\end{align*}
By definition, for each \( r \in [0,s] \), 
\[
\eta^r(t) - \eta^r(s) \in \text{span} \left[ \bigcup_{u \in (s,t]} R(\tilde{Z}^x(u)) \right],
\]
so
\[
\int_0^s (\eta^r(t) - \eta^r(s)) \upsilon(r)\,dr \in \text{span} \left[ \bigcup_{u \in (s,t]} R(\tilde{Z}^x(u)) \right].
\]
For \( r \in (s,t) \),
\[
\eta^r(t) \in \text{span} \left[ \bigcup_{u \in (r,t]} R(\tilde{Z}^x(u)) \right] \subset \text{span} \left[ \bigcup_{u \in (s,t]} R(\tilde{Z}^x(u)) \right],
\]
so
\[
\int_s^t \eta^r(t) \upsilon(r)\,dr \in \text{span} \left[ \bigcup_{u \in (s,t]} R(\tilde{Z}^x(u)) \right]
\]
as well. Thus,
\[
\eta_{\tilde{\Phi}}(t) - \eta_{\tilde{\Phi}}(s) \in \text{span} \left[ \bigcup_{u \in (s,t]} R(\tilde{Z}^x(u)) \right].
\]

Therefore, \( \tilde{\Phi} \) is a derivative process along \( \tilde{Z}^x \) in the direction \( \Upsilon \). 
By pathwise uniqueness, we conclude the proof of the proposition. 
\end{proof}

\section{Exact Simulation of the Triple \texorpdfstring{$\eqref{eq:triple}$}{(88)}}\label{sec:exact-simulation}
Recall that the distribution of the random time point $S \in (t, T)$ is given by:
\begin{equation}
\pr(S < s) = \left\{
\begin{array}{ll}
\frac{1-2\Phi\left(-\sqrt{2\beta(s-t)}\right)}{1-2\Phi\left(-\sqrt{2\beta(T-t)}\right)}, & \beta > 0; \\
\frac{\sqrt{s-t}}{\sqrt{T-t}}, & \beta = 0,
\end{array}\right.
\end{equation}
which can be efficiently simulated using standard methods such as inverse transform. 
When $\beta >0$, the distribution of $S$ is also the same as that of a quadratic function of a truncated normal random variable, 
for which efficient simulation packages exist (in e.g., Julia).
Thus, we focus our attention on exact simulation of the triple $\left(\tau_j, \tilde Z^{t,x}_j(s), B_j(\tau_j\wedge s)\right)$ in expression \eqref{eq:triple}, 
where we recall that $\tilde Z_j^{t,x}$ is the reflected Brownian motion defined in Eq. \eqref{eq:rbm-reference}, and
$\tau_j:= \inf\{r \ge t : \tilde Z_j^{t,x}(r) = 0\}$ is the first hitting time to zero of $\tilde Z_j^{t,x}$ since time $t$. 
To simplify notation, we drop the subscript $j$, and suppose that time starts at $t=0$. 
Thus, we focus on simulating the triple $\left(\tau, \tilde Z^x(s), B(\tau\wedge s)\right)$, 
where $B(\cdot)$ is a one-dimensional standard Brownian motion, $\tilde Z^x$ is the associated reflected Brownian motion with initial state $x \in \R_+$, 
drift $\gamma$ and variance $\sigma^2$, i.e., $\tilde Z^x = \Gamma(\tilde X^x)$, where 
\[
\tilde X^x(t) = x + \gamma t + \sigma B(t)
\]
is the free part of $\tilde Z^x$, and $\Gamma$ the one-dimensional Skorokhod map, 
and $\tau = \inf\{r \ge 0: \tilde Z^x(r) = 0\}$ is the first hitting time of $\tilde Z^x$ (equivalently, $\tilde X^x$) to zero. 
We restrict the drift $\gamma$ to be non-positive, to ensure that $\tilde Z^x$ hits $0$ with probability $1$. 
Our strategy is to first simulate $\tau$, and then simulate 
$\tilde Z^x(t)$ and $B(t\wedge \tau)$, conditioned on the value of $\tau$. 

First, by Section 13.2 of \cite{Steele2001}, $\tau$ has an explicit density function $f_\tau$ given by
\begin{equation}
f_\tau (t) = \frac{x}{s^{3/2}} \phi\left(\frac{x+\gamma t}{\sqrt{t}}\right),
\end{equation}
where $\phi$ is the density function of a standard normal distribution, i.e., 
\begin{equation}
\phi(y) = \frac{1}{\sqrt{2\pi}} e^{-y^2/2}.
\end{equation}
By a simple rescaling, for general $\sigma > 0$, the density function $f_\tau$ 
of $\tau$ is given by 
\begin{equation}
f_\tau (t) = \frac{x}{\sigma t^{3/2}} \phi\left(\frac{x+\gamma t}{\sigma\sqrt{t}}\right).
\end{equation}
This is an inverse Gaussian distribution with parameters $-x/\gamma$ and $x^2/\sigma^2$, 
which can be efficiently simulated using standard packages in programming languages such as Julia.

Suppose that we have generated an instance of $\tau$. 
We next simulate $\tilde Z^x(s)$ and $B(s\wedge \tau)$, conditioned on $\tau$. 
There are two cases to consider, $s\geq \tau$ and $s < \tau$.

In the first case $s \geq \tau$, we have
\[
B(s\wedge \tau) = B(\tau) = \sigma^{-1}\left(\tilde Z^x (\tau) - x - \gamma \tau\right) = -(x+\gamma \tau)/\sigma, 
\]
because $\tilde Z^x(\tau) = 0$. To simulate $\tilde Z^x(s)$, note that by the strong Markov property of $\tilde Z^x$, 
conditioned on $\tilde Z^x(\tau) = 0$, $\tilde Z^x(s)$ has the same distribution as $\tilde{Z}^0(s-\tau)$, 
where $\tilde{Z}^0$ is a reflected Brownian motion with initial state $0$, drift $\gamma$, 
and variance $\sigma^2$. 
To simulate marginals of $\tilde Z^0$, we follow the exact simulation algorithm described in \cite{AsmussenGlynnPitman1995}
(also noted in \cite{Lepingle1984, AsmussenGlynn2007}). 
For completeness, we provide a brief derivation and description of the algorithm here. We have
\[
\tilde Z^0(t) = W(t) + \sup_{0\leq r \leq t} [W(r)]^-
= W(t) + \left[-\inf_{0\le r\le t} W(r)\right]^+.
\]
where $W$ is a Brownian motion with drift $\gamma$, variance $\sigma^2$, and zero initial state. 
In the case $\sigma = 1$, \cite{AsmussenGlynnPitman1995} provided the following explicit expression 
for the conditional distribution of $-\inf_{0\le r\le t} W(r)$, 
conditioned on $-W(t)$:
\[
\pr\left(-\inf_{0\le r\le t} W(r) - y \leq x \mid -W(t) = y \right)
= 1 - e^{-2x(y+x)/t}.
\]
By a simple rescaling, we obtain that for general $\sigma > 0$, 
\[
\pr\left(-\inf_{0\le r\le t} W(r) - y \leq x \mid -W(t) = y \right)
= 1 - \exp\left(-\frac{2x(y+x)}{\sigma^2 t}\right).
\]
This suggests an inverse transform method for simulating $-\inf_{0\le r\le t} W(r)$, conditioned on the value of $-W(t)$. 
From $\left(-W(t), -\inf_{0\le r\le t} W(r)\right)$ we can obtain $\tilde Z^0(t)$ using the formula above.
We now describe the algorithm.
\begin{itemize}
\item[] Step 1. Generate a normal random variable $\tilde W$ with mean $-\gamma t$ 
and variance $\sigma^2 t$. This has the same distribution as $-W(t)$.
\item[] Step 2. Generate $U$ that is uniform on $(0, 1)$, and let 
\[
\tilde M = \frac{\tilde W}{2} + \frac{\sqrt{\tilde W^2 - 2\sigma^2 t\log U}}{2}.
\]
Then, $\left(\tilde M, \tilde W\right)$ has the same distribution as $\left(-\inf_{0\le r\le t} W(r), -W(t)\right)$. 
\item[] Step 3. Output $\tilde Z^0(t) = - \tilde W + \tilde M^+$.
\end{itemize}
Next, we consider the second case, where $s < \tau$. 
Because $\tilde Z^x$ remains positive before time $\tau$ and $s < \tau$, 
$\tilde Z^x(s) = \tilde X^x(s)$. Furthermore, 
$B(s\wedge \tau) = B(s) = \sigma^{-1}\left(\tilde X^x(s) - x - \gamma s\right)$. 
Thus, it suffices to simulate the marginal $\tilde X^x(s)$, given $\tau$ and the fact that $s < \tau$. 
To this end, note that between time $0$ and $\tau$, $\tilde X^x$ is a Brownian meander that starts at $x > 0$ at time $0$, 
remains positive between $0$ and $\tau$, and hits $0$ at time $\tau$. 
It is important to note that Brownian meanders are different from the better known Brownian bridges, 
whose paths between end points are not constrained. 
Chapter 1 in \cite{DevroyeKarasozenKohlerKorn2010} provides the following exact method for simulating $\tilde Z^x(s)$, given $\tau$ and $\tau>s$, with appropriate scaling: 
\begin{itemize}
\item[] Step 1. Generate a standard normal random variable $\tilde N$ and an independent unit-rate exponential random variable $\tilde E$.
\item[] Step 2. Output 
\[
\sqrt{\left(\frac{x(\tau-s)}{\tau} + \sigma \sqrt{s(\tau-s)/\tau}\tilde N\right)^2 + \frac{2\sigma^2 s(\tau-s)}{\tau} \cdot \tilde E}.
\]
\end{itemize}

\section{Proofs for Section \ref{sec:fixed-point}}\label{sec:proofs-fixed-point}

\subsection{Proof of Lemma \ref{lem:norm-rho-banach}}
\begin{proof} Let $\rho \in \R$, and let \( C_\rho = \min\{1, e^{\rho T}\} \) and \( C^\rho = \max\{1, e^{\rho T}\} \).  
Then, it is not difficult to see that  
\begin{equation}\label{eq:equiv-norm}
C_\rho \|v\|_0 \leq \|v\|_\rho \leq C^\rho \|v\|_0.
\end{equation}
Because $\|v\|_0 < \infty$ for all $v \in \cB$, so $\|v\|_\rho < \infty$ as well. 
It is also not difficult to check that \( \|\cdot\|_\rho \) is a well-defined norm. 
By Ineq. \eqref{eq:equiv-norm}, all norms \( \|\cdot\|_\rho \) are equivalent to \( \|\cdot\|_0 \), 
so it suffices to show that \( (\mathcal{B}, \|\cdot\|_0) \) is complete.

Take a Cauchy sequence \( \{v^{(n)}\}_{n=1}^\infty \) in \( (\mathcal{B}, \|\cdot\|_0) \).  
Then  
\[
\lim_{N \to \infty} \sup_{m,n \geq N} \|v^{(n)} - v^{(m)}\|_0 = 0.
\]

For a fixed \( (t,x) \in [0,T) \times \mathbb{R}^d_+ \),  
because \( \{ v^{(n)}(t,x) \}_{n=1}^\infty \) is a Cauchy sequence in \( \mathbb{R}^d \), \( v^{(n)}(t,x) \to v(t,x) \) for some  
\( v(t,x) \in \mathbb{R}^d \), as \( n \to \infty \).
As the pointwise limit of measurable functions $v^{(n)}$, $v$ is a measurable function as well.

Furthermore, for all \( (t,x) \in [0,T) \times \mathbb{R}^d_+ \),  
for all \( \varepsilon > 0 \), there exists \( N \) such that if \( n \geq N \), then  
\begin{align*}
\frac{(T - t)^{1/2} \|v(t,x)\|_\infty}{w(x)} \leq &~\frac{(T - t)^{1/2} \|v^{(n)}(t,x)\|_\infty}{w(x)} + \varepsilon \\
\leq &~\sup_{n \in \mathbb{N}} \|v^{(n)}\|_0 + \varepsilon \\
\leq &~\sup_{n \le N} \|v^{(n)}\|_0 + \sup_{m \geq N} \|v^{(m)} - v^{(N)}\|_0 + \veps
< \infty
\end{align*}

Thus \( \|v\|_0 < \infty \).

Finally, 
\begin{align*}
\|v - v^{(n)}\|_0
& = \sup_{(t,x) \in [0,T) \times \mathbb{R}^d_+} \frac{(T - t)^{1/2}\|v(t,x) - v^{(n)}(t,x)\|_\infty}{w(x)} \\
& = \sup_{(t,x) \in [0,T) \times \mathbb{R}^d_+} \frac{(T - t)^{1/2} \|\lim_{m \to \infty} v^{(m)}(t,x) - v^{(n)}(t,x)\|_\infty}{w(x)} \\
& \leq \sup_{(t,x) \in [0,T) \times \mathbb{R}^d_+} \sup_{m \geq n} \frac{(T - t)^{1/2}\|v^{(m)}(t,x) - v^{(n)}(t,x)\|_\infty}{w(x)} \\
& = \sup_{m \geq n} \|v^{(m)} - v^{(n)}\|_0 \to 0 \quad \text{as } n \to \infty,
\end{align*}
where the convergence follows by the fact that the sequence \( \{v^{(n)}\} \) is Cauchy. 
This completes the proof of the lemma.
\end{proof}

\subsection{Proof of Lemma \ref{lem:Fv-bounded}}
\begin{proof} 
Let \( \mathcal{S} = \{ (t,s) : t \in [0,T), \, s \in (t,T) \} \).  
Define the function \( \zeta^{(1)} : \mathcal{S} \times \mathbb{R}^d_+ \to \mathbb{R}^d \) by  
\begin{equation}\label{eq:zeta1}
\zeta^{(1)}(t,s,x) = \mathbb{E} \Bigg[ \tilde \cH \left(\tilde{Z}^{t,x}(s), v(s, \tilde{Z}^{t,x}(s))\right)
\left( \frac{1}{s - t} \int_t^s \left[\sigma^{-1} D\tilde Z^{t,x}(r)\right]^{\top} dB(r) \right) \Bigg],
\end{equation}
and the function \( \zeta^{(2)} : [0,T) \times \mathbb{R}^d_+ \to \mathbb{R}^d \) by  
\begin{equation}\label{eq:zeta2}
\zeta^{(2)}(t,x) = \mathbb{E} \left[ \xi(\tilde{Z}^{t,x}(T)) 
\left( \frac{1}{T - t} \int_t^T \left[\sigma^{-1} D\tilde Z^{t,x}(r)\right]^{\top} dB(r) \right) \right].
\end{equation}
The term $\E\left[\int_t^T \kappa^{\top} d\left[D\tilde Y^{t,x}(s)\right]\right] = \kappa^{\top} D\tilde Y^{t,x}(T)$ is uniformly bounded over $(t,x) \in [0,T)\times \R_+^d$ for any fixed $T >0$, 
so it suffices to show that \( \|\zeta^{(2)}\|_0 < \infty \), and
\[
\sup_{(t,x) \in [0,T) \times \mathbb{R}_+^d} \frac{(T - t)^{1/2}}{w(x)} \int_t^T \|\zeta^{(1)}(t,s,x)\|_\infty ds < \infty.
\] 

We first show that \( \|\zeta^{(2)}\|_0 < \infty \).  
For \( (t,x) \in [0,T) \times \mathbb{R}^d_+ \), $j = 1, 2, \cdots, d$, 
\begin{align*}
\frac{(T - t)^{1/2} |\zeta_j^{(2)}(t,x)|}{w(x)} 
&= w(x)^{-1} \left| \mathbb{E} \left[ \xi(\tilde{Z}^{t,x}(T)) \left( \frac{1}{\sqrt{T - t}} \int_t^T \langle \left[\sigma^{-1} D\tilde Z^{t,x}(r)\right]_j, dB(r) \rangle \right) \right] \right|\\
&\leq \frac{ \mathbb{E} \left[ \xi^2(\tilde{Z}^{t,x}(T)) \right]^{1/2} }{w(x)} 
\cdot \frac{1}{\sqrt{T-t}}\mathbb{E} \left[\left| \int_t^T \langle \left[\sigma^{-1} D\tilde Z^{t,x}(r)\right]_j, dB(r) \rangle \right|^2\right]^{1/2} \\
& = \frac{ \mathbb{E} \left[ \xi^2(\tilde{Z}^{t,x}(T)) \right]^{1/2} }{w(x)} 
\cdot \frac{1}{\sqrt{T - t}} \E \left[ \int_t^T \left[\sigma^{-1} D\tilde Z^{t,x}(r)\right]_j^2dr \right]^{1/2}.
\end{align*}
Here, the first inequality follows from Cauchy-Schwarz, and the last equality follows from Itô isometry.

By Proposition \ref{prop:moment-Z},  
\[
\mathbb{E} \left[ \xi^2(\tilde{Z}^{t,x}(T)) \right] \leq C (1 + \|x\|^{2\alpha})
\]  
for some constant \( C = C(T,\alpha,\sigma, C_{\tb}, C_\Gamma) \), so  
\[
\mathbb{E} \left[ \xi^2(\tilde{Z}^{t,x}(T)) \right]^{1/2} \leq C'(1 + \|x\|^\alpha)
\]
for some constant $C'$. But \( w(x) = 1 + \|x\|^\alpha \), so  
\[
\frac{ \mathbb{E} \left[ \xi^2(\tilde{Z}^{t,x}(T)) \right]^{1/2} }{w(x)} \leq C'.
\]
By Lemma \ref{lem:bound-DZ}, we also have \( \|D \tilde{Z}^{t,x}\|_T \leq \tilde C \)  
for some constant \( \tilde C\), so  
\[
\frac{1}{(T - t)^{1/2}} \mathbb{E} \left[ \left( \int_t^T \left[ \sigma^{-1} D \tilde{Z}^{t,x}(r) \right]_j^2 dr \right)^{1/2} \right] 
\leq \frac{\tilde C' (T - t)^{1/2}}{(T - t)^{1/2}} = \tilde C'.
\]
Thus there is a finite constant that upper bounds  
$\frac{(T - t)^{1/2} |\zeta_j^{(2)}(t,x)|}{w(x)}$ uniformly over $(t,x) \in [0,T) \times \mathbb{R}_+^d$, $j = 1, 2, \cdots, d$.
Therefore, $\| \zeta^{(2)} \|_0 < \infty$.

Next, consider \( \zeta^{(1)} \).  
Recall that  
\[
\tilde \cH(x,p) = \inf_{a \in \mathcal{A}} \left\{ (b(x,a) - \tb(x))^{\top} p + C(x,a) \right\}.
\]
Since \( b \) is uniformly bounded over all  
\( (x,a) \in \mathbb{R}_+^d \times \mathcal{A} \), \( \tb \) is uniformly bounded  
over \( x \in \mathbb{R}_+^d \), \( |c(x,a)| \leq \alpha_2 (1 + \|x\|^\alpha) \) uniformly  
over \( a \in \mathcal{A} \), we have  
\begin{equation}\label{eq:bound-H}
\tilde \cH^2 (x,p) \leq \tilde \alpha_1 \|p\|^2 + \tilde \alpha_2 (1 + \|x\|^{2\alpha}), \quad (x,p) \in \mathbb{R}_+^d \times \mathbb{R}^d,
\end{equation}
for some constants $\tilde \alpha_1, \tilde \alpha_2$ and $\alpha$. Thus,  
\begin{align*}
& \quad \frac{(T - t)^{1/2}}{w(x)} \int_t^T |\zeta_j^{(1)}(t,s,x)| ds \\
&= \frac{(T - t)^{1/2}}{w(x)} \int_t^T \left|\mathbb{E} \left[ \tilde \cH (\tilde{Z}^{t,x}(s), v(s, \tilde{Z}^{t,x}(s))) \cdot 
\frac{1}{s - t} \int_t^s \langle \left[\sigma^{-1} D\tilde Z^{t,x}(r)\right]_j, dB(r) \rangle \right] \right| ds\\
&\leq \frac{(T - t)^{1/2}}{w(x)} \int_t^T 
\left( \mathbb{E} \left[ \tilde \cH^2(\tilde{Z}^{t,x}(s), v(s, \tilde{Z}^{t,x}(s))) \right] \right)^{1/2} 
\left( \mathbb{E} \left[ \left| \frac{1}{s - t} \int_t^s \langle \left[\sigma^{-1} D\tilde Z^{t,x}(r)\right]_j, dB(r) \rangle \right|^2 \right] \right)^{1/2} ds\\
&= \frac{(T - t)^{1/2}}{w(x)} \int_t^T 
\left( \mathbb{E} \left[ \tilde \cH^2(\tilde{Z}^{t,x}(s), v(s, \tilde{Z}^{t,x}(s))) \right] \right)^{1/2} 
\left( \frac{1}{s - t} \E \left[\int_t^s \left[\sigma^{-1} D\tilde Z^{t,x}(r)\right]_j^2 dr \right]^{1/2}\right) ds \\
&\leq \frac{\tilde C'(T - t)^{1/2}}{w(x)} \int_t^T 
\left( \tilde \alpha_1 \mathbb{E} \left[\|v(s, \tilde{Z}^{t,x}(s))\|^2 \right] + \tilde \alpha_2 \E\left[1 + \|\tilde{Z}^{t,x}(s)\|^{2\alpha} \right]\right)^{1/2} (s - t)^{-1/2} ds \\
&\leq \frac{\tilde C''(T - t)^{1/2}}{w(x)} \int_t^T 
\left((T-s)^{-1/2} \|v\|_0 \mathbb{E} \left[ w^2(\tilde{Z}^{t,x}(s)) \right]^{1/2} + \mathbb{E} \left[ w^2(\tilde{Z}^{t,x}(s)) \right]^{1/2} \right) \cdot (s - t)^{-1/2} ds \\
&\leq \tilde C''' (T - t)^{1/2} \left[\|v\|_0 \int_t^T (s - t)^{-1/2} (T - s)^{-1/2} ds + \int_t^T (s - t)^{-1/2} ds \right] < \infty.
\end{align*}
Here, the first inequality again follows from Cauchy-Schwarz, the second equality follows from Itô isometry, 
the second inequality follows from the bound \eqref{eq:bound-H} on the Hamiltonian $\tilde \cH$, 
the third inequality follows from the definition of $\|v\|_0$, 
the fourth inequality follows by collecting constants and Proposition \ref{prop:moment-Z}, 
from which we deduce that $\E\left[w^2(\tilde Z^x(s))\right]^{1/2} < \alpha_3 w(x)$ for some constant $\alpha_3$, 
and the final inequality follows from the fact that the integrals $\int_t^T (s - t)^{-1/2} (T - s)^{-1/2} ds$ 
and $\int_t^T (s - t)^{-1/2} ds$ are both finite. 

Since our upper bound is uniform over $(t,x) \in [0, T) \times \R_+^d$ and $j = 1, 2, \cdots, d$, 
\[
\sup_{(t,x) \in [0,T) \times \mathbb{R}_+^d} \frac{(T - t)^{1/2}}{w(x)} \int_t^T \|\zeta^{(1)}(t,s,x)\|_\infty ds < \infty.
\]
Thus, we can conclude that for \( v \in \mathcal{B} \),  
$\|F(v)\|_0 < \infty$. This concludes the proof of the lemma. 
\end{proof}

\subsection{Proof of Proposition \ref{prop:F-contraction}}
\begin{proof}
Here we only treat the case of zero discount rate, i.e., \( \beta = 0 \). The more general case follows analogously.
Let \( v, \tilde{v} \in (\mathcal{B}, \|\cdot\|_\rho) \). Let $j \in \{1, 2, \cdots, d\}$.
Then, for any \( (t,x) \in [0,T) \times \mathbb{R}_+^d \),
\begin{align*}
\left|\left[F(v)\right]_j(t,x) - \left[F(\tilde{v})\right]_j(t,x)\right|
&= \mathbb{E} \left[ \int_t^T \left( \tilde \cH(\tilde{Z}^{t,x}(s), v(s, \tilde{Z}^{t,x}(s))) 
- \tilde \cH(\tilde{Z}^{t,x}(s), \tilde{v}(s, \tilde{Z}^{t,x}(s))) \right) \right. \\
& \qquad \quad \left. \cdot \left( \frac{1}{s - t} \int_t^s \left[\sigma^{-1} D_j \tilde{Z}^{t,x}(r)\right]^{\top} dB(r) \right) ds \right] \\
&\leq \int_t^T \mathbb{E} \left[ \left| \tilde \cH(\tilde{Z}^{t,x}(s), v(s, \tilde{Z}^{t,x}(s))) 
- \tilde \cH(\tilde{Z}^{t,x}(s), \tilde{v}(s, \tilde{Z}^{t,x}(s))) \right|^2 \right]^{1/2} \\
&\qquad \quad \cdot \left( \frac{1}{s - t} \mathbb{E} \left[ \left| \int_t^s \left[\sigma^{-1} D_j \tilde{Z}^{t,x}(r)\right]^{\top} dB(r) \right|^2 \right] \right)^{1/2} ds \\
&\leq C \int_t^T (s - t)^{-1/2} \mathbb{E} \left[ \|v(s, \tilde{Z}^{t,x}(s)) - \tilde{v}(s, \tilde{Z}^{t,x}(s))\|^2 \right]^{1/2} ds \\
&\leq C \int_t^T (s - t)^{-1/2} (T - s)^{-1/2} \mathbb{E} \left[ w^2(\tilde{Z}^{t,x}(s)) \right]^{1/2} 
e^{-\rho s} \| v - \tilde{v} \|_\rho ds \\
&\leq \tilde C w(x) \left( \int_t^T (s - t)^{-1/2} (T - s)^{-1/2} e^{-\rho s} ds \right) 
\| v - \tilde{v} \|_\rho,
\end{align*}
for some positive constants $C$ and $\tilde C$. 
Here, the first inequality follows from Cauchy-Schwarz, 
and the second inequality follows from the (uniform in $x$) Lipschitz continuity of $\tilde \cH(x, p)$ in $p$.
Thus, 
\begin{align*}
& \quad w(x)^{-1}e^{\rho t}  \left|\left[F(v)\right]_j(t,x) - \left[F(\tilde{v})\right]_j(t,x)\right| (T - t)^{1/2} \\
&\leq \tilde C \left( \int_t^T e^{-\rho(s - t)} (s - t)^{-1/2} (T - s)^{-1/2} (T - t)^{1/2} ds \right) 
\| v - \tilde{v} \|_\rho.
\end{align*}
The integral
\begin{align*}
&\quad \int_t^T e^{-\rho(s - t)} (s - t)^{-1/2} (T - s)^{-1/2} (T - t)^{1/2} ds \\
&=  \int_t^T e^{-\rho(s - t)} (s - t)^{-1/2} (T - s)^{-1/2} (T - s + s - t)^{1/2} ds \\
&\leq \int_t^T e^{-\rho(s - t)} \left( (s - t)^{-1/2} + (T - s)^{-1/2} \right) ds \\
&= \int_t^T (s - t)^{-1/2} e^{-\rho(s - t)} ds + \int_t^T (T - s)^{-1/2} e^{-\rho(s - t)} ds.
\end{align*}
By Hölder’s inequality, if \( \rho \geq 0 \), then
\begin{align*}
\int_t^T (s - t)^{-1/2} e^{-\rho(s - t)} ds 
&\leq \left( \int_t^T (s - t)^{-3/4} ds \right)^{2/3} 
\left( \int_t^T e^{-3\rho(s - t)} ds \right)^{1/3} \\
&= \left( 4(T - t)^{1/4} \right)^{2/3} 
\left( -\frac{1}{3\rho} e^{-3\rho(s - t)} \Big|_t^{\top} \right)^{1/3} \\
&\leq 4^{2/3} (T - t)^{1/6} \left( \frac{1}{3\rho} \right)^{1/3}.
\end{align*}
Similarly, 
\[
\int_t^T (T - s)^{-1/2} e^{-\rho(s - t)} ds \leq 4^{2/3} (T - t)^{1/6} \left( \frac{1}{3\rho} \right)^{1/3}.
\]
Thus, whenever \( \rho >  4^5 \tilde C^3 T^{1/2}/3\),  
\begin{align*}
\frac{e^{\rho t} \left|\left[F(v)\right]_j(t,x) - \left[F(\tilde{v})\right]_j(t,x)\right|(T - t)^{1/2}}{w(x)} 
&\leq \frac{1}{2} \| v - \tilde{v} \|_\rho
\end{align*}
for all \( (t,x) \in [0,T) \times \mathbb{R}_+^d \), $j = 1,2, \cdots, J$. This implies that
\[
\|F(v) - F(\tilde{v})\|_\rho < \frac{1}{2} \| v - \tilde{v} \|_\rho,
\]
for sufficiently large \( \rho \).  
This concludes the proof of the proposition. 
\end{proof}

\subsection{Proof of Lemma \ref{lem:Fv-continuous}}
\begin{proof} Let \( (t,x) \in [0,T) \times \mathbb{R}_{++}^d \), and let \( \{(t^{(n)}, x^{(n)})\}_n \)  
be a sequence in \( [0,T) \times \mathbb{R}_{++}^d \) such that  
\[
(t^{(n)}, x^{(n)}) \to (t,x) \quad \text{as } n \to \infty.
\]
We wish to show that 
\[
F(v)(t^{(n)}, x^{(n)}) \to F(v)(t,x) \quad \text{as } n \to \infty.
\]
We first show that if $(t^{(n)}, x^{(n)}) \to (t,x)$ in such a way that $t^{(n)} \downarrow t$, as $n \to \infty$, then 
$F(v)(t^{(n)}, x^{(n)}) \to F(v)(t,x)$. The case when $t^{(n)} \uparrow t$ follows analogously. 

Note that 
$\{ (\tilde{Z}^{t,x}(s), D \tilde{Z}^{t,x}(s), \tilde{Y}^{t,x}(s), D \tilde{Y}^{t,x}(s)) : s \in [t, T] \}$
has the same law as $\{ (\tilde{Z}^{x}(s - t), D \tilde{Z}^x(s - t), \tilde{Y}^{x}(s - t), D \tilde{Y}^x(s - t)) : s \in [t, T] \}$.
Thus, we define \( (\tilde{Z}^x, D \tilde{Z}^x, \tilde{Y}^x, D \tilde{Y}^x) \) and \( \{ (\tilde{Z}^{x^{(n)}}, D \tilde{Z}^{x^{(n)}}, \tilde{Y}^{x^{(n)}}, D \tilde{Y}^{x^{(n)}}) \} \)  
on the same probability space such that  
they are driven by the same Brownian motion. 
Consider the terms $\E\left[\int_t^T \kappa^{\top} d\left[D\tilde Y^{t,x}(s)\right]\right] = \E\left[\kappa^{\top} D\tilde Y^{x}(T-t)\right]$ 
and $\E\left[\int_t^T \kappa^{\top} d\left[D\tilde Y^{t^{(n)},x^{(n)}}(s)\right]\right] = \E\left[\kappa^{\top} D\tilde Y^{x^{(n)}}(T-t^{(n)})\right]$. 
For notational simplicity, write 
\begin{equation}\label{eq:rn-and-r}
r^{(n)} = T - t^{(n)} \ \ \text{and} \ \ r = T - t.
\end{equation}
We now proceed to show that as $n \to \infty$, 
\[
\E\left[\kappa^{\top} D\tilde Y^{x^{(n)}}(r^{(n)})\right] \to \E\left[\kappa^{\top} D\tilde Y^{x}(r)\right].
\]
By Proposition \ref{thm:continuity-DZ}, for each $j = 1, 2, \cdots, d$, $D_j\tilde Y^{x^{(n)}} \to D_j\tilde Y^{x}$ in the Skorokhod topology. 
This implies that there exist strictly increasing bijections $\lambda^{(n)} : [0,T] \to [0,T]$ such that as $n \to \infty$, 
\[
\left\|D_j\tilde Y^{x^{(n)}}  - D_j\tilde Y^{x} \circ \lambda^{(n)} \right\|_T \to 0, \text{ and } \|\lambda^{(n)} - \iota\|_T \to 0,
\]
where $\iota : [0,T] \to [0, T]$ is the identity map. We have
\begin{align*}
&~\left\|D_j\tilde Y^{x^{(n)}}(r^{(n)}) - D_j\tilde Y^{x}(r)\right\|_\infty \\
\le &~\left\|D_j\tilde Y^{x^{(n)}}(r^{(n)}) - D_j\tilde Y^{x}(\lambda^{(n)}(r^{(n)}))\right\|_\infty+\left\|D_j\tilde Y^{x}(\lambda^{(n)}(r^{(n)})) - D_j\tilde Y^{x}(r)\right\|_\infty \\
\le &~\left\|D_j\tilde Y^{x^{(n)}}  - D_j\tilde Y^{x} \circ \lambda^{(n)} \right\|_T + \left\|D_j\tilde Y^{x}(\lambda^{(n)}(r^{(n)})) - D_j\tilde Y^{x}(r)\right\|_\infty.
\end{align*}
By Skorokhod convergence, the first term $\left\|D_j\tilde Y^{x^{(n)}}  - D_j\tilde Y^{x} \circ \lambda^{(n)} \right\|_T \to 0$ as $n \to \infty$. 
For the second term, note that 
\begin{align*}
&~\left| \lambda^{(n)} (r^{(n)}) - r\right| 
\le  \left| \lambda^{(n)} (r^{(n)}) - (r^{(n)}) \right| + |r^{(n)} - r| 
\le \|\lambda^{(n)} - \iota\|_T + |t^{(n)} - t| \to 0
\end{align*}
as $n \to \infty$, and by Proposition \ref{prop:LR2019b-main} part (ii) and Lemma \ref{lem:Z>0}, a.s. $r$ is a continuity point of $D_j\tilde Y^{x}$, 
we have that a.s. $\left\|D_j\tilde Y^{x}(\lambda^{(n)}(r^{(n)})) - D_j\tilde Y^{x}(r)\right\|_\infty \to 0$ as $n \to \infty$. 
Thus, we have shown that a.s. for each $j$, $D_j\tilde Y^{x^{(n)}}(r^{(n)}) \to D_j\tilde Y^{x}(r)$ as $n \to \infty$. 
By uniform boundedness of $D_j\tilde Y^{x^{(n)}}$, we have that
\[
\E\left[D_j\tilde Y^{x^{(n)}}(r^{(n)})\right] \to \E\left[D_j\tilde Y^{x}(r\right], \ \text{ so } \
\E\left[\kappa^{\top} D\tilde Y^{t^{(n)}, x^{(n)}}(T)\right] \to \E\left[\kappa^{\top} D\tilde Y^{t,x}(T)\right].
\]
Next, let \( \zeta^{(2)}(t,x) \) be as in the proof of Lemma \ref{lem:Fv-bounded}, defined in Eq. \eqref{eq:zeta2}. 
We now proceed to show that \( \zeta^{(2)}(t^{(n)}, x^{(n)}) \to \zeta^{(2)}(t,x) \)  
as \( n \to \infty \).
To this end, first note that $\{ (\tilde{Z}^{t,x}(s), D \tilde{Z}^{t,x}(s)) : s \in [t, T] \}$ 
has the same law as $\{ (\tilde{Z}^{x}(s - t), D \tilde{Z}^x(s - t)) : s \in [t, T] \}$.

Thus, for each $j = 1, 2, \cdots, d$, 
\[
\zeta^{(2)}_j(t,x) = \mathbb{E} \left[ \xi(\tilde{Z}^x(T - t)) 
\left( \frac{1}{T - t} \int_0^{T - t} \left[\sigma^{-1} D_j \tilde{Z}^x(u)\right]^{\top} dB(u) \right) \right].
\]
We claim the following. 

\textbf{Claim 1.}  
\[
\mathbb{E} \left[ \left( \xi(\tilde{Z}^{x^{(n)}}(r^{(n)})) 
- \xi(\tilde{Z}^x(r)) \right)^2 \right] \to 0 \quad \text{as } n \to \infty.
\]
\textbf{Claim 2.}  
\[
\mathbb{E} \left[ \left( \frac{1}{r^{(n)}}\int_0^{r^{(n)}} \left[\sigma^{-1} D_j \tilde{Z}^{x^{(n)}}(u)\right]^{\top} dB(u) 
- \frac{1}{r}\int_0^r \left[\sigma^{-1} D_j \tilde{Z}^x(u)\right]^{\top} dB(u) \right)^2 \right] \to 0 
\quad \text{as } n \to \infty.
\]
By Claims 1 and 2, it is immediate that for each $j$, 
\[
\zeta^{(2)}_j(t^{(n)}, x^{(n)}) \to \zeta^{(2)}_j(t,x) \quad \text{as } n \to \infty.
\]
\textbf{Proof of Claim 1}. We have
\begin{align*}
\mathbb{E} \left[ \left( \xi(\tilde{Z}^{x^{(n)}}(r^{(n)})) - \xi(\tilde{Z}^x(r)) \right)^2 \right]
& \leq 2 \Bigg\{\mathbb{E} \left[ \left( \xi(\tilde{Z}^{x^{(n)}}(r^{(n)})) - \xi(\tilde{Z}^x(r^{(n)})) \right)^2 \right] \\
& \quad + \mathbb{E} \left[ \left( \xi(\tilde{Z}^x(r^{(n)})) - \xi(\tilde{Z}^x(r)) \right)^2 \right] \Bigg\}.
\end{align*}
By Proposition \ref{prop:moment-Z}, a.s.  
\[
\left\|\tilde{Z}^{x^{(n)}}(r^{(n)}) - \tilde{Z}^x(r^{(n)})\right\| 
\leq \| \tilde{Z}^{x^{(n)}} - \tilde{Z}^x \|_T \leq C \|x^{(n)} - x\| \to 0 \text{ as } n \to \infty,
\]
so by continuity of \( \xi \), a.s.
\[
(\xi(\tilde{Z}^{x^{(n)}}(r^{(n)})) - \xi(\tilde{Z}^x(r^{(n)})))^2 \to 0
\quad \text{as } n \to \infty.
\]
Next, define \( y = (y_j)_{j=1}^J \in \mathbb{R}_+^d \) by  
setting \( y_j = \sup_n x_j^{(n)} \), \( j = 1, \cdots, d \), and define  
a reflecting diffusion \( W \) by \( W = \Gamma(X) \),  
where \( X(t) = y + C_{\tb} \bOne t + \sigma B(t) \), \( t \in \mathbb{R}_+ \).

Then, by Theorem 1.1 of \cite{KellaRamasubramanian2012},  
a.s. \( W(r^{(n)}) \geq \tilde{Z}^{x^{(n)}}(r^{(n)}) \), and  
$W(r^{(n)}) \geq \tilde{Z}^x(r^{(n)})$ for all $n$. Thus, by Assumption \ref{as:cost-poly}, 
\begin{align*}
(\xi(\tilde{Z}^{x^{(n)}}(r^{(n)})) - \xi(\tilde{Z}^x(r^{(n)})))^2 
& \leq 2 \left( \xi^2(\tilde{Z}^{x^{(n)}}(r^{(n)})) + \xi^2(\tilde{Z}^x(r^{(n)})) \right) \\
& \leq 2\alpha_w \left[ w\left(\tilde{Z}^{x^{(n)}}(r^{(n)})\right) + w\left(\tilde{Z}^x(r^{(n)})\right)\right] \\
& \leq 4\alpha_w w\left(W(r^{(n)})\right). \leq 4\alpha_w w\left(\|W\|_T\right)
\end{align*}
By dominated convergence theorem, and the fact that $w\left(\|W\|_T\right)$ is integrable, 
\[
\mathbb{E} \left[ (\xi(\tilde{Z}^{x^{(n)}}(r^{(n)})) - \xi(\tilde{Z}^x(r^{(n)})))^2 \right] \to 0 
\quad \text{as } n \to \infty.
\]
Next, because \( r^{(n)} \to r \), by continuity of the sample paths of \( \tilde{Z}^x \)  
and \( \xi \),  
\[
(\xi(\tilde{Z}^x(r^{(n)})) - \xi(\tilde{Z}^x(r)))^2 \to 0 \quad \text{a.s. as } n \to \infty.
\]
A similar argument shows that  
\[
\mathbb{E} \left[ \left( \xi(\tilde{Z}^x(r^{(n)})) - \xi(\tilde{Z}^x(r)) \right)^2 \right] \to 0 
\quad \text{as } n \to \infty.
\]
Thus, $\mathbb{E} \left[ \left( \xi(\tilde{Z}^{x^{(n)}}(r^{(n)})) - \xi(\tilde{Z}^x(r)) \right)^2 \right] \to 0$ as $n \to \infty$, and we have established Claim 1.

%

\textbf{Proof of Claim 2}. Because \( r^{(n)} \to r \), and  
\[
\mathbb{E} \left[ \left(\int_0^r \left[\sigma^{-1} D_j \tilde{Z}^x(u)\right]^{\top} dB(u) \right)^2 \right] < \infty,
\]
it suffices to show that  
\[
\mathbb{E} \left[ \left(\int_0^{r^{(n)}} \left[\sigma^{-1} D_j \tilde{Z}^{x^{(n)}}(u)\right]^{\top} dB(u) 
- \int_0^r \left[\sigma^{-1} D_j \tilde{Z}^x(u)\right]^{\top} dB(u) \right)^2 \right] \to 0 
\quad \text{as } n \to \infty.
\]
Let $j \in \{1, 2, \cdots, d\}$. We have  
\begin{align*}
&\quad \ \mathbb{E} \left[ \left( \int_0^{r^{(n)}} \left[\sigma^{-1} D_j \tilde{Z}^{x^{(n)}}(u)\right]^{\top} dB(u) - \int_0^r \left[\sigma^{-1} D_j \tilde{Z}^x(u)\right]^{\top} dB(u) \right)^2 \right] \\
& = \mathbb{E} \left[ \left( \int_0^r \left[\sigma^{-1} \left(D_j \tilde{Z}^{x^{(n)}}(u) - D_j \tilde{Z}^{x^{(n)}}(u)\right)\right]^{\top} dB(u) 
- \int_r^{r^{(n)}} \left[\sigma^{-1} D_j \tilde{Z}^{x^{(n)}}(u)\right]^{\top} dB(u) \right)^2 \right] \\
& \leq 2 \mathbb{E} \left[ \left| \int_0^r \left[\sigma^{-1} \left(D_j \tilde{Z}^{x^{(n)}}(u) - D_j \tilde{Z}^{x^{(n)}}(u)\right)\right]^{\top} dB(u) \right|^2 \right] + 
 2 \mathbb{E} \left[ \left| \int_r^{r^{(n)}} \left[\sigma^{-1} D_j \tilde{Z}^{x^{(n)}}(u)\right]^{\top} dB(u) \right|^2 \right]\\
& = 2 \mathbb{E} \left[ \int_0^r \left\|\left[\sigma^{-1} (D_j \tilde{Z}^x(u) - D_j \tilde{Z}^{x^{(n)}}(u))\right] \right\|_2^2 \, du \right] 
+ 2 \mathbb{E} \left[ \int_r^{r^{(n)}} \left\| \left[\sigma^{-1} D_j \tilde{Z}^{x^{(n)}}(u)\right] \right\|_2^2 \, du \right],
\end{align*}
where the last equality follows from Itô isometry. By the fact that \( \| D \tilde{Z}^{y} \|_T \) is uniformly bounded 
over over \( y \in \mathbb{R}_+^d \), $t^{(n)} \to t$, and Corollary \ref{cor:conv-bel}, 
both terms in the last expression converge to  zero as \( n \to \infty \). Thus, we have established Claim 2 as well. 

We now establish continuity of the term  
\begin{align*}
\zeta(t,x) & := \mathbb{E} \left[ \int_t^T e^{-\beta(s - t)} \tilde \cH (\tilde{Z}^{t,x}(s), v(s, \tilde{Z}^{t,x}(s))) 
\left( \frac{1}{s - t} \int_t^s \left[\sigma^{-1} D \tilde{Z}^{t,x}(r)\right]^{\top} dB(r) \right) ds \right]\\
& = \mathbb{E} \left[ \int_0^{T - t} e^{-\beta s} \tilde \cH(\tilde{Z}^x(s), v(s + t, \tilde{Z}^x(s))) 
\left( \frac{1}{s} \int_0^s \left[\sigma^{-1} D \tilde{Z}^x(r)\right]^{\top} dB(r) \right) ds \right]
\end{align*}
in \( (t,x) \in [0,T) \times \mathbb{R}_{++}^d \).

For notational simplicity, we consider the case \( \beta = 0 \).  
The more general case where \( \beta \) can be positive can be established using a similar argument.

Let $j \in \{1, 2, \cdots, d\}$. For \( (t,x), (t',x') \in [0,T) \times \mathbb{R}_{++}^d \), with \( t' > t \),  
\begin{align*}
\zeta_j(t,x) - \zeta_j(t',x') 
&= \mathbb{E} \left[ \int_0^{T - t} \tilde \cH (\tilde{Z}^x(s), v(s + t, \tilde{Z}^x(s))) 
\left( \frac{1}{s} \int_0^s \left[\sigma^{-1} D_j \tilde{Z}^x(r)\right]^{\top} dB(r) \right) ds \right] \\
&\quad - \mathbb{E} \left[ \int_0^{T - t'} \tilde \cH (\tilde{Z}^{x'}(s), v(s + t', \tilde{Z}^{x'}(s))) 
\left( \frac{1}{s} \int_0^s \left[\sigma^{-1} D_j \tilde{Z}^{x'}(r)\right]^{\top} dB(r) \right) ds \right] \\
&\leq A_1(t, t', x, x') + A_2(t, t', x, x') + A_3(t, t', x, x'),
\end{align*}
where
\begin{align}
A_1(t, t', x, x') &= \int_{T - t'}^{T - t} \mathbb{E} \left[ \left|\tilde \cH(\tilde{Z}^x(s), v(s + t, \tilde{Z}^x(s)))\right| 
\left| \frac{1}{s} \int_0^s \left[\sigma^{-1} D_j \tilde{Z}^x(r)\right]^{\top} dB(r) \right| \right] ds; \label{eq:A1}\\
A_2(t, t', x, x') &= \int_0^{T - t'} \mathbb{E} \left[ \left| \tilde\cH(\tilde{Z}^x(s), v(s + t, \tilde{Z}^x(s))) 
- \tilde \cH(\tilde{Z}^{x'}(s), v(s + t', \tilde{Z}^{x'}(s))) \right| \right. \nonumber \\
& \qquad \qquad \quad \cdot \left.\left| \frac{1}{s} \int_0^s \left[\sigma^{-1} D_j \tilde{Z}^x(r)\right]^{\top} dB(r) \right| \right] ds; \label{eq:A2}\\
A_3(t, t', x, x') &= \int_0^{T - t'} \mathbb{E} \left[ \left|\tilde \cH(\tilde{Z}^{x'}(s), v(s + t', \tilde{Z}^{x'}(s)))\right| \right. \nonumber \\ 
&~~~~~~~~~~~~~~~~
 \left. \cdot \left| \frac{1}{s} \int_0^s \left[\sigma^{-1} D_j \tilde{Z}^x(r)\right]^{\top}_j dB(r) - \frac{1}{s} \int_0^s \left[\sigma^{-1} D_j \tilde{Z}^{x'}(r)\right]^{\top} dB(r) \right| \right] ds. \label{eq:A3}
\end{align}

We now bound the terms \( A_1, A_2 \), and \( A_3 \) respectively. By Cauchy-Schwarz inequality,
\begin{align*}
A_1(t, t', x, x') 
&\leq \int_{T - t'}^{T - t} \left( \mathbb{E} \left[ \left|\tilde \cH(\tilde{Z}^x(s), v(s + t, \tilde{Z}^x(s)))\right|^2 \right] \right)^{1/2} 
 \cdot \left( \mathbb{E} \left[ \left| \frac{1}{s} \int_0^s \left[\sigma^{-1} D_j \tilde{Z}^x(r)\right]^{\top} dB(r) \right|^2 \right] \right)^{1/2} ds.
\end{align*}
Using a similar reasoning as in the proof of Lemma \ref{lem:Fv-bounded}, we can show that 
\begin{align*}
A_1(t, t', x, x') & \leq C w(x) \int_{T - t'}^{T - t} s^{-1/2} \left(1+\|v\|_0 (T-t-s)^{-1/2}\right) ds\\
& \leq C w(x) \left(\int_{T - t'}^{T - t} s^{-1/2} ds + \sqrt{2(T-t)} \|v\|_0 \int_{T - t'}^{T - t} \left(s^{-1/2} + (T-t-s)^{-1/2}\right)\right) ds \\
& \leq \tilde C w(x) \left(\sqrt{T-t} - \sqrt{T-t'}\right) +\tilde C w(x) \sqrt{2(T-t)} \|v\|_0 \left(\sqrt{T-t} - \sqrt{T-t'} + \sqrt{t'-t}\right) \\
& \to 0 \text{ as } t' \downarrow t,
\end{align*}
for some positive constants $C$ and $\tilde C$. Here, the second inequality follows from the fact that for $a, b \in \R_+$, 
$\sqrt{a} + \sqrt{b} \leq \sqrt{2(a+b)}$.
Thus, if \( (t^{(n)}, x^{(n)}) \to (t, x) \),  
with \( t^{(n)} \downarrow t \), then  
\[
A_1(t, t^{(n)}, x, x^{(n)}) \to 0.
\]

Next,
\begin{align}
A_2(t, t', x, x') &= \int_0^{T - t'} 
\mathbb{E} \left[ \left| \tilde \cH(\tilde{Z}^x(s), v(s + t, \tilde{Z}^x(s))) 
- \tilde \cH(\tilde{Z}^{x'}(s), v(s + t', \tilde{Z}^{x'}(s))) \right| \right. \nonumber \\
&\qquad \left. \cdot \left| \frac{1}{s} \int_0^s \left[\sigma^{-1} D_j \tilde{Z}^x(r)\right]^{\top} dB(r) \right| \right] ds \nonumber \\
&= \int_0^{T - t} \mbOne_{\{s < T - t'\}} 
\mathbb{E} \left[ \left| \tilde \cH(\tilde{Z}^x(s), v(s + t, \tilde{Z}^x(s))) 
- \tilde \cH(\tilde{Z}^{x'}(s), v(s + t', \tilde{Z}^{x'}(s))) \right| \right. \nonumber \\
&\qquad \left. \cdot \left| \frac{1}{s} \int_0^s \left[\sigma^{-1} D_j \tilde{Z}^x(r)\right]^{\top} dB(r) \right| \right] ds \nonumber \\
&\le \int_0^{T - t} \mbOne_{\{s < T - t'\}} 
\mathbb{E} \left[ \left| \tilde \cH(\tilde{Z}^x(s), v(s + t, \tilde{Z}^x(s))) 
- \tilde \cH(\tilde{Z}^{x'}(s), v(s + t', \tilde{Z}^{x'}(s))) \right|^2 \right]^{1/2} \nonumber \\
& \qquad \cdot\E\left[\left| \frac{1}{s} \int_0^s \left[\sigma^{-1} D_j \tilde{Z}^x(r)\right]^{\top} dB(r) \right|^2 \right]^{1/2} ds.
\label{eq:A2-bound0}
\end{align}
By Lemma \ref{lem:bound-DZ}, and using Itô isometry, 
\begin{equation} \label{eq:A2-bel-bound}
\mathbb{E} \left[ \left( \frac{1}{s} \int_0^s \left[\sigma^{-1} D_j \tilde{Z}^x(r)\right]^{\top} dB(r) \right)^2 \right]^{1/2}
\leq {\tilde C} s^{1/2},
\end{equation}
for some constant \( {\tilde C} \) independent of \( s \in (0,T) \) and $j \in \{1, 2, \cdots, d\}$.

For the term
\[
\mathbb{E} \left[ \left| \tilde \cH(\tilde{Z}^x(s), v(s + t, \tilde{Z}^x(s))) 
- \tilde \cH(\tilde{Z}^{x'}(s), v(s + t', \tilde{Z}^{x'}(s))) \right|^2 \right]^{1/2}, 
\]
consider \( (t^{(n)}, x^{(n)}) \to (t,x) \) with \( t^{(n)} \downarrow t \) 
and \( (t^{(n)}, x^{(n)}) \) in place of \( (t', x') \). Then, we upper bound the term 
\begin{align*}
\mbOne_{\{s \leq T - t^{(n)}\}} (T - t - s)^{1/2} \mathbb{E} \left[ \left|  
\tilde \cH(\tilde{Z}^x(s), v(s + t, \tilde{Z}^x(s))) 
- \tilde \cH(\tilde{Z}^{x^{(n)}}(s), v(s + t^{(n)}, \tilde{Z}^{x^{(n)}}(s))) \right|^2 \right]^{1/2}
\end{align*}
as follows. 
\begin{align}
& \quad \mbOne_{\{s \leq T - t^{(n)}\}} (T - t - s)^{1/2} \mathbb{E} \left[ \left|  
\tilde \cH(\tilde{Z}^x(s), v(s + t, \tilde{Z}^x(s))) 
- \tilde \cH(\tilde{Z}^{x^{(n)}}(s), v(s + t^{(n)}, \tilde{Z}^{x^{(n)}}(s))) \right|^2 \right]^{1/2} \nonumber \\
& \leq 2 \mbOne_{\{s \leq T - t^{(n)}\}} \cdot \mathbb{E} \left[ \left| (T - t - s)^{1/2} 
\tilde \cH(\tilde{Z}^x(s), v(s + t, \tilde{Z}^x(s))) \right. \right. \nonumber \\
&\left. \left. \qquad \qquad \qquad \quad - (T - t^{(n)} - s)^{1/2} \tilde \cH(\tilde{Z}^{x^{(n)}}(s), v(s + t^{(n)}, \tilde{Z}^{x^{(n)}}(s))) \right|^2 \right]^{1/2} \nonumber \\
& + 2  \mbOne_{\{s \leq T - t^{(n)}\}} \left[ (T - t - s)^{1/2} - (T - t^{(n)} - s)^{1/2} \right] 
\mathbb{E} \left[ \tilde \cH^2(\tilde{Z}^{x^{(n)}}(s), v(s + t^{(n)}, \tilde{Z}^{x^{(n)}}(s))) \right]^{1/2},
\label{eq:A2-bound}
\end{align}
where, to derive the last inequality, we used the fact that for two random variables $X_1$ and $X_2$, 
\[
\E\left[(X_1 + X_2)^2\right]^{1/2} \leq 2 \E\left[X_1^2\right]^{1/2} + 2 \E\left[X_2^2\right]^{1/2}.
\]
Next, consider the first term in Ineq. \eqref{eq:A2-bound}:
\begin{align*}
& 2 \mbOne_{\{s \leq T - t^{(n)}\}} \cdot \mathbb{E} \left[ \left| (T - t - s)^{1/2} 
\tilde \cH(\tilde{Z}^x(s), v(s + t, \tilde{Z}^x(s))) \right. \right. \\
&\left. \left. \qquad \qquad \qquad \quad - (T - t^{(n)} - s)^{1/2} \tilde \cH(\tilde{Z}^{x^{(n)}}(s), v(s + t^{(n)}, \tilde{Z}^{x^{(n)}}(s))) \right|^2 \right]^{1/2}.
\end{align*}
We have \( s + t^{(n)} \downarrow s + t \) for all \( s \in (0,T - t) \),  
and a.s. for all \( s \in (0,T - t) \),  
\[
\tilde{Z}^{x^{(n)}}(s) \to \tilde{Z}^x(s) \quad \text{ as } \quad n \to \infty.
\]
By the continuity of \( \tilde \cH \) and that of \( v \) on $[0, T)\times \R_{++}^d$, and by Lemmas \ref{lem:Z>0} and \ref{lem:Z-conv},  
we have that a.s., for all $s \in (0,T - t)$,
\[
(T - t^{(n)} - s)^{1/2} \tilde \cH(\tilde{Z}^{x^{(n)}}(s), v(s + t^{(n)}, \tilde{Z}^{x^{(n)}}(s))) \to
(T - t - s)^{1/2} \tilde \cH(\tilde{Z}^x(s), v(s + t, \tilde{Z}^x(s))).
\]
Furthermore, 
\begin{align*}
&\quad \ \mbOne_{\{s \leq T - t^{(n)}\}} \cdot \mathbb{E} \left[ \left| (T - t - s)^{1/2} 
\tilde \cH(\tilde{Z}^x(s), v(s + t, \tilde{Z}^x(s))) \right. \right. \\
&\left. \left. \qquad \qquad \qquad - (T - t^{(n)} - s)^{1/2} \tilde \cH(\tilde{Z}^{x^{(n)}}(s), v(s + t^{(n)}, \tilde{Z}^{x^{(n)}}(s))) \right|^2 \right]^{1/2} \\
&\leq 2 \mathbb{E} \left[ (T - t  - s) 
\tilde \cH^2(\tilde{Z}^x(s), v(s + t, \tilde{Z}^x(s))) \right]^{1/2} \\
&\quad + 2\mbOne_{\{s \leq T - t^{(n)}\}} \mathbb{E} \left[ (T - t^{(n)} - s) 
\tilde \cH^2(\tilde{Z}^{x^{(n)}}(s), v(s + t^{(n)}, \tilde{Z}^{x^{(n)}}(s)))\right]^{1/2}. 
\end{align*}
Using a similar reasoning as in the proof of Lemma \ref{lem:Fv-bounded}, we can show that 
\[
\mathbb{E} \left[ (T - t  - s) \tilde \cH^2(\tilde{Z}^x(s), v(s + t, \tilde{Z}^x(s))) \right]^{1/2} \leq C w(x), 
\]
and 
\[
\mbOne_{\{s \leq T - t^{(n)}\}} \mathbb{E} \left[ (T - t^{(n)} - s) 
\tilde \cH^2(\tilde{Z}^{x^{(n)}}(s), v(s + t^{(n)}, \tilde{Z}^{x^{(n)}}(s)))\right]^{1/2} \leq C w(x), 
\]
for some $C>0$ that does not depend on $s$ or $t$. Thus, the sequence
\begin{align*}
&\quad \ \mbOne_{\{s \leq T - t^{(n)}\}} \cdot \mathbb{E} \left[ \left| (T - t - s)^{1/2} 
\tilde \cH(\tilde{Z}^x(s), v(s + t, \tilde{Z}^x(s))) \right. \right. \\
&\left. \left. \qquad \qquad \qquad - (T - t^{(n)} - s)^{1/2} \tilde \cH(\tilde{Z}^{x^{(n)}}(s), v(s + t^{(n)}, \tilde{Z}^{x^{(n)}}(s))) \right|^2 \right]^{1/2}
\end{align*}
is uniformly integrable, so it converges to zero,  
uniformly in \( s \in (0, T - t) \).

We then consider the second term in Ineq. \eqref{eq:A2-bound}: 
\[
\mbOne_{\{s \leq T - t^{(n)}\}} \left[ (T - t - s)^{1/2} - (T - t^{(n)} - s)^{1/2} \right] 
\mathbb{E} \left[ \tilde \cH^2(\tilde{Z}^{x^{(n)}}(s), v(s + t^{(n)}, \tilde{Z}^{x^{(n)}}(s))) \right]^{1/2}.
\]
Again using a similar reasoning as in the proof of Lemma \ref{lem:Fv-bounded}, we can show that 
\begin{align}
&\quad \mbOne_{\{s \leq T - t^{(n)}\}} \left[ (T - t - s)^{1/2} - (T - t^{(n)} - s)^{1/2} \right] 
\mathbb{E} \left[ \tilde \cH^2(\tilde{Z}^{x^{(n)}}(s), v(s + t^{(n)}, \tilde{Z}^{x^{(n)}}(s))) \right]^{1/2} \nonumber \\
& \leq Cw(x) \left[ (T - t - s)^{1/2} - (T - t^{(n)} - s)^{1/2} \right] \left(1 + (T-t^{(n)}-s)^{1/2} \mbOne_{\{s \leq T - t^{(n)}\}} \right).
\label{eq:A2-bound2}
\end{align}
Therefore, by Ineqs. \eqref{eq:A2-bound0}, \eqref{eq:A2-bel-bound}, \eqref{eq:A2-bound} and \eqref{eq:A2-bound2}, 
\begin{align*}
A_2(t, t^{(n)}, x, x^{(n)}) 
&\leq \tilde C \int_0^{T - t} s^{-1/2} (T - t-s)^{-1/2} \\
& \quad \Bigg[2 \mbOne_{\{s \leq T - t^{(n)}\}} \cdot \mathbb{E} \left[ \left| (T - t - s)^{1/2} 
\tilde \cH(\tilde{Z}^x(s), v(s + t, \tilde{Z}^x(s))) \right. \right. \\
&\left. \left. \qquad \qquad \qquad \qquad - (T - t^{(n)} - s)^{1/2} \tilde \cH(\tilde{Z}^{x^{(n)}}(s), v(s + t^{(n)}, \tilde{Z}^{x^{(n)}}(s))) \right|^2 \right]^{1/2} \\
& \quad + 2 Cw(x)  \left[ (T - t - s)^{1/2} - (T - t^{(n)} - s)^{1/2} \right] \left(1 + (T-t^{(n)}-s)^{1/2} \mbOne_{\{s \leq T - t^{(n)}\}} \right) \Bigg] ds.
\end{align*}
Expanding bracket, and consider the first term  
\begin{align*}
& \tilde C \int_0^{T - t} s^{-1/2} (T - s)^{-1/2} \cdot \left\{ 2 \mbOne_{\{s < T - t^{(n)}\}} 
\mathbb{E} \left[ \left| (T - t - s)^{1/2} \tilde \cH(\tilde{Z}^x(s), v(s + t, \tilde{Z}^x(s))) \right. \right. \right. \\
&\left. \left. \left. \quad - (T - t^{(n)} - s)^{1/2} \tilde \cH(\tilde{Z}^{x^{(n)}}(s), v(s + t^{(n)}, \tilde{Z}^{x^{(n)}}(s))) \right|^2 \right] \right\}^{1/2} ds.
\end{align*}
We claim that it converges to zero as \( (x^{(n)}, t^{(n)}) \to (x, t) \),  with \( t^{(n)} \downarrow t \).
This is because the term
\begin{align*}
&\quad \ 2 \mbOne_{\{s \leq T - t^{(n)}\}} \cdot \mathbb{E} \left[ \left| (T - t - s)^{1/2} 
\tilde \cH(\tilde{Z}^x(s), v(s + t, \tilde{Z}^x(s))) \right. \right. \\
&\left. \left. \qquad \qquad \qquad - (T - t^{(n)} - s)^{1/2} \tilde \cH(\tilde{Z}^{x^{(n)}}(s), v(s + t^{(n)}, \tilde{Z}^{x^{(n)}}(s))) \right|^2 \right]^{1/2}
\end{align*}
converges to $0$ for all \( s \in (0, T - t) \), 
is dominated by $8Cw(x)$, with the integrand $s^{-1/2} (T - s)^{-1/2} w(x)$ being integrable. 
Thus, convergence follows by dominated convergence theorem. 

Next, consider the second term (up to a factor of $2\tilde C C w(x)$)
\begin{align*}
\int_0^{T - t} s^{-1/2} (T - t - s)^{-1/2} 
\left[ (T - t + s)^{1/2} - (T - t^{(n)} + s)^{1/2} \right] \cdot \left(1 + (T - s - t^{(n)})^{-1/2}) \mbOne_{\{s < T - t^{(n)}\}}\right) ds.
\end{align*}
We have 
\begin{align*}
& \quad \int_0^{T - t} s^{-1/2} (T - t - s)^{-1/2} 
\left[ (T - t + s)^{1/2} - (T - t^{(n)} + s)^{1/2} \right] \\
& \qquad \cdot \left(1 + (T - s - t^{(n)})^{-1/2}) \mbOne_{\{s < T - t^{(n)}\}}\right) ds \\
& =   \int_0^{T - t} s^{-1/2} (T - t - s)^{-1/2} \left[ (T - t + s)^{1/2} - (T - t^{(n)} + s)^{1/2} \right] ds \\
& \quad + \int_0^{T - t^{(n)}} s^{-1/2} (T - t - s)^{-1/2} (T - s - t^{(n)})^{-1/2}\left[ (T - t + s)^{1/2} - (T - t^{(n)} + s)^{1/2} \right] ds \\
& \le \left(t^{(n)} - t\right)^{1/2} \int_0^{T - t} s^{-1/2} (T - t - s)^{-1/2}ds \\
& \quad + \int_0^{T - t^{(n)}} s^{-1/2} \left[(T - s - t^{(n)})^{-1/2} - (T - t - s)^{-1/2}\right] ds, 
\end{align*} 
where, in the last inequality, we used the fact that 
\[
(T - t + s)^{1/2} - (T - t^{(n)} + s)^{1/2} \leq \left(t^{(n)} - t\right)^{1/2}. 
\]
The integral $\int_0^{T - t} s^{-1/2} (T - t - s)^{-1/2}ds$ is finite, so the term 
\[
\left(t^{(n)} - t\right)^{1/2} \int_0^{T - t} s^{-1/2} (T - t - s)^{-1/2}ds
\]
converges to zero as $n \to \infty$. The integral 
\[
\int_0^{T - t^{(n)}} s^{-1/2} \left[(T - s - t^{(n)})^{-1/2} - (T - t - s)^{-1/2}\right] ds
\]
is also finite, with the integrand $s^{-1/2} \left[(T - s - t^{(n)})^{-1/2} - (T - t - s)^{-1/2}\right]$ converging to zero for each $s \in (0, T-t)$, 
so the integral 
\[
\int_0^{T - t^{(n)}} s^{-1/2} \left[(T - s - t^{(n)})^{-1/2} - (T - t - s)^{-1/2}\right] ds \to 0.
\]
{In summary}, we have proved that as $(t^{(n)}, x^{(n)}) \to (t,x)$ with $t^{(n)} \downarrow t$, 
$A_2(t, t^{(n)}, x, x^{(n)})$ is upper bounded by the sum of two integrals, both of which converge to zero.  
Thus, $A_2(t, t^{(n)}, x, x^{(n)}) \to 0$ as well.

Finally, consider \( A_3(t, t^{(n)}, x, x^{(n)}) \).  
We have
\begin{align*}
A_3(t, t^{(n)}, x, x^{(n)}) 
&\leq \int_0^{T - t^{(n)}} \mathbb{E} \left[ \left| \tilde \cH(\tilde{Z}^{x^{(n)}}(s), v(s + t^{(n)}, \tilde{Z}^{x^{(n)}}(s))) \right|^2 \right]^{1/2} \\
&\quad \cdot \frac{1}{s} \mathbb{E} \left[ \left( \int_0^s \left[\sigma^{-1} \left( D_j \tilde{Z}^x(r) - D_j \tilde{Z}^{x^{(n)}}(r) \right)\right]^{\top} dB(r) \right)^2 \right]^{1/2} ds
\end{align*}
For $(t^{(n)}, x^{(n)})$ such that \( \|(t^{(n)}, x^{(n)}) - (t,x)\|_2 \leq 1 \), since $w(x^{(n)}) \leq \alpha w(x)$ for some constant $\alpha >0$, we can  
show that
\begin{align*}
\mathbb{E} \left[ \left| \tilde \cH(\tilde{Z}^{x^{(n)}}(s), v(s + t^{(n)}, \tilde{Z}^{x^{(n)}}(s))) \right|^2 \right]^{1/2} 
\leq \alpha \left( 1 + (T - t^{(n)} - s)^{-1/2} \right) w(x), \quad s \in (0, T - t^{(n)}),
\end{align*}
and
\begin{align*}
\mathbb{E} \left[ \left( \int_0^s \left[\sigma^{-1}(D_j \tilde{Z}^x(r) - D_j \tilde{Z}^{x^{(n)}}(r))\right]^{\top} dB(r) \right)^2 \right]^{1/2} \to 0
\end{align*}
as \( n \to \infty \), for all \( s \in (0, T - t^{(n)}) \), with $\alpha$ being redefined appropriately. 
Thus, the integrand
\[
\mathbb{E} \left[ \left| \tilde \cH(\tilde{Z}^{x^{(n)}}(s), v(s + t^{(n)}, \tilde{Z}^{x^{(n)}}(s))) \right|^2 \right]^{1/2}
\cdot \frac{1}{s} \mathbb{E} \left[ \left( \int_0^s \left[\sigma^{-1} \left( D \tilde{Z}^x(r) - D \tilde{Z}^{x^{(n)}}(r) \right)\right]^{\top}_j dB(r) \right)^2 \right]^{1/2}
\] 
goes to zero for \( s \in (0, T - t^{(n)}) \).

In addition,
\begin{align*}
\frac{1}{s} \cdot \mathbb{E} \left[ \left( \int_0^s \left[\sigma^{-1}(D_j \tilde{Z}^x(r) - D_j\tilde{Z}^{x^{(n)}}(r))\right] dB(r) \right)^2 \right]^{1/2}
\leq \hat C s^{-1/2} \quad \text{for some constant } \hat C, 
\end{align*}
so by the fact that the integral
\[
\int_0^{T - t^{(n)}} \left( 1 + (T - t^{(n)} - s)^{-1/2} \right) s^{-1/2} ds
\]
is finite, and by dominated convergence theorem, we can show that  
\[
A_3(t, t^{(n)}, x, x^{(n)}) \to 0.
\]

In conclusion, \( A_i(t, t^{(n)}, x, x^{(n)}) \to 0 \)  
as \( n \to \infty \), \( i = 1, 2, 3 \), where \( (t^{(n)}, x^{(n)}) \to (t,x) \)  
such that \( t^{(n)} \downarrow t \).
Similarly, we can show that  
\( A_i(t, t^{(n)}, x, x^{(n)}) \to 0 \) as \( (t^{(n)}, x^{(n)}) \to (t,x) \) such that \( t^{(n)} \uparrow t \), $i=1, 2, 3$.

Thus, we have established the continuity  
of $\zeta$. 
Combined with the continuity of \( \zeta^{(2)} \) and $\E\left[\kappa^{\top} D\tilde Y^{t,x}(T)\right]$,  
we have proved that \( F(v) \) is a continuous  
function on \( [0,T) \times \mathbb{R}_{++}^d \) whenever \( v \) is  
continuous on \( [0,T) \times \mathbb{R}_{++}^d \). 
\end{proof}

\section{Proof of Lemma \ref{lem:oblique-boundary}}\label{sec:appendix-oblique-boundary}
For the proof of Lemma \ref{lem:oblique-boundary}, we recall the definitions of $H_x$ and $R(x)$ from Section \ref{ssec:deriv-prob}, 
and some results from \cite{LipshutzRamanan2018}. 
\begin{lemma} [Lemma 8.1 in \cite{LipshutzRamanan2018}]\label{lem:decompose}
For each \( x \in \partial \mathbb{R}_+^d \), \( H_x \cap \text{span}[R(x)] = \{ 0 \} \).  
For each \( y \in \mathbb{R}^d \), there exists a unique pair $(v_y, w_y) \in H_x \times \text{span}[R(x)]$ such that 
\[
y = v_y + w_y.
\]
\end{lemma}
\begin{lemma}[Lemma 8.3 in \cite{LipshutzRamanan2018}]\label{lem:L_x}
Given \( x \in \partial \mathbb{R}_+^d \), 
there exists a unique mapping  
$\cL_x : \mathbb{R}^d \to H_x$ that satisfies $\cL_x(y) - y \in \text{span}[R(x)]$ for all $y \in \mathbb{R}^d$.
Furthermore, \( \cL_x \) is a linear map.
\end{lemma}
We can deduce the following result on the derivative process $D\tilde Z^x$ and pathwise derivatives.
\begin{lemma}\label{lem:deriv-proc}
Let \( x \in \partial \mathbb{R}_+^d \) and \( y \in \text{span}[R(x)] \). Then a.s. the derivative process $D\tilde Z^x$ 
evaluated in the direction $y$ is zero everywhere, i.e.,
\[
D \tilde Z^x(t; y) = 0 \text{ for all } t \geq 0.
\]
As a result, a.s. $\partial_y \tilde{Z}^x(t) = 0$ for almost all $t \geq 0$.
\end{lemma}
\begin{proof}
Since \( y \in \text{span}[R(x)] \), \( \cL_x(y) = 0 \). Thus, a.s. the derivative process \( D\tilde{Z}^x \) evaluated in the direction $y$ satisfies
\[
D\tilde{Z}^x(t; y) = \int_0^t D\tb(\tilde{Z}^x(s)) D\tilde{Z}^x(s; y) ds + \eta_y(t),
\]
where the process $\eta_y$ satisfies the properties of $\eta$ described in Definition \ref{df:deriv-prob}.
Because the zero process solves the derivative problem for the zero process, by Lipschitz continuity of \( \Lambda_{\tilde Z^x} \),
\[
\| D\tilde{Z}^x(\cdot; y) \|_t \leq C_\Lambda C_{\tb} \int_0^t \| D\tilde{Z}^x(\cdot; y) \|_s ds.
\]
By Grönwall's inequality, we have that a.s.
$D\tilde{Z}^x(t; y) = 0$ for all $t \geq 0$. Since $D\tilde{Z}^x(\cdot; y)$ is the right-continuous regularization of the pathwise derivative $\partial_y \tilde{Z}^x$, 
we also have that a.s. $\partial_y \tilde{Z}^x(t) = 0$ for almost all $t \geq 0$. This concludes the proof of the lemma.
\end{proof}
The following two results are useful consequences of Lemma \ref{lem:deriv-proc}.
\begin{corollary}\label{cor:deriv-boundary}
For all measurable \( \zeta : \mathbb{R}_+^d \to \R \) with polynomial growth and \( t > 0 \), if \( x \in \partial \mathbb{R}_+^d \), 
then for $j \in \cI(x)$,
\begin{equation}\label{eq:deriv-boundary}
\left\langle R_j, D_x \E\left[\zeta(\tilde{Z}^x(t))\right] \right\rangle = 0.
\end{equation}
\end{corollary}
\begin{proof}
Let $x\in \R_+^d$ and $j \in \cI(x)$.
By a density argument that is similar to the proof of Theorem \ref{thm:bel}, it suffices to establish Eq. \eqref{eq:deriv-boundary} for 
bounded continuously differentiable $\zeta : \R_+^d \to \R$ with bounded gradients. 
For such $\zeta$, by chain rule, boundedness of $\zeta$ and $D\zeta$, and boundedness of $D\tilde Z^x$, we have
\[
\left\langle R_j, D_x \E\left[\zeta(\tilde{Z}^x(t))\right] \right\rangle = \E\left[D\zeta(\tilde Z^x(t) D\tilde Z^x(t; R_j)\right] = 0.
\]
This concludes the proof of the corollary.
\end{proof}
\begin{corollary}\label{cor:DY-boundary}
Let \( x \in \partial \mathbb{R}_+^d \) and \( j \in \cI(x) \). Then a.s. the derivative process $D\tilde Y^x$ 
evaluated in the direction $R_j$ is $-e_j$ everywhere, i.e.,
\[
D \tilde Y^x(t; R_j) = -e_j \text{ for all } t \geq 0,
\]
where we recall that $e_j$ is the $j$th standard unit vector.
\end{corollary}
\begin{proof}
We have 
\[
D\tilde{Z}^x(t; R_j) = R_j + \int_0^t D\tb(\tilde{Z}^x(s)) D\tilde{Z}^x(s; R_j) ds + R\left(D\tilde Y^x(t; R_j)\right).
\]
By Lemma \ref{lem:deriv-proc},
\[
R\left(D\tilde Y^x(t; R_j)\right) = -R_j.
\]
Applying $R^{-1}$ to both sides of the equality gives us the desired result. 
\end{proof}

\begin{proof}[Proof of Lemma \ref{lem:oblique-boundary}]
Recall that $\tilde V$ is differentiable with $\tilde v$ being its gradient in $x \in \R_+^d$. 
Let $x\in \partial \R_+^d$ and $j \in \cI(x)$. 
From the form of $\tilde V$, and Corollary \ref{cor:deriv-boundary}, it suffices to prove that 
\begin{equation*}
\left\langle R_j, D_x \E\left[\int_t^T e^{-\beta(s-t)} \kappa^{\top} d\tilde Y^x (s)\right] \right\rangle = -\kappa_j.
\end{equation*}
By product rule, we have 
\[
\E\left[\int_t^T e^{-\beta(s-t)} \kappa^{\top} d\tilde Y^x (s)\right] = \E\left[e^{-\beta(T-t)} \kappa^{\top} \tilde Y^x (T) - \int_t^T \kappa^{\top} \tilde Y^x(s) de^{-\beta(s-t)}\right].
\]
Thus, 
\begin{align*}
&~\left\langle R_j, D_x \E\left[\int_t^T e^{-\beta(s-t)} \kappa^{\top} d\tilde Y^x (s)\right] \right\rangle \\
=&~\left\langle R_j, D_x \E\left[e^{-\beta(T-t)} \kappa^{\top} \tilde Y^x (T) - \int_t^T \kappa^{\top} \tilde Y^x(s) de^{-\beta(s-t)}\right] \right\rangle \\
=&~\E\left[e^{-\beta(T-t)} \kappa^{\top} D\tilde Y^x (T; R_j)\right] - \E\left[\int_t^T \kappa^{\top} D\tilde Y^x(s; R_j) de^{-\beta(s-t)}\right] \\
=&~ e^{-\beta(T-t)}(-\kappa_j) - (-\kappa_j) \left(e^{-\beta(T-t)} - 1\right) = -\kappa_j, 
\end{align*}
where the second last equality follows from Corollary \ref{cor:DY-boundary}. 
This concludes the proof of the lemma.
\end{proof}

\section{Proof Sketch of Proposition \ref{prop:viscosity-unique}}\label{app:viscosity-unique}
The proof of Proposition \ref{prop:viscosity-unique} follows 
the same line of reasoning as that of Theorems 3.6 and 3.7 in \cite{borkar2004ergodic}. 
First, we establish the uniqueness of viscosity solutions for 
an appropriately localized version of the HJB equation \eqref{eq:hjb1} -- \eqref{eq:hjb3}, 
subject to appropriate boundary conditions (this corresponds to the proof of Theorem 3.6 in \cite{borkar2004ergodic}). 
Then, we follow the proof of Theorem 3.7 there to 
establish uniqueness over the entire domain. 

We start with the definition of parabolic jets.
\begin{definition}\label{def:jets}
For an upper semi-continuous (USC) function $u : [0,T]\times \R_+^d \to \R$, define the parabolic superjet of $u$ at $(t,x) \in [0,T]\times \R_+^d$ by
\begin{align*}
\cP^{+(1,2)} u(t,x) := &~\Big\{ (q,p,X) \in \R \times \R^d \times \R^{d\times d} : u(s,y) \leq
u(t,x) + q(s-t) + \\
&~ + \langle p, y-x\rangle + \frac{1}{2} \langle X(y-x), y-x \rangle + o\left(|s-t| + |y-x|^2\right), (s,y) \in [0,T]\times \R_+^d\Big\}.
\end{align*}
Similarly, define the parabolic subjet of $u$ at $(t,x) \in [0,T]\times \R_+^d$ by
\begin{align*}
\cP^{-(1,2)} u(t,x) := &~\Big\{ (q,p,X) \in \R \times \R^d \times \R^{d\times d} : u(s,y) \geq
u(t,x) + q(s-t) + \\
&~ + \langle p, y-x\rangle + \frac{1}{2} \langle X(y-x), y-x \rangle + o\left(|s-t| + |y-x|^2\right), (s,y) \in [0,T]\times \R_+^d\Big\}.
\end{align*}
Define also the closure $c\cP^{\pm(1,2)} u(t,x)$ of $\cP^{\pm(1,2)} u(t,x)$ to be the set 
of all triples $(q,p,X) \in \R \times \R^d \times \R^{d\times d}$ such that
there are sequences $(t_n,x_n) \in [0,T]\times \R_+^d \to (t,x)$ 
and $(q_n,p_n,X_n) \in \cP^{\pm(1,2)} u(t_n,x_n)$ with $(q_n,p_n,X_n) \to (q,p,X)$.
\end{definition}
Next, we proceed to sketch the proof of the uniqueness of viscosity solutions for an appropriately localized version of the HJB equation \eqref{eq:hjb1} -- \eqref{eq:hjb3}. 
Specifically, let $B_n$ be the open ball of radius $n$ centered at origin, 
i.e., $B_n := \{x\in\R^d : \|x\|_2 < n\}$. Define $G_n := \R_+^d \cap B_n$. 
Also let $\xi_n : \left([0,T] \times \partial B_n\right)\cup \left(\{T\} \times G_n\right) \to \R$ be a continuous function so that $\xi_n(T,x) = \xi(x)$ for all $x \in \partial B_n$.
We consider the following parabolic PDE with mixed boundary conditions over the domain $[0,T]\times G_n$.
\begin{equation}\label{eq:hjb1-n}
  V_t(t,x) + \frac{1}{2}\mathrm{Tr}\left(\sigma\sigma^{\top} D_{xx}V(t,x)\right) + \cH\left(x, D_x V(t,x)\right) - \beta V(t,x) = 0, 
  \quad (t,x) \in [0,T) \times G_n,
\end{equation}
  with boundary conditions
\begin{equation}\label{eq:hjb2-n}
  V(t,x) = \xi_n(x), \quad \text{ if } (t, x) \in \left([0,T] \times \partial B_n \right) \cup \left(\{T\} \times G_n\right),
\end{equation}
  and 
\begin{equation}\label{eq:hjb3-n}
  \langle R_j, D_x V (t,x) \rangle = -\kappa_j, \quad \text{if } x_j = 0 \text{ and } t<T.
\end{equation}

\begin{lemma}\label{lem:unique-viscosity-n}
The parabolic PDE \eqref{eq:hjb1-n} -- \eqref{eq:hjb3-n} has a unique viscosity solution.
\end{lemma}

The following parabolic analog of Lemma 3.1 in \cite{DupuisIshii1991} 
(also Lemma A.1 in \cite{borkar2004ergodic}) is essential for the proof of Lemma \ref{lem:unique-viscosity-n}. 
Its proof is an application of the well-known Crandall-Ishii Lemma (Theorem 6.1 in Chapter V of \cite{FlemingSoner2006}; 
also Theorem 8.3 in \cite{CrandallIshiiLions1992}), 
using a construction that is similar to that used in the proof of Lemma 3.1 in \cite{DupuisIshii1990}\footnote{
We note a small error in the proof of Lemma 3.1 in \cite{DupuisIshii1990}, 
where the terms $2\lambda \langle \bar x-\bar y, x\rangle$ (which appeared twice) used 
in the constructions of $\tilde u$ and $\tilde v$ should be replaced by $2\lambda \langle \bar x-\bar y, x-\bar x\rangle$ 
and $2\lambda \langle \bar x-\bar y, x-\bar y\rangle$, respectively.}.
We omit details.
\begin{lemma}\label{lem:crandall-ishii-parabolic}
Let $\Omega$ be a compact domain in $\R_+^d$.
Suppose $u$ and $-v$ are USC functions such that for every $M>0$ there exists a constant $C(M)$ such that 
for any $(t,x) \in [0,T]\times \Omega$, 
\begin{equation}
  (q,p,X) \in c\cP^{+(1,2)} u(t,x) \text{ and } \|(t,x,p,X,u(t,x))\|_2 \leq M\implies q \geq -C(M);
\end{equation}
\begin{equation}
  (q,p,X) \in c\cP^{-(1,2)} v(t,x) \text{ and } \|(t,x,p,X,v(t,x))\|_2 \leq M \implies q \leq C(M).
\end{equation}
Define $w(t,x,y) := u(t,x) - v(t,y)$. Let $\alpha_1, \alpha_2 > 0$, $q \in \R$, $p_1, p_2 \in \R^d$, 
$\bar t \in (0,T)$ and $\bar x, \bar y \in \R_+^d$ satisfy 
\begin{equation}
\left(q, p_1, p_2, \alpha_1
\begin{pmatrix}
I & -I\\
-I & I
\end{pmatrix}
\;+\;
\alpha_2
\begin{pmatrix}
I & 0\\
0 & I
\end{pmatrix}
\right) \in \cP^{+(1,2)} w(\bar t, \bar x, \bar y).
\end{equation}
Then, there exist symmetric matrices $X_1, X_2 \in \R^{d\times d}$ and $q_1, q_2 \in \R$ with $q_1-q_2 = q$, such that 
\[
- 3\alpha_1
\begin{pmatrix}
I & 0\\
0 & I
\end{pmatrix}
\;\le\;
\begin{pmatrix}
X_1-2\alpha_2 I & 0\\
0 & X_2-2\alpha_2 I
\end{pmatrix}
\;\le\;
3\alpha_1
\begin{pmatrix}
I & -I\\
-I & I
\end{pmatrix},
\]
and 
\[
(u(\bar t, \bar x), q_1, p_1, X_1) \in c\cP^{+(1,2)} u(\bar t, \bar x), \ \text{ and } \
(v(\bar t, \bar y), q_2, -p_2, -X_2) \in c\cP^{-(1,2)} v(\bar t, \bar y).
\]
\end{lemma}
\begin{proof}[Proof sketch of Lemma \ref{lem:unique-viscosity-n}] 
The proof is similar to that of Theorem 3.6 in \cite{borkar2004ergodic}, so we only outline the steps 
and highlight main differences. As in \cite{borkar2004ergodic}, 
let $V_1(t,x)$ and $V_2(t,x)$ be viscosity super- and sub-solution of the PDE \eqref{eq:hjb1-n} -- \eqref{eq:hjb3-n}, respectively, 
with $V_1 \ge \xi_n \ge V_2$ on $\left([0,T]\times \partial B_n \right) \cup \left(\{T\} \times G_n\right)$. 
Define $V_{\alpha_1,\alpha_2}(t,x) := V_1(t,x) + \alpha_1 g(x) + \alpha_2$ 
and $U_{\alpha_1,\alpha_2}(t,x) := V_2(t,x) - \alpha_1 g(x) - \alpha_2$ for any given $\alpha_1, \alpha_2 > 0$. 
By way of contradiction, suppose that $V_2 \nleq V_1$ on $[0,T]\times G_n$. 
Then, it can be shown that for some $\alpha_1, \alpha_2 > 0$, $U_{\alpha_1,\alpha_2} \nleq V_{\alpha_1,\alpha_2}$ on $[0,T]\times G_n$, 
and, as in Theorem 2.1 of \cite{DupuisIshii1991}, that for appropriately modified 
$\tilde \Phi_*$ and $\tilde \Phi^*$ of $\Phi_*$ and $\Phi^*$ defined in Eqs. \eqref{eq:supersolution} and \eqref{eq:subsolution} respectively, 
\[
\tilde \Phi_*(t,x,U_{\alpha_1,\alpha_2}(t,x),q,p,X) \le 0 \text{ for } (t,x) \in [0,T)\times G_n \text{ and } (q,p,X) \in \cP^{+(1,2)} U_{\alpha_1,\alpha_2}(t,x),
\]
\[
\tilde \Phi^*(t,x,U_{\alpha_1,\alpha_2}(t,x),q,p,X) \ge 0 \text{ for } (t,x) \in [0,T)\times G_n \text{ and } (q,p,X) \in \cP^{-(1,2)} V_{\alpha_1,\alpha_2}(t,x).
\]
By standard maximum principle arguments, it can be shown that 
\[
\max_{(t,x) \in [0,T]\times \bar G_n} \{U_{\alpha_1,\alpha_2}(t,x) - V_{\alpha_1,\alpha_2}(t,x)\} 
\]
is achieved at some $(t^*, x^*) \in [0,T)\times \left(\partial \R_+^d \cap B_n\right)$, 
with 
\[\kappa_0 := U_{\alpha_1,\alpha_2}(t^*, x^*) - V_{\alpha_1,\alpha_2}(t^*, x^*) > 0.\] 
The rest of the proof follows the same steps as in the proof of Theorem 3.6 in \cite{borkar2004ergodic}, 
until Eq. (A.13) there, where we also write $U_{\alpha_1,\alpha_2}$ and $V_{\alpha_1,\alpha_2}$ as $u$ and $v$, respectively, and define instead
\[
\tilde u(t,x) := u(t,x) - \frac{\delta}{2} \left(\left\|x-x^*\right\|_2^2 + (t-t^*)^2\right),
\]
and for $\veps > 0$,
\[
\phi(t,x,y) := \tilde u(t,x) - v(t,y) - w_{\veps}(x,y).
\]
Let $\left(\bar t, \bar x, \bar y\right)$ a maximum point of $\phi$. Then, similar to the proof of (3.23) in \cite{DupuisIshii1991}, as $\veps \to 0$, we have
\[
\frac{\|\bar x-\bar y\|^2_2}{\varepsilon} \to 0; \ (\bar t, \bar x, \bar y) \to (t^*, x^*, x^*); \ 
\tilde u (\bar t, \bar x) \to u(t^*, x^*); \ \tilde v (\bar t, \bar y) \to v(t^*, x^*).
\]
Define $w(t,x,y) := \tilde u(t,x) - v(t,y)$, and use $p_1$ and $p_2$ to denote $p(\veps, \bar x, \bar y)$ and $q(\veps, \bar x, \bar y)$
in (A.12) of \cite{borkar2004ergodic}, respectively (this is to respect the notational convention used in this paper). 
Then, from (A.12) of \cite{borkar2004ergodic}, we can show that 
\[
\left(0, p_1, p_2,\ \frac{1}{\varepsilon}
\begin{pmatrix}
I & -I\\
-I & I
\end{pmatrix}
\;+\;
\frac{\lvert \bar{x}-\bar{y}\rvert^{2}}{\varepsilon}
\begin{pmatrix}
I & 0\\
0 & I
\end{pmatrix}
\right)
\in \cP^{+(1,2)} w(\bar t, \bar{x},\bar{y}).
\]
By Lemma \ref{lem:crandall-ishii-parabolic}, we can then find symmetric matrices $X_1, X_2 \in \R^{d\times d}$ and $q\in \R$ such that 
\[
- \frac{3}{\veps}
\begin{pmatrix}
I & 0\\
0 & I
\end{pmatrix}
\;\le\;
\begin{pmatrix}
X_1-\frac{2\lvert \bar{x}-\bar{y}\rvert^{2}}{\varepsilon} I & 0\\
0 & X_2-\frac{2\lvert \bar{x}-\bar{y}\rvert^{2}}{\varepsilon} I
\end{pmatrix}
\;\le\;
\frac{3}{\veps}
\begin{pmatrix}
I & -I\\
-I & I
\end{pmatrix},
\]
and 
\[
(\tilde u(\bar t, \bar x), q, p_1, X_1) \in c\cP^{+(1,2)} \tilde u(\bar t, \bar x), \ \text{ and } \
(v(\bar t, \bar y), q, -p_2, -X_2) \in c\cP^{-(1,2)} v(\bar t, \bar y).
\]
The rest of the proof is essentially the same as the correponding parts in the proof of Theorem 3.6 in \cite{borkar2004ergodic}. We omit details.
\end{proof}
\begin{proof}[Proof sketch of Proposition \ref{prop:viscosity-unique}]
Given Lemma \ref{lem:unique-viscosity-n}, 
The rest of the proof of Proposition \ref{prop:viscosity-unique} is a standard bootstrapping argument similar to the proof of Theorem 3.7 in \cite{borkar2004ergodic}. 
Define $\tau_n = \tau_n(x)$ as in the proof of Theorem 3.7 in \cite{borkar2004ergodic}. 
For any two viscosity solutions $V_1$ and $V_2$ of the HJB equation \eqref{eq:hjb1} -- \eqref{eq:hjb3}, for any $(t,x) \in [0,T)\times \R_+^d$,
we can obtain the following inequality: 
\[
|V_1(t,x) - V_2(t,x)| \le \tilde \alpha_w w(n) \sup_{a(\cdot)} \pr\left(\tau_n(x) < T\right), 
\]
for all sufficiently large $n$ and some $\tilde \alpha_w > 0$. The fact that $|V_1(t,x) - V_2(t,x)| = 0$ can be obtained by letting $n \to \infty$, 
and noting that $\pr\left(\tau_n(x) < T\right)$ converges to $0$ exponentially fast in $n^2$, outpacing the polynomial growth of $w(n)$. 
\end{proof}

\section{Additional Numerical Results for \texorpdfstring{$N$}{N}-network with \texorpdfstring{$h_1 = h_2 = 1.0$}{h1 = h2 = 1.0}}\label{ssec:results-tables}

All results reported here are for the cost configuration $h_1 = h_2 = 1.0$.
The following table reports the estimated value function (mean) and standard error for the diagonal initial states $x=(u,u)$, $u\in\{0.0,0.1,\dots,1.0\}$, under the least control policy.
\begin{longtable}{@{}rcc@{}}
\toprule
\textbf{Initial state $x$} & \textbf{Mean $V$} & \textbf{Std. error} \\
\midrule
\endfirsthead
\toprule
\textbf{Initial state $x$} & \textbf{Mean $V$} & \textbf{Std. error} \\
\midrule
\endhead
$(0.0,\,0.0)$ & 0.06710860936627899 & 0.0003174430623739718 \\
$(0.1,\,0.1)$ & 0.07286877607579945 & 0.0003463517851296053 \\
$(0.2,\,0.2)$ & 0.08732756692792475 & 0.00041100406091468646 \\
$(0.3,\,0.3)$ & 0.1096587344315208 & 0.0004817627558674379 \\
$(0.4,\,0.4)$ & 0.13715314324440586 & 0.0005569592881856829 \\
$(0.5,\,0.5)$ & 0.17090623288993287 & 0.0006135801670639973 \\
$(0.6,\,0.6)$ & 0.2067179307852707 & 0.0006629498266768171 \\
$(0.7,\,0.7)$ & 0.24426238402885178 & 0.0006847332427289391 \\
$(0.8,\,0.8)$ & 0.2825880553629785 & 0.0006982786724378977 \\
$(0.9,\,0.9)$ & 0.3200067843584761 & 0.0007188021349059385 \\
$(1.0,\,1.0)$ & 0.36058317902988024 & 0.0007177691212418337 \\
\bottomrule
\end{longtable}

For each diagonal initial state $x=(u,u)$, the following tables report, by level $n$ and sample size $M$, the value function estimate $V$ formatted as mean $V$ (\(\pm\) std\%), for three drift upper bounds $C_{\mathcal A}\in\{1,2,5\}$.

\begin{longtable}{@{}r r c c c@{}}
\caption{$x=(0.0,0.0)$}\label{tab:mlp-x00}\\
\toprule
\textbf{Level} & \textbf{$M$} & \textbf{$C_{\mathcal A}=1$} & \textbf{$C_{\mathcal A}=2$} & \textbf{$C_{\mathcal A}=5$} \\
\midrule
\endfirsthead
\toprule
\textbf{Level} & \textbf{$M$} & \textbf{$C_{\mathcal A}=1$} & \textbf{$C_{\mathcal A}=2$} & \textbf{$C_{\mathcal A}=5$} \\
\midrule
\endhead
1 & 196608 & 0.077251 (\(\pm\) 0.25\%) & 0.077103 (\(\pm\) 0.30\%) & 0.077155 (\(\pm\) 0.23\%) \\
2 & 768    & 0.074330 (\(\pm\) 0.19\%) & 0.071667 (\(\pm\) 0.28\%) & 0.063457 (\(\pm\) 1.93\%) \\
3 & 192    & 0.074984 (\(\pm\) 0.17\%) & 0.074045 (\(\pm\) 0.66\%) & 0.073954 (\(\pm\) 1.66\%) \\
4 & 60     & 0.074689 (\(\pm\) 0.30\%) & 0.073104 (\(\pm\) 0.44\%) & 0.069430 (\(\pm\) 2.62\%) \\
5 & 48     & 0.074775 (\(\pm\) 0.11\%) & 0.073355 (\(\pm\) 0.58\%) & 0.070788 (\(\pm\) 6.07\%) \\
\bottomrule
\end{longtable}

\begin{longtable}{@{}r r c c c@{}}
\caption{$x=(0.1,0.1)$}\label{tab:mlp-x01}\\
\toprule
\textbf{Level} & \textbf{$M$} & \textbf{$C_{\mathcal A}=1$} & \textbf{$C_{\mathcal A}=2$} & \textbf{$C_{\mathcal A}=5$} \\
\midrule
\endfirsthead
\toprule
\textbf{Level} & \textbf{$M$} & \textbf{$C_{\mathcal A}=1$} & \textbf{$C_{\mathcal A}=2$} & \textbf{$C_{\mathcal A}=5$} \\
\midrule
\endhead
1 & 196608 & 0.082100 (\(\pm\) 0.08\%) & 0.082224 (\(\pm\) 0.18\%) & 0.082105 (\(\pm\) 0.30\%) \\
2 & 768    & 0.079257 (\(\pm\) 0.24\%) & 0.076384 (\(\pm\) 0.45\%) & 0.067843 (\(\pm\) 1.63\%) \\
3 & 192    & 0.079934 (\(\pm\) 0.26\%) & 0.079108 (\(\pm\) 0.20\%) & 0.078292 (\(\pm\) 2.38\%) \\
4 & 60     & 0.079676 (\(\pm\) 0.27\%) & 0.077607 (\(\pm\) 0.24\%) & 0.075803 (\(\pm\) 2.36\%) \\
5 & 48     & 0.079748 (\(\pm\) 0.18\%) & 0.078134 (\(\pm\) 0.55\%) & 0.074847 (\(\pm\) 3.39\%) \\
\bottomrule
\end{longtable}

\begin{longtable}{@{}r r c c c@{}}
\caption{$x=(0.2,0.2)$}\label{tab:mlp-x02}\\
\toprule
\textbf{Level} & \textbf{$M$} & \textbf{$C_{\mathcal A}=1$} & \textbf{$C_{\mathcal A}=2$} & \textbf{$C_{\mathcal A}=5$} \\
\midrule
\endfirsthead
\toprule
\textbf{Level} & \textbf{$M$} & \textbf{$C_{\mathcal A}=1$} & \textbf{$C_{\mathcal A}=2$} & \textbf{$C_{\mathcal A}=5$} \\
\midrule
\endhead
1 & 196608 & 0.096169 (\(\pm\) 0.22\%) & 0.095970 (\(\pm\) 0.31\%) & 0.096131 (\(\pm\) 0.27\%) \\
2 & 768    & 0.092959 (\(\pm\) 0.14\%) & 0.090153 (\(\pm\) 0.46\%) & 0.080232 (\(\pm\) 0.73\%) \\
3 & 192    & 0.093992 (\(\pm\) 0.36\%) & 0.092531 (\(\pm\) 0.35\%) & 0.092654 (\(\pm\) 2.19\%) \\
4 & 60     & 0.093628 (\(\pm\) 0.21\%) & 0.091610 (\(\pm\) 0.26\%) & 0.089000 (\(\pm\) 3.46\%) \\
5 & 48     & 0.093623 (\(\pm\) 0.11\%) & 0.091832 (\(\pm\) 0.30\%) & 0.088047 (\(\pm\) 2.25\%) \\
\bottomrule
\end{longtable}

\begin{longtable}{@{}r r c c c@{}}
\caption{$x=(0.3,0.3)$}\label{tab:mlp-x03}\\
\toprule
\textbf{Level} & \textbf{$M$} & \textbf{$C_{\mathcal A}=1$} & \textbf{$C_{\mathcal A}=2$} & \textbf{$C_{\mathcal A}=5$} \\
\midrule
\endfirsthead
\toprule
\textbf{Level} & \textbf{$M$} & \textbf{$C_{\mathcal A}=1$} & \textbf{$C_{\mathcal A}=2$} & \textbf{$C_{\mathcal A}=5$} \\
\midrule
\endhead
1 & 196608 & 0.117271 (\(\pm\) 0.16\%) & 0.117336 (\(\pm\) 0.20\%) & 0.117267 (\(\pm\) 0.21\%) \\
2 & 768    & 0.114306 (\(\pm\) 0.15\%) & 0.111133 (\(\pm\) 0.58\%) & 0.102496 (\(\pm\) 0.33\%) \\
3 & 192    & 0.115404 (\(\pm\) 0.20\%) & 0.113656 (\(\pm\) 0.43\%) & 0.112090 (\(\pm\) 1.04\%) \\
4 & 60     & 0.114777 (\(\pm\) 0.13\%) & 0.112800 (\(\pm\) 0.42\%) & 0.109656 (\(\pm\) 1.82\%) \\
5 & 48     & 0.114781 (\(\pm\) 0.07\%) & 0.112981 (\(\pm\) 0.51\%) & 0.111153 (\(\pm\) 3.09\%) \\
\bottomrule
\end{longtable}

\begin{longtable}{@{}r r c c c@{}}
\caption{$x=(0.4,0.4)$}\label{tab:mlp-x04}\\
\toprule
\textbf{Level} & \textbf{$M$} & \textbf{$C_{\mathcal A}=1$} & \textbf{$C_{\mathcal A}=2$} & \textbf{$C_{\mathcal A}=5$} \\
\midrule
\endfirsthead
\toprule
\textbf{Level} & \textbf{$M$} & \textbf{$C_{\mathcal A}=1$} & \textbf{$C_{\mathcal A}=2$} & \textbf{$C_{\mathcal A}=5$} \\
\midrule
\endhead
1 & 196608 & 0.144095 (\(\pm\) 0.16\%) & 0.144050 (\(\pm\) 0.20\%) & 0.143981 (\(\pm\) 0.17\%) \\
2 & 768    & 0.141176 (\(\pm\) 0.15\%) & 0.138000 (\(\pm\) 0.12\%) & 0.128984 (\(\pm\) 0.83\%) \\
3 & 192    & 0.142151 (\(\pm\) 0.22\%) & 0.141356 (\(\pm\) 0.28\%) & 0.140676 (\(\pm\) 1.33\%) \\
4 & 60     & 0.141834 (\(\pm\) 0.19\%) & 0.139630 (\(\pm\) 0.56\%) & 0.137248 (\(\pm\) 2.43\%) \\
5 & 48     & 0.141933 (\(\pm\) 0.13\%) & 0.140871 (\(\pm\) 0.22\%) & 0.137891 (\(\pm\) 0.98\%) \\
\bottomrule
\end{longtable}

\begin{longtable}{@{}r r c c c@{}}
\caption{$x=(0.5,0.5)$}\label{tab:mlp-x05}\\
\toprule
\textbf{Level} & \textbf{$M$} & \textbf{$C_{\mathcal A}=1$} & \textbf{$C_{\mathcal A}=2$} & \textbf{$C_{\mathcal A}=5$} \\
\midrule
\endfirsthead
\toprule
\textbf{Level} & \textbf{$M$} & \textbf{$C_{\mathcal A}=1$} & \textbf{$C_{\mathcal A}=2$} & \textbf{$C_{\mathcal A}=5$} \\
\midrule
\endhead
1 & 196608 & 0.174892 (\(\pm\) 0.16\%) & 0.175164 (\(\pm\) 0.10\%) & 0.174823 (\(\pm\) 0.09\%) \\
2 & 768    & 0.172568 (\(\pm\) 0.16\%) & 0.169768 (\(\pm\) 0.46\%) & 0.161343 (\(\pm\) 0.69\%) \\
3 & 192    & 0.173227 (\(\pm\) 0.22\%) & 0.172325 (\(\pm\) 0.26\%) & 0.171178 (\(\pm\) 1.16\%) \\
4 & 60     & 0.173107 (\(\pm\) 0.07\%) & 0.171748 (\(\pm\) 0.24\%) & 0.168328 (\(\pm\) 0.63\%) \\
5 & 48     & 0.173203 (\(\pm\) 0.14\%) & 0.171984 (\(\pm\) 0.20\%) & 0.170499 (\(\pm\) 1.23\%) \\
\bottomrule
\end{longtable}

\begin{longtable}{@{}r r c c c@{}}
\caption{$x=(0.6,0.6)$}\label{tab:mlp-x06}\\
\toprule
\textbf{Level} & \textbf{$M$} & \textbf{$C_{\mathcal A}=1$} & \textbf{$C_{\mathcal A}=2$} & \textbf{$C_{\mathcal A}=5$} \\
\midrule
\endfirsthead
\toprule
\textbf{Level} & \textbf{$M$} & \textbf{$C_{\mathcal A}=1$} & \textbf{$C_{\mathcal A}=2$} & \textbf{$C_{\mathcal A}=5$} \\
\midrule
\endhead
1 & 196608 & 0.209231 (\(\pm\) 0.22\%) & 0.208997 (\(\pm\) 0.10\%) & 0.209174 (\(\pm\) 0.09\%) \\
2 & 768    & 0.206685 (\(\pm\) 0.07\%) & 0.204191 (\(\pm\) 0.14\%) & 0.196985 (\(\pm\) 0.30\%) \\
3 & 192    & 0.207689 (\(\pm\) 0.19\%) & 0.207117 (\(\pm\) 0.18\%) & 0.206461 (\(\pm\) 1.76\%) \\
4 & 60     & 0.207649 (\(\pm\) 0.09\%) & 0.206594 (\(\pm\) 0.38\%) & 0.201241 (\(\pm\) 0.89\%) \\
5 & 48     & 0.207703 (\(\pm\) 0.06\%) & 0.206914 (\(\pm\) 0.28\%) & 0.204436 (\(\pm\) 1.44\%) \\
\bottomrule
\end{longtable}

\begin{longtable}{@{}r r c c c@{}}
\caption{$x=(0.7,0.7)$}\label{tab:mlp-x07}\\
\toprule
\textbf{Level} & \textbf{$M$} & \textbf{$C_{\mathcal A}=1$} & \textbf{$C_{\mathcal A}=2$} & \textbf{$C_{\mathcal A}=5$} \\
\midrule
\endfirsthead
\toprule
\textbf{Level} & \textbf{$M$} & \textbf{$C_{\mathcal A}=1$} & \textbf{$C_{\mathcal A}=2$} & \textbf{$C_{\mathcal A}=5$} \\
\midrule
\endhead
1 & 196608 & 0.245402 (\(\pm\) 0.11\%) & 0.245331 (\(\pm\) 0.18\%) & 0.245300 (\(\pm\) 0.12\%) \\
2 & 768    & 0.243227 (\(\pm\) 0.08\%) & 0.241008 (\(\pm\) 0.08\%) & 0.233330 (\(\pm\) 0.41\%) \\
3 & 192    & 0.244489 (\(\pm\) 0.20\%) & 0.244027 (\(\pm\) 0.24\%) & 0.241501 (\(\pm\) 0.86\%) \\
4 & 60     & 0.244266 (\(\pm\) 0.07\%) & 0.243162 (\(\pm\) 0.11\%) & 0.238534 (\(\pm\) 1.12\%) \\
5 & 48     & 0.244357 (\(\pm\) 0.09\%) & 0.243187 (\(\pm\) 0.29\%) & 0.240461 (\(\pm\) 0.89\%) \\
\bottomrule
\end{longtable}

\begin{longtable}{@{}r r c c c@{}}
\caption{$x=(0.8,0.8)$}\label{tab:mlp-x08}\\
\toprule
\textbf{Level} & \textbf{$M$} & \textbf{$C_{\mathcal A}=1$} & \textbf{$C_{\mathcal A}=2$} & \textbf{$C_{\mathcal A}=5$} \\
\midrule
\endfirsthead
\toprule
\textbf{Level} & \textbf{$M$} & \textbf{$C_{\mathcal A}=1$} & \textbf{$C_{\mathcal A}=2$} & \textbf{$C_{\mathcal A}=5$} \\
\midrule
\endhead
1 & 196608 & 0.283229 (\(\pm\) 0.13\%) & 0.283432 (\(\pm\) 0.12\%) & 0.283068 (\(\pm\) 0.19\%) \\
2 & 768    & 0.280910 (\(\pm\) 0.06\%) & 0.278621 (\(\pm\) 0.05\%) & 0.271616 (\(\pm\) 0.18\%) \\
3 & 192    & 0.282855 (\(\pm\) 0.17\%) & 0.282250 (\(\pm\) 0.53\%) & 0.277051 (\(\pm\) 0.85\%) \\
4 & 60     & 0.282225 (\(\pm\) 0.13\%) & 0.280553 (\(\pm\) 0.22\%) & 0.273290 (\(\pm\) 0.72\%) \\
5 & 48     & 0.282219 (\(\pm\) 0.04\%) & 0.281848 (\(\pm\) 0.32\%) & 0.277789 (\(\pm\) 1.65\%) \\
\bottomrule
\end{longtable}

\begin{longtable}{@{}r r c c c@{}}
\caption{$x=(0.9,0.9)$}\label{tab:mlp-x09}\\
\toprule
\textbf{Level} & \textbf{$M$} & \textbf{$C_{\mathcal A}=1$} & \textbf{$C_{\mathcal A}=2$} & \textbf{$C_{\mathcal A}=5$} \\
\midrule
\endfirsthead
\toprule
\textbf{Level} & \textbf{$M$} & \textbf{$C_{\mathcal A}=1$} & \textbf{$C_{\mathcal A}=2$} & \textbf{$C_{\mathcal A}=5$} \\
\midrule
\endhead
1 & 196608 & 0.321761 (\(\pm\) 0.17\%) & 0.321827 (\(\pm\) 0.11\%) & 0.321313 (\(\pm\) 0.21\%) \\
2 & 768    & 0.319409 (\(\pm\) 0.12\%) & 0.316904 (\(\pm\) 0.21\%) & 0.310259 (\(\pm\) 0.23\%) \\
3 & 192    & 0.321317 (\(\pm\) 0.22\%) & 0.320620 (\(\pm\) 0.27\%) & 0.318565 (\(\pm\) 1.02\%) \\
4 & 60     & 0.320732 (\(\pm\) 0.06\%) & 0.319946 (\(\pm\) 0.51\%) & 0.309771 (\(\pm\) 1.07\%) \\
5 & 48     & 0.321070 (\(\pm\) 0.03\%) & 0.320558 (\(\pm\) 0.28\%) & 0.316143 (\(\pm\) 0.59\%) \\
\bottomrule
\end{longtable}

\begin{longtable}{@{}r r c c c@{}}
\caption{$x=(1.0,1.0)$}\label{tab:mlp-x10}\\
\toprule
\textbf{Level} & \textbf{$M$} & \textbf{$C_{\mathcal A}=1$} & \textbf{$C_{\mathcal A}=2$} & \textbf{$C_{\mathcal A}=5$} \\
\midrule
\endfirsthead
\toprule
\textbf{Level} & \textbf{$M$} & \textbf{$C_{\mathcal A}=1$} & \textbf{$C_{\mathcal A}=2$} & \textbf{$C_{\mathcal A}=5$} \\
\midrule
\endhead
1 & 196608 & 0.360660 (\(\pm\) 0.25\%) & 0.360843 (\(\pm\) 0.12\%) & 0.360900 (\(\pm\) 0.15\%) \\
2 & 768    & 0.358493 (\(\pm\) 0.07\%) & 0.355827 (\(\pm\) 0.13\%) & 0.348379 (\(\pm\) 0.33\%) \\
3 & 192    & 0.360270 (\(\pm\) 0.20\%) & 0.359344 (\(\pm\) 0.38\%) & 0.360107 (\(\pm\) 1.02\%) \\
4 & 60     & 0.359885 (\(\pm\) 0.11\%) & 0.358296 (\(\pm\) 0.33\%) & 0.348312 (\(\pm\) 1.27\%) \\
5 & 48     & 0.360570 (\(\pm\) 0.08\%) & 0.359739 (\(\pm\) 0.17\%) & 0.353010 (\(\pm\) 0.99\%) \\
\bottomrule
\end{longtable}
\end{appendices}

\end{document}